\documentclass[10pt]{amsart}
\input{preamble.tex}
\usetikzlibrary{patterns,snakes}

\usepackage{etoolbox}
\usepackage{scalerel}
\usepackage{arydshln}
\tracingpatches
\makeatletter
\patchcmd{\@setref}{\bfseries ??}{\bfseries\color{red} OWE A COFFEE/BEER}{}{}
\makeatother

\usepackage{stackengine}
\newcommand\ostar{\stackMath\mathbin{\stackinset{c}{0ex}{c}{0ex}{\star}{\odot}}}

\newcommand{\TFT}{\mathcal{Z}}
\newcommand{\BNA}{\bn}
\newcommand{\gpm}{\Gamma_{\!\pm}}

\newcommand{\diff}{\mathrm{Diff}^+}

\newcommand{\PB}{\mathbb{P}}
\newcommand{\Dst}{\mathbb{D}}
\newcommand{\Ast}{\mathbb{A}}
\renewcommand{\MM}{\EuScript{F}}
\newcommand{\PMS}{C}
\newcommand{\PMStwo}{C'}
\newcommand{\sCat}{\CS}
\newcommand{\std}{\mathrm{std}}
\newcommand{\nn}{\mathbf{n}}
\newcommand{\mm}{\mathbf{m}}
\newcommand{\LO}{\operatorname{LinR}}
\newcommand{\supp}{\operatorname{supp}}
\newcommand{\TG}{\tilde{\Gamma}}

\newcommand{\leftangle}{\begin{tikzpicture}[anchorbase,scale=.15]
	\draw
   (0,1) to (0,0) \lu(-1,1);
	\end{tikzpicture}}
	\newcommand{\rightangle}{\begin{tikzpicture}[anchorbase,scale=-.15]
		\draw[white] 
		(0.2,0) to (0,0);
		\draw 
	(0,1) to (0,0) \lu(-1,1);
		\end{tikzpicture}}

\begin{document}

\author{Matthew Hogancamp}
\address{Department of Mathematics, Northeastern University, 360 Huntington Ave, Boston,
MA 02115, USA}
\email{m.hogancamp@northeastern.edu}

\author{David~E.~V.~Rose}
\address{Department of Mathematics, University of North Carolina, 
Phillips Hall, CB \#3250, UNC-CH, 
Chapel Hill, NC 27599-3250, USA
\href{https://davidev.web.unc.edu/}{davidev.web.unc.edu}}
\email{davidrose@unc.edu}

\author{Paul Wedrich}
\address{Fachbereich Mathematik, Universit\"at Hamburg, 
Bundesstra{\ss}e 55, 
20146 Hamburg, Germany
\href{http://paul.wedrich.at}{paul.wedrich.at}}
\email{paul.wedrich@uni-hamburg.de}

\title{Bordered invariants from Khovanov homology}

\begin{abstract} 
To every compact oriented surface that is composed entirely out of
$2$-dimensional $0$- and $1$-handles, we construct a dg category using
structures arising in Khovanov homology. These dg categories form part of the
$2$-dimensional layer (a.k.a.~modular functor) of a categorified version of the
$\mathfrak{sl}(2)$ Turaev--Viro topological field theory. As a byproduct, we
obtain a unified perspective on several hitherto disparate constructions in
categorified quantum topology, including the Rozansky--Willis invariants,
Asaeda--Przytycki--Sikora homologies for links in thickened surfaces,
categorified Jones--Wenzl projectors and associated spin networks, and dg
horizontal traces.
\end{abstract}

\maketitle

\setcounter{tocdepth}{1}

\tableofcontents

\section{Introduction}
	\label{s:intro}
Since its inception \cite{MR1295461}, one of the central aims of the
categorification program has been the construction of 4-dimensional analogues of
the Reshetikhin--Turaev \cite{MR1091619} and Turaev--Viro \cite{TuraevViro}
topological field theories (TFTs). While the original vision involved
building such theories using the notion of ``Hopf category'', Khovanov's
categorification of the Jones polynomial \cite{Kho} opened the door for an
approach based on structures arising from link homology theories. 

Developments in this direction include categorifications of the Kauffman bracket
skein theory on (thickened) surfaces \cite{MR2113902,MR2475122} and the
invariants of smooth $4$-manifolds from \cite{2019arXiv190712194M}. While such
constructions have proven surprisingly powerful (e.g.~detecting exotic smooth
structure \cite{ren2024khovanov}), the full potential of TFTs based on link
homology theories can only be harnessed by working in a homotopy coherent
setting. This can be seen, for example, in the context of the dg horizontal
trace \S \ref{sss:dgtr}.

In this paper, the homotopy coherent perspective enables us to take the first
steps toward a categorification of the $\mathfrak{sl}_2$ Turaev--Viro TFT, by
providing dg categorified skein invariants of surfaces $\Sigma$ with non-empty
boundary. Our approach utilizes a novel perspective on various structures
arising in the Khovanov theory, and recovers (and gives a fresh perspective on)
the aforementioned dg horizontal trace in the case that $\Sigma$ is an annulus.
As evidence for the relation with the Turaev--Viro theory, we find that the
$\Hom$-pairing on our dg categories yields a categorified version of the
hermitian inner product on the associated TFT Hilbert spaces.

\subsection{The category associated to a surface}
\label{ss:intro summary}
We now proceed with a summary of our main construction.  
Let $\Sigma$ be a compact, connected, oriented surface with $\del \Sigma \neq \emptyset$. 
Fix a collection $\Pi$ of boundary intervals 
and a set of points $\pp=\{p_1,\ldots,p_{2r}\}\subset \partial\Sigma$ 
contained in the union of those intervals (we allow $\pp=\emptyset$). 
Choose also a finite set $\Gamma$ of properly embedded arcs 
(called \emph{seams}) 
that do not meet the boundary intervals and which cut 
$\Sigma$ into a disjoint union of disks. 
An example is provided on the left side
of Figure~\ref{fig:surfaceandtangle} below; 
here the boundary arcs are depicted in blue and the seams in purple.

\begin{figure}[ht]
	\[
	\begin{tikzpicture}[anchorbase,scale=.7]
		\draw[line width=.8mm, seam] (-1.05,-1.05) to (-.37,-.37);
		\draw[line width=.8mm, seam] (3.55,-1.05) to (2.87,-.37);
		\draw[line width=.5mm, gray, fill=gray, fill opacity=.2] 
			(-1.5,0) \ur (0,1.5) \pr (1.25,1.25) \pr (2.5,1.5) \rd (4,0) \dl (2.5,-1.5) \pl (1.25,-1.25) \pl (0,-1.5) \lu (-1.5,0);
		\draw[line width=.5mm, gray, fill=white] (0,0) circle (.5);
		\draw[line width=.5mm, gray, fill=white] (2.5,0) circle (.5);
		\draw[line width=.5mm, blue] (0,1.5) \pr (1.25,1.25) \pr (2.5,1.5);
		\draw[line width=.5mm, blue] (0,-1.5) \pr (1.25,-1.25) \pr (2.5,-1.5);
		\draw[line width=.5mm, blue] (0,0.5) \rd (0.5,0) \dl (0,-0.5);
		\draw[line width=.5mm, blue] (2.5,0.5) \ld (2,0) \dr (2.5,-0.5);
		\end{tikzpicture}
		\quad,\quad
	\begin{tikzpicture}[anchorbase,scale=.7]
		\draw[line width=.8mm, seam] (-1.05,-1.05) to (-.37,-.37);
		\draw[line width=.8mm, seam] (3.55,-1.05) to (2.87,-.37);
		\draw[thick,double] (0.5,0) to (2,0);
		\draw[line width=.5mm, gray] (0,0) circle (.5);
		\draw[line width=.5mm, gray] (0,0) circle (1.5);
		\draw[line width=.5mm, gray] (2.5,0) circle (.5);
		\draw[line width=.5mm, gray] (2.5,0) circle (1.5);
		\fill[white] (1,-1.5) rectangle (1.5,1.5);
		\fill[white] (0,1) rectangle (2.5,1.6);
		\fill[white] (0,-1) rectangle (2.5,-1.6);
		\draw[thick,double] (1.25,-1.25) to (1.25,1.25);
		\draw[line width=.5mm, gray, fill=gray, fill opacity=.2] 
			(-1.5,0) \ur (0,1.5) \pr (1.25,1.25) \pr (2.5,1.5) \rd (4,0) \dl (2.5,-1.5) \pl (1.25,-1.25) \pl (0,-1.5) \lu (-1.5,0);
		\draw[line width=.5mm, gray, fill=white] (0,0) circle (.5);
		\draw[line width=.5mm, gray, fill=white] (2.5,0) circle (.5);
		\draw[line width=.5mm, blue] (0,1.5) \pr (1.25,1.25) \pr (2.5,1.5);
		\draw[line width=.5mm, blue] (0,-1.5) \pr (1.25,-1.25) \pr (2.5,-1.5);
		\draw[line width=.5mm, blue] (0,0.5) \rd (0.5,0) \dl (0,-0.5);
		\draw[line width=.5mm, blue] (2.5,0.5) \ld (2,0) \dr (2.5,-0.5);
		\draw[thick,double] (0,0) circle (1);
		\draw[thick,double] (2.5,0) circle (1);
		\filldraw[white] (1.25,0) circle (1.3em);
		\draw[thick] (1.25,0) circle (1.3em);
		\node  at (1.25,0) {$T$};
		\end{tikzpicture}
\]
	\caption{
		\label{fig:surfaceandtangle}
	Left: a surface with seams (purple) and intervals in the boundary
	(blue). Right: a tangle in the surface with ends on the boundary intervals.  
	The doubled strands indicate an arbitrary finite number of parallel copies of that strand.  
	}
\end{figure}
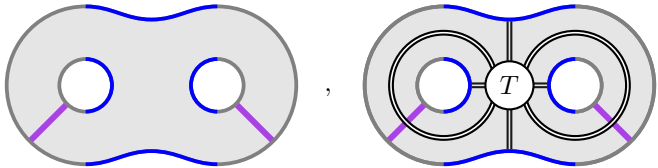

Alternatively, the data of $\Gamma$ may be considered as a
presentation of $\Sigma$ as a union of 2-dimensional 0- and 1-handles, 
with the seams appearing as co-cores of the 1-handles. 
We construct\footnote{We notationally suppress the dependence on $\Gamma$ for the introduction; 
later, we prove that our construction does not depend on this choice, up to quasi-equivalence.} 
a dg category $\sCat(\Sigma,\mathbf{p})$ as follows.
The objects of $\sCat(\Sigma,\mathbf{p})$ are unoriented
	tangles $T$, neatly embedded in $\Sigma$, with boundary $\partial T =
	\mathbf{p}$ and transversely intersecting the seams $T \pitchfork \gamma$
	for all $\gamma\in \Gamma$. 
The right side of Figure~\ref{fig:surfaceandtangle} gives a graphical depiction of such 
	a tangle.
	
	Given two such objects $S,T$, the chain complex of morphisms from
	$S$ to $T$ in $\sCat(\Sigma,\mathbf{p})$ is computed as follows.
	We first proceed to the doubled surface 
	$D(\Sigma):=\Sigma \cup_{\partial\Sigma} \Sigma^\vee$, which is obtained by gluing on its orientation reversal
	$\Sigma^\vee$ along the boundary, and which contains the embedded link
	$D(T,S):=T\cup_{\pp} S^\vee$. We then embed $D(\Sigma)$, together with $D(T,S)$, 
	in $\R^3$ as the boundary of the $3$-dimensional handlebody with compression disks
	specified by the circles in $D(\Sigma)$ determined by the seams $\gamma\in \Gamma$. 
	Finally, at each compression disk we insert a copy of Rozansky's 
	(through-degree zero) \emph{bottom projector} $\Bproj$ from \cite{rozansky2010categorification} and
	compute the Khovanov chain complex of the result; see Figure~\ref{fig:cuthandlebody}.

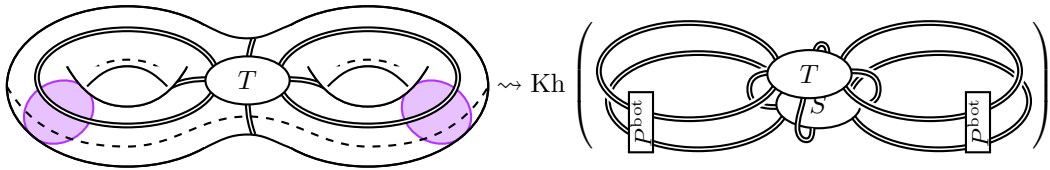
\begin{figure}[ht]
	\[
		\begin{tikzpicture}[anchorbase,scale=.8]
			\toruscdisk{-1.2}{.2}{.6}{.5}
			\begin{scope}[xscale=-1,shift={(-4,0)}]
			\toruscdisk{-1.2}{.2}{.6}{.5}
			\end{scope}
			\draw[thick] (0,0) ellipse (2 and 4/3);
			\draw[thick] (4,0) ellipse (2 and 4/3);
			\torusnoeq{0}{0}{2}{thick}
			\torusnoeq{4}{0}{2}{thick}
			\fill[white] (1,-1) rectangle (3,1);
			\draw[thick,double] (2.1,.8) to [out=250,in=90] (2,0) to [out=270,in=110] (2.1,-.8);
			\draw[thick,double] (0.5,-.15) to [out=40,in=180] (2,0.1) to [out=0,in=140] (3.5,-.15);
			\torushole{0}{0}{2}{thick}
			\torushole{4}{0}{2}{thick}
			\draw[thick, dashed] (-2,0) to [out=300,in=180] (0,-2/2) \pr (2,-.5) \pr (4,-2/2) to [out=0,in=240] (6,0);
			\draw[thick, fill=white] (1.3,-1.015) to [out=26,in=180] (2,-.8) to [out=0,in=160] (2.7,-1.015);
			\draw[thick, fill=white] (1.3,1.015) to [out=-26,in=180] (2,0.8) to [out=0,in=-160] (2.7,1.015);
			\draw[thick,double] (0,.15) ellipse (1.5 and .8);
			\draw[thick,double] (4,.15) ellipse (1.5 and .8);
			\draw[thick, fill=white] (2,.1) ellipse (0.7 and .4);
			\node  at (2,.1) {$T$};
			\end{tikzpicture}
			\rightsquigarrow
			\KhEval{
			\begin{tikzpicture}[anchorbase,scale=.8]
			\begin{scope}[shift={(.15,-.5)}]
				\draw[thick,double] (0,.15) ellipse (1.5 and .8);
				\draw[thick,double] (4,.15) ellipse (1.5 and .8);
			\end{scope}
			\draw[white, line width=1.5mm] (2.15,-.4) \lu (1,-.15) \ur (2,0.1) \rd (3.15,-.15) \dl (2.15,-.4);
			\draw[thick,double] (2.15,-.4) \lu (1,-.15) \ur (2,0.1) \rd (3.15,-.15) \dl (2.15,-.4);
			\begin{scope}[shift={(.15,-.5)}]
				\draw[thick, fill=white] (2,.1) ellipse (0.7 and .4);
				\node  at (2,.1) {$S$};
			\end{scope}
			\draw[thick,double] (1.9,-.3) to [out=260,in=90] (1.8,-.8) \dr (1.95,-1) to [out=0,in=270] (2.05,-.8) (2.1,.5) to [out=70,in=180] (2.2,.6) \rd (2.3,.3);
			\draw[white, line width=1.5mm] (0,.15) ellipse (1.5 and .8);
			\draw[white, line width=1.5mm] (4,.15) ellipse (1.5 and .8);
			\draw[thick,double] (0,.15) ellipse (1.5 and .8);
			\draw[thick,double] (4,.15) ellipse (1.5 and .8);
				\draw[thick, fill=white] (2,.1) ellipse (0.7 and .4);
				\node  at (2,.1) {$T$};
			\draw[thick, fill=white] (-1,-1.2) rectangle  (-0.6,-.2);
			\node[rotate=90] at (-.8,-.7) {\small $\Bproj$};
			\draw[thick, fill=white] (5,-1.2) rectangle  (4.6,-.2);
			\node[rotate=90] at (4.8,-.7) {\small $\Bproj$};
				\end{tikzpicture}
			}
	\]
	\caption{
		\label{fig:cuthandlebody}
		Left: the doubled surface embedded in $\R^3$ with compression disks
	specified by the seams. Right: the Khovanov chain complex of the link $D(T,S)$ with
	bottom projectors inserted at the compression disks. }
\end{figure}

The operation of composing morphisms in $\sCat(\Sigma,\mathbf{p})$ is subtle, 
and will only be fully developed in the body of the paper.
It follows the geometric intuition of stacking the $3$-dimensional handlebodies shown on
the left of Figure~\ref{fig:cuthandlebody}. 
The main technical ingredient needed to make this precise
is a new ``sideways'' composition operation 
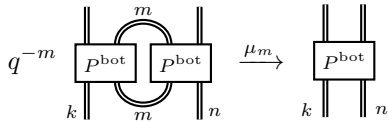
\[
q^{-m} \begin{tikzpicture}[anchorbase]
	\draw[thick, double] (.75,.25) to (.75,.75);
	\draw[thick, double] (-.75,.25) to (-.75,.75);
	\draw[thick, double] (.75,-.25) to (.75,-.75);
	\draw[thick, double] (-.75,-.25) to (-.75,-.75);
	\draw[thick, double] (-.35,.25) to[out=90,in=180]  (0,.55) to[out=0,in=90] (.35,.25);
	\draw[thick, double] (-.35,-.25) to[out=-90,in=180]  (0,-.55) to[out=0,in=-90] (.35,-.25);
	\draw[thick] (.1,-.25) rectangle (.9,.25); 
	\draw[thick] (-.1,-.25) rectangle (-.9,.25); 
	\node at (.5,0) {\scs$\Bproj$};
	\node at (-.5,0) {\scs$\Bproj$};
	\node at (0,.70) {\scriptsize $m$};
	\node at (0,-.70) {\scriptsize $m$};
	\node at (-.95,-.65) {\scriptsize $k$};
	\node at (.95,-.65) {\scriptsize $n$};
	\end{tikzpicture}
	\xrightarrow{\mu_m}
	\begin{tikzpicture}[anchorbase]
	\draw[thick, double] (.25,.25) to (.25,.75);
	\draw[thick, double] (-.25,.25) to (-.25,.75);
	\draw[thick, double] (.25,-.25) to (.25,-.75);
	\draw[thick, double] (-.25,-.25) to (-.25,-.75);
	\draw[thick] (-.4,-.25) rectangle (.4,.25); 
	\node at (0,0) {\scs$\Bproj$};
	\node at (-.5,-.65) {\scriptsize $k$};
	\node at (.5,-.65) {\scriptsize $n$};
	\end{tikzpicture}
\]
that we introduce in \S\ref{s:Rozanskybar}, 
and which is defined via the Eilenberg--Zilber shuffle product.

If $\Sigma$ is a disk, 
then (up to quasi-equivalence) the above construction produces the usual Bar-Natan categories \cite{BN2}, 
which categorify the Temperley--Lieb skein theory and feature in many constructions of Khovanov homology 
(e.g.~they are closely related to Khovanov's arc rings \cite{MR1928174}). 
More generally, as explained in greater detail in \S \ref{sec:surfacelinkhom}, 
our dg category $\CS(\Sigma,\pp)$ is a derived version of the Bar-Natan category $\bn(\Sigma,\pp)$ 
and both versions categorify the Temperley--Lieb skein module of crossingless tangles on $\Sigma$ with boundary $\pp$, 
i.e.~the 2-dimensional layer of the Turaev--Viro TFT for quantum $\mathfrak{sl}_2$ \cite{SN-TV}. 

In \S \ref{ss:spin-networks} we develop the parallels with the Turaev--Viro
theory further by showing that (a natural completion of) 
$\CS(\Sigma,\pp)$ is generated by certain infinite complexes indexed 
by quantum spin networks on $\Sigma$. 
Only in the derived setting of $\CS(\Sigma,\pp)$ will these objects be
orthogonal under the (appropriately symmetrized) $\Hom$-pairing.
Moreover, as mentioned above, the self-pairing categorifies 
a well-known formula in Turaev--Viro theory; 
see Theorem~\ref{thm:pairing} for the precise result.

\subsection{Higher categorical interpretation}

A modern perspective on knot homology theories and their corresponding tangle
invariants is to view them as morphisms in certain \emph{braided monoidal
2-categories}\footnote{Or, more accurately, locally graded-linear
$\mathbb{E}_2$-monoidal $(\infty, 2)$-categories.}
\cite{2019arXiv190712194M,stroppel2022categorification,liu2024braided}, in analogy to how many
\emph{decategorified} quantum invariants appear as morphisms in \emph{braided
monoidal 1-categories}, e.g.~representation categories of quantum groups. 
The Baez--Dolan--Lurie cobordism hypothesis \cite{MR1355899, MR2555928, MR2994995} and
its generalization provide a close connection between (braided) monoidal categories 
(and their higher analogues) and local topological field theories (TFTs). 
Applying these tools to higher categories originating in link homology
theory is one possibly strategy for constructing new TFTs 
and associated invariants of smooth manifolds.

From this perspective,
our present work is concerned with the local TFT arising from the monoidal
$2$-category underlying Khovanov homology \cite{Kho}, 
with the braiding arising from Khovanov's ``cube of resolutions'' having been forgotten. 
The analogous decategorified setting is the monoidal Temperley--Lieb category 
(or, its idempotent completion, $\Rep(U_q(\slnn{2}))$) with its braiding forgotten,
which gives rise to TFTs of Turaev--Viro type.
Because we have forgotten the braiding (at least for now), 
the construction does not depend on the specifics of Khovanov's link homology theory, 
but only on the \emph{graded} commutative Frobenius algebra that appears as the invariant of the unknot. 
As a somewhat simplified illustration of what this means, 
we are considering categorified skein modules of surfaces 
in which elements are \emph{embedded} 1-manifolds in the surface itself, 
hence have crossingless diagrams. 

\begin{table}[h]
	\begin{center}
	\begin{tabular}{|c||l c|l c|}\hline
		 && monoidal && braided monoidal  \\ \hline\hline 
		1-category & 
		$n{=}2$:& Turaev--Viro family \cite{TuraevViro} 
		& 
		$n{=}3$:& Crane--Yetter family \cite{CraneYetter}  
		\\
		(classical)
		&&
		&&
		(Jones polynomial)
		\\\hline 
		2-category &
		$n{=}3$:&Asaeda--Frohman--Kaiser \cite{MR2370224,MR2503518}
		
		& 
		$n{=}4$:& Morrison--Walker--W. \cite{2019arXiv190712194M} 
	 \\ 
	 (categorified, underived)
		  && 
		  Douglas--Reutter \cite{douglas2018fusion}
		  
		&&
		(Khovanov homology)
	 \\  
	 \hdashline
	 2-category &
		$n{=}3$:&  first steps here
		
		& 
		$n{=}4$: & ???
	 \\ 
	 (categorified, derived)
		  && 
		 
		&&
		(Khovanov complexes)
	 \\
	 \hline
	\end{tabular}
	\caption{\label{table:landscape}Some types of local (partially defined) TFTs
	in $n+\epsilon$ dimensions (and related invariants) by input data: these
	assign vector spaces or chain complexes to closed $n$-manifolds, linear
	categories or dg categories to closed $(n-1)$-manifolds, as well as higher
	categories to manifolds of higher codimension.	
	}
	\end{center}
	\end{table}

Table~\ref{table:landscape} outlines the positioning of our work relative to
some reference points in the landscape of link homologies and local TFTs.  
Here we distinguish between underived and derived versions of the categorified TFTs.
As indicated there,
we expect that our work in this paper is a part of a local
$(3+\epsilon)$-dimensional TFT for oriented manifolds. 
While the full development of this TFT, henceforth denoted $\TFT$, 
goes beyond the scope of this paper, it is useful to have the following outline in mind:

Our construction of $\TFT$ starts with the graded commutative Frobenius algebra
$H^*(\C P^1)=\cring[X]/(X^2)$---the Khovanov homology of the unknot 
over a commutative ring $\cring$---and proceeds 
by building a locally graded $\cring$-linear monoidal bicategory $\BNA$. 
The monoidal bicategory $\BNA$ can be described in terms of Bar-Natan's dotted cobordisms 
from \cite[Section 11.2]{BN2} (see also \cite[\S 4.2]{HRW3}) 
and it should be the value of $\TFT$ on the positively oriented point. 
To a compact, connected oriented $1$-manifold (an interval) the TFT
$\TFT$ should assign the regular $\BNA$-bimodule, i.e.~$\BNA$, 
acted upon by itself on both sides via the monoidal structure, 
but with its own monoidal structure forgotten.

The actual work in this paper starts at the $2$-dimensional layer, specifically
for \emph{marked surfaces}. 
To this end, we consider pairs $(\Sigma,\Pi)$, where
$\Sigma$ is a compact, oriented surface with boundary, 
which has no closed components, 
and $\Pi$ is a set of finitely many oriented intervals, 
disjointly embedded in $\partial \Sigma$.
The complement of the arcs of $\Pi$ in $\partial \Sigma$ is then considered as the \emph{marking}. 

A main open question about the construction of $\TFT$, 
which we plan to pursue in future work, 
is the appropriate symmetric monoidal higher category to use as its target. 
In any case, $\TFT$ should assign to a marked surface $(\Sigma, \Pi)$ an algebraic
gadget that carries actions of the bicategories $\BNA$, 
namely one for each of the intervals from $\Pi$.
Informed by this, we are able to package our dg categories $\sCat(\Sigma,\pp)$
as the values (on objects) of $2$-functors with domain an appropriate 
tensor product $\bn(\Pi)$ of Bar--Natan bicategories.

\begin{thm}[Theorems~\ref{thm:final2} and \ref{thm:final3}]
	\label{thm:main}
Let $(\Sigma,\Pi)$ be a marked surface 
and let $\bn(\Pi)$ denote the tensor product of Bar-Natan bicategories (and their opposites) 
associated to the collection $\Pi$ of oriented boundary intervals (see Definition \ref{def:bn Pi}). 
Then the assignment $\pp \mapsto \sCat(\Sigma,\pp)$ extends to a 2-functor
\[
\MM \colon \bn(\Pi) \Rightarrow \MorDGCat
\]
where $\MorDGCat$ denotes the Morita bicategory of 
differential $(\Z \times \Z)$-graded $\cring$-linear categories, 
bimodules of such dg categories, and bimodule homomorphisms
(see Definition~\ref{def:morita}).
\end{thm}

\begin{remark}
Strictly speaking, the 2-functor from Theorem \ref{thm:main} requires some additional choices, 
most significantly the aforementioned decomposition of $\Sigma$ into $0$- and $1$-handles. 
We address the (in)dependence on this choice, 
considering questions of both invariance and coherence, 
in \S \ref{s:coarsening} and \S \ref{s:coherence}. This approach is inspired by \cite{DyKa,HKK}. 
\end{remark}

The invariants from Theorem~\ref{thm:main} carry mapping class group actions
(Construction~\ref{constr:diffeos}) and can be glued along a pair of intervals
from $\Pi$ (Construction \ref{const:gluing}). 
In fact, the invariants of general surfaces are constructed by gluing together the invariants
of polygons. 
We expect that this amounts to the structure of a categorified version of an \emph{open modular functor} 
or a \emph{topological conformal field theory}; see e.g.~\cite[Definition 2.1]{MR4696498}, 
\cite{MR2298823,costello2004ainfinity}, and \cite[\S 4.2]{MR2555928}.

\subsection{Relation to other work}
\label{s:relotherwork}

In this section, we relate our work to previous constructions that 
extend Khovanov homology beyond links in $\R^3$.
Our invariants can be seen as the unification, and the derived generalization, 
of much of this work.

\subsubsection{Rozansky--Willis invariants} 
Rozansky first used bottom projectors to define a version of Khovanov homology for links embedded in
$S^1\times S^2$ \cite{rozansky2010categorification}, which was subsequently
generalized by Willis to links in connected sums of several copies of $S^1\times
S^2$ \cite{MR4332675}. 
The Rozansky--Willis invariant can be computed by taking
standard Heegaard splittings of these $3$-manifolds, pushing the relevant link
into one of the two handlebodies, and then performing a Khovanov homology
computation in $\R^3$ after inserting bottom projectors. 
In particular, the morphism complexes in our dg categories 
$\sCat(\Sigma,\mathbf{p})$
can be interpreted as Rozansky--Willis chain complexes. 
In \cite[Theorem 6.7]{rozansky2010categorification}, 
Rozansky shows that the bottom projectors
can be approximated by (suitably normalized) chain complexes of high powers of
full twist braids.
Consequently, Rozansky--Willis invariants are locally finitely generated
(i.e.~finitely generated in each bidegree) and can be approximated as a stable
limit of Khovanov homologies of a suitable family of links in $\R^3$; 
see \cite[Corollary 3.8]{MR4332675}. 
Hence, we obtain:

	\begin{corollary}
		\label{cor:fg}
		The cohomology of the morphism complexes in 
		$\sCat(\Sigma,\mathbf{p})$ is finitely generated in
		each bidegree and, for a fixed bidegree, it can be effectively computed
		by a finite Khovanov homology computation. 
	\end{corollary}

\subsubsection{Link homology in thickened surfaces}
\label{sec:surfacelinkhom} The
Asaeda--Przytzcki--Sikora homology theories \cite{MR2113902} for links in a
thickened surface $\Sigma \times [0,1]$ factor through a construction \cite{MR2475122}
involving the bounded homotopy category of chain complex over the graded, linear
category $\bn(\Sigma)$ of Bar-Natan's (dotted) cobordisms in $\Sigma\times[0,1]$
(for more details, see the next item below). 
Our dg categories $\sCat(\Sigma,\mathbf{p})$ appear as a derived version of this
construction:

\begin{thm}[Theorem~\ref{thm:surfaceBN}]
	The zeroth cohomology category of $\sCat(\Sigma,\emptyset)$ is
	equivalent to $\bn(\Sigma)$.
\end{thm}

\subsubsection{The Asaeda--Frohman--Kaiser $(3+\epsilon)$-TFT}
Given a compact oriented $3$-manifold $M$, one can consider the 
\emph{Bar-Natan skein module} spanned by embedded (and dotted) surfaces, 
modulo ambient isotopy and certain local skein relations \eqref{eq:BNrels}. 
Given a ``boundary condition'' taking the form of an embedded $1$-manifold in $\partial M$, 
one analogously considers a relative skein module spanned by properly
embedded surfaces with this specified boundary; 
see e.g.~\cite{MR2370224,Russell,MR3220480,MR2503518,thesis-Fadali}. 
For $M=\Sigma\times [0,1]$, 
these skein modules recover the morphism spaces of the category $\bn(\Sigma)$. 
There are few computations of Bar-Natan skein modules of other $3$-manifolds, 
but it is known that they are sensitive to compressibility of
surfaces \cite{MR2370224,kaiser2022barnatan}. 

In the landscape of Table~\ref{table:landscape}, the Bar-Natan skein modules
form the $3$-dimensional layer of a $(3+\epsilon)$-dimensional
\emph{Asaeda--Frohman--Kaiser TFT}, which is based on the monoidal bicategory $\bn$. 
Unlike the Douglas--Reutter TFTs \cite{douglas2018fusion}, which are also
constructed from monoidal bicategories, 
this theory is not expected to extend to
all oriented $4$-manifolds, due to the non-semisimplicity of $\bn$. 
On the positive side, however, such non-semisimplicity may enable greater topological
sensitivity of the associated TFT, e.g.~towards smooth structure \cite{reutter2020semisimple}, 
provided the theory extends to \emph{some} $4$-manifolds. 
As with all extensions of TFTs upwards in dimension, this would
require finiteness properties, which tend to be difficult to achieve with skein modules (c.f.~\S\ref{sss:lasagna}), 
but which can be relaxed in a derived setting based on chain complexes. 

As an intermediate step, we conjecture the existence of a partial $3$-dimensional extension of
our surface invariants, which assigns graded chain complexes to (certain)
oriented $3$-manifolds, with zeroth cohomology recovering the Bar-Natan skein
module.

\subsubsection{The dg horizontal trace}
\label{sss:dgtr}
Arguably, the most important case of the Asaeda--Przytzcki--Sikora link homology
theories arises for the annulus $\Sigma=S^1\times [0,1]$. The corresponding
category $\bn(S^1\times [0,1])$ can be described as the horizontal trace
\cite{BHLZ,QR2} of the bicategory $\bn$. Over the last decade, this
higher-categorical tracing perspective has been highly influential in link
homology theory and beyond, 
e.g.~featuring in the Gorsky--Negut--Rasmussen conjecture \cite{GNR}. 
One reason for its importance is that the annular category 
for a link homology theory should control all possible versions of \emph{colored} link homology, 
as well as the behavior of the link homology under
cabling operations \cite[Section 6.4]{2019arXiv190404481G}. 
Unfortunately, the ordinary horizontal trace is unable 
to capture the rich homotopical data required for such applications. 
As mentioned above, 
the appropriate derived replacement, the \emph{dg horizontal trace} 
was introduced by Gorsky and two of the authors \cite{2002.06110} to remedy this issue.
In the present paper, we recover the dg horizontal trace as the dg category
associated to the annulus with one seam and no boundary points.

\begin{thm}[Theorem~\ref{thm:dgtrace}]
	\label{thm:introtrace}
The dg category $\sCat(S^1\times [0,1],\emptyset)$ of the annulus 
	(with one radial seam) is canonically equivalent to the dg horizontal trace of the bicategory $\bn$.
\end{thm}

In particular, our results on the independence of the category $\sCat$ from the system of
seams $\Gamma$ give an a posteriori justification of the construction of 
the dg horizontal trace, which uses only one seam.

If we additionally allow boundary points for tangles on the boundary of the annulus, 
the resulting categories form the $\Hom$-categories of what could be called the 
\emph{monoidal trace}\footnote{Or, possibly, the factorization homology over the circle; 
\emph{monoidal} trace as opposed to the \emph{horizontal} trace.} of the
monoidal bicategory $\bn$. We discuss the additional
compositional structure on these $\Hom$-categories in \S\ref{ss:composing in affbn}.

\subsubsection{Categorification of skein algebras}
The category $\bn(\Sigma)$, and thus also $\sCat(\Sigma,\emptyset)$,
descends on the level of K-theory to the Temperley--Lieb skein module of $\Sigma$, 
a variant of the underlying module of the Kauffman bracket skein algebra of $\Sigma$. 
It is a natural question to ask whether the multiplication on skein algebras 
can also be categorified by means of link homology \cite{MR3934851,1806.03416}. 
However, existing approaches that use the homotopy category
of chain complexes over $\bn(\Sigma)$ immediately run into problems caused by a
lack of homotopy coherence\footnote{This is why the work in \cite{1806.03416,2209.08794}
relies on an additional technical semisimplification step; 
see e.g.~the discussion following Conjecture 1.9 in \cite{1806.03416}.}. 

We expect that the dg categories $\sCat(\Sigma,\emptyset;\Gamma)$ 
(or, more precisely, their pretriangulated hulls) form a more natural setting for a
categorified skein algebra multiplication. 
Since this will involve the braiding on $\bn$, it belongs to the $(4+\epsilon)$-sector of 
Table~\ref{table:landscape} and will be left for future work.

\subsubsection{Skein lasagna modules}
\label{sss:lasagna}

We were led to the developments in this paper by thinking about the
$(4+\epsilon)$-dimensional TFT associated to Khovanov
homology~\cite{2019arXiv190712194M}. 
As this theory will play a minor role here, 
we only mention three important facts: 
it is constructed based on Khovanov \emph{homology} (i.e. not on the level of chain complexes), 
it is capable of detecting exotic pairs of oriented $4$-manifolds \cite{ren2024khovanov}, 
and the associated manifold invariants are difficult to compute and can fail to be locally finitely generated, 
see \cite[Theorem 1.4]{2206.04616}.

One may speculate that a \emph{derived} analogue of the $4$-manifold invariant
from \cite{2019arXiv190712194M}, which is constructed on the level of
\emph{chain complexes}, may have better finiteness properties. 
Moreover, the invariants for $4$-dimensional $1$-handlebodies in such theory 
may already be known: the Rozansky--Willis invariants, in which $1$-handles are modeled by
bottom projectors, see \cite[Section 4.7]{2206.04616}. 

A next step in the exploration of such a theory would be to search for a similar
model for $2$-handles. The (underived) theory from \cite{2019arXiv190712194M}
admits such a model, namely the \emph{Kirby color for Khovanov homology}
constructed in previous work of the authors \cite{HRW3}; 
see also \cite{2020arXiv200908520M} for the inspiration for this construction and 
\cite{sullivan2024kirby} for a recent application.
This Kirby color is an object of an appropriate completion of $\bn(S^1\times [0,1])$;
interestingly, this object, which controls $2$-handle attachments in a $(4+\epsilon)$-dimensional theory, 
can already be constructed in the $(3+\epsilon)$-dimensional theory based on the same monoidal bicategory,
but with the braiding forgotten. 
This parallels how the Kirby color for the Crane-Yetter theory agrees with the Kirby color of the corresponding
Turaev--Viro theory, when forgetting the braiding; 
see the first row of Table~\ref{table:landscape}. 
In follow-up work, we pursue a construction of a \emph{derived} Kirby colored for Khovanov homology, 
which will be an object of (a completion of) the derived version of $\bn(S^1\times [0,1])$, namely the dg
category from Theorem~\ref{thm:introtrace}.

\subsubsection{Bordered Heegaard--Floer theory and Fukaya categories of symmetric products}

Our constructions have structural similarities with the bordered Heegaard--Floer homology
of Lipshitz--Ozsvath--Thurston~\cite{MR3827056} and the bordered sutured Floer
homology of Zarev~\cite{MR2941830}; see also \cite{lipshitz2023floer} for a
recent survey. The dg category\footnote{A dg algebra with a distinguished
collection of idempotents is nothing but a dg category.} that bordered
Heegaard--Floer theory assigns to a surface depends on a handle decomposition,
with $2$-handles playing a special role, and $3$-manifolds with boundary yield
$A_\infty$-bimodules between these dg algebras. 
The invariants developed here thus appear
located at the same categorical level as bordered Heegaard--Floer
theory\footnote{Unlike Khovanov homology, which appears one level
higher when incarnated in a braided monoidal bicategory.}. Related candidates
for interesting comparisons are the (partially wrapped versions of) Fukaya
categories of symmetric products of the surface $\Sigma$; see~\cite{MR2827825}.
We plan to study such relations in future work. 
Here we would like to point out one conceptual difference: 
bordered and cornered Heegaard--Floer theories
\cite{MR3180613,MR4045965, manion2020higher} appear as \emph{extensions
downward} from the Heegaard--Floer invariants of $3$-manifolds (and even from
the Seiberg--Witten invariants of $4$-manifolds), whereas our theory is
\emph{extended upwards} from the point, and thus is manifestly local.

\subsection{A generalization beyond Khovanov homology and TFT for stratified manifolds}
In the final stages of this project, we observed\footnote{For this reason,
and for ease of exposition, we have restricted to the archetypical case of $H^*(\C P^1)$.} 
that all of our constructions work in a much more general setting than that of
Khovanov homology. 
Namely, we may start from any non-negatively graded commutative
Frobenius algebra or, equivalently, the associated $2$-dimensional graded TFT.
Indeed, Bar-Natan skein modules have been constructed for any commutative
Frobenius algebra \cite{MR2503518} and the grading serves to make the associated
analogues of bottom projectors locally finitely generated, 
which is then inherited by all morphism complexes. 

This striking observation raises the question: how can such a simple algebraic
structure (a graded commutative Frobenius algebra) give rise to the
complexity of higher-dimensional TFTs exemplified by the many relations
discussed in \S\ref{s:relotherwork}. The answer is necessarily related to the
interaction of TFTs with stratifications on manifolds, as we outline in \S \ref{ss:stratified}. 

Another starting point for analogous constructions as in this paper are the
categories of foams that appear in combinatorial constructions of
Khovanov--Rozansky $\glnn{N}$ link homology theories \cite{Kho3,MSV,QR,RoW}. 
We expect the resulting
invariants to be essentially equivalent to those based on the graded commutative
Frobenius algebras $H^*(\C P^{N-1})=\cring[\mathsf{x}]/(\mathsf{x}^N)$.

\subsection*{Acknowledgements}
We thank
Christian Blanchet, 
Jesse Cohen,
Tobias Dyckerhoff, 
Jonte G\"odicke, 
Mikhail Kapranov,
Slava Krushkal,
Jake Rasmussen,
and Lev Rozansky
for helpful discussions, references to the literature, 
and encouragement.

\subsection*{Funding}

D.R. was partially supported by 
NSF CAREER grant DMS-2144463 
and Simons Collaboration Grant 523992. 
P.W. acknowledges support from the Deutsche
Forschungsgemeinschaft (DFG, German Research Foundation) under Germany's
Excellence Strategy - EXC 2121 ``Quantum Universe'' - 390833306 and the
Collaborative Research Center - SFB 1624 ``Higher structures, moduli spaces and
integrability'' - 506632645.

\section{Background}
\label{sec:background}

\subsection{Categorical setup and conventions}
\label{ss:setup}

All results in this paper, 
with the exception of
a few decategorification statements in \S \ref{sec:spin}, 
hold over an arbitrary commutative ring of scalars $\cring$, which will be fixed throughout.  
Tensor products are always taken over this ring, unless otherwise specified.

We will consider (differential) graded $\cring$-modules and categories, with gradings valued in 
various abelian groups $\Grp$ (in reality, we will only need $\Grp=\{0\}$, $\Z$, or $\Z\times \Z$).  To 
disambiguate, we will always explicitly include the group $\Grp$ in our notation, for example by writing 
$\dgMod{\Grp}{\cring}$ for the category of differential $\k$-modules with gradings valued 
in $\Grp$. In more detail, 
to discuss differential $\k$-modules and the like, we must choose a 
\emph{grading datum} $(\Grp,\ip{\ ,\ },\iota)$, consisting of
\begin{itemize}
\item an abelian group $\Grp$, where the gradings take values,
\item a symmetric bilinear pairing $\ip{\ , \ }\colon \Grp\times \Grp \to \Z/2$, 
used in the Koszul sign convention,
\item a distinguished element $\iota\in \Grp$ with $\ip{\iota,\iota}=1$, which is 
to be the degree of all differentials.
\end{itemize}
Note that the element $\iota$ is only relevant when we want to consider complexes,
i.e.~the choice of $(\Grp,\ip{\ , \ })$ alone determines a symmetric monoidal category 
$\gMod{\Grp}{\cring}$ of $\Grp$-graded $\cring$-modules and degree zero 
$\k$-linear morphisms.  This category is monoidal in the standard way, with 
symmetric braiding $M\otimes N\rightarrow N\otimes M$ determined by the Koszul sign rule
\[
m\otimes n \mapsto (-1)^{\ip{\deg(m),\deg(n)}} n\otimes m.
\]

A $\Grp$-graded $\cring$-linear category is defined to be a category enriched in
$\gMod{\Grp}{\cring}$.  The category $\gMod{\Grp}{\cring}$ admits an internal
$\Hom$, consisting of homogeneous $\cring$-linear morphisms of arbitrary degree.
The notation $\Hom_\cring(-,-)$ will \emph{always} mean this internal $\Hom$;
the space of morphisms in $\gMod{\Grp}{\cring}$ is recovered as its
degree-zero component.  The internal hom defines a lift of $\gMod{\Grp}{\cring}$ 
to a $\Grp$-graded $\cring$-linear category that we denote $\gModen{\Grp}{\cring}$.

Let $\dgMod{\Grp}{\cring}$ be the symmetric monoidal category of differential 
$\Grp$-graded $\cring$-modules and degree zero chain maps.  

\begin{example}
The usual category of complexes of $\cring$-modules would be denoted using the above notation as 
$\dgMod{\Z}{\cring}$, 
with $\ip{i,j}\equiv ij$ (mod 2), and $\iota= 1$ (or $\iota=-1$ if one prefers).
\end{example}

A \emph{differential $\Grp$-graded $\cring$-linear category}
(or simply \emph{dg category}, for short) will mean a category enriched in $\dgMod{\Grp}{\cring}$.
A dg functor between dg categories will, as usual, mean a functor in the enriched sense, 
thus its components are degree-zero chain maps.
We will denote the category of dg categories and dg functors by $\DGCat$.

As was the case for graded modules, 
$\dgMod{\Grp}{\cring}$ possesses an internal $\Hom$ given by
\[
\Hom_\cring\big((M,\d_M),(N,\d_N) \big)  
	= \big(\Hom_\cring(M,N)\,  , \  f\mapsto \d_N\circ f - (-1)^{\ip{\iota,\deg(f)}} f\circ \d_M \big) \, .
\]
The space of morphisms in $\dgMod{\Grp}{\cring}$ is obtained from this by taking 
the closed morphisms in degree-zero.  The internal hom defines a lift of the 
ordinary (unenriched) category $\dgMod{\Grp}{\cring}$ to a dg category, 
which we denote $\dgModen{\Grp}{\cring}$.
If $\CS$ is a dg category, 
then we write 
$H^0(\CS)$ for the $\cring$-linear categories 
with the same objects as $\CS$ 
and morphisms given by 
degree-zero closed morphisms 
modulo homotopy.

\begin{conv}
An ordinary $\cring$-linear category can and frequently will 
be regarded as a dg category (for any choice of grading datum $(\Grp,\ip{\ , \ },\iota)$) with trivial grading and differentials.
\end{conv}

\begin{conv}
In this paper, 
all categories will graded by $\Z\times \Z$ or its subgroup $\{0\}\times \Z$, 
with differentials of degree $\iota=(1,0)$, and 
sign rule $\ip{(i,j),(i',j')}=ii'$ (mod 2).  In particular, ordinary graded categories are 
$(\{0\}\times \Z)$-graded, while their categories of complexes are $(\Z\times \Z)$-graded.  
We let $t$ and $q$ denote the corresponding grading shift autoequivalences, 
so that for $M\in \gMod{\Z\times\Z}{\cring}$, the homogeneous pieces of $t^iq^j M$ are given by $(t^iq^j M)^{k,l} 
= M^{k-i,l-j}$.
\end{conv}

If $\CS$ is a dg category, then we write $\Ch^b(\CS)$ (resp.~$\Ch^-(\CS)$) for the dg categories of
bounded (resp.~bounded above) one-sided twisted complexes over 
the additive completion of $\CS$. 
When $\CS$ is $(\{0\}\times \Z)$- or $(\Z \times \Z)$-graded, 
we will additionally formally adjoin $q$-shifts of our objects before forming twisted complexes.
If $\CS$ is simply a $\cring$-linear category, 
this gives the usual notions of bounded (resp.~bounded above) chain complexes. 
If $(X,\d)$ is a chain complex with differential $\d$ 
and $\alpha$ is an endomorphism of $X$ with $(\alpha+\d)^2=0$, 
then we write $\tw_\alpha(X):=(X,\d+\alpha)$ for the twist of $(X,\d)$ by $\alpha$; 
see e.g.~\cite[Definition 2.15]{HRW1}.

If $\CS$ and $\DS$ are dg categories, then $\CS\otimes \DS$ is the dg category with objects given by pairs 
$(X,Y)$ with $X\in \CS$, $Y\in \DS$, $\Hom$-complexes given by
\[
\Hom_{\CS\otimes \DS}\Big((X,Y),(X',Y')\Big):=\Hom_\CS(X,X')\otimes \Hom_\DS(Y,Y'),
\]
and composition defined using the symmetric monoidal structure on $\dgMod{\Grp}{\cring}$ 
(i.e.~the Koszul sign rule):
\[
(f\otimes g)\circ (f'\otimes g') = (-1)^{\ip{\deg(g),\deg(f')}} (f\circ f')\otimes (g\circ g') \, .
\]
There is a canonical \emph{totalization} functor 
\begin{equation}
	\label{eq:totalization}
\Ch(\CS) \otimes \Ch(\DS) \to \Ch(\CS \otimes \DS)
\end{equation}
given akin to the usual tensor product of chain complexes of abelian groups.
We will tacitly use \eqref{eq:totalization} to view objects $(X,Y) \in \Ch(\CS) \otimes \Ch(\DS)$ 
as objects in $\Ch(\CS \otimes \DS)$.

We will also consider dg bicategories, wherein the $1$-morphism categories are dg categories 
(with some fixed data $(\Grp, \ip{\ , \ }, \iota)$).  
For example, there is a dg bicategory of dg categories
with objects dg categories, $1$-morphisms dg functors, and $2$-morphisms dg natural transformations. 
By abuse of notation, we will also denote this bicategory by $\DGCat$.
For many constructions in this paper, 
we will also utilize various instances of Morita bicategories, whose objects are dg categories, 
1-morphisms are dg bimodules over these categories, and 2-morphisms are dg maps of bimodules.  
See \S \ref{ss:BNmm} and \S \ref{ss:MM of surface} for more.

Given dg categories $\CS$ and $\DS$, there is an invertible dg functor
$\CS\otimes \DS \xrightarrow{\cong} \DS\otimes \CS$ sending $(X,Y) \mapsto
(Y,X)$ and $f\otimes g\mapsto (-1)^{\ip{\deg(f),\deg(g)}} g\otimes f$. 
In this way, the dg bicategory of dg categories inherits a symmetric monoidal structure
from the symmetric monoidal structure on $\dgMod{\Grp}{\cring}$. 
By a further step of enrichment, we also obtain a symmetric monoidal structure on dg bicategories: 
if $\mathbf{C}$ and $\mathbf{D}$ are dg bicategories 
(in particular, they may be ordinary $\cring$-linear bicategories),
then we let $\mathbf{C}\otimes \mathbf{D}$ be the dg bicategory wherein
\begin{itemize}
\item objects are pairs $(p,q)$ where $p$ is an object of $\mathbf{C}$ and $q$ is an object of $\mathbf{D}$,
\item the 1-morphism dg category from $(p,q)$ to $(p',q')$ is
\[
(\mathbf{C}\otimes \mathbf{D})_{(p,q)}^{(p',q')} :=\mathbf{C}_p^{p'}\otimes \mathbf{D}_q^{q'},
\]
where $\mathbf{C}_p^{p'}$ denotes the 1-morphism dg category from $p$ to $p'$, 
and similarly for $\mathbf{D}_q^{q'}$, and
\item the horizontal composition $\hComp$ of $1$-morphisms in $\mathbf{C}\otimes \mathbf{D}$ 
is given component-wise, i.e. $(X',Y') \hComp (X,Y) = (X' \hComp X , Y' \hComp Y)$.

\end{itemize}

\subsection{Bar-Natan categories as Temperley--Lieb categorification}
	\label{ss:BN}

We recall the definition of the (pre-additive) Bar-Natan categories for surfaces
from \cite[Definition 3.6]{HRW3}. 
See \cite{BN2} for the original construction, 
which here corresponds to the case where $\Sigma$ is a disk.

\begin{definition}
	\label{def:BN}
Let $\Sigma$ be an orientable surface with (possibly empty) boundary and let
$\mathbf{p} \subset \partial \Sigma$ be finite. 
The (pre-additive) \emph{Bar-Natan category} is the
$\Z$-graded $\cring$-linear category $\bn(\Sigma ; \mathbf{p})$ defined as follows.
Objects in $\bn(\Sigma ; \mathbf{p})$ are
smoothly embedded $1$-manifolds $C \subset \Sigma$ with boundary
$\partial C = \mathbf{p}$ meeting $\partial \Sigma$ transversely. Given objects
$C_1,C_2$, $\Hom_{\bn}(C_1,C_2)$ is the $\Z$-graded $\cring$-module spanned by
embedded orientable cobordisms $W \subset \Sigma \times [0,1]$ 
with corners (when $\mathbf{p} \neq \emptyset$) from
$C_1$ to $C_2$, modulo the following local relations:
\begin{equation}
	\label{eq:BNrels}
\begin{tikzpicture} [fill opacity=0.2,anchorbase, scale=.375]
	\path[fill=red, opacity=.2] (1,0) arc[start angle=0, end angle=180,x radius=1,y radius=.5] 
		to (-1,4) arc[start angle=180, end angle=0,x radius=1,y radius=.5] to (1,0);
	\path[fill=red, opacity=.2] (1,0) arc[start angle=360, end angle=180,x radius=1,y radius=.5] 
		to (-1,4) arc[start angle=180, end angle=360,x radius=1,y radius=.5] to (1,0);
	\draw [very thick] (0,4) ellipse (1 and 0.5);
	\draw [very thick] (0,0) ellipse (1 and 0.5);
	\draw[very thick] (1,4) -- (1,0);
	\draw[very thick] (-1,4) -- (-1,0);
\end{tikzpicture}
\, = \,
\begin{tikzpicture} [fill opacity=0.2,anchorbase, scale=.375,rotate=180]
	\path[fill=red,opacity=.2] (1,4) arc[start angle=0, end angle=180,x radius=1,y radius=.5] 
		to [out=270,in=180] (0,2.5) to [out=0,in=270] (1,4);
	\path[fill=red,opacity=.2] (1,4) arc[start angle=360, end angle=180,x radius=1,y radius=.5] 
		to [out=270,in=180] (0,2.5) to [out=0,in=270] (1,4);
	\draw[very thick] (0,4) ellipse (1 and 0.5);
	\draw[very thick] (-1,4) to [out=270,in=180] (0,2.5) to [out=0,in=270] (1,4);
	\path[fill=red,opacity=.2] (1,0) arc[start angle=0, end angle=180,x radius=1,y radius=.5] 
		to [out=90,in=180] (0,1.5) to [out=0,in=90] (1,0);
	\path[fill=red,opacity=.2] (1,0) arc[start angle=360, end angle=180,x radius=1,y radius=.5] 
		to [out=90,in=180] (0,1.5) to [out=0,in=90] (1,0);
	\draw[very thick] (0,0) ellipse (1 and 0.5);
	\draw[very thick] (-1,0) to [out=90,in=180] (0,1.5) to [out=0,in=90] (1,0);
	\node[opacity=1] at (0,1) {\footnotesize$\bullet$};
\end{tikzpicture}
+
\begin{tikzpicture} [fill opacity=0.2,anchorbase, scale=.375]
	\path[fill=red,opacity=.2] (1,4) arc[start angle=0, end angle=180,x radius=1,y radius=.5] 
		to [out=270,in=180] (0,2.5) to [out=0,in=270] (1,4);
	\path[fill=red,opacity=.2] (1,4) arc[start angle=360, end angle=180,x radius=1,y radius=.5] 
		to [out=270,in=180] (0,2.5) to [out=0,in=270] (1,4);
	\draw[very thick] (0,4) ellipse (1 and 0.5);
	\draw[very thick] (-1,4) to [out=270,in=180] (0,2.5) to [out=0,in=270] (1,4);
	\path[fill=red,opacity=.2] (1,0) arc[start angle=0, end angle=180,x radius=1,y radius=.5] 
		to [out=90,in=180] (0,1.5) to [out=0,in=90] (1,0);
	\path[fill=red,opacity=.2] (1,0) arc[start angle=360, end angle=180,x radius=1,y radius=.5] 
		to [out=90,in=180] (0,1.5) to [out=0,in=90] (1,0);
	\draw[very thick] (0,0) ellipse (1 and 0.5);
	\draw[very thick] (-1,0) to [out=90,in=180] (0,1.5) to [out=0,in=90] (1,0);
	\node[opacity=1] at (0,1) {\footnotesize$\bullet$};
\end{tikzpicture}
\, , \quad
\begin{tikzpicture}[anchorbase, scale=.375]
	\path [fill=red,opacity=0.3] (0,0) circle (1);
	\draw (-1,0) .. controls (-1,-.4) and (1,-.4) .. (1,0);
	\draw[dashed] (-1,0) .. controls (-1,.4) and (1,.4) .. (1,0);
	\draw[very thick] (0,0) circle (1);
\end{tikzpicture}
= 0
\, , \quad
\begin{tikzpicture}[anchorbase, scale=.375]
	\path [fill=red,opacity=0.3] (0,0) circle (1);
	\draw (-1,0) .. controls (-1,-.4) and (1,-.4) .. (1,0);
	\draw[dashed] (-1,0) .. controls (-1,.4) and (1,.4) .. (1,0);
	\draw[very thick] (0,0) circle (1);
	\node at (0,0.6) {\footnotesize$\bullet$};
\end{tikzpicture}
= 1
\, , \quad
\begin{tikzpicture}[fill opacity=.3, scale=.5, anchorbase]
	\filldraw [very thick,fill=red] (-1,-1) rectangle (1,1);
	\node [opacity=1] at (0,-.25) {$\bullet$};
	\node [opacity=1] at (0,.25) {$\bullet$};
	\end{tikzpicture}
= 0 \, .
\end{equation}
 The degree of a cobordism with corners $W \colon C_1 \to C_2$ is given
by $\deg(W) = \frac{1}{2}|\mathbf{p}| - \chi(W)$, and a dot on a surface is used
as shorthand for taking a connect sum with a torus at that point and multiplying
by $\frac{1}{2}$. For example, 
\begin{equation}
	\label{eq:dotdef}
\begin{tikzpicture} [fill opacity=0.2,anchorbase, scale=.375]
	\path[fill=red,opacity=.2] (1,4) arc[start angle=0, end angle=180,x radius=1,y radius=.5] 
		to [out=270,in=180] (0,2.5) to [out=0,in=270] (1,4);
	\path[fill=red,opacity=.2] (1,4) arc[start angle=360, end angle=180,x radius=1,y radius=.5] 
		to [out=270,in=180] (0,2.5) to [out=0,in=270] (1,4);
	\draw[very thick] (0,4) ellipse (1 and 0.5);
	\draw[very thick] (-1,4) to [out=270,in=180] (0,2.5) to [out=0,in=270] (1,4);
	\node[opacity=1] at (0,3) {\footnotesize$\bullet$};
\end{tikzpicture}
:=
\frac{1}{2}
\begin{tikzpicture}[scale=.5,anchorbase]
\begin{scope}
    \clip (-.65,4) arc[start angle=180, end angle=360,x radius=.65,y radius=.25]
    	to (1.25,4) to (1.25,.4) to (-1.25,.4) to  (-1.25,4) to (-.65,4);
\fill[red,opacity=.3] (0,2.5) ellipse (1 and 2);
\draw[very thick] (0,2.5) ellipse (1 and 2);
\fill[white] (0,3) to [out=300,in=60] (0,2) to [out=120,in=240] (0,3);
\draw[very thick] (0,3) to [out=300,in=60] (0,2);
\draw[very thick] (0.1,1.8) to [out=125,in=235] (0.1,3.2);
\end{scope}
\fill[red,opacity=.2] (0,4) ellipse (.65 and .25);
\draw[very thick] (0,4) ellipse (.66 and .25);
\end{tikzpicture} \, .
\end{equation}
The (additive) \emph{Bar-Natan category} is the graded additive completion
$\BN(\Sigma ; \mathbf{p}) := \Mat\big(\bn(\Sigma ; \mathbf{p})\big)$.
\end{definition}

The following categories will play a special role.
For $n\in \N$, 
let $\mathbf{p}^n \subset (0,1)$ denote a chosen set of $n$ distinct points 
(that is symmetric about $1/2$)
and consider
\begin{equation}
	\label{eq:BNmn}
\begin{aligned}
\bn_m^n &:= \bn \big( [0,1]^2 ; \mathbf{p}^m \times \{0\} \cup \mathbf{p}^n \times \{1\} \big)\quad, \quad 
\BN_m^n &= \BN \big( [0,1]^2 ; \mathbf{p}^m \times \{0\} \cup \mathbf{p}^n \times \{1\} \big) 
\end{aligned}
\end{equation}
the (pre-additive and additive) Bar-Natan categories of planar $(m,n)$-tangles
in $[0,1]^2$ with (dotted) cobordisms in $[0,1]^3$ as morphisms between them.
The categories $\bn_m^n$ possess the following structure, which is then induced
on $\BN_m^n$ as well: 
\begin{itemize}

\item A covariant functor  $r_x\colon \bn^n_m\rightarrow \bn^n_m$ 
	which reflects tangles and cobordisms in the first coordinate.
	\item A covariant functor $r_y \colon \bn^n_m \rightarrow \bn^m_n$ 
	which reflects tangles and cobordisms in the second coordinate.
\item A contravariant functor $r_z \colon \bn^n_m \rightarrow \bn^n_m$ which is
	the identity on tangles in $\bn^n_m$ and reflects cobordisms in the third
	coordinate. On $\BN_m^n$, it also inverts grading shifts on objects.
\item A horizontal composition functor $\star\colon \bn^n_m\otimes \bn^m_k\rightarrow \bn^n_k$ 
	induced by composition of tangles.
\item An identity tangle $\one_n$ in $\bn^n_n$ for each $n$, the units for the horizontal composition.
\item An external tensor product $\boxtimes \colon \bn^n_m \otimes \bn_k^\ell \rightarrow \bn_{m+k}^{n+\ell}$ 
	induced by placing tangles side-by-side. 
\end{itemize}
The reflection functors commute and we will write $r_{xy}$, $r_{xz}$, $r_{yz}$, and $r_{xyz}$ for their composites.
For additional details, see \cite[Section 4.2]{HRW3}.
The following is folklore.

\begin{prop}\label{prop:BN2cat}
The operations $\star$ and $\boxtimes$ endow 
$\bn := \coprod_{m,n \geq 0} \bn_m^n$  and 
$\BN := \coprod_{m,n \geq 0} \BN_m^n$ 
with the structure of a monoidal bicategory. \qed
\end{prop}

The endomorphism category $\BN^0_0$ of the tensor unit $\mathbf{p}^0 \in \BN$
is naturally a braided monoidal category. In fact it is symmetric. The following is well-known.

\begin{lemma}
	\label{BN:equiv}
	The representable functor 
	\[\BN^0_0 \to \gMod{\Z}{\cring} \, , \quad X \mapsto \Hom_{\BN^0_0}(\emptyset,X) \]
	is an equivalence of braided monoidal categories. As a consequence, the natural braiding on $\BN^0_0$ is symmetric.
\end{lemma}

\begin{definition}
	\label{def:Kh}
	From now onwards, we will use the notation $\KhEval{-}$ for the representable functor from Lemma~\ref{BN:equiv} 
	and also for its pre-composition with the inclusion $\bn_0^0\hookrightarrow \BN_0^0$.
\end{definition}

On the level of objects $L\in \bn^0_0$, i.e.~planar $1$-manifolds, $\KhEval{L}$
indeed computes Khovanov homology, as suggested by the notation. Moreover, the
graded $\cring$-module $\KhEval{L}$ is free with a \emph{standard basis} parametrized by
$\{1,\mathsf{x}\}^{\pi_0(L)}$, i.e. labelings of the connected components of $L$ by
symbols $1$ or $\mathsf{x}$. These basis elements are realised by dotted cobordisms in
$\Hom_{\bn^0_0}(\emptyset,L)$ consisting of a disjoint union of cup cobordisms
(for components labeled $1$) and dotted cup cobordisms (for components labeled $\mathsf{x}$). 
By passing to the graded additive closure, we also have standard bases of
$\KhEval{X}$ for every $X\in \BN^0_0$

\begin{lemma} For any $m\in \N_0$, there is a canonical natural
equivalence $\mathrm{sph}\colon \mathrm{close}_R \To \mathrm{close}_L \colon \BN^m_m \to \BN^0_0$ between
the right- and left-closure functors
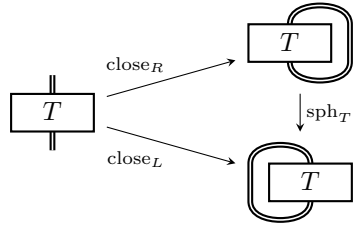
\[
	\begin{tikzcd}
		&& \begin{tikzpicture}[anchorbase,xscale=-1]
			\draw[thick,double] (0,-.25) to (0,.25) \ul (-.4,.5) \ld (-.8,.25) to (-.8,-.25) \dr (-.4,-.5) \ru (0,-.25); 
			\draw[thick, fill=white] (-.55,-.25) rectangle (.55,.25);
			\node at (0,-.1) {$T$};
		\end{tikzpicture}
		\arrow[dd, "\mathrm{sph}_T"]
		\\[-30pt]
		\begin{tikzpicture}[anchorbase,xscale=-1]
			\draw[thick,double] (0,-.5) to (0,.5); 
			\draw[thick, fill=white] (-.55,-.25) rectangle (.55,.25);
			\node at (0,-.1) {$T$};
		\end{tikzpicture}
		\arrow[urr, "\mathrm{close}_R "]
		\arrow[drr, "\mathrm{close}_L",swap]&&
	\\[-30pt]
	&&
	\begin{tikzpicture}[anchorbase,xscale=1]
		\draw[thick,double] (0,-.25) to (0,.25) \ul (-.4,.5) \ld (-.8,.25) to (-.8,-.25) \dr (-.4,-.5) \ru (0,-.25); 
		\draw[thick, fill=white] (-.55,-.25) rectangle (.55,.25);
		\node at (0,-.1) {$T$};
	\end{tikzpicture}
	\end{tikzcd}	
		\]
whose components are the isomorphisms, which under the equivalence $\KhEval{-}$, are induced by the evident canonical bijection of standard basis elements.
\end{lemma}

After having discussed the sphericality, we consider aspects of pivotality.

\begin{lemma}
	\label{lem:fullrotation}
	Let $N\in \N_0$, then the operation that sends a cap tangle $T\in \BN^0_N$
	to its rotation $\mathrm{rot}(T)$ by one click 
	
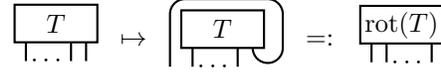
\[
		\begin{tikzpicture}[anchorbase]
		\draw[thick] 	
		(-.4,-.25) to (-.4,-.5) 
		(.2,-.25) to (.2,-.5)
		(.4,-.25) to (.4,-.5);
		\draw[thick, fill=white] (-.55,-.25) rectangle (.55,.25);
		\node at (0,0) {$T$};
		\node  at (0-.1,-.5) {$\cdots$};
		\end{tikzpicture}
		\;\;\mapsto\;\;
		\begin{tikzpicture}[anchorbase]
		\draw[thick] 	
		(-.4,-.25) to (-.4,-.5) 
		(.2,-.25) to (.2,-.5)
		(.4,-.25) \dr (.6,-.45) \ru (.8,-.25) to (.8,.2) \ul (.6,.4) to (-.5,.4) \ld (-.7,.2)  to (-.7,-.5) ;
		\draw[thick, fill=white] (-.55,-.25) rectangle (.55,.25);
		\node at (0,0) {$T$};
		\node  at (0-.1,-.5) {$\cdots$};
		\end{tikzpicture}
		\;\;=:\;\;
		\begin{tikzpicture}[anchorbase]
		\draw[thick] 	
		(-.4,-.25) to (-.4,-.5) 
		(-.2,-.25) to (-.2,-.5)
		(.4,-.25) to (.4,-.5);
		\draw[thick, fill=white] (-.55,-.25) rectangle (.55,.25);
		\node at (0,0) {$\mathrm{rot}(T)$};
		\node  at (0+.1,-.5) {$\cdots$};
		\end{tikzpicture}
		\]
	extends to an
	endofunctor $\mathrm{rot}$ on $\BN^0_N$. Moreover, there exists a canonical natural isomorphism 
	\[
		\id \xRightarrow{\mathrm{twist}_+ }\mathrm{rot}^N
		\]
		of endofunctors of $\BN^0_N$, whose target can be interpreted as the
		$2\pi$-rotation. In particular, $\mathrm{rot}$ is an auto-equivalence.
		Analogous results hold for the opposite rotation $\mathrm{rot}_-$ and an associated canonical natural isomorphism 
		\[
			\id \xRightarrow{\mathrm{twist}_- }\mathrm{rot}_-^N
			\]
		as well as for rotations of objects in any $\BN^M_N$ for $M,N\in \N_0$.
\end{lemma}

\begin{proposition}
	\label{prop:sphericality} 
For $N\in \N_0$ and a pair of
objects $S, T\in \BN^0_N$ we consider the composite isomorphism:
\[\phi\colon T\star r_y(S) =
			\begin{tikzpicture}[anchorbase]
		\draw[thick] 	(-.4,-.25) to (-.4,-.75) 
						(.2,-.25) to (.2,-.75)
						(.4,-.25) to (.4,-.75); 
		\draw[thick, fill=white] (-.55,-.25) rectangle (.55,.25);
			\node at (0,0) {$T$};
			\draw[thick, fill=white] (-.55,-.75) rectangle (.55,-1.25);
			\node at (0,-1) {$r_y(S)$};
		\node  at (0-.1,-.5) {$\cdots$};
	\end{tikzpicture}
	\;
	\cong 
	\;
	\begin{tikzpicture}[anchorbase]
						\draw[thick] 	
						(-.4,-.25) to (-.4,-.75) 
						(.2,-.25) to (.2,-.75)
						(.4,-.25)\dr (.6,-.45) \ru (.8,-.25) to (.8,.2) \ur (1,.4) \rd (1.2,.2)  to (1.2,-1.2) \dl (1,-1.4) \lu (0.8,-1.2) to (0.8,-0.75) \ul (.6,-.55)  \ld (.4,-.75);
		\draw[thick, fill=white] (-.55,-.25) rectangle (.55,.25);
			\node at (0,0) {$T$};
			\draw[thick, fill=white] (-.55,-.75) rectangle (.55,-1.25);
			\node at (0,-1) {$r_y(S)$};
		\node  at (0-.1,-.5) {$\cdots$};
	\end{tikzpicture}
\xrightarrow{\mathrm{sph}}
\begin{tikzpicture}[anchorbase]
	\draw[thick] 	
	(-.4,-.25) to (-.4,-.75) 
	(.2,-.25) to (.2,-.75)
	(.4,-.25) \dr (.6,-.45) \ru (.8,-.25) to (.8,.2) \ul (.6,.4) to (-.5,.4) \ld (-.7,.2)  to (-.7,-1.2) \dr (-.5,-1.4) to (.6,-1.4) \ru (0.8,-1.2) to (0.8,-0.75) \ul (.6,-.55)  \ld (.4,-.75);
\draw[thick, fill=white] (-.55,-.25) rectangle (.55,.25);
\node at (0,0) {$T$};
\draw[thick, fill=white] (-.55,-.75) rectangle (.55,-1.25);
\node at (0,-1) {$r_y(S)$};
\node  at (0-.1,-.5) {$\cdots$};
\end{tikzpicture}
= \mathrm{rot}(T)\star r_y(\mathrm{rot}(S)) \] 
Upon iterating such morphisms $N$ times and using the natural isomorphisms from Lemma~\ref{lem:fullrotation} we have:
		\[(\mathrm{twist}^{-1}_+(T)\star r_y(\mathrm{twist}^{-1}_+(S)))\circ \phi^N = \id_{T\star S}.\] 
	\end{proposition}

\begin{rem}\label{rem:2pirot} In the following, we will consistently
	suppress the difference between objects of $\bn^n_m$ or $\BN^n_m$ that only
	differ by the natural isomorphisms $\mathrm{twist}_{\pm}$. 
\end{rem}

For a careful study of monoidal bicategories with duals, including pivotality and sphericality structures, we refer to \cite{barrett2018gray}.

\subsection{Khovanov's rings}

Consider the hom-functor of $\bn^n_m$
 \begin{align*}
 	\bn^n_m \otimes (\bn^n_m)^\op &\longrightarrow \gMod{\Z}{\cring}\\
 		(a,b) & \longmapsto \Hom_{\bn^n_m}(b,a) 
 \end{align*}
where $a,b \in \bn^n_m$. By duality, this hom functor is naturally isomorphic to the composite
\[
	\begin{tikzcd}[row sep=.5cm]
		\bn^n_m \otimes (\bn^n_m)^\op 
		\arrow[r, "\id \otimes r_{xz}"] 
		& 
		\bn^n_m \otimes (\bn^n_m)^{x\op}
		\arrow[r] 
		&
		\bn^0_0
		\arrow[r, "q^{\frac{1}{2}(n+m)}\KhEval{-}"] 
		&
		\gMod{\Z}{\cring}
		\\
		(a,b)
		\arrow[mapsto,r] 
		& 
		(a, r_x(b))
		\arrow[mapsto,r] 
		&
		\begin{tikzpicture}[anchorbase]
			\draw[thick] (.1,-.25) rectangle (.9,.25); 
			\draw[thick] (-.1,-.25) rectangle (-.9,.25); 
			\node at (.5,-.1) {\small$r_x(b)$};
			\node at (-.5,-.1) {$a$};
			\draw[thick, double] (-.5,.25) to[out=90,in=180]  (0,.55) to[out=0,in=90] (.5,.25);
			\draw[thick, double] (-.5,-.25) to[out=-90,in=180]  (0,-.55) to[out=0,in=-90] (.5,-.25);
			\node at (0,.70) {\scriptsize $n$};
			\node at (0,-.75) {\scriptsize $m$};
			\end{tikzpicture}
		\arrow[mapsto,r] 
		&
		q^{\frac{1}{2}(n+m)} \Kh\left(\begin{tikzpicture}[anchorbase]
			\draw[thick] (.1,-.25) rectangle (.9,.25); 
			\draw[thick] (-.1,-.25) rectangle (-.9,.25); 
			\node at (.5,-.1) {\small$r_x(b)$};
			\node at (-.5,-.1) {$a$};
			\draw[thick, double] (-.5,.25) to[out=90,in=180]  (0,.55) to[out=0,in=90] (.5,.25);
			\draw[thick, double] (-.5,-.25) to[out=-90,in=180]  (0,-.55) to[out=0,in=-90] (.5,-.25);
			\node at (0,.70) {\scriptsize $n$};
			\node at (0,-.75) {\scriptsize $m$};
			\end{tikzpicture}\right)
	\end{tikzcd}
\]
The latter planar expression is important and deserves separate notation. For $a,b\in \bn^n_m$ we thus set
\begin{equation}
	\label{eq:planarHom}
\Khring(a,b) := q^{\frac{1}{2}(n+m)} \Kh\left(\begin{tikzpicture}[anchorbase]
\draw[thick] (.1,-.25) rectangle (.9,.25); 
\draw[thick] (-.1,-.25) rectangle (-.9,.25); 
\node at (.5,0) {\small$r_x(b)$};
\node at (-.5,0) {$a$};
\draw[thick, double] (-.5,.25) to[out=90,in=180]  (0,.55) to[out=0,in=90] (.5,.25);
\draw[thick, double] (-.5,-.25) to[out=-90,in=180]  (0,-.55) to[out=0,in=-90] (.5,-.25);
\node at (0,.70) {\scriptsize $n$};
\node at (0,-.70) {\scriptsize $m$};
\end{tikzpicture}\right) \, .
\end{equation}
Since $\Khring(a,b) \cong \Hom_{\BN^n_m}(b,a)$, the former inherit a
multiplicative structure with unit from the composition and identity morphisms
in the latter. These admit a planar description as follows. Let $a\in \bn^n_m$
be given. There are canonical degree zero maps $\cring\rightarrow H(a,a)$ and
$H(a,b)\otimes H(b,c)\rightarrow H(a,c)$ (induced by saddle
cobordisms and tubes), which give $\bigoplus_{a,b}H(a,b)$ the
structure of a graded locally unital algebra. By construction, we then have:

\begin{lemma}
	Via the duality isomorphisms of the form $\Khring(a,b) \cong \Hom_{\BN^n_m}(b,a)$, the
	multiplication $\Khring(a,b)\otimes \Khring(b,c)\rightarrow \Khring(a,c)$
	agrees with composition of morphisms in $\bn^n_m$.\qed
	\end{lemma}

\begin{definition}
	\label{def:Khrings}
We refer to the graded locally unital algebra
$\Khring^n_m:=\bigoplus_{a,b}\Khring(a,b)$, where $a,b\in \bn^n_m$, as the
\emph{big Khovanov ring}.  Let $I^n_m$ be a set of representatives of isotopy
classes of planar $(n,m)$-tangles having no closed (circle) components, i.e.~$I^n_m$
is a set of representatives of isomorphism classes of indecomposable objects of
$\bn^n_m$.  Then the \emph{small (or usual)} Khovanov ring may be identified
with subalgebra $\bigoplus_{a,b\in I^n_m}\Khring(a,b)$.
\end{definition}

\begin{remark}
	The big Khovanov ring $\bigoplus_{a,b}\Khring(a,b)$ is just the graded
$\cring$-linear category $\bn^n_m$, regarded as a graded locally unital algebra. In
doing so, we forget that $\bn^n_m$ appears as a hom category in a bicategory.
The main reason we sometimes prefer to work with the big Khovanov ring instead
of the small one, however, is that the set of objects for the larger ring is
closed under tangle composition.
\end{remark}

\subsection{The bar complex and Rozansky's projector}
\label{s:Rozanskybar}
We now consider bimodules over the big Khovanov rings $\Khring^n_m$. 

\begin{definition}
Given integers $m_1,m_2,n_1,n_2 \in \Z_{\geq 0}$ with $n_i\equiv m_i$ (mod 2), 
let $\BS^{n_1,n_2}_{m_1,m_2}$ denote the category of finitely generated, graded $(\Khring^{n_1}_{m_1},\Khring^{n_2}_{m_2})$-bimodules 
which are projective from the right and left. 
Let $\BS$ be the 2-category with objects pairs of integers $(m,n)$ with $n\equiv m$ (mod 2), 
1-morphism categories given by $\BS^{n_1,n_2}_{m_1,m_2}$,
and with composition of 1-morphisms in $\BS$ is given by tensoring over the appropriate ring $\Khring^n_m$.
\end{definition}

The usual Yoneda embeddings give functors
\[
  (\bn^n_m)^{\op}\rightarrow \Khring^n_m\text{-mod} \, , \quad \bn^n_m\rightarrow \text{mod-}\Khring^n_m .
\]
sending $a \mapsto \Hom_{\bn^n_m}(a,-) \cong \Khring(-,a)$ and 
$a\mapsto \Hom_{\bn^n_m}(-,a) \cong \Khring(a,-)$, respectively. 
Using \eqref{eq:planarHom}, we obtain the following planar forms:
\[
	\Khring(-,a) := q^{\frac{1}{2}(n+m)} \Kh\left(\begin{tikzpicture}[anchorbase]
		\draw[thick] (.05,-.25) rectangle (.95,.25); 
		\draw[thick] (-.1,-.25) rectangle (-.9,.25); 
		\node at (.5,0) {$r_x(a)$};
		\node at (-.5,0) {$-$};
		\draw[thick, double] (-.5,.25) to[out=90,in=180]  (0,.55) to[out=0,in=90] (.5,.25);
		\draw[thick, double] (-.5,-.25) to[out=-90,in=180]  (0,-.55) to[out=0,in=-90] (.5,-.25);
		\node at (0,.70) {\scriptsize $n$};
		\node at (0,-.70) {\scriptsize $m$};
		\end{tikzpicture}\right)
\, , \quad
\Khring(a,-) := q^{\frac{1}{2}(n+m)} \Kh\left(\begin{tikzpicture}[anchorbase]
	\draw[thick] (.05,-.25) rectangle (.95,.25); 
	\draw[thick] (-.1,-.25) rectangle (-.9,.25); 
	\node at (.5,0) {$r_x({\!-\!})$};
	\node at (-.5,0) {$a$};
	\draw[thick, double] (-.5,.25) to[out=90,in=180]  (0,.55) to[out=0,in=90] (.5,.25);
	\draw[thick, double] (-.5,-.25) to[out=-90,in=180]  (0,-.55) to[out=0,in=-90] (.5,-.25);
	\node at (0,.70) {\scriptsize $n$};
	\node at (0,-.70) {\scriptsize $m$};
	\end{tikzpicture}\right) \, .
\]
We may regard these Yoneda modules as bimodules, i.e.~as 1-morphisms in $\BS$:
\[
\Khring(-,a)\in \BS^{n,0}_{m,0} \, ,\quad \Khring(a,-)\in \BS^{0,n}_{0,m}.
\]

\begin{definition}
	\label{def:bar}
Let $\Bar^n_m(-,-) \in \Ch^-(\BS^{n,n}_{m,m})$ denote the bar complex of $\Khring^n_m$.  
Precisely, given objects $b,b' \in \bn_m^n$, $\Bar^n_m(b,b')$ is the complex defined as follows:
\[
\Bar^n_m(b,b') := \tw_\d\left(\bigoplus_{r\geq 0} t^{-r} \bigoplus_{a_0,a_1,\ldots,a_{r}} 
	\Khring(b,a_0)\otimes \Big(H(a_0,a_1)\otimes\cdots\otimes H(a_{r-1},a_r)\Big)\otimes \Khring(a_r,b')\right).
\]
where $\d$ is the usual ``bar'' differential.
\end{definition}

An element of $\Bar^n_m(b,b')$ in cohomological degree $-r$ is a linear combination of elements of the form 
\[
g\otimes f_1\otimes \cdots \otimes f_r\otimes g'
	\in \Khring(b,a_0)\otimes \Khring(a_0,a_1)\otimes \cdots \otimes \Khring(a_{r-1},a_r)\otimes \Khring(a_r,b') \, .
	\] 
Following standard conventions, we will denote such elements by $g(f_1|\cdots | f_r)g'$.
We will also use the notation $\Bar^n_m$ (without the arguments)
to denote the Yoneda preimage: 
\begin{equation}
	\label{eq:BarY}
	\Bar^n_m := \tw_\d\left(\bigoplus_{r\geq 0} t^{-r} \bigoplus_{a_0,a_1,\ldots,a_{r}}  
		a_0
		\otimes \Big(H(a_0,a_1)\otimes\cdots\otimes H(a_{r-1},a_r)\Big)\otimes a_r\right)  \, ,
	\end{equation}
which is an object of $\Ch^-((\bn_m^n)^{\op}\otimes \bn^n_m)$.

The composition of tangles (i.e.~the composition $\hComp$ of 1-morphisms in $\bn$)
endows the categories $\BS^{n_1,n_2}_{n_1,n_2}$ with an additional composition, 
denoted $\ostar$, which we study next.  First, consider the composition $\bn^n_m\otimes 
\bn^m_k\rightarrow \bn^n_k$ sending $(a,b) \mapsto a\star b$.  
Under the correspondence between $\cring$-linear categories and $\cring$-algebras,
this functor corresponds to an algebra map $\Khring^n_m\otimes \Khring^m_k \rightarrow \Khring^n_k$.
Thus, we can view $\Khring^n_k$ as a (left or right) module over $\Khring^n_m\otimes \Khring^m_k$.

\begin{definition}
Given $M\in \BS^{n,n'}_{m,m'}$ and $N\in \BS^{m,m'}_{k,k'}$, let
\[
M \ostar N := \Khring_{k}^{n} \otimes_{\Khring_{m}^{n} \otimes \Khring_{k}^{m}} 
	(M\otimes N)\otimes_{\Khring_{m'}^{n'} \otimes \Khring_{k'}^{m'}} \Khring_{k'}^{n'}
	\in \BS^{n,n'}_{k,k'} \, .
\]
\end{definition}

A mild extension of \cite[\S 6.1]{2002.06110}, 
which itself is an adaptation of \cite[\S 5]{EML} to the categorical setting, 
defines a chain map
\begin{equation}
	\label{eq:EZ}
\big(\Bar^n_m(-,-) \ostar \Bar^m_k(-,-)\big) \longrightarrow \Bar^n_k(-,-)
\end{equation}
called the \emph{Eilenberg--Zilber shuffle product}. 
Explicitly, when evaluated at the pair $(b,b')$ of objects in $\Khring^n_k$, 
elements of the domain of \eqref{eq:EZ} are sums of terms of the form
$F(f_1|\cdots | f_r) \otimes (g_1|\cdots | g_s) G$ for
$F \in H(b, a_0 \hComp a_0')$, 
$f_i \in H(a_{i-1},a_i)$, 
$g_i \in H(a_{i-1}',a_i')$,
and $G \in H(a_r \hComp a_s',b')$.
The shuffle product of such an element
is a sum of terms of the form $\pm F (e_1 | \cdots | e_{r+s}) G$
where $\{e_1,\ldots,e_r\}$ is a shuffle of 
$\{f_1 \hComp \id,\ldots,f_r \hComp \id\}$ and $\{\id \hComp g_1,\ldots,\id \hComp g_s\}$ 
for appropriate identity morphisms.
The shuffle product is clearly unital, 
and \cite[Theorem 5.2]{EML} implies that it is (strictly) associative.
Taking the Yoneda preimage yields the following.

\begin{proposition}\label{prop:Bar is algebra}
The Eilenberg--Zilber shuffle product defines a chain map 
$\mu_m\colon \Bar^n_m \hComp \Bar^m_k \rightarrow \Bar^n_k$
in $\Ch^{-}((\bn_k^n)^{\op} \otimes \bn_k^n)$
which is strictly associative and unital. \qed
\end{proposition}

Later, we will use the following compatibilities of $\Bar_m^n$ with the 
symmetries of the Bar-Natan category $\bn_m^n$ described in \S \ref{ss:BN}.
All follow immediately from the definition of $\Bar^n_m$.

\begin{proposition}[Symmetries of $\Bar^n_m$]
	\label{prop:symmetries of B}
We have:
\begin{enumerate}
\item \label{BS1} $(r_x\otimes r_x)(\Bar^n_m)\cong \Bar^n_m$.
\item \label{BS2} $(r_y\otimes r_y)(\Bar^m_n)\cong \Bar^n_m$.
\item \label{BS3} $\sigma \circ (r_z \otimes r_z)(\Bar^n_m)\cong \Bar^n_m$, 
where $\sigma$ is (the functor on complexes induced from) the
transposition of tensor factors $(\bn^n_m)^{\op}\otimes \bn^n_m \to \bn^n_m \otimes (\bn^n_m)^{\op}$.
\end{enumerate}
Moreover, each of the above is compatible with the multiplication maps $\mu_m$ from Proposition~\ref{prop:Bar is algebra}. \qed
\end{proposition}

We now use \eqref{eq:planarHom} to give a graphical interpretation of $\Bar_m^n$ 
and relate the latter to the Rozansky's {bottom projector}, 
defined in \cite{rozansky2010categorification}.
For integers $m,n$, consider the functor 
$\iota_m^n \colon (\bn^n_m)^\op \otimes \bn^n_m \rightarrow \bn_{m+n}^{m+n}$ 
defined as the composition:
\begin{equation}
	\label{eq:iota}
	\begin{tikzcd}[row sep=.5cm]
		(\bn^n_m)^\op \otimes \bn^n_m 
		\arrow[r, "r_{xz}\otimes \id "] 
		& 
		\bn^n_m \otimes \bn^n_m 
		\arrow[r] 
		&
		\bn_{m+n}^{m+n}
		&
		\\
		(a,b)
		\arrow[mapsto,r] 
		& 
		(r_x(a),b)
		\arrow[mapsto,r] 
		&
		\begin{tikzpicture}[anchorbase]
			\draw[thick, double] (-.75,-1) \ur  (-.35,-.5) to (.35,-.5)  \rd (.75,-1);
			\draw[thick, double] (-.75,1) \dr  (-.35,.5) to (.35,.5) \ru (.75,1);
			\filldraw[white] (-.35,.2) rectangle (.35, .8);
			\filldraw[white ] (-.35,-.1) rectangle (.35, -1);
			\draw[thick] (-.35,-.2) rectangle (.35, -.8);
			\draw[thick] (-.35,.1) rectangle (.35, 1);
			\node at (.9,.75) {\scriptsize $n$};
			\node at (.9,-.75) {\scriptsize $n$};
			\node at (-.9,.75) {\scriptsize $m$};
			\node at (-.9,-.75) {\scriptsize $m$};
			\node[rotate = 270] at (-0.1,.55) {\small$r_x(a)$};
			\node[rotate = 270] at (-0.1,-.5) {$b$};
			\end{tikzpicture}
	\end{tikzcd}
\end{equation}
As usual, we will denote the extension of $\iota_m^n$ to the category of complexes 
$\Ch^-((\bn^n_m)^\op\otimes \bn^n_m) \rightarrow \Ch^-(\bn_{m+n}^{m+n})$ similarly.

\begin{definition}\label{def:bottom proj}
Set $\Bproj_{m+n}:= q^{\frac{1}{2}(m+n)} \iota_m^n(\Bar^n_m)\in \Ch^-(\bn_{m+n}^{m+n})$.
\end{definition}

Comparing with Definition \ref{def:bar}, we see that
\begin{equation}
	\label{eq:Bar-visual}
	\Bar^n_m(b,b')	
	=
	q^{\frac{1}{2}(m+n)} 
	\KhEval{
		\begin{tikzpicture}[anchorbase]
		\draw[thick] (1.1,-.25) rectangle (1.9,.25); 
		\draw[thick] (-.1,-.25) rectangle (-.9,.25); 
		\node at (1.5,0) {\scs$r_x(b')$};
		\node at (-.5,0) {$b$};
		\draw[thick, double] (-.5,.25) to[out=90,in=180]  (0.5,.65) to[out=0,in=90] (1.5,.25);
		\draw[thick, double] (-.5,-.25) to[out=-90,in=180]  (0.5,-.65) to[out=0,in=-90] (1.5,-.25);
		\rtallbox{0.2}{0}{m+n}
		\node at (-.5,.70) {\scriptsize $n$};
		\node at (-.5,-.70) {\scriptsize $m$};
		\node at (1.5,.70) {\scriptsize $n$};
		\node at (1.5,-.70) {\scriptsize $m$};
		\end{tikzpicture}} \, .
\end{equation}
It is straightforward to see that, as the notation suggests, 
$\Bproj_{m+n}$ depends only on the sum $m+n \in \Z_{\geq 0}$ 
up to isomorphism.
The object $(r_{xz}\otimes \id)(\Bar_m^n) \in \Ch^-(\bn^n_m\otimes \bn^n_m)$ appearing in the 
definition of $\Bproj_{m+n}$ will occur sufficiently often that we 
introduce the following shorthand and graphical notation.

\begin{definition}\label{def:tbar}
Set $\tBar^n_m:= (r_{xz}\otimes\id)(\Bar^n_m)$,
which we depict graphically as:
\[
\tBar^n_m =
		\left(
			\begin{tikzpicture}[anchorbase,xscale=-1]
				\draw[thick, double] (0,-0.6) node[below=-2pt]{\scs $m$} 
					to (0,0.6) node[above=-2pt]{\scs $n$};
				\ljf{0}{0}{}
				\end{tikzpicture}	
\, , \,
		\begin{tikzpicture}[anchorbase,scale=1]
			\draw[thick, double] (0,-0.6) node[below=-2pt]{\scs $m$} 
				to (0,0.6) node[above=-2pt]{\scs $n$};
			\ljf{0}{0}{}
		\end{tikzpicture}	
		\right)
\]
\end{definition}

\begin{remark}\label{rmk:einstein}
The above graphical notation is used in the spirit of a
\emph{categorified Einstein summation convention}: 
when appearing in a ``planar evaluation'' 
(e.g.~after applying $\boxtimes \colon \Ch^-(\bn^n_m\otimes \bn^n_m) \to \Ch^-(\bn^{2n}_{2m})$),
paired boxes represent a complex of split tangles, 
i.e.~a \emph{direct sum} of (homological and quantum grading shifts) of pairs of objects from $\bn^n_m$, 
one in each box, related only by the differential. 
\end{remark}

\begin{remark}\label{rmk:purple boxes merge to roz}
Definitions \ref{def:bottom proj} and \ref{def:tbar} give that
\begin{equation}
	\label{eq:purpleprojector}
\tBar^n_m \xmapsto{\boxtimes}
\begin{tikzpicture}[anchorbase,xscale=-1]
				\draw[thick, double] (0,-0.6) node[below=-2pt]{\scs$m$} 
					to (0,0.6) node[above=-2pt]{\scs$n$};
				\ljf{0}{0}{}
				\end{tikzpicture}	
		\quad
		\begin{tikzpicture}[anchorbase,scale=1]
			\draw[thick, double] (0,-0.6) node[below=-2pt]{\scs$m$} 
				to (0,0.6) node[above=-2pt]{\scs$n$};
			\ljf{0}{0}{}
		\end{tikzpicture}
= q^{-\frac{m+n}{2}}
			\begin{tikzpicture}[anchorbase,rotate=-90,scale=.75]
				\draw[thick, double] (-1.25,-.75) node[above=-2pt]{\scs$n$} to [out=0,in=270] (-.5,-.25);
				\draw[thick, double] (-1.25,.75) node[above=-2pt]{\scs$n$} to [out=0,in=90] (-.5,.25);
				\draw[thick, double] (1.25,-.75) node[below=-2pt]{\scs$m$} to [out=180,in=270] (.5,-.25);			
				\draw[thick, double] (1.25,.75) node[below=-2pt]{\scs$m$} to [out=180,in=90] (.5,.25);	
				\draw[thick, fill=white] (-1,.25) rectangle (1, -.25);
				\node[rotate=90] at (0,0) {\scs$\Bproj_{m+n}$};
				\end{tikzpicture} \, .
\end{equation}
i.e.~adjacent, appropriately oriented purple boxes can be replaced by $q^{-\frac{m+n}{2}}\Bproj_{m+n}$.
Note, however, that more general planar evaluations involving $\tBar_m^n$ need not result in nearby purple boxes.
\end{remark}

In fact, there is another relative orientation of purple boxes 
that again yields $q^{-\frac{m+n}{2}} \Bproj_{m+n}$, 
a fact that will be important in \S \ref{s:coarsening}.

\begin{lem}
	\label{lem:revpurple}
For $m,n \geq 0$,
\begin{equation}
	\label{eq:inpurple}
\begin{tikzpicture}[anchorbase,xscale=-1]
				\draw[thick, double] (0,-0.6) node[below=-2pt]{\scs$m$} 
					to (0,0.6) node[above=-2pt]{\scs$n$};
				\rjf{0}{0}{}
				\end{tikzpicture}	
		\quad
		\begin{tikzpicture}[anchorbase,scale=1]
			\draw[thick, double] (0,-0.6) node[below=-2pt]{\scs$m$} 
				to (0,0.6) node[above=-2pt]{\scs$n$};
			\rjf{0}{0}{}
		\end{tikzpicture}
\cong q^{-\frac{m+n}{2}}
			\begin{tikzpicture}[anchorbase,rotate=-90,scale=.75]
				\draw[thick, double] (-1.25,-.75) node[above=-2pt]{\scs$n$} to [out=0,in=270] (-.5,-.25);
				\draw[thick, double] (-1.25,.75) node[above=-2pt]{\scs$n$} to [out=0,in=90] (-.5,.25);
				\draw[thick, double] (1.25,-.75) node[below=-2pt]{\scs$m$} to [out=180,in=270] (.5,-.25);			
				\draw[thick, double] (1.25,.75) node[below=-2pt]{\scs$m$} to [out=180,in=90] (.5,.25);	
				\draw[thick, fill=white] (-1,.25) rectangle (1, -.25);
				\node[rotate=90] at (0,0) {\scs$\Bproj_{m+n}$};
				\end{tikzpicture} \, .
\end{equation}
\end{lem}

\begin{proof}
The left-hand side of \eqref{eq:inpurple} is given by 
$\boxtimes \circ \sigma (\tBar) = \boxtimes \circ \sigma \circ (r_{xz}\otimes\id)(\Bar^n_m)$, 
where $\sigma \colon \Ch^-(\bn_m^n \otimes \bn_m^n) \to \Ch^{-}(\bn_m^n \otimes \bn_m^n)$ 
is the functor induced from the transposition of tensor factors.
We compute
\[
\sigma \circ (r_{xz}\otimes\id) = (\id \otimes r_{xz}) \circ \sigma 
	= (r_{xz}\otimes\id) \circ (r_{xz} \otimes r_{xz}) \circ \sigma 
	= (r_{xz}\otimes\id) \circ (r_{x} \otimes r_{x}) \circ \sigma \circ (r_{z} \otimes r_{z})
\]
and the result follows from Proposition \ref{prop:symmetries of B}.
\end{proof}

We next establish further structure on $\Bproj_{m+n}$.
Recall that the bar complex from Definition \ref{def:bar} possesses a canonical counit,
i.e.~a degree zero chain map $\tilde{\e}_{m}^n \colon \Bar^n_m(b,b') \rightarrow H(b,b')$
given by
\[
g(f_1|\cdots|f_r)g'\mapsto 
\begin{cases} 
	gg'\in H(b,b') & \text{ if $r=0$}\\ 
	0 & \text{else}.
		\end{cases}
	\]
This map is clearly natural in $b$ and $b'$, 
and realizes $\Bar_m^n(-,-)$ as a projective resolution 
of the identity $(\Khring_m^n,\Khring_m^n)$-bimodule.

\begin{definition}\label{def:bproj counit}
The \emph{counit} $\e_{m+n}\colon\Bproj_{m+n}\rightarrow \one_{m+n}$ 
is the map induced from $\tilde{\e}_{m}^n$ by taking Yoneda preimages and applying $\iota_{m}^n$. 
Explicitly, $\e_{m+n}$ is the composition
\begin{equation}
	\label{eq:counitcomp}
\Bproj_{m+n} \longrightarrow q^{\frac{n+m}{2}}\bigoplus_{a} \begin{tikzpicture}[anchorbase]
					\draw[thick] (-.5,-.25) rectangle (.5,.25); 
					\node at (0,0) {\small$r_y(a)$};
					\draw[thick, double] (.25,.25) to (.25,.75);
					\draw[thick, double] (-.25,.25) to (-.25,.75);
					\node at (-.5,.65) {\scriptsize $m$};
					\node at (.5,.65) {\scriptsize $n$};
					\begin{scope}[shift={(0,-.75)}] 
						\draw[thick] (-.5,-.25) rectangle (.5,.25); 
						\node at (0,0) {\small$a$};
						\node at (-.5,-.65) {\scriptsize $m$};
						\node at (.5,-.65) {\scriptsize $n$};
						\draw[thick, double] (.25,-.25) to (.25,-.75);
						\draw[thick, double] (-.25,-.25) to (-.25,-.75);
						\end{scope}
					\end{tikzpicture}
\longrightarrow \one_{m+n}
\end{equation}
where the first map projects onto the degree zero chain object and the second map is the canonical cobordism.  
\end{definition}

Observe that, although projecting onto the degree zero object in $\Bproj_{m+n}$
is not a chain map, the composition in \eqref{eq:counitcomp} is indeed chain
map. The map $\e_{m+n}$ allows us to give the following abstract
characterization of $\Bproj_{m+n}$, which follows from \cite[\S
5.2]{hogancamp2020constructing}; see also \cite{Cooper_2015}.

\begin{prop}
	\label{prop:Punique}
The following properties uniquely characterize $(\Bproj_{m+n},\e_{m+n})$, up to homotopy equivalence.
\begin{enumerate}
\item The chain objects of $\Bproj_{n+m}$ have through-degree zero, 
meaning that they are all direct sums of objects of the form $b \hComp b'$ 
for $b \in \bn_0^{m+n}$ and $b' \in \bn_{m+n}^0$.
\item $\Cone(\e_{m+n}) \star a \simeq 0 \simeq a \star \Cone(\e_{m+n})$ 
for every through-degree zero $a \in \bn^{n+m}_{n+m}$.
\end{enumerate}
Explicitly, if $(P',\e')$ is another pair satisfying (1) and (2), 
then there is a unique-up-to-homotopy chain map $\nu\colon \Bproj_{m+n}\rightarrow P'$ 
such that $\e'\circ \nu\simeq \e_{m+n}$, and this $\nu$ is a homotopy equivalence.
\end{prop}
\begin{proof}
	It is clear from \eqref{eq:iota} that $\Bproj_{m+n}$ satisfies (1), 
	while (2) is simply a reformulation 
	of the fact that $\tilde{\e}_{m}^n$ is a projective resolution.
	The uniqueness follows by considering the diagram
	\[
	\begin{tikzcd}
	\Bproj_{m+n} \hComp P' \ar[rr,"\e_{m+n}\hComp \id"] \ar[d,"\id \hComp \e' "] & & \one_{m+n} \hComp P' \ar[d,"\id \hComp \e' "] \\
	\Bproj_{m+n} \hComp \one_{m+n} \ar[rr,"\e_{m+n}\hComp \id"] & & \one_{m+n}
	\end{tikzcd} \, .
	\]
	Indeed, since all complexes involved are bounded above,
	property (2) implies that the cones of the top and left-most maps are contractible.
	Hence, these maps are homotopy equivalences and 
	we can let $\nu$ be the composition of $\e_{m+n}\hComp \id$ with a (homotopy) inverse to $\id \hComp \e'$.
	\end{proof}

Since Rozansky's bottom projector from \cite{rozansky2010categorification} satisfies the conditions of 
Proposition \ref{prop:Punique}, this immediately gives the following.

\begin{cor}\label{cor:bar is roz}
$\Bproj_{m+n}$ agrees with Rozansky's bottom projector, up to homotopy equivalence. \qed
\end{cor}

\begin{remark}\label{rmk:roz is coalg}
	A similar argument to that in the proof of Proposition \ref{prop:Punique} shows that 
	$\e_{m+n} \hComp \id \colon \Bproj_{m+n} \hComp \Bproj_{m+n} \to \Bproj_{m+n}$ 
	is a homotopy equivalence. 
	Its homotopy inverse and the counit equip $\Bproj_{m+n}$ with the structure of a counital coalgebra:
	\begin{equation}
		\label{eq:coalgebra}
				\begin{tikzpicture}[anchorbase]
					\draw[thick, double] (.25,-.75) node[below=-2pt]{\scriptsize $n$} to (.25,.75);
					\draw[thick, double] (-.25,-.75) node[below=-2pt]{\scriptsize $m$} to (-.25,.75);
					\draw[thick, fill=white] (-.4,-.25) rectangle (.4,.25); 
					\node at (0,0) {\scs$\Bproj_{m+n}$};
					\end{tikzpicture}
					\xrightarrow{(\e_{m+n} \hComp \id)^{-1}} 
					\begin{tikzpicture}[anchorbase]
						\draw[thick, double] (.25,-.75) node[below=-2pt]{\scriptsize $n$} to (.25,.75);
						\draw[thick, double] (-.25,-.75) node[below=-2pt]{\scriptsize $m$} to (-.25,.75);
						\draw[thick, fill=white] (-.4,.1) rectangle (.4,.6); 
						\node at (0,.35) {\scs$\Bproj_{m+n}$};
						\draw[thick, fill=white] (-.4,-.1) rectangle (.4,-.6); 
						\node at (0,-.35) {\scs$\Bproj_{m+n}$};
						\end{tikzpicture} 
	\quad ,\quad
		\begin{tikzpicture}[anchorbase]
			\draw[thick, double] (.25,-.5) node[below=-2pt]{\scriptsize $n$} to (.25,.5);
			\draw[thick, double] (-.25,-.5) node[below=-2pt]{\scriptsize $m$} to (-.25,.5);
			\draw[thick, fill=white] (-.4,-.25) rectangle (.4,.25); 
			\node at (0,0) {\scs$\Bproj_{m+n}$};
			\end{tikzpicture}
			\xrightarrow{\e_{m+n}}
			\begin{tikzpicture}[anchorbase]
			\draw[thick, double] (.25,-.5) node[below=-2pt]{\scriptsize $n$} to (.25,.5);
			\draw[thick, double] (-.25,-.5) node[below=-2pt]{\scriptsize $m$} to (-.25,.5);
				\end{tikzpicture} \, .
		\end{equation}
	The algebra structure $\mu_m
	\colon \Bar^n_m \hComp \Bar^m_k \rightarrow \Bar^n_k$
	from Proposition \ref{prop:Bar is algebra} now gives rise to a locally
	unital algebra structure on $\bigoplus_{m,n} \Bproj_{m+n}$ with respect to a
	``sideways'' (note: \emph{not} $\hComp$) composition of tangles:
	\begin{equation}
		\label{eq:algebra}
	q^{-m} \begin{tikzpicture}[anchorbase]
	\draw[thick, double] (.75,.25) to (.75,.75);
	\draw[thick, double] (-.75,.25) to (-.75,.75);
	\draw[thick, double] (.75,-.25) to (.75,-.75);
	\draw[thick, double] (-.75,-.25) to (-.75,-.75);
	\draw[thick, double] (-.35,.25) to[out=90,in=180]  (0,.55) to[out=0,in=90] (.35,.25);
	\draw[thick, double] (-.35,-.25) to[out=-90,in=180]  (0,-.55) to[out=0,in=-90] (.35,-.25);
	\draw[thick] (.1,-.25) rectangle (.9,.25); 
	\draw[thick] (-.1,-.25) rectangle (-.9,.25); 
	\node at (.5,0) {\scs$\Bproj_{m+n}$};
	\node at (-.5,0) {\scs$\Bproj_{k+m}$};
	\node at (0,.70) {\scriptsize $m$};
	\node at (0,-.70) {\scriptsize $m$};
	\node at (-.95,-.65) {\scriptsize $k$};
	\node at (.95,-.65) {\scriptsize $n$};
	\end{tikzpicture}
	\xrightarrow{\mu_m}
	\begin{tikzpicture}[anchorbase]
	\draw[thick, double] (.25,.25) to (.25,.75);
	\draw[thick, double] (-.25,.25) to (-.25,.75);
	\draw[thick, double] (.25,-.25) to (.25,-.75);
	\draw[thick, double] (-.25,-.25) to (-.25,-.75);
	\draw[thick] (-.4,-.25) rectangle (.4,.25); 
	\node at (0,0) {\scs$\Bproj_{k+n}$};
	\node at (-.5,-.65) {\scriptsize $k$};
	\node at (.5,-.65) {\scriptsize $n$};
	\end{tikzpicture}
	\quad , \quad
	\begin{tikzpicture}[anchorbase]
	\draw[thick, double] (-.25,-.40) to[out=90,in=180]  (0,-.15) to[out=0,in=90] (.25,-.40);
	\draw[thick, double] (-.25,.40) to[out=-90,in=180]  (0,.15) to[out=0,in=-90] (.25,.40);
	\node at (.55,.3) {\scriptsize $m$};
	\node at (.55,-.3) {\scriptsize $m$};
	\end{tikzpicture}
	\xrightarrow{\nu_m}
	q^{-m} 
	\begin{tikzpicture}[anchorbase]
	\draw[thick, double] (.25,.25) to (.25,.75);
	\draw[thick, double] (-.25,.25) to (-.25,.75);
	\draw[thick, double] (.25,-.25) to (.25,-.75);
	\draw[thick, double] (-.25,-.25) to (-.25,-.75);
	\draw[thick] (-.4,-.25) rectangle (.4,.25); 
	\node at (0,0) {\scs$\Bproj_{m+m}$};
	\node at (-.5,-.65) {\scriptsize $m$};
	\node at (.5,-.65) {\scriptsize $m$};
	\end{tikzpicture} \, .
	\end{equation}
	Here, the unit map $\nu_m$ is given by inclusion of the indicated object of $\bn_{2m}^{2m}$ 
	into the homological degree zero term of $\Bproj_{m+m}$.
	More generally, we can consider the composition induced by $\ostar$ on $\bn^{2m}_{2m}$, 
	which is given analogously to the first diagram in \eqref{eq:algebra}; 
	together with the 
	composition $\hComp$, it endows this category with the structure of a \emph{duoidal category}.
	(We omit the definition here, since it will not be needed in the present work.)
	In this setting, $\Bproj_{m+m}$ is a dg bialgebra with respect to the maps in \eqref{eq:algebra} and \eqref{eq:coalgebra}.
	\end{remark}

The counit maps from Definition \ref{def:bproj counit} satisfy the following compatibility 
with the multiplication maps $\mu_m$.

\begin{lemma}
	\label{lem:counitsquare}
The counit maps from Definition \ref{def:bproj counit} form the (strictly)
commutative square depicted on the left in \eqref{fig:counit}.
\begin{equation}
	\label{fig:counit}
	\begin{tikzcd}
	q^{-m} \begin{tikzpicture}[anchorbase]
			\draw[thick, double] (.75,.25) to (.75,.75);
			\draw[thick, double] (-.75,.25) to (-.75,.75);
			\draw[thick, double] (.75,-.25) to (.75,-.75);
			\draw[thick, double] (-.75,-.25) to (-.75,-.75);
			\draw[thick, double] (-.35,.25) to[out=90,in=180]  (0,.55) to[out=0,in=90] (.35,.25);
			\draw[thick, double] (-.35,-.25) to[out=-90,in=180]  (0,-.55) to[out=0,in=-90] (.35,-.25);
			\draw[thick] (.1,-.25) rectangle (.9,.25); 
			\draw[thick] (-.1,-.25) rectangle (-.9,.25); 
			\node at (.5,-.05) {\scs$\Bproj_{m+n}$};
			\node at (-.5,-.05) {\scs$\Bproj_{k+m}$};
			\node at (0,.70) {\scriptsize $m$};
			\node at (0,-.80) {\scriptsize $m$};
			\node at (-.95,-.65) {\scriptsize $k$};
			\node at (.95,-.65) {\scriptsize $n$};
				\end{tikzpicture}
	\arrow[rr,"\mu_m"] \arrow[d,"\e \ostar \e"] 
	& &
		\begin{tikzpicture}[anchorbase]
			\draw[thick, double] (.25,.25) to (.25,.75);
			\draw[thick, double] (-.25,.25) to (-.25,.75);
			\draw[thick, double] (.25,-.25) to (.25,-.75);
			\draw[thick, double] (-.25,-.25) to (-.25,-.75);
			\draw[thick] (-.4,-.25) rectangle (.4,.25); 
			\node at (0,-.05) {\scs$\Bproj_{k+n}$};
			\node at (-.5,-.65) {\scriptsize $k$};
			\node at (.5,-.65) {\scriptsize $n$};
			\end{tikzpicture}
		\arrow[d,"\e"]
		\\
			q^{m} \begin{tikzpicture}[anchorbase,scale=.75]
					\draw[thick, double] (.75,.-.75) to (.75,.75);
					\draw[thick, double] (-.75,-.75) to (-.75,.75);
					\draw[thick, double] (-.35,-.25)to (-.35,.25) to[out=90,in=180]  (0,.55) to[out=0,in=90] (.35,.25) to (.35,-.25);
					\draw[thick, double] (-.35,-.25) to[out=-90,in=180]  (0,-.55) to[out=0,in=-90] (.35,-.25);
					\node at (0,.65) {\scriptsize $m$};
					\node at (-.95,-.55) {\scriptsize $k$};
					\node at (.95,-.55) {\scriptsize $n$};
					\end{tikzpicture}
			\arrow[rr, "\mathrm{cap}","\mathrm{cobordisms}"']& &				 
				\begin{tikzpicture}[anchorbase,scale=.75]
					\draw[thick, double] (.25,-.75) to (.25,.75);
					\draw[thick, double] (-.25,-.75) to (-.25,.75);
					\node at (-.55,-.65) {\scriptsize $k$};
					\node at (.55,-.65) {\scriptsize $n$};
					\end{tikzpicture}
		\end{tikzcd}
	\quad , \quad
\begin{tikzcd}
			q^{-m} \begin{tikzpicture}[anchorbase,scale=.875]
				\draw[thick] (.1,-.25) rectangle (1,.25); 
					\draw[thick] (-.1,-.25) rectangle (-1,.25); 
				\node at (.55,-.1) {\scs$r_y(b)$};
				\node at (-.55,-.1) {\scs$r_y(a)$};
				\draw[thick, double] (.75,.25) to (.75,.75);
				\draw[thick, double] (-.75,.25) to (-.75,.75);
				\draw[thick, double] (-.35,.25) to[out=90,in=180]  (0,.55) to[out=0,in=90] (.35,.25);
				\node at (0,.65) {\scriptsize $m$};	
				\begin{scope}[shift={(0,-.75)}] 
					\draw[thick] (.1,-.25) rectangle (1,.25); 
					\draw[thick] (-.1,-.25) rectangle (-1,.25); 
					\node at (.55,-.1) {\small$b$};
					\node at (-.55,-.1) {\small$a$};
					\draw[thick, double] (.75,-.25) to (.75,-.75);
					\draw[thick, double] (-.75,-.25) to (-.75,-.75);
					\draw[thick, double] (-.35,-.25) to[out=-90,in=180]  (0,-.55) to[out=0,in=-90] (.35,-.25);
					\node at (0,-.75) {\scriptsize $m$};
					\node at (-.95,-.65) {\scriptsize $k$};
					\node at (.95,-.65) {\scriptsize $n$};
					\end{scope}
				\end{tikzpicture}
				 \arrow[rr, "\id"] \arrow[d] 
				 & &
				 \begin{tikzpicture}[anchorbase,scale=.875]
					\draw[thick] (-.5,-.25) rectangle (.5,.25); 
					\node at (0,-.1) {\scs$r_y(c)$};
					\draw[thick, double] (.25,.25) to (.25,.75);
					\draw[thick, double] (-.25,.25) to (-.25,.75);
					\begin{scope}[shift={(0,-.75)}] 
						\draw[thick] (-.5,-.25) rectangle (.5,.25); 
						\node at (0,-.1) {\small$c$};
						\draw[thick, double] (.25,-.25) to (.25,-.75);
						\draw[thick, double] (-.25,-.25) to (-.25,-.75);
						\node at (-.5,-.65) {\scriptsize $k$};
						\node at (.5,-.65) {\scriptsize $n$};
						\end{scope}
					\end{tikzpicture}
				   \arrow[d] \\
				   q^{m} \begin{tikzpicture}[anchorbase,scale=.75]
					\draw[thick, double] (.75,.-.75) to (.75,.75);
					\draw[thick, double] (-.75,-.75) to (-.75,.75);
					\draw[thick, double] (-.35,-.25)to (-.35,.25) to[out=90,in=180]  (0,.55) to[out=0,in=90] (.35,.25) to (.35,-.25);
					\draw[thick, double] (-.35,-.25) to[out=-90,in=180]  (0,-.55) to[out=0,in=-90] (.35,-.25);
					\node at (0,.65) {\scriptsize $m$};
					\node at (-.95,-.55) {\scriptsize $k$};
					\node at (.95,-.55) {\scriptsize $n$};
					\end{tikzpicture}
			\arrow[rr, "\mathrm{cap}","\mathrm{cobordisms}"']& &				 
				\begin{tikzpicture}[anchorbase,scale=.75]
					\draw[thick, double] (.25,-.75) to (.25,.75);
					\draw[thick, double] (-.25,-.75) to (-.25,.75);
					\node at (-.55,-.65) {\scriptsize $k$};
					\node at (.55,-.65) {\scriptsize $n$};
					\end{tikzpicture}
				\end{tikzcd}
	\end{equation}
	\end{lemma}
\begin{proof}
Since the counit morphisms are supported in homological degree zero, 
so are both of the relevant compositions.
The right square in \eqref{fig:counit} shows the restriction of the square to a summand of the degree-zero 
object of the domain.
As indicated, the top horizontal map is simply the identity
(where we denote by $c$ the cap tangle obtained by joining $a$ and $b$ as indicated).
This matches the shuffle product from Proposition \ref{prop:Bar is algebra}, 
since in degree zero there are no morphisms to shuffle.
The left vertical map is the cobordism realizing the pairing of $a$ with $r_y(a)$ and $b$ with $r_y(b)$. 
Post-composition with the cap cobordisms yields the cobordism pairing $c$ with $r_y(c)$. 
\end{proof}

We need one more result for later:

\begin{proposition}
\label{prop:centrality}
Let $m,n\in \N_0$ and $D\in \Ch^-(\bn^{n+m}_{n'+m'})$.
\begin{enumerate}
\item If $D$ has through-degree zero, then the maps
\[
\e_{n,m}\star \id \colon \Bproj_{n,m}\star D \rightarrow D
\]
and
\[
\id \star \e_{n,m}\colon  D\star \Bproj_{n',m'}\rightarrow D
\]
are homotopy equivalences.
\item 
For any $D$, the maps
\[
\e_{n,m}\star D \star \id\colon \Bproj_{n,m}\star D\star \Bproj_{n',m'}\rightarrow D\star \Bproj_{n',m'}
\]
and
\[
\id\star D \star \e_{n,m}\colon \Bproj_{n,m}\star D\star \Bproj_{n',m'}\rightarrow \Bproj_{n,m}\star D
\]
are homotopy equivalences.
\end{enumerate}
\end{proposition}
\begin{proof}
Statement (1) implies (2) by taking $D= \Bproj_{n,m}\star D$ or $D= D\star
\Bproj_{n',m'}$.  Statement (1) holds since $\Cone(\e_{n,m})$ is annihilated by
through-degree zero objects on the left and right.
\end{proof}

\subsection{Two-dimensional TFT and stratifications}
\label{ss:stratified}
In this section, we discuss how the Bar-Natan bicategory and related structures
arise from evaluating $2$-dimensional TFTs on embedded submanifolds constrained
by a stratification of the ambient manifold.

It is well-known that commutative Frobenius algebras correspond to
$2$-dimensional topological field theories, i.e.~symmetric monoidal
functors out of the 2-dimensional oriented cobordism category. For a graded
commutative Frobenius algebra, such as $A=\cring[\mathsf{x}]/(\mathsf{x}^2)$ 
with $\mathsf{x}$ in degree $2$, we adopt
the convention that the associated invariant of $S^1$ is the shifted downward
version $A$, which is concentrated in degrees in degrees symmetric around zero. This comes at
the expense of having cobordisms act by homogeneous morphisms of a degree
proportional to their Euler characteristic. Additionally we will assume that $A$
carries the structure of an extended commutative Frobenius algebra in the sense
of \cite{MR2253441}, which determines an unoriented $2$-dimensional TFT
$\mathcal{F}_A$. Thus we may evaluate the TFT on unoriented $1$-manifolds and
cobordisms between them, although we will only need orientable cobordisms at
present. (Similarly, one can also work with an oriented version with appropriate modifications.)

In particular, we may evaluate $\mathcal{F}_A$ on any $1$-submanifold $X$ of
$S^2$, by first forgetting the embedding. The value can be interpreted, by a
\emph{holographic} change of perspective, as the Bar-Natan skein module for the
Frobenius algebra $A$ of the manifold $B^3$ with boundary condition $X$, for which we write $\bn_A(B^3,X)$.

Now we consider a Whitney stratification $F_\bullet=\left(F_0 \subset F_1
\subset F_2=S^2\right)$ of the $2$-sphere and the set $\mathrm{Lin}(F_\bullet)$
of compact $1$-submanifolds of $S^2$ that are transverse to the strata of
$F_\bullet$. The transversality conditions imply these $1$-manifolds are
disjoint from the $0$-dimensional strata in $F_0$ and cross the $1$-dimensional
strata in $F_1\setminus F_0$ in isolated points. 

The stratifications $F_\bullet$ of $S^2$ serve to organize gluing maps relating
the various Bar-Natan skein modules $\bn_A(B^3,X)$ under boundary connected
sums. More specifically, suppose we are given two stratifications $F_\bullet$,
$G_\bullet$ of $S^2$, each of which contains a simply connected $2$-dimensional
stratum that is identified with the other one by an orientation reversing
diffeomorphism. Then, for every pair of elements $X\in \mathrm{Lin}(F_\bullet)$
and $Y\in \mathrm{Lin}(G_\bullet)$ that agree when restricted to the paired
strata, there exists a distinguished boundary connected sum gluing map

\[\bn_A(B^3,X) \otimes \bn_A(B^3,Y) \to \bn_A(B^3,\mathrm{glue}(X,Y)) \] where
$\mathrm{glue}(X,Y)$ results from gluing the restrictions of $X$ and $Y$ to the
complement of the paired strata. These gluing maps, for suitably chosen
stratifications on $S^2$, give rise to the following structures when $A=\cring[\mathsf{x}]/(\mathsf{x}^2)$. We set $I:=[0,1]$.

\begin{itemize}
	\item $S^2= \partial (D^2 \times I)$ with exclusively vertical intersection
	in $J\times I \subset (\partial D^2)\times I$ for an interval $J$: Khovanov's arc rings \cite{MR1928174}.
	\item $S^2= \partial (D^2 \times I)$ with vertical intersection with
	$(\partial D^2)\times I$: a planar algebra valued in linear categories
	(canopolis) \cite{BN2}.
	\item $S^2= \partial I^3$ with empty intersection with $(\partial I)\times
	I^2$ and vertical intersection with $I \times (\partial I)\times I$: the
	monoidal bicategory $\bn$.
	\item $S^2= \partial (k\mathrm{-gon} \times I)$ without intersection with
	$\mathrm{corners}\times I$: a functor from the $k$-fold tensor product of
	bicategories $\bn$ to the Morita bicategory of linear categories, see
	\S\ref{ss:BNmm}. 
\end{itemize}

\section{Bar-Natan modules for disks}

In this section, we define the basic ingredients needed in our construction 
of the $2$-functors from Theorem \ref{thm:main}.
Namely, we associate a free graded $\cring$-module to each disk depending on boundary data associated with its thickening.
By appropriately partitioning this boundary data, 
this allows us to associate a module over the tensor product of Bar-Natan bicategories 
to a disk with distinguished proper $1$-submanifold of its boundary.

\subsection{Bar-Natan skein spaces of prisms}
\label{ss:BNprism}

Set $I = [0,1]$. An \emph{arc} in $\Sigma$ is an oriented $1$-submanifold with boundary 
given as the image of an embedding $\gamma \colon I \hookrightarrow \Sigma$.
(We will sometimes use the symbol $\gamma$ to also denote the corresponding submanifold.)
We begin by establishing conventions for disks with certain distinguished arcs in their boundary. 
Fix a choice of $k \in \N$.

\begin{conv}\label{conv:disk conventions}
A \emph{disk with $k$ arcs in its boundary} is a pair $(D,A)$
consisting of a compact, oriented surface $D$ with boundary 
 that is homeomorphic to the unit disk in $\R^2$ and a compact 1-submanifold 
$A \subsetneq \partial D$ equipped with the induced orientation such that 
$|\pi_0(A)| = k$.
	\end{conv}

We will refer to $(D,A)$ simply as a \emph{disk}, when $k$ is implicit (as in much of the following).
We can picture the pair $(D,A)$ as an oriented $2k$-gon whose boundary alternates between 
arcs in $A$ (depicted in {\color{seam2} green}) and its complement (in black), 
or equivalently as a truncated $k$-gon: 

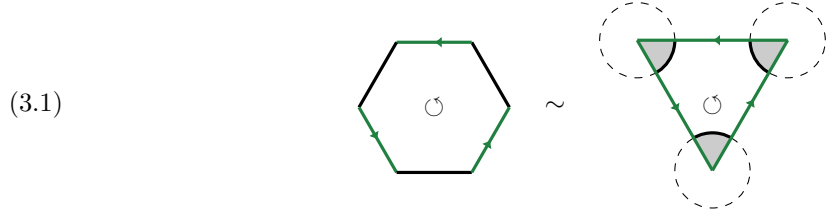
\begin{equation}
	\label{eq:polygon}
\begin{tikzpicture}[anchorbase,scale=1]
\draw[very thick] (1,0) to (.5,.866);
\draw[very thick,seam2,directed=.5] (.5,.866) to (-.5,.866);
\draw[very thick] (-.5,.866) to (-1,0);
\draw[very thick,seam2,directed=.5] (-1,0) to (-.5,-.866);
\draw[very thick] (-.5,-.866) to (.5,-.866);
\draw[very thick,seam2,directed=.5] (.5,-.866) to (1,0);
\node at (0,0) {$\circlearrowleft$};
\end{tikzpicture}
\quad
\sim
\quad
\begin{tikzpicture}[anchorbase,scale=1]
%
	\draw[dashed] (-1,1.732) circle (.5);
	\draw[dashed] (1,1.732) circle (.5);
	\draw[dashed] (0,0) circle (.5);
\begin{scope}
	\clip (1,1.732) to (-1,1.732) to (0,0) to (1,1.732);
	\draw[very thick,fill,fill opacity=.2] (-1,1.732) circle (.5);
	\draw[very thick,fill,fill opacity=.2] (1,1.732) circle (.5);
	\draw[very thick,fill,fill opacity=.2] (0,0) circle (.5);
\end{scope}
\draw[very thick,seam2,directed=.5] (1,1.732) to (-1,1.732);
\draw[very thick,seam2,rdirected=.5] (1,1.732) to (0,0);
\draw[very thick,seam2,rdirected=.5] (0,0) to (-1,1.732);
\node at (0,.866) {$\circlearrowleft$};
\end{tikzpicture} \, .
\end{equation}

Next, recall that a $1$-submanifold $C \subset \Sigma$ is \emph{neatly embedded}
provided $C \hookrightarrow \Sigma$ is a proper embedding 
and $\partial C = C \pitchfork \partial \Sigma$.
Thus, an arc $\gamma$ is neatly embedded if and only if $\del\gamma = \gamma \pitchfork \del\Sigma$
and, in particular, the $1$-submanifold $A$ in Convention \ref{conv:disk conventions} is not neatly embedded.
We will refer to neatly embedded tangles $T$ in $D$ such that $\del T \subset A$ as \emph{tangles in $(D,A)$}.

\begin{definition}\label{def:boundary condition on prism}
Let $(D,A)$ be as in Convention \ref{conv:disk conventions} and let 
$\pi_0(A) = \{A_1,\ldots,A_k\}$.
A \emph{boundary condition} for $(D\times I, A\times I)$ consists of a tuple $(T|V_1,\ldots,V_k|S)$ 
in which $T,S$ 
are tangles in $(D,A)$, 
each $V_i$ 
is a tangle in $(A_i\times I, A_i\times \partial I)$, 
and 
\[
\textstyle
(T\times\{1\}) \cup \big( \bigcup_{i=1}^k V_i \big) \cup (S\times\{0\})
\subset (\del D \times I) \cup (D \times \del I)
\] 
is a closed $1$-manifold.
\end{definition}

The pair $(D \times I, A \times I)$ will be referred to as a \emph{prism}, 
and we will visualize prisms and their boundary conditions as follows:
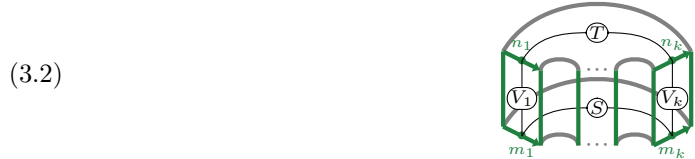
\begin{equation}
	\label{eq:prismpic1}
\begin{tikzpicture}[anchorbase, scale=.5]
\draw[ultra thick, gray] (0,.5) to [out=60,in=120] (5,.5);
\draw[ultra thick, gray] (1,0) to [out=90,in=90] (2,0);
\draw[ultra thick, gray] (3,0) to [out=90,in=90] (4,0);
\node[gray] at (2.5,0) {$\mydots$};
\draw[ultra thick,seam2,->] (0,.5) to node {\scs$\bullet$} node[below] (a1) {\tiny$m_1$} (1,0); 
\draw[ultra thick,seam2,->] (4,0) to node{\scs$\bullet$} node[below] (ak) {\tiny$m_k$} (5,.5);
\node (Stangle) at (2.5,1) {};
\draw (a1.north) to [out=60,in=180] (Stangle);
\draw (Stangle) to [out=0,in=120] (ak.north);
\node at (Stangle) {$\NB{S}$};
\draw[ultra thick,seam2] (0,.5) to (0,2.5);
\draw[ultra thick,seam2] (1,0) to (1,2);
\draw[ultra thick,seam2] (2,0) to (2,2);
\draw[ultra thick,seam2] (3,0) to (3,2);
\draw[ultra thick,seam2] (4,0) to (4,2);
\draw[ultra thick,seam2] (5,.5) to (5,2.5);
\draw[ultra thick, gray] (0,2.5) to [out=60,in=120] (5,2.5);
\draw[ultra thick, gray] (1,2) to [out=90,in=90] (2,2);
\draw[ultra thick, gray] (3,2) to [out=90,in=90] (4,2);
\node[gray] at (2.5,2) {$\mydots$};
\draw[ultra thick,seam2,->] (0,2.5) to node {\scs$\bullet$} node[above] (b1) {\tiny$n_1$} (1,2); 
\draw[ultra thick,seam2,->] (4,2) to node{\scs$\bullet$} node[above] (bk) {\tiny$n_k$} (5,2.5);
\draw (a1) to node{$\NB{V_1}$} (b1);
\draw (ak) to node{$\NB{V_k}$} (bk);
\node (Ttangle) at (2.5,3) {};
\draw (b1.south) to [out=60,in=180] (Ttangle);
\draw (Ttangle) to [out=0,in=120] (bk.south);
\node at (Ttangle) {$\NB{T}$};
\end{tikzpicture}
\end{equation}
We wish to consider the ``Bar-Natan skein module'' of the prism $(D\times I, A\times I)$ 
with a given boundary condition,
i.e.~the graded $\cring$-module spanned by embedded dotted surfaces with corners 
having prescribed boundary \eqref{eq:prismpic1},
modulo the Bar-Natan relations in \eqref{eq:BNrels}.
Here, the orientation of the arcs $A_i$ corresponds to the positive $x$-axis 
in \eqref{eq:BNmn}.

For precision and rigor, 
we will work with the following planar incarnation of this Bar-Natan skein module,
though we will return to the above geometric picture to motivate and elucidate our constructions.

\begin{definition}
	\label{def:partitioned cap tangles}
For each $k$-tuple $(n_1,\ldots,n_k) \in \N^k$, 
let us denote $\bn_{n_1|\cdots|n_k}:=\bn_{n_1+\cdots+n_k}^{0}$.  
Objects of this category will be referred to as \emph{$k$-partitioned planar cap tangles}.
Given $\apS \in \bn_{m_1|\cdots|m_k}$, $\apT \in \bn_{n_1|\cdots|n_k}$,
and objects $V_j\in \bn_{m_j}^{n_j}$ for $1\leq j \leq k$, we set $M:=m_1+\cdots+m_k$ and $N:=n_1+\cdots+n_k$,
and
\begin{equation}
	\label{eqn:BNprismlinear} 
\PMS_{\std}(\apT |V_1,\dots, V_k| \apS):= q^{(N+M)/4}  
\KhEval{ \
\begin{tikzpicture}[anchorbase,scale=1]
	\draw[thick, double] (-.625,-.5) to (-.625,.5);
	\draw[thick, double] (.625,-.5) to (.625,.5);
	\filldraw[white] (-1,.5) rectangle (1, 1);
	\draw[thick] (-1,.5) rectangle (1, 1);
	\node at (0,.75) {$\apT$};
	\tangleboxn{-.625}{0}{V_1}
	\node at (0,0) {$\mydots$};
	\tangleboxn{.625}{0}{V_{k}};
	\filldraw[white] (-1,-.5) rectangle (1, -1);
	\draw[thick] (-1,-.5) rectangle (1, -1);
	\node at (0,-.75) {$r_y(\apS)$};
\end{tikzpicture} \ } \, .
\end{equation}
\end{definition}

We refer to the $\Z$-graded $\cring$-module in \eqref{eqn:BNprismlinear}
as the \emph{prism space} (with boundary condition $(T|V_1,\ldots,V_k|S)$) 
and will sometimes abbreviate as 
$\PMS_{\std}(\apT | V_\bullet | \apS) := \PMS_{\std}(\apT|V_1,\dots, V_k| \apS)$.
The $q$-shift in \eqref{eqn:BNprismlinear} guarantees that for indecomposable 
$T\in \bn_{n_1|\dots|n_k}$, 
the free graded $\cring$-module $\PMS_{\std}(\apT|\one_{n_1},\ldots,\one_{n_k}|\apT)$ 
is supported in non-negative $q$-degrees and is free of rank 1 in $q$-degree zero. 

\begin{remark}
	\label{rem:BNhoms}	
We have a canonical isomorphism
\[
\PMS_{\std}(\apT | \one_{n_1},\ldots, \one_{n_k} | \apS) \cong H(\apT,\apS)\cong \Hom_{\bn}(\apS , \apT).
	\]
with $H(\apT,\apS)$ as in \eqref{eq:planarHom}.
\end{remark}

The prism spaces admits various actions and operations.
We first record the following, which is an immediate consequence of 
functoriality of $\KhEval{-} = \Hom_{\BN_0^0}(\emptyset,-)$.
\begin{lem}
	\label{lem:prismfunctor}
Given $\apS \in \bn_{m_1|\cdots|m_k}$, $\apT \in \bn_{n_1|\cdots|n_k}$,
the assignment \eqref{eqn:BNprismlinear} extends to a graded $\cring$-linear functor
\[
\PMS_{\std}(\apT| - |\apS) \colon \bn^{n_1}_{m_1}\otimes \cdots \otimes \bn^{n_k}_{m_k}
	\to \gMod{\Z}{\cring} \, . 
\]
\qed
\end{lem}

Note that the space of morphisms from $V_\bullet = (V_1,\ldots,V_k)$ to $W_\bullet = (W_1,\ldots,W_k)$ is
\[
\Hom_{\bigotimes_{i=1}^k \bn_{m_i}^{n_i}}(V_\bullet,W_\bullet) 
	:= \bigotimes_{i=1}^k \Hom_{\bn_{m_i}^{n_i}}(V_i,W_i) \, .
\]
By slight abuse of notation, we will sometimes denote elements in this $\Hom$-space by 
$f_\bullet$, although a general such morphism is a $\cring$-linear combinations of pure tensors 
$f_1\otimes \cdots \otimes f_k$ where $f_i \in \Hom_{\bn_{m_i}^{n_i}}(V_i,W_i)$.

\begin{remark}
Given $k$-partitioned cap tangles 
$\apS \in \bn_{m_1|\cdots|m_k}$ and $\apT \in \bn_{n_1|\cdots|n_k}$, 
we will refer to the functor
\begin{equation}
	\label{eq:stprismfunctor}
\PMS_{\std}(\apT| - |\apS) \colon \bigotimes_{i=1}^k \bn_{m_i}^{n_i} \to \gMod{\Z}{\cring}
\end{equation}
from Lemma \ref{lem:prismfunctor} as the \emph{standard prism module} 
of $\apS$ and $\apT$.
\end{remark}

\begin{remark}
	\label{rem:Pfunc}
Let us continue the geometric picture from \eqref{eq:prismpic1}. A generic cobordism 
with boundary condition $(\apT | V_\bullet | \apS)$ will be denoted schematically as:
\begin{equation}
	\label{eq:gencob}
\begin{tikzpicture}[anchorbase, scale=.5]
\draw[ultra thick, gray] (0,.5) to [out=60,in=120] (5,.5);
\draw[ultra thick, gray] (1,0) to [out=90,in=90] (2,0);
\draw[ultra thick, gray] (3,0) to [out=90,in=90] (4,0);
\node[gray] at (2.5,0) {$\mydots$};
\draw[ultra thick,seam2] (0,.5) to node {\scs$\bullet$} node[below] (a1) {\tiny$m_1$} (1,0); 
\draw[ultra thick,seam2] (4,0) to node{\scs$\bullet$} node[below] (ak) {\tiny$m_k$} (5,.5);
\node (Stangle) at (2.5,1) {};
\draw (a1.north) to [out=60,in=180] (Stangle);
\draw (Stangle) to [out=0,in=120] (ak.north);
\node at (Stangle) {$\NB{\apS}$};
\draw[ultra thick,seam2] (0,.5) to (0,2.5);
\draw[ultra thick,seam2] (1,0) to (1,2);
\draw[ultra thick,seam2] (2,0) to (2,2);
\draw[ultra thick,seam2] (3,0) to (3,2);
\draw[ultra thick,seam2] (4,0) to (4,2);
\draw[ultra thick,seam2] (5,.5) to (5,2.5);
\draw[ultra thick, gray] (0,2.5) to [out=60,in=120] (5,2.5);
\draw[ultra thick, gray] (1,2) to [out=90,in=90] (2,2);
\draw[ultra thick, gray] (3,2) to [out=90,in=90] (4,2);
\node[gray] at (2.5,2) {$\mydots$};
\draw[ultra thick,seam2] (0,2.5) to node {\scs$\bullet$} node[above] (b1) {\tiny$n_1$} (1,2); 
\draw[ultra thick,seam2] (4,2) to node{\scs$\bullet$} node[above] (bk) {\tiny$n_k$} (5,2.5);
\node (Ttangle) at (2.5,3) {};
\draw (b1.south) to [out=60,in=180] (Ttangle) to [out=0,in=120] (bk.south);
\fill[black, opacity=.3] (b1.south) to [out=60,in=180] (2.5,3) to [out=0,in=120] (bk.south) 
	to (ak.north) to [out=120,in=0] (2.5,1) to [out=180,in=60] (a1.north) to (b1.south);
\draw (a1) to node{$\NB{V_1}$} (b1);
\draw (ak) to node{$\NB{V_k}$} (bk);
\node at (Ttangle) {$\NB{\apT}$};
\end{tikzpicture}.
\end{equation}
Such a cobordism may (after the appropriate dualization) be regarded as a cobordism from the empty 
diagram to the diagram depicted in \eqref{eq:prismpic1}, and thus an element of the prism space 
$\PMS_{\std}(\apT|V_\bullet|\apS)$ 
(in fact, the images of these cobordisms span $\PMS_{\std}(\apT|V_\bullet|\apS)$).  

Using the same schematic, a generic pure tensor 
$f_1 \otimes \cdots \otimes f_k
\in \Hom_{\bigotimes_{i=1}^k \bn_{m_i}^{n_i}}(V_\bullet,W_\bullet) 
$ and its action $\PMS_{\std}(\apT | f_1 \otimes \cdots \otimes f_k | \apS)$ 
from Lemma \ref{lem:prismfunctor} may be visualized as:
\begin{equation}
	\label{eq:2morprism}
\begin{tikzpicture}[anchorbase, yscale=.75]
\draw[ultra thick,seam2] (0,0) to node {\scs$\bullet$} node[below] (ai) {\tiny$m_1$} (1,0);
\draw[ultra thick,seam2] (-.25,.75) to node {\scs$\bullet$} node[above] (bi) {\tiny$n_1 \ \ $} (.75,.75);
\draw[ultra thick,seam2] (0,0) to (-.25,.75);
\draw[ultra thick,seam2] (1,0) to (.75,.75);
\draw (ai.north) to node{$\NB{V_1}$} (bi.south);
\draw[ultra thick,seam2] (0,0) to (0,1.5);
\draw[ultra thick,seam2] (1,0) to (1,1.5);
\draw[ultra thick,seam2] (-.25,.75) to (-.25,2.25);
\draw[ultra thick,seam2] (.75,.75) to (.75,2.25);
\draw[ultra thick,seam2] (-.25,2.25) to node {\scs$\bullet$} node[above] (b) {\tiny$n_1$} (.75,2.25);
\draw (bi.south) to (b.south);
\draw[ultra thick,seam2] (0,1.5) to node {\scs$\bullet$} node[below] (a) {\tiny$\ \ m_1$} (1,1.5);
\draw[ultra thick,seam2] (0,1.5) to (-.25,2.25);
\draw[ultra thick,seam2] (1,1.5) to (.75,2.25);
\fill[black, opacity=.3] (ai.north) to (bi.south) to (b.south) to (a.north) to (ai.north);
\draw (a.north) to (ai.north);
\draw (a.north) to node{$\NB{W_1}$} (b.south);
\end{tikzpicture}
\otimes
\cdots
\otimes
\begin{tikzpicture}[anchorbase, yscale=.75]
\draw[ultra thick,seam2] (0,0) to node {\scs$\bullet$} node[below] (ai) {\tiny$m_k$} (1,0);
\draw[ultra thick,seam2] (-.25,.75) to node {\scs$\bullet$} node[above] (bi) {\tiny$n_k \ \ $} (.75,.75);
\draw[ultra thick,seam2] (0,0) to (-.25,.75);
\draw[ultra thick,seam2] (1,0) to (.75,.75);
\draw (ai.north) to node{$\NB{V_k}$} (bi.south);
\draw[ultra thick,seam2] (0,0) to (0,1.5);
\draw[ultra thick,seam2] (1,0) to (1,1.5);
\draw[ultra thick,seam2] (-.25,.75) to (-.25,2.25);
\draw[ultra thick,seam2] (.75,.75) to (.75,2.25);
\draw[ultra thick,seam2] (-.25,2.25) to node {\scs$\bullet$} node[above] (b) {\tiny$n_k$} (.75,2.25);
\draw (bi.south) to (b.south);
\draw[ultra thick,seam2] (0,1.5) to node {\scs$\bullet$} node[below] (a) {\tiny$\ \ m_k$} (1,1.5);
\draw[ultra thick,seam2] (0,1.5) to (-.25,2.25);
\draw[ultra thick,seam2] (1,1.5) to (.75,2.25);
\fill[black, opacity=.3] (ai.north) to (bi.south) to (b.south) to (a.north) to (ai.north);
\draw (a.north) to (ai.north);
\draw (a.north) to node{$\NB{W_k}$} (b.south);
\end{tikzpicture} \quad \cdot \quad
\begin{tikzpicture}[anchorbase, scale=.5]
\draw[ultra thick, gray] (0,.5) to [out=60,in=120] (5,.5);
\draw[ultra thick, gray] (1,0) to [out=90,in=90] (2,0);
\draw[ultra thick, gray] (3,0) to [out=90,in=90] (4,0);
\node[gray] at (2.5,0) {$\mydots$};
\draw[ultra thick,seam2] (0,.5) to node {\scs$\bullet$} node[below] (a1) {\tiny$m_1$} (1,0); 
\draw[ultra thick,seam2] (4,0) to node{\scs$\bullet$} node[below] (ak) {\tiny$m_k$} (5,.5);
\node (Stangle) at (2.5,1) {};
\draw (a1.north) to [out=60,in=180] (Stangle);
\draw (Stangle) to [out=0,in=120] (ak.north);
\node at (Stangle) {$\NB{\apS}$};
\draw[ultra thick,seam2] (0,.5) to (0,2.5);
\draw[ultra thick,seam2] (1,0) to (1,2);
\draw[ultra thick,seam2] (2,0) to (2,2);
\draw[ultra thick,seam2] (3,0) to (3,2);
\draw[ultra thick,seam2] (4,0) to (4,2);
\draw[ultra thick,seam2] (5,.5) to (5,2.5);
\draw[ultra thick, gray] (0,2.5) to [out=60,in=120] (5,2.5);
\draw[ultra thick, gray] (1,2) to [out=90,in=90] (2,2);
\draw[ultra thick, gray] (3,2) to [out=90,in=90] (4,2);
\node[gray] at (2.5,2) {$\mydots$};
\draw[ultra thick,seam2] (0,2.5) to node {\scs$\bullet$} node[above] (b1) {\tiny$n_1$} (1,2); 
\draw[ultra thick,seam2] (4,2) to node{\scs$\bullet$} node[above] (bk) {\tiny$n_k$} (5,2.5);
\node (Ttangle) at (2.5,3) {};
\draw (b1.south) to [out=60,in=180] (Ttangle) to [out=0,in=120] (bk.south);
\fill[black, opacity=.3] (b1.south) to [out=60,in=180] (2.5,3) to [out=0,in=120] (bk.south) 
	to (ak.north) to [out=120,in=0] (2.5,1) to [out=180,in=60] (a1.north) to (b1.south);
\draw (a1) to node{$\NB{V_1}$} (b1);
\draw (ak) to node{$\NB{V_k}$} (bk);
\node at (Ttangle) {$\NB{\apT}$};
\end{tikzpicture}
=
\begin{tikzpicture}[anchorbase, scale=.5]
\draw[ultra thick, gray] (0,.5) to [out=60,in=120] (5,.5);
\draw[ultra thick, gray] (1,0) to [out=90,in=90] (2,0);
\draw[ultra thick, gray] (3,0) to [out=90,in=90] (4,0);
\node[gray] at (2.5,0) {$\mydots$};
\draw[ultra thick,seam2] (0,.5) to node {\scs$\bullet$} node[below] (a1) {\tiny$m_1$} (1,0); 
\draw[ultra thick,seam2] (4,0) to node{\scs$\bullet$} node[below] (ak) {\tiny$m_k$} (5,.5);
\node (Stangle) at (2.5,1) {};
\draw (a1.north) to [out=60,in=180] (Stangle);
\draw (Stangle) to [out=0,in=120] (ak.north);
\node at (Stangle) {$\NB{\apS}$};
\draw[ultra thick,seam2] (0,.5) to (0,2.5);
\draw[ultra thick,seam2] (1,0) to (1,2);
\draw[ultra thick,seam2] (2,0) to (2,2);
\draw[ultra thick,seam2] (3,0) to (3,2);
\draw[ultra thick,seam2] (4,0) to (4,2);
\draw[ultra thick,seam2] (5,.5) to (5,2.5);
\draw[ultra thick, gray] (0,2.5) to [out=60,in=120] (5,2.5);
\draw[ultra thick, gray] (1,2) to [out=90,in=90] (2,2);
\draw[ultra thick, gray] (3,2) to [out=90,in=90] (4,2);
\node[gray] at (2.5,2) {$\mydots$};
\draw[ultra thick,seam2] (0,2.5) to node {\scs$\bullet$} node[above] (b1) {\tiny$\ n_1$} (1,2); 
\draw[ultra thick,seam2] (4,2) to node{\scs$\bullet$} node[above] (bk) {\tiny$n_k$} (5,2.5);
\node (Ttangle) at (2.5,3) {};
\draw (b1.south) to [out=60,in=180] (Ttangle) to [out=0,in=120] (bk.south);
\fill[black, opacity=.3] (b1.south) to [out=60,in=180] (2.5,3) to [out=0,in=120] (bk.south) 
	to (ak.north) to [out=120,in=0] (2.5,1) to [out=180,in=60] (a1.north) to (b1.south);
\draw (a1) to node{$\NB{V_1}$} (b1);
\draw (ak) to node{$\NB{V_k}$} (bk);
\node at (Ttangle) {$\NB{\apT}$};
\draw[ultra thick, seam2] (0,0.5) to (-1.25,-1) to node {\scs$\bullet$} node[below] (al) {\tiny$m_1$} (-.25,-1.5) to (1,0);
\draw (a1.north) to (al.north);
\draw[ultra thick, seam2] (-1.25,-1) to (-1.25,1);
\draw[ultra thick, seam2] (-.25,-1.5) to (-.25,.5);
\draw[ultra thick, seam2] (0,2.5) to (-1.25,1) to node {\scs$\bullet$} node[above] (bl) {\tiny$\ n_1$} (-.25,.5) to (1,2);
\draw (b1.south) to (bl.south);
\fill[black, opacity=.3] (b1.south) to (bl.south) to (al.north) to (a1.north) to (b1.south);
\draw (al.north) to node{$\NB{W_1}$} (bl.south);
\draw[ultra thick, seam2] (5,0.5) to (6.25,-1) to node {\scs$\bullet$} node[below] (ar) {\tiny$n_1$} (5.25,-1.5) to (4,0);
\draw (ak.north) to (ar.north);
\draw[ultra thick, seam2] (6.25,-1) to (6.25,1);
\draw[ultra thick, seam2] (5.25,-1.5) to (5.25,.5);
\draw[ultra thick, seam2] (5,2.5) to (6.25,1) to node {\scs$\bullet$} node[above] (br) {\tiny$\ n_1$} (5.25,.5);
\draw (bk.south) to (br.south);
\fill[black, opacity=.3] (bk.south) to (br.south) to (ar.north) to (ak.north) to (bk.south);
\draw (ar.north) to node{$\NB{W_k}$} (br.south);
\end{tikzpicture} \quad .
\end{equation}
\end{remark}

The prism spaces $\PMS_{\std}(\apT | V_1,\ldots,V_k | \apS)$ 
are also functorial in $\apT$ and $\apS$. 
More generally, the interpretation of prism spaces from Remark \ref{rem:Pfunc}
suggests that they may be endowed with composition maps
that correspond to the stacking of cobordisms in the prisms \eqref{eq:prismpic1}
as follows:
\begin{equation}
	\label{eq:prismstack}
\begin{tikzpicture}[anchorbase, scale=.5]
\draw[ultra thick, gray] (0,.5) to [out=60,in=120] (5,.5);
\draw[ultra thick, gray] (1,0) to [out=90,in=90] (2,0);
\draw[ultra thick, gray] (3,0) to [out=90,in=90] (4,0);
\node[gray] at (2.5,0) {$\mydots$};
\draw[ultra thick,seam2] (0,.5) to node {\scs$\bullet$} node[below] (a1) {\tiny$m_1$} (1,0); 
\draw[ultra thick,seam2] (4,0) to node{\scs$\bullet$} node[below] (ak) {\tiny$m_k$} (5,.5);
\node (Stangle) at (2.5,1) {};
\draw (a1.north) to [out=60,in=180] (Stangle);
\draw (Stangle) to [out=0,in=120] (ak.north);
\node at (Stangle) {$\NB{\apS}$};
\draw[ultra thick,seam2] (0,.5) to (0,2.5);
\draw[ultra thick,seam2] (1,0) to (1,2);
\draw[ultra thick,seam2] (2,0) to (2,2);
\draw[ultra thick,seam2] (3,0) to (3,2);
\draw[ultra thick,seam2] (4,0) to (4,2);
\draw[ultra thick,seam2] (5,.5) to (5,2.5);
\draw[ultra thick, gray] (0,2.5) to [out=60,in=120] (5,2.5);
\draw[ultra thick, gray] (1,2) to [out=90,in=90] (2,2);
\draw[ultra thick, gray] (3,2) to [out=90,in=90] (4,2);
\node[gray] at (2.5,2) {$\mydots$};
\draw[ultra thick,seam2] (0,2.5) to node {\scs$\bullet$} node[above] (b1) {\tiny$n_1$} (1,2); 
\draw[ultra thick,seam2] (4,2) to node{\scs$\bullet$} node[above] (bk) {\tiny$n_k$} (5,2.5);
\node (Ttangle) at (2.5,3) {};
\draw (b1.south) to [out=60,in=180] (Ttangle) to [out=0,in=120] (bk.south);
\fill[black, opacity=.3] (b1.south) to [out=60,in=180] (2.5,3) to [out=0,in=120] (bk.south) 
	to (ak.north) to [out=120,in=0] (2.5,1) to [out=180,in=60] (a1.north) to (b1.south);
\draw (a1) to node{$\NB{W_1}$} (b1);
\draw (ak) to node{$\NB{W_k}$} (bk);
\node at (Ttangle) {$\NB{\apT}$};
\end{tikzpicture}
\otimes
\begin{tikzpicture}[anchorbase, scale=.5]
\draw[ultra thick, gray] (0,.5) to [out=60,in=120] (5,.5);
\draw[ultra thick, gray] (1,0) to [out=90,in=90] (2,0);
\draw[ultra thick, gray] (3,0) to [out=90,in=90] (4,0);
\node[gray] at (2.5,0) {$\mydots$};
\draw[ultra thick,seam2] (0,.5) to node {\scs$\bullet$} node[below] (a1) {\tiny$l_1$} (1,0); 
\draw[ultra thick,seam2] (4,0) to node{\scs$\bullet$} node[below] (ak) {\tiny$l_k$} (5,.5);
\node (Stangle) at (2.5,1) {};
\draw (a1.north) to [out=60,in=180] (Stangle);
\draw (Stangle) to [out=0,in=120] (ak.north);
\node at (Stangle) {$\NB{\apR}$};
\draw[ultra thick,seam2] (0,.5) to (0,2.5);
\draw[ultra thick,seam2] (1,0) to (1,2);
\draw[ultra thick,seam2] (2,0) to (2,2);
\draw[ultra thick,seam2] (3,0) to (3,2);
\draw[ultra thick,seam2] (4,0) to (4,2);
\draw[ultra thick,seam2] (5,.5) to (5,2.5);
\draw[ultra thick, gray] (0,2.5) to [out=60,in=120] (5,2.5);
\draw[ultra thick, gray] (1,2) to [out=90,in=90] (2,2);
\draw[ultra thick, gray] (3,2) to [out=90,in=90] (4,2);
\node[gray] at (2.5,2) {$\mydots$};
\draw[ultra thick,seam2] (0,2.5) to node {\scs$\bullet$} node[above] (b1) {\tiny$m_1$} (1,2); 
\draw[ultra thick,seam2] (4,2) to node{\scs$\bullet$} node[above] (bk) {\tiny$m_k$} (5,2.5);
\node (Ttangle) at (2.5,3) {};
\draw (b1.south) to [out=60,in=180] (Ttangle) to [out=0,in=120] (bk.south);
\fill[black, opacity=.3] (b1.south) to [out=60,in=180] (2.5,3) to [out=0,in=120] (bk.south) 
	to (ak.north) to [out=120,in=0] (2.5,1) to [out=180,in=60] (a1.north) to (b1.south);
\draw (a1) to node{$\NB{V_1}$} (b1);
\draw (ak) to node{$\NB{V_k}$} (bk);
\node at (Ttangle) {$\NB{\apS}$};
\end{tikzpicture}
\mapsto
\begin{tikzpicture}[anchorbase, scale=.5]
\draw[ultra thick, gray] (0,.5) to [out=60,in=120] (5,.5);
\draw[ultra thick, gray] (1,0) to [out=90,in=90] (2,0);
\draw[ultra thick, gray] (3,0) to [out=90,in=90] (4,0);
\node[gray] at (2.5,0) {$\mydots$};
\draw[ultra thick,seam2] (0,.5) to node {\scs$\bullet$} node[below] (a1) {\tiny$l_1$} (1,0); 
\draw[ultra thick,seam2] (4,0) to node{\scs$\bullet$} node[below] (ak) {\tiny$l_k$} (5,.5);
\node (Stangle) at (2.5,1) {};
\draw (a1.north) to [out=60,in=180] (Stangle);
\draw (Stangle) to [out=0,in=120] (ak.north);
\node at (Stangle) {$\NB{\apR}$};
\draw[ultra thick,seam2] (0,.5) to (0,2.5);
\draw[ultra thick,seam2] (1,0) to (1,2);
\draw[ultra thick,seam2] (2,0) to (2,2);
\draw[ultra thick,seam2] (3,0) to (3,2);
\draw[ultra thick,seam2] (4,0) to (4,2);
\draw[ultra thick,seam2] (5,.5) to (5,2.5);
\draw[ultra thick, gray] (0,2.5) to [out=60,in=120] (5,2.5);
\draw[ultra thick, gray] (1,2) to [out=90,in=90] (2,2);
\draw[ultra thick, gray] (3,2) to [out=90,in=90] (4,2);
\node[gray] at (2.5,2) {$\mydots$};
\draw[ultra thick,seam2] (0,2.5) to node {\scs$\bullet$} node[above] (b1) {} (1,2); 
\draw[ultra thick,seam2] (4,2) to node{\scs$\bullet$} node[above] (bk) {} (5,2.5);
\node (Ttangle) at (2.5,3) {};
\draw (b1.south) to [out=60,in=180] (Ttangle) to [out=0,in=120] (bk.south);
\fill[black, opacity=.3] (b1.south) to [out=60,in=180] (2.5,3) to [out=0,in=120] (bk.south) 
	to (ak.north) to [out=120,in=0] (2.5,1) to [out=180,in=60] (a1.north) to (b1.south);
\draw (a1) to node{$\NB{V_1}$} (b1);
\draw (ak) to node{$\NB{V_k}$} (bk);
\node at (Ttangle) {$\NB{\apS}$};
\begin{scope}[shift={(0,2)}] 
\path[ultra thick, gray] (0,.5) to [out=60,in=120] (5,.5);
\path[ultra thick, gray] (1,0) to [out=90,in=90] (2,0);
\path[ultra thick, gray] (3,0) to [out=90,in=90] (4,0);
\path[ultra thick,seam2] (0,.5) to node {\scs$\bullet$} node[below] (a1) {} (1,0); 
\path[ultra thick,seam2] (4,0) to node{\scs$\bullet$} node[below] (ak) {} (5,.5);
\node (Stangle) at (2.5,1) {};
\path (a1.north) to [out=60,in=180] (Stangle);
\path (Stangle) to [out=0,in=120] (ak.north);
\draw[ultra thick,seam2] (0,.5) to (0,2.5);
\draw[ultra thick,seam2] (1,0) to (1,2);
\draw[ultra thick,seam2] (2,0) to (2,2);
\draw[ultra thick,seam2] (3,0) to (3,2);
\draw[ultra thick,seam2] (4,0) to (4,2);
\draw[ultra thick,seam2] (5,.5) to (5,2.5);
\draw[ultra thick, gray] (0,2.5) to [out=60,in=120] (5,2.5);
\draw[ultra thick, gray] (1,2) to [out=90,in=90] (2,2);
\draw[ultra thick, gray] (3,2) to [out=90,in=90] (4,2);
\node[gray] at (2.5,2) {$\mydots$};
\draw[ultra thick,seam2] (0,2.5) to node {\scs$\bullet$} node[above] (b1) {\tiny$n_1$} (1,2); 
\draw[ultra thick,seam2] (4,2) to node{\scs$\bullet$} node[above] (bk) {\tiny$n_k$} (5,2.5);
\node (Ttangle) at (2.5,3) {};
\draw (b1.south) to [out=60,in=180] (Ttangle) to [out=0,in=120] (bk.south);
\fill[black, opacity=.3] (b1.south) to [out=60,in=180] (2.5,3) to [out=0,in=120] (bk.south) 
	to (ak.north) to [out=120,in=0] (2.5,1) to [out=180,in=60] (a1.north) to (b1.south);
\draw (a1) to node{$\NB{W_1}$} (b1);
\draw (ak) to node{$\NB{W_k}$} (bk);
\node at (Ttangle) {$\NB{\apT}$};
\end{scope}
\end{tikzpicture} \, .
\end{equation}

In the planar setup, this is given as follows.

\begin{definition}[Prism unit and stacking]
	\label{def:prismstackingmaps}
Fix $k\in \N$. 
For $\apT\in \bn_{n_1|\cdots|n_k}$, let
\[
e_{\apT} \colon \cring \to \PMS_{\std}( \apT | \one,\dots, \one | \apT)
\]
denote the graded $\cring$-linear map 
that sends $1\in \cring$ to $\id_{\apT} \in \End_{\bn}(\apT)$ 
under the isomorphism from Remark~\ref{rem:BNhoms}.

Further, given $\apR \in \bn_{l_1|\cdots | l_k}$, $\apS \in \bn_{m_1|\cdots|m_k}$ 
and objects $V_i\in \bn_{l_i}^{m_i}$ and $W_i \in \bn_{m_i}^{n_i}$ for $1\leq i \leq k$, 
let
\begin{equation}
	\label{eq:Pcomp}
	m_{\apS} \colon  
	\PMS_{\std}(\apT |W_1,\dots, W_k| \apS)	\otimes \PMS_{\std}(\apS |V_1,\dots, V_k| \apR) 
		\to \PMS_{\std}(\apT|W_1\star V_1,\dots, W_k\star V_k|\apR)
\end{equation}
denote the graded $\cring$-linear map:
\begin{equation}
	\label{eq:prismstackingpicture}
	q^{|\del \apS\cup \del \apT|/4+|\del \apR\cup \del \apS|/4}  
	\KhEval{ \
	\begin{tikzpicture}[anchorbase,scale=.8]
		\draw[thick, double] (-.625,-.5) to (-.625,.5);
		\draw[thick, double] (.625,-.5) to (.625,.5);
		\filldraw[white] (-1,.5) rectangle (1, 1);
		\draw[thick] (-1,.5) rectangle (1, 1);
		\node at (0,.75) {\small$\apT$};
		\tangleboxw{-.625}{0}{W_1}
		\node at (0,0) {$\mydots$};
		\tangleboxw{.625}{0}{W_{k}};
		\filldraw[white] (-1,-.5) rectangle (1, -1);
		\draw[thick] (-1,-.5) rectangle (1, -1);
		\node at (0,-.75) {\small$r_y(\apS)$};
		\begin{scope}[shift={(0,-2.25)}]
			\draw[thick, double] (-.625,-.5) to (-.625,.5);
			\draw[thick, double] (.625,-.5) to (.625,.5);
			\filldraw[white] (-1,.5) rectangle (1, 1);
			\draw[thick] (-1,.5) rectangle (1, 1);
			\node at (0,.75) {\small$\apS$};
			\tangleboxn{-.625}{0}{V_1}
			\node at (0,0) {$\mydots$};
			\tangleboxn{.625}{0}{V_{k}};
			\filldraw[white] (-1,-.5) rectangle (1, -1);
			\draw[thick] (-1,-.5) rectangle (1, -1);
			\node at (0,-.75) {\small$r_y(\apR)$};
		\end{scope}
	\end{tikzpicture} \ }
	\xrightarrow{m_{\apS}}
	q^{|\del \apR\cup \del \apT|/4}  
	\KhEval{ \
	\begin{tikzpicture}[anchorbase,scale=.8]
		\draw[thick, double] (-.625,-2.5) to (-.625,.5);
		\draw[thick, double] (.625,-2.5) to (.625,.5);
		\filldraw[white] (-1,.5) rectangle (1, 1);
		\draw[thick] (-1,.5) rectangle (1, 1);
		\node at (0,.75) {\small$\apT$};
		\tangleboxw{-.625}{0}{W_1}
		\node at (0,0) {$\mydots$};
		\tangleboxw{.625}{0}{W_{k}};
		\begin{scope}[shift={(0,-2.25)}]
			\draw[thick, double] (-.625,-.5) to (-.625,.5);
			\draw[thick, double] (.625,-.5) to (.625,.5);
			\tangleboxn{-.625}{0}{V_1}
			\node at (0,0) {$\mydots$};
			\tangleboxn{.625}{0}{V_{k}};
			\filldraw[white] (-1,-.5) rectangle (1, -1);
			\draw[thick] (-1,-.5) rectangle (1, -1);
			\node at (0,-.75) {\small$r_y(\apR)$};
		\end{scope}
	\end{tikzpicture} \ }		
\end{equation}
induced by pairing $\apS$ and $r_y(\apS)$ via saddle and tube cobordisms.
\end{definition}

In \eqref{eq:prismstackingpicture},
we have used the symmetric monoidality of $\Kh$ for the identification of the domain
with $\PMS_{\std}(\apT|W_1,\dots, W_k| \apS) \otimes \PMS_{\std}(\apS | V_1,\dots, V_k | \apR)$.
Further, we notationally suppress the dependence of $m_{\apS}$ 
on the objects $W_i$ and $V_i$.
We refer to the maps in \eqref{eq:Pcomp} as \emph{prism stacking maps}, 
due to their interpretation as in \eqref{eq:prismstack}.

The following expresses the naturality of the prism stacking maps.
\begin{lem}\label{lem:prismstacking-star} 
Given $k$-partitioned cap tangles
$\apR \in \bn_{\ell_1 | \cdots | \ell_k}$,  
$\apS \in \bn_{m_1 | \cdots | m_k}$,
$\apT \in \bn_{n_1 | \cdots | n_k}$ and morphisms 
$f_\bullet \colon V_\bullet \to V'_\bullet$ in $\bigotimes_{i=1}^k \bn_{\ell_i}^{m_i}$ 
and $g_\bullet \colon W_\bullet \to W_\bullet'$ in $\bigotimes_{i=1}^k \bn_{m_i}^{n_i}$, 
there is a commutative square of graded linear maps:
	\[
		\begin{tikzcd}
			\PMS_{\std}(\apT|W_\bullet |\apS) \otimes \PMS_{\std}(\apS | V_\bullet | \apR)
			\arrow[r,"m_{\apS}"] 
			 \arrow[d,"\PMS_{\std}(\apT|g_\bullet|\apS)\otimes \PMS_{\std}(\apS|f_\bullet|\apR)"'] 
			&  
			\PMS_{\std}(\apT|W_\bullet \star V_\bullet|\apR)
			\arrow[d,"\PMS_{\std}(\apT|g_\bullet \star f_\bullet|\apR)"]
			\\
			\PMS_{\std}(\apT|W_\bullet'|\apS) \otimes \PMS_{\std}(\apS|V_\bullet'|\apR)
			\arrow[r,"m_{\apS}"] 
			&  
			\PMS_{\std}(\apT|W_\bullet' \star V_\bullet'|\apR)
		\end{tikzcd}
	\]
\end{lem}
\begin{proof}
This follows from the far-commutativity of morphisms induced by cobordisms
supported over the boxes labeled $V_i$ and $W_i$ in \eqref{eq:prismstackingpicture}
and the saddle and tube cobordisms implementing $m_{\apS}$.
\end{proof}

Next we show that the prism stacking maps enjoy associativity and unitality properties.
\begin{lemma}
	\label{lem:assocunit-prismstacking}
	Fix $k\in \N$. 
	Let $\apQ$, $\apR$, $\apS$, $\apT$ be $k$-partitioned planar cap tangles 
	and let $W_i$, $V_i$, $U_i$ be objects for $1 \leq i \leq k$ 
	such that $\PMS_{\std}(\apT | W_\bullet | \apS)$, 
			$\PMS_{\std}(\apS | V_\bullet | \apR)$, 
			and $\PMS_{\std}(\apR | U_\bullet | \apQ)$ 
	are defined as in \eqref{eqn:BNprismlinear}.
	Then, the prism stacking maps from \eqref{eq:Pcomp} are associative, i.e.~the composites
	\[
			m_{\apR}\circ (m_{\apS}\otimes \id) \colon 
			\PMS_{\std}(\apT|W_\bullet|\apS)\otimes \PMS_{\std}(\apS|V_\bullet|\apR)
				\otimes \PMS_{\std}(\apR|U_\bullet|\apQ) \to 
			\PMS_{\std}(\apT|(W_\bullet \star V_\bullet)\star U_\bullet| \apQ)
			\]
	and
	\[
			m_{\apS}\circ (\id\otimes m_{\apR}) \colon
			\PMS_{\std}(\apT|W_\bullet|\apS)\otimes \PMS_{\std}(\apS|V_\bullet|\apR)
				\otimes \PMS_{\std}(\apR|U_\bullet|\apQ) \to 
			\PMS_{\std}(\apT|W_\bullet\star (V_\bullet\star U_\bullet)|\apQ)
			\]
	are equal (up to the coherent isomorphism induced by 
	$(W_\bullet\star V_\bullet)\star U_\bullet\cong W_\bullet\star (V_\bullet\star U_\bullet)$). 
	Additionally, the prism stacking maps are also unital in the sense that the composites
	\[
			m_{\apS}\circ (\id \otimes e_{\apS}) \colon 
			\PMS_{\std}(\apT|W_\bullet|\apS)\otimes \cring \to 
			\PMS_{\std}(\apT|W_\bullet\star \one_\bullet|\apS)
			\]
	and
	\[
			m_{\apT}\circ (e_{\apT} \otimes \id) \colon
			\cring \otimes \PMS_{\std}(\apT|W_\bullet|\apS) \to 
			\PMS_{\std}(\apT|\one_\bullet \star W_\bullet|\apS)
			\]
	are both identity maps (after applying the isomorphisms induced by 
	$W_\bullet\star \one_\bullet \cong W_\bullet \cong \one_\bullet\star W_\bullet$).
\end{lemma}
\begin{proof}
This is immediate from the definition of $\KhEval{-}$ and the corresponding 
properties of the Bar-Natan $2$-category.
(Alternatively, it can be explicitly checked using the 
standard bases of the free graded $\cring$-modules $\KhEval{L}$ discussed in \S\ref{ss:BN}.)
\end{proof}

\subsection{Prism $2$-modules}
	\label{ss:BNmm}
	
Note that prism stacking is inherently $2$-categorical since 
it makes use of the horizontal composition $\star$ in the bicategory $\bn$ 
(or, rather, $\bn^{\otimes k}$).
We now establish further compatibilities between the
prism stacking maps and functoriality of the prism spaces.
These compatibilities allow us to ultimately
assemble the prism spaces into a 2-functor from $\bn^{\otimes k}$ 
to an appropriate Morita bicategory of bimodules.

\begin{definition}
\label{def:bimodules}
Let $\AS$ and $\BS$ be $\Z$-graded $\cring$-linear categories. 
An \emph{$(\AS,\BS)$-bimodule} is a graded, $\cring$-linear
functor $\AS\otimes \BS^{\op} \rightarrow \gModen{\Z}{\cring}$.  
For fixed $\AS,\BS$, the $(\AS,\BS)$-bimodules form a 
category, denoted $\Bim_{\AS,\BS}$, 
with morphisms given by natural transformations.  
Given $M\in \Bim_{\AS,\BS}$ and $N\in \Bim_{\BS,\CS}$, 
we define their (underived) tensor product as coequalizer of the left- and right-action of $\BS$, namely
\begin{equation}\label{eq:tensor product}
	(M\otimes_\BS N)(a,c):=  
	\mathrm{CoKer}\left( \bigoplus_{b,b'\in \BS} M(a,b')\otimes_{\cring} \Hom_{\BS}(b,b') \otimes N(b,c) 
		\to \bigoplus_{b\in \BS} M(a,b)\otimes N(b,c) \right)
\end{equation}
where the map sends 
$f\otimes g \otimes h \mapsto (f g) \otimes h - f \otimes (g h)$. 
See e.g.~\cite[\S 2.1.5]{MR3692883} or \cite[\S 2.4]{2002.06110}.
\end{definition}

Note that the usual categories of left and right $\AS$-modules 
are obtained as $\Bim_{\AS,\cring}$ and $\Bim_{\cring,\AS}$,
respectively, where here $\cring$ is regarded as a category with one object.
When we want to emphasize the categories acting on $M\in \Bim_{\AS,\BS}$, 
we will sometimes write ${}_\AS M_{\BS}$. 
Similarly, given objects $X\in \Obj(\AS)$, $Y\in \Obj(\BS)$ we introduce notation
\begin{equation}
	\label{eq:bimshorthand}
{}_X M_Y:=M(X,Y) \, , \quad  {}_X M:= {}_X M_\BS := M(X,-) \, , \quad M_Y := {}_\AS M_Y:= M(-,Y) \, .
\end{equation}
For fixed $X$ and $Y$, the latter two give right- and left-modules, respectively.

\begin{conv}\label{conv:A(x,y)}
A graded $\cring$-linear category $\AS$ can be viewed as an $(\AS,\AS)$-bimodule 
defined by $\AS(X,X'):=\Hom_\AS(X',X)$. (Note the reversal of order.) 
This is the left- and right-unit for the tensor product $\otimes_{\AS}$ from \eqref{eq:tensor product}.
\end{conv}

\begin{remark}\label{rmk:left action}
Suppose $M\in \Bim_{\AS,\BS}$ is a bimodule. 
The left action of $\AS$ is encoded in a collection of $\cring$-linear maps
\[
\Hom_{\AS}(X,X') \rightarrow \Hom_\cring\left(M(X,Y),M(X',Y)\right)
\]
satisfying appropriate unit and associativity properties. 
Using tensor-hom adjunction and Convention \ref{conv:A(x,y)}, 
we may write this action in the more traditional fashion, 
i.e.~as a family of $\cring$-linear maps
\[
\AS(X',X)\otimes M(X,Y)\rightarrow M(X',Y) \, .
\]
Similar remarks apply to the right $\BS$-action.
\end{remark}

\begin{definition}
\label{def:morita}
Let $\PF[\Z]$ denote the \emph{Morita bicategory} of graded $\cring$-linear categories wherein
\begin{itemize}
	\item objects are $\Z$-graded $\cring$-linear categories $\AS$, and
	\item the $1$-morphism category $\AS \leftarrow \BS$ is the graded 
		$\cring$-linear category $\Bim_{\AS,\BS}$.
\end{itemize}
The (vertical) composition of $2$-morphisms is simply the usual composition of natural transformations, 
and the (horizontal) composition $\Bim_{\AS,\BS} \otimes \Bim_{\BS,\CS} \to \Bim_{\AS,\CS}$
is given by tensor product $\otimes_{\BS}$.
\end{definition}

We now extend the prism module construction to a $2$-functor valued in $\PF[\Z]$. 
The relevant domain is the tensor product $\bn^{\otimes k}$ of Bar-Natan bicategories.

\begin{conv}\label{conv:bnk}
Given $\nn=(n_1,\ldots,n_k)$ with $n_i\in \N$, 
we have the object $(\pp^{n_1},\ldots,\pp^{n_k})\in \bn^{\otimes k}$.
By abuse of notation, 
we will frequently denote this object simply by the tuple $\nn=(n_1,\ldots,n_k)$.
On the level of 1-morphisms in $\bn^{\otimes k}$, we will abbreviate by writing 
\[
(\bn^{\otimes k})_{\mm}^{\nn} 
	= \bigotimes_{i=1}^k \bn^{n_i}_{m_i} 
\]
A 1-morphism in $\bn^{\otimes k}$ is a tuple $V_\bullet=(V_1,\ldots,V_k)$; 
from Section \ref{s:surfaceinvt} onwards,
we will drop the bullet subscript for aesthetic reasons. 
The identity 1-morphism on an object $\nn$ will be denoted 
by $\one_\nn:=(\one_{n_1},\ldots,\one_{n_k})$.
\end{conv}

\begin{definition}\label{def:standard bimod}
For each object $\nn$ of $\bn^{\otimes k}$, let 
$\MM(\nn):=\bn_{n_1|\cdots|n_k}$.  For each 1-morphism 
$V_\bullet\in (\bn^{\otimes k})^{\nn}_{\mm}$, let $\MM(V)\in \Bim_{\MM(\nn),\MM(\mm)}$ 
denote the bimodule 
$\PMS_{\std}(-|V_\bullet|-)$, i.e.~in the notation of \eqref{eq:bimshorthand}:
\[
{}_{\apT}\MM(V_\bullet)_{\apS} := \PMS_{\std}(\apT|V_\bullet|\apS) \, .
\]
The bimodule structure is given 
(via Remark \ref{rmk:left action})
by prism stacking maps
\[
\PMS_{\std}(\apT'|\one_\nn|\apT) \otimes \PMS_{\std}(\apT|V_\bullet|\apS) 
	\otimes \PMS_{\std}(\apS|\one_\mm|\apS')
		\rightarrow \PMS_{\std}(\apT'|V_\bullet|\apS')
\]
using the identification of Remark \ref{rem:BNhoms}.
\end{definition}

We may also denote the $2$-functor from Definition \ref{def:standard bimod} by 
$\MM=\MM_{\std}$ when we want to distinguish it from similar constructions 
appearing e.g.~in \S \ref{ss:arbitrary disks} below.

\begin{example}\label{ex:KhRing as sCat}
In case $k=1$, 
we have that $\MM(\pp^{n})$ is the big Khovanov ring $H_n^0$ from Definition \ref{def:Khrings} 
(here viewed as a category).
The bimodule $\MM(V)$ is the analogue for this ring of the 
bimodule Khovanov associates to an $(n,m)$-tangle $V$ in \cite{MR1928174}.
\end{example}

\begin{prop}
	\label{prop:stprism2M}
Given $k \in \N$, the assignments $\nn \mapsto \MM(\nn)$ and $V\mapsto \MM(V)$ from 
	Definition \ref{def:standard bimod} extend to a $2$-functor $\bn^{\otimes k} \to \PF[\Z]$.
\end{prop}

\begin{proof}
Let $\mm, \nn$ be objects of $\bn^{\otimes k}$, let $V_\bullet,W_\bullet\in (\bn^{\otimes k})_{\mm}^{\nn}$ 
be 1-morphisms, and let $f_\bullet\colon V_\bullet \rightarrow W_\bullet$ be a 2-morphism. 
Let ${}_\apT\MM(f_\bullet)_\apS$ be the linear map
\[
\PMS_{\std}(\apT | f_\bullet |\apS)
\colon \PMS_{\std}(\apT | V_\bullet |\apS) \to \PMS_{\std}(\apT | W_\bullet | \apS)
\]
from Lemma \ref{lem:prismfunctor}.  
That these maps define a bimodule map 
(i.e.~that they give the components of a natural transformation of functors) 
$\MM(V_\bullet)\rightarrow \MM(W_\bullet)$ follows from 
the naturality of the prism stacking maps in Lemma \ref{lem:prismstacking-star}.

Having now defined the $2$-functor on $0$-, $1$-, and $2$-morphisms in $\bn^{\otimes k}$, 
it remains to check the necessary coherences:
that it is functorial with respect to vertical composition 
and functorial with respect to horizontal composition, 
up to coherent $2$-isomorphism in the codomain.
Since vertical composition of natural transformations is given component-wise, 
the former is simply the functoriality of the standard prism module $\PMS_{\std}(\apT | - | \apS)$
provided by Lemma \ref{lem:prismfunctor}.

We thus turn out attention to functoriality with respect to horizontal composition.
Namely, given 1-morphisms $V_\bullet \in (\bn^{\otimes k})_{\mm}^{\nn}$ 
and $W_\bullet\in (\bn^{\otimes k})_{\mathbf{l}}^{\mm}$, 
we must argue that $\MM(V_\bullet)\otimes_{\MM(\qq)}\MM(W_\bullet) \cong \MM(V_\bullet\star W_\bullet)$.
Since this is formally very similar to \cite[Theorem 1]{MR1928174} 
(and we give similar arguments in \S \ref{ss:MM of surface} below), we omit the details here.
Further, our $2$-functor preserves identity $1$-morphisms, 
since Remark \ref{rem:BNhoms} implies that 
$\MM(\one_\nn)$ is precisely the identity bimodule of $\MM(\nn)$.

Lastly, the coherences for horizontal composition follows from 
the associativity and unitality of the prism stacking morphisms given 
in Lemma \ref{lem:assocunit-prismstacking}.
\end{proof}

\begin{remark}
	\label{rem:prismpictures}
To help digest Proposition \ref{prop:stprism2M}, 
we describe the $2$-functor $\MM$ in terms of the schematic of \eqref{eq:prismpic1}.
It is given on objects by
\[
(\pp^{n_1},\dots,\pp^{n_k}) \mapsto 
\bn \Bigg(
\begin{tikzpicture}[anchorbase, scale=.5]
\draw[ultra thick, gray] (0,.5) to [out=60,in=120] (5,.5);
\draw[ultra thick, gray] (1,0) to [out=90,in=90] (2,0);
\draw[ultra thick, gray] (3,0) to [out=90,in=90] (4,0);
\node[gray] at (2.5,0) {$\mydots$};
\draw[ultra thick,seam2] (0,.5) to node{\scs$\bullet$} node[below]{$n_1$} (1,0); 
\draw[ultra thick,seam2] (4,0) to node{\scs$\bullet$} node[below]{$n_k$} (5,.5);
\end{tikzpicture}
\Bigg) \, ,
\]
which we may identify with $\bn_{n_1|\cdots|n_k}$.
Given $V_\bullet\in (\bn^{\otimes k})_{\mm}^{\nn}$, 
the bimodule $\MM(V_\bullet)$ associates to a pair of tangles $\apT\in \MM(\nn)$, $\apS\in \MM(\mm)$ 
the prism space $\PMS_{\std}(\apT|V_\bullet|\apS)$, 
which we may interpret via Remark \ref{rem:Pfunc}.

The bimodule structure on $\MM(V_\bullet)$ is given by maps
\[
\PMS_{\std}(\apT_2|\one_{\nn}|\apT_1)\otimes \PMS_{\std}(\apT_1|V_\bullet|\apS_1)
	\otimes \PMS_{\std}(\apS_1|\one_\mm|\apS_2) 
		\rightarrow \PMS_{\std}(\apT_2|V_\bullet|\apS_2) \, ,
\]
which may visualized as
\[
\begin{tikzpicture}[anchorbase, scale=.35]
\draw[ultra thick, gray] (0,.5) to [out=60,in=120] (5,.5);
\draw[ultra thick, gray] (1,0) to [out=90,in=90] (2,0);
\draw[ultra thick, gray] (3,0) to [out=90,in=90] (4,0);
\node[gray] at (2.5,0) {$\mydots$};
\draw[ultra thick,seam2] (0,.5) to node {\scs$\bullet$} node[below] (a1) {\tiny$n_1$} (1,0); 
\draw[ultra thick,seam2] (4,0) to node{\scs$\bullet$} node[below] (ak) {\tiny$n_k$} (5,.5);
\node (Stangle) at (2.5,1) {};
\draw (a1.north) to [out=60,in=180] (Stangle);
\draw (Stangle) to [out=0,in=120] (ak.north);
\node at (Stangle) {$\NBt{\apT_1}$};
\draw[ultra thick,seam2] (0,.5) to (0,2.5);
\draw[ultra thick,seam2] (1,0) to (1,2);
\draw[ultra thick,seam2] (2,0) to (2,2);
\draw[ultra thick,seam2] (3,0) to (3,2);
\draw[ultra thick,seam2] (4,0) to (4,2);
\draw[ultra thick,seam2] (5,.5) to (5,2.5);
\draw[ultra thick, gray] (0,2.5) to [out=60,in=120] (5,2.5);
\draw[ultra thick, gray] (1,2) to [out=90,in=90] (2,2);
\draw[ultra thick, gray] (3,2) to [out=90,in=90] (4,2);
\node[gray] at (2.5,2) {$\mydots$};
\draw[ultra thick,seam2] (0,2.5) to node {\scs$\bullet$} node[above] (b1) {} (1,2); 
\draw[ultra thick,seam2] (4,2) to node{\scs$\bullet$} node[above] (bk) {} (5,2.5);
\node (Ttangle) at (2.5,3) {};
\draw (b1.south) to [out=60,in=180] (Ttangle) to [out=0,in=120] (bk.south);
\fill[black, opacity=.3] (b1.south) to [out=60,in=180] (2.5,3) to [out=0,in=120] (bk.south) 
	to (ak.north) to [out=120,in=0] (2.5,1) to [out=180,in=60] (a1.north) to (b1.south);
\draw (a1) to (b1);
\draw (ak) to (bk);
\node at (Ttangle) {$\NBt{\apT_2}$};
\node at (0,0) {$f$};
\end{tikzpicture}
\otimes
\begin{tikzpicture}[anchorbase, scale=.5]
\draw[ultra thick, gray] (0,.5) to [out=60,in=120] (5,.5);
\draw[ultra thick, gray] (1,0) to [out=90,in=90] (2,0);
\draw[ultra thick, gray] (3,0) to [out=90,in=90] (4,0);
\node[gray] at (2.5,0) {$\mydots$};
\draw[ultra thick,seam2] (0,.5) to node {\scs$\bullet$} node[below] (a1) {\tiny$m_1$} (1,0); 
\draw[ultra thick,seam2] (4,0) to node{\scs$\bullet$} node[below] (ak) {\tiny$m_k$} (5,.5);
\node (Stangle) at (2.5,1) {};
\draw (a1.north) to [out=60,in=180] (Stangle);
\draw (Stangle) to [out=0,in=120] (ak.north);
\node at (Stangle) {$\NB{\apS_1}$};
\draw[ultra thick,seam2] (0,.5) to (0,2.5);
\draw[ultra thick,seam2] (1,0) to (1,2);
\draw[ultra thick,seam2] (2,0) to (2,2);
\draw[ultra thick,seam2] (3,0) to (3,2);
\draw[ultra thick,seam2] (4,0) to (4,2);
\draw[ultra thick,seam2] (5,.5) to (5,2.5);
\draw[ultra thick, gray] (0,2.5) to [out=60,in=120] (5,2.5);
\draw[ultra thick, gray] (1,2) to [out=90,in=90] (2,2);
\draw[ultra thick, gray] (3,2) to [out=90,in=90] (4,2);
\node[gray] at (2.5,2) {$\mydots$};
\draw[ultra thick,seam2] (0,2.5) to node {\scs$\bullet$} node[above] (b1) {\tiny$n_1$} (1,2); 
\draw[ultra thick,seam2] (4,2) to node{\scs$\bullet$} node[above] (bk) {\tiny$n_k$} (5,2.5);
\node (Ttangle) at (2.5,3) {};
\draw (b1.south) to [out=60,in=180] (Ttangle) to [out=0,in=120] (bk.south);
\fill[black, opacity=.3] (b1.south) to [out=60,in=180] (2.5,3) to [out=0,in=120] (bk.south) 
	to (ak.north) to [out=120,in=0] (2.5,1) to [out=180,in=60] (a1.north) to (b1.south);
\draw (a1) to node{$\NBt{V_1}$} (b1);
\draw (ak) to node{$\NBt{V_k}$} (bk);
\node at (Ttangle) {$\NB{\apT_1}$};
\end{tikzpicture}
\otimes
\begin{tikzpicture}[anchorbase, scale=.35]
\draw[ultra thick, gray] (0,.5) to [out=60,in=120] (5,.5);
\draw[ultra thick, gray] (1,0) to [out=90,in=90] (2,0);
\draw[ultra thick, gray] (3,0) to [out=90,in=90] (4,0);
\node[gray] at (2.5,0) {$\mydots$};
\draw[ultra thick,seam2] (0,.5) to node {\scs$\bullet$} node[below] (a1) {\tiny$m_1$} (1,0); 
\draw[ultra thick,seam2] (4,0) to node{\scs$\bullet$} node[below] (ak) {\tiny$m_k$} (5,.5);
\node (Stangle) at (2.5,1) {};
\draw (a1.north) to [out=60,in=180] (Stangle);
\draw (Stangle) to [out=0,in=120] (ak.north);
\node at (Stangle) {$\NBt{\apS_2}$};
\draw[ultra thick,seam2] (0,.5) to (0,2.5);
\draw[ultra thick,seam2] (1,0) to (1,2);
\draw[ultra thick,seam2] (2,0) to (2,2);
\draw[ultra thick,seam2] (3,0) to (3,2);
\draw[ultra thick,seam2] (4,0) to (4,2);
\draw[ultra thick,seam2] (5,.5) to (5,2.5);
\draw[ultra thick, gray] (0,2.5) to [out=60,in=120] (5,2.5);
\draw[ultra thick, gray] (1,2) to [out=90,in=90] (2,2);
\draw[ultra thick, gray] (3,2) to [out=90,in=90] (4,2);
\node[gray] at (2.5,2) {$\mydots$};
\draw[ultra thick,seam2] (0,2.5) to node {\scs$\bullet$} node[above] (b1) {} (1,2); 
\draw[ultra thick,seam2] (4,2) to node{\scs$\bullet$} node[above] (bk) {} (5,2.5);
\node (Ttangle) at (2.5,3) {};
\draw (b1.south) to [out=60,in=180] (Ttangle) to [out=0,in=120] (bk.south);
\fill[black, opacity=.3] (b1.south) to [out=60,in=180] (2.5,3) to [out=0,in=120] (bk.south) 
	to (ak.north) to [out=120,in=0] (2.5,1) to [out=180,in=60] (a1.north) to (b1.south);
\draw (a1) to (b1);
\draw (ak) to (bk);
\node at (Ttangle) {$\NBt{\apS_1}$};
\end{tikzpicture}
\mapsto
\begin{tikzpicture}[anchorbase, scale=.45]
\begin{scope}[shift={(0,-2)}] 
\draw[ultra thick, gray] (0,.5) to [out=60,in=120] (5,.5);
\draw[ultra thick, gray] (1,0) to [out=90,in=90] (2,0);
\draw[ultra thick, gray] (3,0) to [out=90,in=90] (4,0);
\node[gray] at (2.5,0) {$\mydots$};
\draw[ultra thick,seam2] (0,.5) to node {\scs$\bullet$} node[below] (a1) {\tiny$m_1$} (1,0); 
\draw[ultra thick,seam2] (4,0) to node{\scs$\bullet$} node[below] (ak) {\tiny$m_k$} (5,.5);
\node (Stangle) at (2.5,1) {};
\draw (a1.north) to [out=60,in=180] (Stangle);
\draw (Stangle) to [out=0,in=120] (ak.north);
\node at (Stangle) {$\NB{\apS_2}$};
\draw[ultra thick,seam2] (0,.5) to (0,2.5);
\draw[ultra thick,seam2] (1,0) to (1,2);
\draw[ultra thick,seam2] (2,0) to (2,2);
\draw[ultra thick,seam2] (3,0) to (3,2);
\draw[ultra thick,seam2] (4,0) to (4,2);
\draw[ultra thick,seam2] (5,.5) to (5,2.5);
\path[ultra thick, gray] (0,2.5) to [out=60,in=120] (5,2.5);
\path[ultra thick, gray] (1,2) to [out=90,in=90] (2,2);
\path[ultra thick, gray] (3,2) to [out=90,in=90] (4,2);
\path[ultra thick,seam2] (0,2.5) to node {\scs$\bullet$} node[above] (b1) {} (1,2); 
\path[ultra thick,seam2] (4,2) to node{\scs$\bullet$} node[above] (bk) {} (5,2.5);
\fill[black, opacity=.3] (b1.south) to [out=60,in=180] (2.5,3) to [out=0,in=120] (bk.south) 
	to (ak.north) to [out=120,in=0] (2.5,1) to [out=180,in=60] (a1.north) to (b1.south);
\draw (a1) to (b1);
\draw (ak) to (bk);
\node at (Ttangle) {};
\end{scope}
\draw[ultra thick, gray] (0,.5) to [out=60,in=120] (5,.5);
\draw[ultra thick, gray] (1,0) to [out=90,in=90] (2,0);
\draw[ultra thick, gray] (3,0) to [out=90,in=90] (4,0);
\node[gray] at (2.5,0) {$\mydots$};
\draw[ultra thick,seam2] (0,.5) to node {\scs$\bullet$} node[below] (a1) {} (1,0); 
\draw[ultra thick,seam2] (4,0) to node{\scs$\bullet$} node[below] (ak) {} (5,.5);
\node (Stangle) at (2.5,1) {};
\draw (a1.north) to [out=60,in=180] (Stangle);
\draw (Stangle) to [out=0,in=120] (ak.north);
\node at (Stangle) {$\NB{\apS_1}$};
\draw[ultra thick,seam2] (0,.5) to (0,2.5);
\draw[ultra thick,seam2] (1,0) to (1,2);
\draw[ultra thick,seam2] (2,0) to (2,2);
\draw[ultra thick,seam2] (3,0) to (3,2);
\draw[ultra thick,seam2] (4,0) to (4,2);
\draw[ultra thick,seam2] (5,.5) to (5,2.5);
\draw[ultra thick, gray] (0,2.5) to [out=60,in=120] (5,2.5);
\draw[ultra thick, gray] (1,2) to [out=90,in=90] (2,2);
\draw[ultra thick, gray] (3,2) to [out=90,in=90] (4,2);
\node[gray] at (2.5,2) {$\mydots$};
\draw[ultra thick,seam2] (0,2.5) to node {\scs$\bullet$} node[above] (b1) {} (1,2); 
\draw[ultra thick,seam2] (4,2) to node{\scs$\bullet$} node[above] (bk) {} (5,2.5);
\node (Ttangle) at (2.5,3) {};
\draw (b1.south) to [out=60,in=180] (Ttangle) to [out=0,in=120] (bk.south);
\fill[black, opacity=.3] (b1.south) to [out=60,in=180] (2.5,3) to [out=0,in=120] (bk.south) 
	to (ak.north) to [out=120,in=0] (2.5,1) to [out=180,in=60] (a1.north) to (b1.south);
\draw (a1) to node{$\NB{V_1}$} (b1);
\draw (ak) to node{$\NB{V_k}$} (bk);
\node at (Ttangle) {$\NB{\apT_1}$};
\begin{scope}[shift={(0,2)}] 
\path[ultra thick, gray] (0,.5) to [out=60,in=120] (5,.5);
\path[ultra thick, gray] (1,0) to [out=90,in=90] (2,0);
\path[ultra thick, gray] (3,0) to [out=90,in=90] (4,0);
\path[ultra thick,seam2] (0,.5) to node {\scs$\bullet$} node[below] (a1) {} (1,0); 
\path[ultra thick,seam2] (4,0) to node{\scs$\bullet$} node[below] (ak) {} (5,.5);
\node (Stangle) at (2.5,1) {};
\path (a1.north) to [out=60,in=180] (Stangle);
\path (Stangle) to [out=0,in=120] (ak.north);
\draw[ultra thick,seam2] (0,.5) to (0,2.5);
\draw[ultra thick,seam2] (1,0) to (1,2);
\draw[ultra thick,seam2] (2,0) to (2,2);
\draw[ultra thick,seam2] (3,0) to (3,2);
\draw[ultra thick,seam2] (4,0) to (4,2);
\draw[ultra thick,seam2] (5,.5) to (5,2.5);
\draw[ultra thick, gray] (0,2.5) to [out=60,in=120] (5,2.5);
\draw[ultra thick, gray] (1,2) to [out=90,in=90] (2,2);
\draw[ultra thick, gray] (3,2) to [out=90,in=90] (4,2);
\node[gray] at (2.5,2) {$\mydots$};
\draw[ultra thick,seam2] (0,2.5) to node {\scs$\bullet$} node[above] (b1) {\tiny$n_1$} (1,2); 
\draw[ultra thick,seam2] (4,2) to node{\scs$\bullet$} node[above] (bk) {\tiny$n_k$} (5,2.5);
\node (Ttangle) at (2.5,3) {};
\draw (b1.south) to [out=60,in=180] (Ttangle) to [out=0,in=120] (bk.south);
\fill[black, opacity=.3] (b1.south) to [out=60,in=180] (2.5,3) to [out=0,in=120] (bk.south) 
	to (ak.north) to [out=120,in=0] (2.5,1) to [out=180,in=60] (a1.north) to (b1.south);
\draw (a1) to (b1);
\draw (ak) to (bk);
\node at (Ttangle) {$\NB{\apT_2}$};
\end{scope}
\end{tikzpicture} \, .
\]
Lastly, $\MM(f_\bullet)$ is given as in \eqref{eq:2morprism}.
\end{remark}

\subsection{The prism 2-module associated to an arbitrary disk}
\label{ss:arbitrary disks}

We now extend the constructions in this section to a general pair $(D,A)$ as in Convention 
\ref{conv:disk conventions}. 
This will depend on a choice of identification relating $(D,A)$ to a standard situation, described next.

\begin{definition}
	\label{def:stdDA}
The \emph{standard disk} with $k$ arcs in its boundary is the pair $(\Dst, \Ast)$ where
\[
\Dst := I \times I \, , \quad 
	\Ast := \big( {\textstyle \bigcup_{i=1}^k} \Ast_i  \big) \times \{0\}
\]
with $\Ast_i := \big[ \frac{2i-2}{2k-1} , \frac{2i-1}{2k-1} \big]$.
\end{definition}

We will take artistic license and depict the pair $(\Dst,\Ast)$ 
e.g.~as in the top boundary of \eqref{eq:prismpic1}.

\begin{conv}\label{conv:Dstvskcap}
Suppose we are given a finite subset $\pp\subset \Ast$ such that $|\pp\cap \Ast_i| = n_i$ for 
$1 \leq i \leq k$.  We will choose once and for all an identification $\bn(\Dst,\pp)=\bn_{n_1|\cdots|n_k}$, for 
instance by taking the horizontal composition with the straight-line tangle from $\pp$ to the 
standard boundary $\pp^{(n_1+\cdots+n_k)}$. 
Any two choices will be canonically isomorphic.
\end{conv}

\begin{definition}
	\label{def:std}
Given $(D,A)$ as in Convention \ref{conv:disk conventions}, 
a \emph{standardization} is an orientation preserving homeomorphism of pairs
\[
\pi \colon (D,A) \to (\Dst,\Ast) \, .
\]
\end{definition}

\begin{conv}\label{conv:disk boundary}
Given a standardization $\pi$ of $(D,A)$ and a finite subset $\pp\subset A$ with $|\pi(\pp)\cap \Ast_i|=n_i$, 
we have an induced isomorphism of categories 
$\bn(D,\pp) \xrightarrow{\cong} \bn(\Dst,\pi(\pp)) \cong \bn_{n_1|\cdots|n_k}$, 
where the first isomorphism is induced by $\pi$ 
and the second is as in Convention \ref{conv:Dstvskcap}.  
This isomorphism will frequently be denoted by $\pi$, by abuse of notation, 
and will be tacitly used in the sequel.
\end{conv}

A standardization allows us to extend constructions \eqref{eqn:BNprismlinear}  
and Definition \ref{def:standard bimod}
to arbitary disks $(D,A)$, as follows.

\begin{definition}\label{def:BNprismarbitrary}
Let $(D,A)$ be as in Convention \ref{conv:disk conventions}, 
and let $\pi \colon (D,A) \to (\Dst, \Ast)$ be a choice of standardization.  
For $\pp,\qq \subset \Ast$ with $|\qq\cap \Ast_i| = m_i$ and $|\pp\cap \Ast_i| = n_i$,
suppose we are given $T\in \bn(D,\pi\inv(\pp))$ and $S\in \bn(D,\pi\inv(\qq))$, 
and 1-morphisms $V_i\in \bn^{n_i}_{m_i}$. 
Associated to these data, 
we have the \emph{prism space}
\begin{equation}\label{eq:prism arbitrary}
\PMS_{D,A,\pi}(T|V_1,\ldots,V_k|S):= \PMS_{\std}(\pi T|V_1,\ldots,V_k|\pi S) \, .
\end{equation}
\end{definition}

The following generalizes Definition \ref{def:standard bimod} 
to an arbitrary disk $(D,A)$ equipped with a standardization. 

\begin{definition}\label{def:prism2M}
The \emph{prism $2$-module} associated to $(D,A,\pi)$ is the assignment
$\MM_{D,A,\pi}\colon \bn^{\otimes k}\rightarrow \PF[\Z]$ given as follows.
On the objects $\pp=(\pp^{n_1},\ldots,\pp^{n_k}) \in \bn^{\otimes k}$, 
we let 
\[
\MM_{D,A,\pi}(\pp) := \bn(D,\pi\inv(\pp))
\]
and on 1-morphisms $V_{\bullet} \in (\bn^{\otimes k})_{\mm}^{\nn}$ 
we let $\MM_{D,A,\pi}(V_\bullet)$ be the $(\MM_{D,A,\pi}(\pp), \MM_{D,A,\pi}(\qq))$-bimodule
\[
{}_T\MM_{D,A,\pi}(V_\bullet)_S:=\PMS_{D,A,\pi}(T|V_\bullet|S).
\]
The behavior of $\MM_{D,A,\pi}$ on 2-morphisms is defined as in the proof of Proposition \ref{prop:stprism2M}.
\end{definition}

The following is immediate from Proposition \ref{prop:stprism2M}.

\begin{corollary}
	\label{cor:prism2M1}
The construction in Definition \ref{def:prism2M} is a well-defined 2-functor 
$\MM_{D,A,\pi}\colon \bn^{\otimes k}\rightarrow \PF[\Z]$. \qed
\end{corollary}

\begin{remark}
In case $(D,A) = (\Dst,\Ast)$ is already standard, we have $\PMS_{\Dst,\Ast,\id}
(T|V_1,\ldots,V_k|S)=\PMS_{\std}(\apT|V_1,\ldots,V_k|\apS)$, where $\apT$ and 
$\apS$ are the $k$-partitioned planar cap tangles associated to $T$ and $S$ 
(as in Convention \ref{conv:Dstvskcap}). 
Hence in this case $\MM_{\Dst,\Ast,\id}=\MM_{\std}$.
\end{remark}

The 2-functor $\MM_{D,A,\pi}$ depends very coarsely on the choice of standardization. 
To make this precise, we use the following.

Given $(D,A)$ as in Convention \ref{conv:disk conventions}, 
the orientation of $D$ induces an action of the cyclic group 
$\ip{\zeta \mid \zeta^k=1}$ on the set of components $\pi_0(A)$. 
(Explicitly, $\zeta(\ma)$ occurs after $\ma$ according to the induced orientation on $\partial D$.)
In this way, $\pi_0(A)$ is a cyclically ordered set.

\begin{definition}\label{def:cyclic refinement}
Let $\LO(A)$ denote the set of linear refinements of the cyclic ordering on $\pi_0(A)$, 
i.e.~the set of bijections $\psi \colon \pi_0(A) \xrightarrow{\cong} \Z/k\Z$ 
such that $\psi(\zeta \ma) \equiv \psi(\ma)+1$ (mod $k$) for all $\ma\in \pi_0(A)$. 
If $\pi$ is a standardization, then we let $[\pi]\in \LO(A)$ denote the induced linear refinement, 
defined by  $\pi\inv(\Ast_i) \mapsto i$.
\end{definition}

\begin{lem}
	\label{lem:stiso}
If $\pi$ and $\pi'$ are standardizations of a disk $(D,A)$ with $[\pi]=[\pi']$
then there is a canonical isomorphism of 2-functors $\MM_{D,A,\pi}\cong \MM_{D,A,\pi'}$.
\end{lem}
\begin{proof}
This follows from the isotopy invariance of all constructions involved. 
Indeed, without loss of generality, we may assume that the homeomorphism 
$\pi' \circ \pi^{-1} \colon (\Dst,\Ast) \to (\Dst,\Ast)$ fixes each $\Ast_i$ point-wise.
It follows that this $\pi$ and $\pi'$ are isotopic rel boundary. 
From this, we obtain canonical isomorphisms of the prism spaces
\[
\PMS_{D,A,\pi}(T|V_\bullet|S)=\PMS_{\std}(\pi T|V_\bullet|\pi S)
	\cong \PMS_{\std}(\pi' T|V_\bullet|\pi' S)=\PMS_{D,A,\pi'}(T|V_\bullet|S),
\]
which are easily seen to be functorial in $V$ and compatible with the evident prism stacking maps.
\end{proof}

Next we investigate the dependence on the linear refinement $[\pi]$.  
The key is provided by the following.

\begin{prop}
	\label{prop:rot2functor}
Let $\rho \colon \bn^{\otimes k} \to \bn^{\otimes k}$ be the $2$-functor that
cyclically permutes the tensor factors. The rotation functor from
Lemma~\ref{lem:fullrotation} determines a pseudonatural equivalence
(natural isomorphism of $2$-functors) $\upsilon \colon
\MM_{\std}\xRightarrow{\cong} \MM_{\std} \circ \rho$ such that the $k$-fold
iteration is canonically identified with the identity equivalence.
\end{prop}

\begin{proof}
Lemma~\ref{lem:fullrotation} gives an equivalence
\begin{equation}\label{eq:prism rotator on objects}
\mathrm{rot}^{n_k} \colon \bn_{n_1|n_2|\cdots|n_k} \rightarrow \bn_{n_k|n_1|\cdots|n_{k-1}}
\end{equation}
and Proposition \ref{prop:sphericality} gives an isomorphism
\begin{equation}
	\label{eq:prismonerotation}
\mathrm{r}_{m_k}^{n_k} \colon \PMS_{\std}(\apT|V_1,\dots, V_k|\apS) \xrightarrow{\cong} 
	\PMS_{\std}(\mathrm{rot}^{n_k} \apT|V_k, V_1,\dots, V_{k-1}| \mathrm{rot}^{m_k} \apS)
\end{equation}
that is functorial in the $V_j\in \bn_{m_j}^{n_j}$ and compatible with the 
prism stacking maps.
Further, Proposition \ref{prop:sphericality} shows that the composition
\[
\PMS_{\std}(\apT | V_1,\dots, V_k | \apS) 
	\xrightarrow{\mathrm{r}_{m_1}^{n_1} \circ \cdots \circ \mathrm{r}_{m_k}^{n_k}} 
	\PMS_{\std}(\mathrm{rot}^{N} \apT|V_1,\dots, V_{k}| \mathrm{rot}^{N} \apS)
\xrightarrow{\cong} \PMS_{\std}(\apT|V_1,\dots, V_k|\apS)
\]
is the identity.
Here, the last isomorphism is 
induced from the natural isomorphism 
$\mathrm{rot}^N \xrightarrow{(\mathrm{twist}_+)^{-1}} \id$ of Lemma~\ref{lem:fullrotation}
(as in the final statement of Proposition~\ref{prop:sphericality}).

Now we must assemble these isomorphisms into the 
pseudonatural equivalence $\upsilon$ as in the statement. 
For this, 
we must specify its $0$- and $1$-morphism components, 
and check the necessary naturality/compatibilities.

Let $\mathbf{n} = (\pp^{n_1},\ldots,\pp^{n_{k-1}},\pp^{n_k})$ be an object of $\bn^{\otimes k}$, 
and let $\mathbf{n}'=\rho(\mathbf{n}) = (\pp^{n_k},\pp^{n_1},\ldots,\pp^{n_{k-1}})$.
The component $\upsilon_{\mathbf{n}}$ is an invertible
$\big(\MM(\nn'),\MM(\nn)\big)$-bimodule.
Since $\MM(\mathbf{n}) = \bn_{n_1 | \cdots | n_k}$ 
and $\MM(\mathbf{n}') = \bn_{n_k | n_1 \cdots | n_{k-1}}$,
we let this be the bimodule incarnation of the functor \eqref{eq:prism rotator on objects},
i.e.~
\[
{}_{\apT}(\upsilon_{\nn})_{\apS} = \Hom_{\bn_{n_k | n_1 | \cdots | n_{k-1}}}\big( \mathrm{rot}^{n_k}(\apS),\apT \big).
\]
This is an invertible bimodule since $\mathrm{rot}^{n_k}$ is an equivalence of categories.

Given a $1$-morphism $V_\bullet = (V_1,\ldots,V_k) \colon \mathbf{m} \to
\mathbf{n}$, the corresponding component of $\upsilon$ is an invertible bimodule
morphism $\upsilon_{V_\bullet} \colon (\MM\circ\rho)(V_\bullet)
\otimes_{\MM(\mathbf{m}')} \upsilon_{\mathbf{m}} \to \upsilon_{\mathbf{n}}
\otimes_{\MM(\mathbf{n})} \MM(V_\bullet)$ that is natural in $V_\bullet$ and
compatible with identity $1$-morphisms and horizontal composition.  
Equivalently (since $\upsilon_{\mathbf{n}}$ is an invertible bimodule), 
it suffices to find such a bimodule morphism 
$\upsilon_{\mathbf{n}}^{-1} \otimes_{\MM(\mathbf{n}')} (\MM \circ \rho)(V_\bullet) 
	\otimes_{\MM(\mathbf{m}')} \upsilon_{\mathbf{m}} 
	\to \MM(V_\bullet)$.
The latter is simply given as the isomorphism
\[
\PMS_{\std} \big(\mathrm{rot}^{n_k}(-) | V_k,V_1,\ldots,V_{k-1} | \mathrm{rot}^{m_k}(-) \big)
	\xrightarrow{\cong} \PMS_{\std}( - | V_1,\ldots,V_k | - )
\]
of $(\bn_{n_1 | \cdots | n_k},\bn_{m_1 | \cdots | m_k})$-bimodules 
(inverse to the isomorphism)
from \eqref{eq:prismonerotation}, 
whose naturality and compatibility properties are straightforward.
Lastly, the $k$-fold iteration is naturally isomorphic to the identity 
via Proposition~\ref{prop:sphericality}.
\end{proof}

\begin{cor}
	\label{cor:prism2M2}
Let $(D,A)$ be a disk as in Convention \ref{conv:disk conventions}, with standardizations $\pi$ and $\pi'$.  
If $\ell\in \Z/k\Z$ is such that $[\pi'](\ma) = [\pi](\zeta^{\ell}\ma)$ for all $\ma\in \pi_0(A)$, 
then we have a canonical isomorphism $\MM_{D,A,\pi'}\circ \rho^{\ell} \cong \MM_{D,A,\pi}$. \qed
\end{cor}

\section{The surface invariant}
\label{s:surfaceinvt}

Corollaries \ref{cor:prism2M1} and \ref{cor:prism2M2} give an invariant of
surfaces that are homeomorphic to disks (with a marked submanifold of the
boundary) that takes the form of a $2$-functor 
from the tensor product of Bar-Natan bicategories to the Morita bicategory.
In this section, we extend this invariant to marked surfaces $(\Sigma,\Pi)$ without
closed components. Recall the convention from the introduction that $\Pi$ is a
finite set of oriented intervals, disjointly embedded in $\partial \Sigma$. The
resulting invariant will take values in dg bimodule-valued modules over an
appropriate tensor product of Bar-Natan bicategories $\bni(\Pi)$, e.g.~for each
object of the latter, we associate a dg category to our marked surface. The $1$-
and $2$-morphisms of $\bni(\Pi)$ are then assigned dg bimodules and natural
transformations. When $\Sigma=D$, this construction recovers Definition
\ref{def:prism2M}, up to quasi-equivalence.

\subsection{Marked surfaces and seams}
\label{ss:marks and seams}

We begin by specifying our conventions for the types of surfaces we consider.

\begin{definition}\label{def:marked surface}
A \emph{marked surface} is a pair $(\Sigma,\Pi)$ where $\Sigma$ is a
compact, oriented surface and $\Pi=\{\ma_1,\ldots,\ma_\ell \}$ is a finite, 
totally ordered set of pairwise disjoint oriented arcs $\ma_i \subset \del \Sigma$. 
\end{definition}

Contrasting the setup of Convention \ref{conv:disk conventions}, 
in Definition \ref{def:marked surface} we do not require the orientation 
of the arcs $\ma_i$ to match that of $\del \Sigma$.
If $(\Sigma,\Pi)$ is such a marked surface, 
then we let $\sgn_\Sigma\colon\Pi\rightarrow \{\pm 1\}$ be defined by
\[
\sgn_{\Sigma}(\ma_i) = 
\begin{cases}
1 & \text{if the orientations on $\ma_i$ and $\del \Sigma$ agree} \\
-1 & \text{if the orientations on $\ma_i$ and $\del \Sigma$ are opposite}.
\end{cases}
\]
We say that $(\Sigma,\Pi)$ is \emph{positive} if $\sgn_\Sigma(\ma_i)=+1$ for all $i$.
Further, set $A(\Pi):=\bigcup_{i=1}^\ell \ma_i$,
which is regarded as an oriented 1-manifold equipped with a function 
$\sgn_{\Sigma} \colon \pi_0(A(\Pi))\rightarrow \{\pm1\}$. 

\begin{rem}
	\label{rem:truncation}
In the literature (e.g.~\cite{MR830043,DyKa}), 
one sometimes encounters a different notion of marked surface, 
namely a surface $X$ with a chosen finite collection of ``marked'' points, 
such that each boundary component contains at least one marked point.
It is possible to translate between this formulation and 
our notion of marked surface from Definition \ref{def:marked surface} as follows.

Suppose $X$ is a compact oriented surface, 
$\BB\subset \del X$ is a finite set of marked points intersecting each component of $\del X$, 
and $\CB\subset \operatorname{int} X$ is a finite set of marked points in the interior of $X$.  
Choose a total order on $\pi_0(\partial X \smallsetminus \BB)$, 
as well as an orientation on each of these components.
	
Given this data, 
we construct a marked surface in the sense of Definition \ref{def:marked surface} 
by choosing a tubular neighborhood $U$ of $\BB\cup \CB$ and setting
\[
\Sigma = \overline{X \smallsetminus U} \, , \quad
\Pi = \pi_0 \big(\overline{ (\del X) \smallsetminus U}\big)
\]
where here $\overline{(-)}$ denotes the closure. 
In other words, 
each marked point in $\BB$ 
contributes an interval to $\Pi$ 
(lying to one side of it on a component of $\del X$),
while each marked point in $\CB$ gives rise to an $S^1$ component of 
$\del \Sigma$ that is disjoint from $A(\Pi)$.

The ``inverse'' construction is provided by contracting all components of 
$\partial \Sigma \smallsetminus \operatorname{int}(A(\Pi))$ to (marked) points.
Hence, the connected components of $\partial\Sigma \smallsetminus \operatorname{int} A(\Pi)$ 
in our setup correspond to markings in the literature.
As an example, \eqref{eq:polygon} depicts a marked surface $(\Sigma,\Pi)$ 
on the left and the corresponding $(X,\BB,\CB=\emptyset)$ on the right.
	\end{rem}
	
We now formulate what it means to cut a surface $\Sigma$ along a neatly embedded arc $\gamma$. 
Such arcs meet the boundary of $\Sigma$ only at their boundary, 
hence we will refer to them as \emph{internal arcs}. 
(These should be viewed in contrast to the arcs $\beta \in \Pi$ from Definition \ref{def:marked surface}.)

\begin{construction}[Cutting along an arc]
		\label{const:cutting along arc}
Let $\gamma$ be an internal arc in a compact, oriented surface $\Sigma$. 
Let $\overline{\Sigma\smallsetminus \gamma}$ denote the result of cutting $\Sigma$ along $\gamma$, 
pictured as in:
		
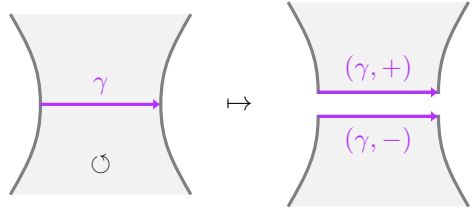
\[
		\begin{tikzpicture}
		[anchorbase,scale=.8,rotate=90]
		\fill [fill=gray, opacity=.1]
		(-1.5,1.5) to[out=-30,in=180] (0,1) to[out=0,in=210] (1.5,1.5) to 
		(1.5,-1.5) to[out=150,in=0] (0,-1) to[out=180,in=30] (-1.5,-1.5) to (-1.5,1.5);
		\draw[very thick, gray]  (-1.5,1.5) to[out=-30,in=180] (0,1) to[out=0,in=210] (1.5,1.5);
		\draw[very thick, gray] (1.5,-1.5) to[out=150,in=0] (0,-1) to[out=180,in=30] (-1.5,-1.5);
		\path[very thick,seam,<-] (0,-1) edge node[above] {$\scs{\gamma}$} (0,1);
		\node at (-1,0) {$\circlearrowleft$};
		\end{tikzpicture}
		\ \ \ \mapsto  \ \ \ 
		\begin{tikzpicture}
		[anchorbase,scale=.8,rotate=90]
		\begin{scope}[shift={(-.2,0)}]
		\fill [fill=gray, opacity=.1]
		(-1.5,1.5) to[out=-30,in=180] (0,1) to (0,-1) to[out=180,in=30] (-1.5,-1.5) to (-1.5,1.5);
		\draw[very thick, gray]
		(-1.5,1.5) to[out=-30,in=180] (0,1) to (0,-1) to[out=180,in=30] (-1.5,-1.5);
		\path[very thick,seam,<-] (0,-1) edge node[below] {$\scs{(\gamma,-)}$} (0,1);
		\end{scope}
		\begin{scope}[shift={(.2,0)}]
		\fill[fill=gray, opacity=.1]
		(1.5,-1.5) to[out=150,in=0] (0,-1) to  (0,1) to[out=0,in=210] (1.5,1.5) to (1.5,-1.5);
		\draw[very thick, gray]
		(1.5,-1.5) to[out=150,in=0] (0,-1) to  (0,1) to[out=0,in=210] (1.5,1.5);
		\path[very thick,seam,<-] (0,-1) edge node[above] {$\scs(\gamma,+)$} (0,1);
		\end{scope}
		\end{tikzpicture}
		\]
The net result of this construction is to delete (a tubular neighborhood of) $\gamma$ from $\Sigma$
and then close by gluing in two disjoint copies of $\gamma$, 
denoted by $(\gamma,\pm)$ according to their compatibility with the orientation induced from $\Sigma$. 
More generally, suppose $\Gamma$ is a finite set of pairwise disjoint internal arcs 
in an oriented surface $\Sigma$.
Iterating Construction \ref{const:cutting along arc} for all the arcs in 
$\Gamma$ defines $\overline{\Sigma \smallsetminus \Gamma}$.

		\end{construction}

\begin{definition}
	\label{def:seamed surface}
A \emph{seamed marked surface} is a tuple $(\Sigma,\Pi;\Gamma)$ 
where $(\Sigma,\Pi)$ is a marked surface 
and $\Gamma$ is a finite set of pairwise disjoint internal arcs in $\Sigma$, 
each disjoint from $A(\Pi)$, 
such that the \emph{cut surface} $\cutS:=\overline{\Sigma \smallsetminus \Gamma}$ 
is homeomorphic to a disjoint union of closed disks. 

The arcs in $\Gamma$ will be called \emph{seams}. 
The connected components of $\cutS$ will be called \emph{regions}, 
and we denote the set of regions by $\Reg(\Sigma;\Gamma)$.
	\end{definition}

Observe that a marked surface $\Sigma$ admits the structure of a seamed marked
surface if and only if $\Sigma$ admits a handle decomposition 
consisting only of $0$- and $1$-handles or, equivalently, 
if every connected component has non-empty boundary. 
Moving forward, \emph{we will only consider these marked surfaces}.
We will typically drop either $\Gamma$ or $\Pi$ 
from the notation in Definition \ref{def:seamed surface} when they are empty,
or when a given construction is agnostic to their presence.

\begin{construction} (Cutting a seamed marked surface)
	\label{constr:cutseamedsurf}
Let $(\Sigma,\Pi;\Gamma)$ be a seamed marked surface 
and write $\Pi=\{\b_1,\dots,\b_\ell\}$ in its given total order. 
The additional choice of a total order on $\Gamma=\{\gamma_1,\ldots,\gamma_g\}$ 
equips the cut surface $\cutS$ with the structure of 
a seamed marked surface $(\cutS,\cutP) = (\cutS,\cutP;\emptyset)$, 
where $\cutP:= \gpm\cup \Pi$ is ordered via
\begin{equation}
	\label{eq:totalorder}
(\gamma_1,-)<(\gamma_1,+)< (\gamma_2,-) < (\gamma_2,+)<
	\cdots<(\gamma_g,-)<(\gamma_g,+)<\b_1<\cdots<\b_\ell \, .
\end{equation}
The arcs in $\gpm$ will be referred to as \emph{cut seams}.
\end{construction}

Let $\reg\colon \cutP \to \Reg(\Sigma;\Gamma)$ 
denote the map sending an arc to the region it abuts.
For each $D\in \Reg(\Sigma;\Gamma)$, let
$\cutP|_D := \reg\inv(D)$, 
so $(D,\cutP|_D)$ is a marked surface 
using the total order of $\cutP|_D$ given by restricting \eqref{eq:totalorder}.

\begin{example}
Here we display $(\Sigma, \{\ma_1,\ma_2,\ma_3\}; \{\gamma\})$ and its (lone) region:
	
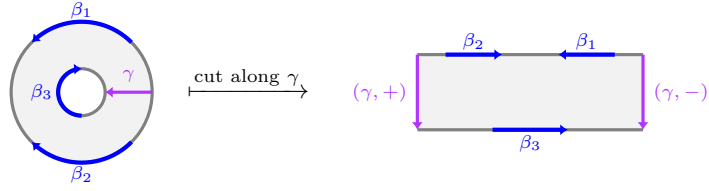
\[
	\begin{tikzpicture}[anchorbase,scale=1.25]
	\fill [fill=gray, opacity=.1] (.25,0) arc[start angle=0,end angle=360,radius=.25] to (.75,0)
		arc[start angle=0,end angle=-360,radius=.75] to (.25,0);
	\draw[very thick, gray] (0,0) circle (.25);
	\draw[very thick, gray] (0,0) circle (.75);
	\draw[very thick,seam,<-] (.25,0) to node[above]{\scs$\gamma$} (.75,0);
	\draw [blue,ultra thick,domain=45:135,->] plot ({.75*cos(\x)}, {.75*sin(\x)});
	\draw [blue,ultra thick,domain=225:315,<-] plot ({.75*cos(\x)}, {.75*sin(\x)});
	\draw [blue,ultra thick,domain=90:270,<-] plot ({.25*cos(\x)}, {.25*sin(\x)});
	\node[blue] at (0,.875) {\scs$\ma_1$};
	\node[blue] at (0,-.875) {\scs$\ma_2$};
	\node[blue] at (-.425,0) {\scs$\ma_3$};
	\end{tikzpicture}
	\quad 
	\xmapsto{\text{cut along } \gamma} 
	\quad
	\begin{tikzpicture}[anchorbase,scale=1]
	\fill [fill=gray, opacity=.1] (-1.5,-.5) rectangle (1.5,.5);
	\draw[very thick, gray] (1.5,.5) to (-1.5,.5);
	\draw[very thick, gray] (1.5,-.5) to (-1.5,-.5);
	\draw[very thick,seam,<-] (1.5,-.5) to node[right]{\scs$(\gamma,-)$} (1.5,.5);
	\draw[very thick,seam,<-] (-1.5,-.5) to node[left]{\scs$(\gamma,+)$} (-1.5,.5);
	\draw [blue,ultra thick,->] (-1.125,.5) to (-.375,.5);
	\draw [blue,ultra thick,<-] (.375,.5) to (1.125,.5);
	\draw [blue,ultra thick,->] (-.5,-.5) to (.5,-.5);
	\node[blue] at (.75,.675) {\scs$\ma_1$};
	\node[blue] at (-.75,.675) {\scs$\ma_2$};
	\node[blue] at (0,-.675) {\scs$\ma_3$};
	\end{tikzpicture}
	\]
	With the orientation induced by the standard orientation of the plane, 
	we have $\sgn(\ma_1) = +1 = \sgn(\ma_3)$ and $\sgn(\ma_2) = -1$.
	\end{example}

If $\iota\colon I\rightarrow \Sigma$ is an arc with image $C$, 
then a finite set of points $\pp\subset C$ will be called \emph{standard} if $\pp = \iota(\pp^n)$, 
where $\pp^n \subset (0,1)$ is our chosen set of $n$ points from \S \ref{ss:BN}. 
If $C_1\cup\cdots \cup C_r$ is a union of finitely pairwise disjoint arcs in $\Sigma$, 
then $\pp\subset C_1\cup \cdots \cup C_r$ is \emph{standard} if each $\pp\cap C_i$ is standard in $C_i$.

\begin{definition}\label{def:tangle in SMS}
A (crossingless) \emph{tangle} in a seamed marked surface $(\Sigma,\Pi;\Gamma)$ 
is a neatly embedded $1$-submanifold 
$(T,\partial T) \subset (\Sigma, A(\Pi))$ 
such that all intersections $T\cap \ma_i$ and $T\cap \gamma_j$ are transverse and standard. 
We let $\Tan(\Sigma,\Pi;\Gamma)$ denote the set of all such tangles.
	\end{definition}	
	
Refining the notation in Definition \ref{def:tangle in SMS},
if $\pp \subset A(\Pi)$ 
is standard, 
then we let $\Tan(\Sigma,\mathbf{p};\Gamma)$ denote the 
subset of $\Tan(\Sigma,\Pi;\Gamma)$ consisting of tangles with $\del T = \pp$.
In the event that $\Gamma$ is empty, 
we abbreviate $\Tan(\Sigma,\mathbf{p}):=\Tan(\Sigma,\mathbf{p};\emptyset)$.

\subsection{Chain complexes for glued disks}
\label{ss:cflat i}

Fix a seamed marked surface $(\Sigma,\Pi;\Gamma)$. 
In this section we define the complex $\PMS_{\Sigma,\Pi;\Gamma}(T|V|S)$ 
associated to a pair of tangles $T,S \in \Tan(\Sigma,\Pi;\Gamma)$
and a 1-morphism $V$ in the bicategory $\bn(\Pi)$ defined as follows.

\begin{definition}\label{def:bn Pi}
Let $(\Sigma,\Pi)$ be a marked surface.  
Associate to $\Pi = \{\ma_1,\ldots,\ma_{\ell}\}$ the following bicategory:
\begin{equation}\label{eq:boundary bicat}
\bn(\Pi) 
:= \bn^{\sgn_\Sigma(\ma_1)} \otimes \cdots \otimes \bn^{\sgn_\Sigma(\ma_\ell)}
\end{equation}
where $\bn^+=\bn$ and $\bn^-=\bn^{\op}$.
\end{definition}

Note that the order on the tensor factors in \eqref{eq:boundary bicat} agrees with the 
total order on $\Pi$ given in Definition \ref{def:marked surface}.
If $(\Sigma,\Pi)$ is positive, then we have a canonical identification 
$\bn(\Pi) \cong \bn^{\otimes \ell}$. 
Otherwise we will use the following identification. 

\begin{definition}\label{def:prismBC to standardBC} 
Given a marked surface $(\Sigma,\Pi)$,
let $\Phi=\Phi_{\Sigma,\Pi}\colon \bn(\Pi) \to \bn^{\otimes \ell}$ 
be the 2-functor obtained as the tensor product of 2-functors 
$\Id \colon \bn \to \bn$ on factors for which $\sgn_\Sigma(\ma_i)=+1$ and 
2-functors $r_{xz} \colon \bn^{\op} \to \bn$ on factors for which $\sgn_\Sigma(\ma_i)=-1$. 
Explicitly, on 1-morphisms we have
\begin{equation}\label{eq:Vpm}
\Phi(V_1,\dots,V_\ell) := (V_1^{\sgn_\Sigma(\ma_1)},\ldots,V_\ell^{\sgn_\Sigma(\ma_\ell)}) 
\quad \text{where} \quad
V_i^{\sgn_\Sigma(\ma_i)} = \begin{cases} V_i & \text{ if $\sgn_\Sigma(\ma_i)=+1$} \\ 
				r_{xz}(V_i) & \text{ if $\sgn_\Sigma(\ma_i)=-1$.}
				\end{cases}
\end{equation}
\end{definition}

\begin{example}
	\label{ex:instances of boundary bicats}
Let $(\Sigma,\Pi;\Gamma)$ be a seamed marked surface. 
Choose a total order $\Gamma=\{\gamma_1,\ldots,\gamma_g\}$ and consider the 
marked surface $(\cutS,\cutP)$ of Construction~\ref{constr:cutseamedsurf}.
As an important special cases of Definition \ref{def:bn Pi}, 
we have
\[
\bn(\cutP) = \bn(\gpm)\otimes \bn(\Pi)
\]
where $\bn(\gpm) :=\bigotimes_{i=1}^g (\bn^\op \otimes \bn)$
and (as above) $\bn(\Pi)=\bigotimes_{i=1}^\ell\bn^{\sgn_\Sigma(\b_i)}$.
\end{example}

\begin{conv}
	\label{conv:tangles to objects}
Given $S,T \in \Tan(\Sigma,\Pi;\Gamma)$, 
we introduce the following shorthand for certain 1-morphism categories 
of the bicategories appearing in Example \ref{ex:instances of boundary bicats}:
\begin{gather*}
\bni(\Pi)^{\partial T}_{\partial S}:= 
	\bigotimes_{i=1}^\ell \left(\bn^{|\b_i \cap \partial T|}_{|\b_i \cap \partial S|}\right)^{\sgn(\b_i)} 
\, , \quad 
\bn(\gpm)^{T}_{S} := 
	\bigotimes_{i\in 1}^g \left(\left(\bn^{|\gamma_i \cap T|}_{|\gamma_i \cap S|}\right)^\op 
	\otimes \bn^{|\gamma_i \cap T|}_{|\gamma_i \cap S|}\right)
\\
\bn(\cutP)^{T}_{S} := \bn(\gpm)^{T}_{S} \otimes  \bni(\Pi)^{\partial T}_{\partial S} \, .
\end{gather*}
Observe that each of these categories depends only very coarsely 
on the tangles $S,T \in \Tan(\Sigma,\Pi;\Gamma)$, 
i.e.~they only depend on the (standard) intersection of the tangles with the boundary edges and/or seams.
Consequently, given standard subsets $\pp,\qq \subset A(\Pi)$ we similarly 
denote the evident $1$-morphism category by $\bni(\Pi)^{\qq}_{\pp}$
\end{conv}

The bicategory $\bn(\cutP)$ is isomorphic to a tensor product of bicategories 
associated to the regions of $(\Sigma,\Pi;\Gamma)$.
Given an ordering $\Reg(\Sigma;\Gamma) = \{D_1,\ldots,D_R\}$
and tangles $S, T \in \Tan(\Sigma,\Pi;\Gamma)$, 
set $S_i := S|_{D_i}$ and $T_i := T|_{D_i}$.
There are then isomorphisms of (bi)categories
\begin{equation}
	\label{eq:reorder}
\bni(\cutP)\cong \bigotimes_{i=1}^R \bni\big(\cutP|_{D_i}\big) 
\, , \quad
\bni(\cutP)^{T}_{S} \cong 
	\bigotimes_{i=1}^R \bni\big(\cutP|_{D_i}\big)^{\partial T_i}_{\partial S_i}
\end{equation}
given by reordering tensor factors via the total ordering of $\cutP$.  
Since our regions are disks, 
we can now apply the prism spaces/modules from \S \ref{ss:arbitrary disks} 
to the present setup.

\begin{definition}[Cut surface modules]
\label{def:cut surface module}
Let $(\Sigma,\Pi;\Gamma)$ be a seamed marked surface. Choose:
\begin{enumerate}
	\item a linear ordering on the set of seams $\Gamma=\{\gamma_1,\ldots,\gamma_g\}$,
		\label{los}
	\item a linear ordering on the set of regions 
		$\Reg(\Sigma;\Gamma) = \{D_1,\ldots,D_R\}$, and
		\label{lor}
	\item a standardization $\pi_i$ of $(D_i,A_i)$, where $A_i=A(\cutP|_{D_i})$.
		\label{cos}
\end{enumerate}
With these choices made, 
and given tangles $T,S \in \Tan(\Sigma,\Pi;\Gamma)$, let
\begin{equation}
	\label{eq:MTSexpl}
\PMS_{\cutS,\cutP}(T|-|S) \colon 
		\bni(\cutP)^{T}_{S} \to \gModen{\Z}{\cring}
\end{equation}
be the tensor product $\bigotimes_{i=1}^R \PMS_{D_i,A_i,\pi_i}(T_i|-|S_i)$ of functors 
from Definition \ref{def:BNprismarbitrary},
precomposed with the tensor product $\bigotimes_{i=1}^R \Phi_{D_i,\cutP|_{D_i}}$ 
of functors from Definition \ref{def:prismBC to standardBC}, precomposed with the 
isomorphism \eqref{eq:reorder} which reorders the tensor factors.
\end{definition}

Unravelling Definition \ref{def:cut surface module}, 
the action of $\PMS_{\cutS,\cutP}(T|-|S)$ on objects 
$V \in \bni(\cutP)^{T}_{S}$ is given as follows.  
If $V=(V_h)_{h \in \cutP}$, then
\[
\PMS_{\cutS,\cutP}(T|V|S) = 
\bigotimes_{i=1}^R \PMS_{\std} \big( \pi_i T_i | 
	V^{\sgn_{\cutS}(h_{i,1})}_{h_{i,1}},\ldots,V^{\sgn_{\cutS}(h_{i,k_{i}})}_{h_{i,k_{i}}}
		 | \pi_i S_i \big)
\]
where here $h_{i,1},\ldots,h_{i,k_{i}}$
denote the elements of $\cutP|_{D_i}$ 
written in the order $[\pi_i] \in \LO(A(\cutP|_{D_i}))$ 
from Definition \ref{def:cyclic refinement}.
We will denote the extension of \eqref{eq:MTSexpl} to categories of bounded above chain 
complexes similarly:
\begin{equation}
	\label{eq:MTScxexpl}
	\PMS_{\cutS,\cutP}(T|-|S) \colon \Ch^- \left(\bni(\cutP)^T_S \right) 
		\rightarrow \Ch^-\big(\gModen{\Z}{\cring}\big) \, .
\end{equation}

We can view the cut surface module as being associated to the surface
$\overline{\Sigma \smallsetminus \Gamma}$ obtained by cutting $\Sigma$ along $\Gamma$. 
We now associate a $\bni(\Pi)^{\partial T}_{\partial S}$-module to $\Sigma$ itself
by ``gluing the cut surface module along $\Gamma_\pm$.''

\begin{definition}
\label{def:Cflat} 
Given $T,S \in \Tan(\Sigma,\Pi;\Gamma)$ consider the ``tensor product'' 
of bar complexes
	\begin{equation}
		\label{eq:bar}
	B^T_S:= \Big( \Bar^{|\gamma_1\cap T|}_{|\gamma_1\cap S|} , \ldots ,
	\Bar^{|\gamma_g \cap T|}_{|\gamma_g \cap S|} \Big) \in \Ch^-\left(\bni(\gpm)^T_S \right)
	\end{equation} 
and define
	\begin{equation}
		\label{eq:MTScx}
		\PMS_{\Sigma,\Pi;\Gamma}(T|-|S)
		\colon \bni(\Pi)^{\partial T}_{\partial S} 
			\rightarrow \Ch^-\big(\gModen{\Z}{\cring}\big)
	\end{equation} 
to be the functor given by the contraction of
the dg functor $\PMS_{\cutS,\cutP}(T|-|S)$ from \eqref{eq:MTScxexpl} with $B^T_S$, i.e.~
	\begin{equation}
		\label{eq:CflatcxW}
	\PMS_{\Sigma,\Pi;\Gamma}(T|V|S) := 
		\PMS_{\cutS,\cutP}(T | B_S^T,V_1,\ldots,V_\ell | S)
	\end{equation}
for all $V = (V_1,\ldots,V_\ell) \in \bn(\Pi)^{\partial T}_{\partial S}$.
	\end{definition}

Note that $\PMS_{\Sigma,\Pi;\Gamma}(T|V|S)$ is an object of $\Ch^-\big(\gModen{\Z}{\cring}\big)$. 
When any of the data $\Sigma,\Pi$ or $\Gamma$ are understood, 
we will omit them from the notation, 
e.g.~writing $\PMS(T|V|S)$ or $\PMS_{\Sigma,\Pi}(T|V|S)$.
In the special case where $\partial S = \partial T = \pp$ and $V=\one_\pp$, we abbreviate by writing
	\begin{equation}
		\label{eq:Cflatcx}
	\PMS_{\Sigma,\Pi;\Gamma}(T,S):= \PMS_{\Sigma,\Pi;\Gamma}(T|\one_\pp|S).
	\end{equation}
 
\begin{theorem}
Different choices for \eqref{los}, \eqref{lor}, or \eqref{cos}
in Definition \ref{def:cut surface module} yield 
canonically isomorphic functors 
$\PMS_{\Sigma,\Pi;\Gamma}(T|-|S)$
in Definition \ref{def:Cflat}.
\end{theorem}
\begin{proof}
Different choices for \eqref{los} and \eqref{lor} will yield canonically isomorphic
complexes using the symmetric monoidal structure on dg categories
to permute tensor factors in the domain.
Since the ordering of the tensor factors in $\bni\big(\cutP|_{D_i}\big)$ is given by 
$[\pi_i] \in \LO(A(\cutP|_{D_i}))$, 
Corollary \ref{cor:prism2M2} gives that different choices \eqref{cos} 
of standardizations $\pi_i$ of the regions 
$\big(D_i, A(\cutP|_{D_i}) \big)$ for $i=1,\ldots,R$ 
will yield canonically isomorphic functors $\PMS_{\Sigma,\Pi;\Gamma}(T|-|S)$.
\end{proof}

\subsection{Graphical model}
\label{sec:graphicalmodel}

The chain complexes $\PMS_{\Sigma,\Pi;\Gamma}(T,S)$ from \eqref{eq:Cflatcx} are defined by
evaluating prism modules on bar complexes and identity tangles. 
In terms of planar evaluations, these functors treat the two sides of each bar complex differently:
given \eqref{eq:totalorder} and Definition \ref{def:cut surface module},
the functor $r_{xz}$ is applied to the first factor of each $\Bar_{m_i}^{n_i}$.
In Section~\ref{s:Rozanskybar}, we introduced graphical notation 
for $(r_{xz}\otimes\id)(\Bar^n_m) = \tBar_m^n$.
We extend this notation to multiple tensor factors as follows:
\[
\big(\Bar^{n_1}_{m_1} , \ldots , \Bar^{n_g}_{m_g} \big)
\xmapsto{(r_{xz}\otimes\id)^{\otimes g}}
\big( \tBar^{n_1}_{m_1} , \ldots , \tBar^{n_g}_{m_g} \big) =
\left(\begin{tikzpicture}[anchorbase,scale=1]
	\draw[thick, double] (0,-0.6) node[below=-2pt]{\scs $m_1$} 
					to (0,0.6) node[above=-2pt]{\scs $n_1$};
	\rjf{0}{0}{1}
	\end{tikzpicture}
,
	\begin{tikzpicture}[anchorbase,scale=1]
	\draw[thick, double] (0,-0.6) node[below=-2pt]{\scs $m_1$} 
					to (0,0.6) node[above=-2pt]{\scs $n_1$};
	\ljf{0}{0}{1}
\end{tikzpicture}
\, , \, 
\begin{tikzpicture}[anchorbase,scale=1]
	\draw[thick, double] (0,-0.6) node[below=-2pt]{\scs $m_2$} 
					to (0,0.6) node[above=-2pt]{\scs $n_2$};
	\rjf{0}{0}{2}
	\end{tikzpicture}
,
	\begin{tikzpicture}[anchorbase,scale=1]
	\draw[thick, double] (0,-0.6) node[below=-2pt]{\scs $m_2$} 
					to (0,0.6) node[above=-2pt]{\scs $n_2$};
	\ljf{0}{0}{2}
\end{tikzpicture}
\, , \,
\ldots \, , \,
	\begin{tikzpicture}[anchorbase,scale=1]
	\draw[thick, double] (0,-0.6) node[below=-2pt]{\scs $m_g$} 
					to (0,0.6) node[above=-2pt]{\scs $n_g$};
	\rjf{0}{0}{$g$}
	\end{tikzpicture}
,
	\begin{tikzpicture}[anchorbase,scale=1]
	\draw[thick, double] (0,-0.6) node[below=-2pt]{\scs $m_g$} 
					to (0,0.6) node[above=-2pt]{\scs $n_g$};
	\ljf{0}{0}{$g$}
\end{tikzpicture}
\right) \, .
\]

Since we can evaluate any 
$\bigotimes_{i=1}^g \big(\bn_{m_i}^{n_i}\big)^{\otimes 2}$-module
on $\big(\tBar^{n_1}_{m_1} , \ldots , \tBar^{n_g}_{m_g} \big)$, 
planar Bar-Natan diagrams containing one or more pairs of such purple boxes are well-defined.
Recall from Remark \ref{rmk:purple boxes merge to roz} 
that when two purple boxes are paired (i.e.~they correspond to the same copy of $\tBar_m^n$)
and occur adjacent to one another in the plane with arrows directed away from one another, 
we may replace them with a Rozansky projector, as in \eqref{eq:purpleprojector}.
We will often suppress the arrows of the purple boxes using the following bookkeeping device. 
If $E$ is a finite set and $h \in E \times \{-1,1\}$, then we set
\[
\begin{tikzpicture}[anchorbase,scale=1]
	\draw[thick, double] 
	(0,-0.6) to (0,0.6);
	\draw[seam, thick, fill=white] (-0.2,-.4) rectangle (0.2,+.4);
	\node at (0,0) {\small$h$};
	\node at (0,.8) {\tiny$n$};
	\node at (0,-.8) {\tiny$m$};
\end{tikzpicture}
:=\begin{tikzpicture}[anchorbase,scale=1]
	\draw[thick, double] 
	(0,-0.6) to (0,0.6);
	\node at (0,.8) {\tiny$n$};
	\node at (0,-.8) {\tiny$m$};
	\rjf{0}{0}{e}
\end{tikzpicture}
\ \ \text{ if } \ \ h =(e,-) \ \  , \qquad  
\begin{tikzpicture}[anchorbase,scale=1]
	\draw[thick, double] 
	(0,-0.6) to (0,0.6);
	\node at (0,.8) {\tiny$n$};
	\node at (0,-.8) {\tiny$m$};
	\draw[seam, thick, fill=white] (-0.2,-.4) rectangle (0.2,+.4);
	\node at (0,0) {\small$h$};
\end{tikzpicture}
:=\begin{tikzpicture}[anchorbase,scale=1]
	\draw[thick, double] 
	(0,-0.6) to (0,0.6);
	\node at (0,.8) {\tiny$n$};
	\node at (0,-.8) {\tiny$m$};
	\ljf{0}{0}{e};
\end{tikzpicture}
\ \ \text{ if } \ \ h=(e,+) \, .
\]

Using these conventions,
the complex $\PMS_{\Sigma,\Pi;\Gamma}(T,S)$
from \eqref{eq:Cflatcx} is described as follows:
\begin{equation}
\label{eq:homcomplex}
\begin{aligned}
\PMS_{\Sigma,\Pi;\Gamma}(T,S) 
&:= 
\bigotimes_{i=1}^R
q^{|\del S_i\cup \del T_i|/4}  \KhEval{\ \ 
\begin{tikzpicture}[anchorbase,scale=1]
\begin{scope}[shift={(0,.75)}]\fTbox{1.8}{\pi_i T_i}\end{scope}
\begin{scope}[shift={(0,-.75)}]\fTbox{1.8}{r_y(\pi_i S_i)}\end{scope}
%
\begin{scope}[shift={(-1,0)}]
	\dottedv{-.5} 
\draw[thick,double] (0,-0.5) to (0,0.5); 
\draw[seam, thick, fill=white] (-0.25,-.3) rectangle (0.25,+.3); 
\node at (0,0) {\tiny$h_{i,1}$}; 
	\dottedv{.5} 
\end{scope}
\node at (0,0) {$\cdots$};
%
\begin{scope}[shift={(1,0)}]
\dottedv{-.5}
\draw[thick,double] (0,-0.5) to (0,0.5); 
\draw[seam, thick, fill=white] (-0.275,-.3) rectangle (0.275,+.3); 
\node at (0,0) {\tiny$h_{i,\ell_i}$}; 
\dottedv{.5}
\end{scope}
\end{tikzpicture} \ \ 
	} 
\end{aligned}
\end{equation}
where $\{h_{i,1},\ldots,h_{i,\ell_{i}}\} = \cutP|_{D_i} \cap \Gamma_{\pm}$ 
is the set of cut seams in $\del D_i$ ordered via the restriction of $[\pi_i]$ to this subset,
and the dotted vertical strands indicate (several or zero) copies of $\one$
corresponding to arcs in $\cutP|_{D_i} \cap \Pi$ that (may) lie between 
the cuts seams in $\cutP|_{D_i} \cap \Gamma_{\pm}$.
An analogous diagrammatic description is available for
the chain complexes $\PMS_{\Sigma,\Pi;\Gamma}(T|V_{\bullet}|S)$ from \eqref{eq:CflatcxW}. 
In this case, the objects $V_h$ or $r_x(V_h)$ are inserted in place of the dotted vertical strands, 
depending on the orientation of the arcs $h \in \cutP|_{D_i} \cap \Pi$.

\begin{remark}
	\label{rem:rotatefactor}
	The presentation of each tensor factor in \eqref{eq:homcomplex}
	can be adjusted by picking a standardization $\pi_i' \colon (D_i,A_i) \to (\Dst,\Ast)$
	that gives a different linear refinement $[\pi_i'] \in \LO( A(\cutP|_{D_i}))$, 
	and using the pivotal sphericality of the Bar-Natan categories.
	In particular, we may pick any cut seam $h\in \cutP|_{D_i} \cap \Gamma_{\pm}$ 
	and express the $i$-{th} tensor factor of \eqref{eq:homcomplex} 
	with the $h$-labelled box appearing on the far left or far right:
\[
\KhEval{ \ \ 
\begin{tikzpicture}[anchorbase,scale=1]
\begin{scope}[shift={(0,.75)}]\fTbox{1.8}{\pi_i' T_i}\end{scope}
\begin{scope}[shift={(0,-.75)}]\fTbox{1.8}{r_y(\pi_i' S_i)}\end{scope}
%
\begin{scope}[shift={(-1,0)}]
\dottedv{0} 
\ljf{0}{0}{\gamma}
\dottedv{.4} 
\end{scope}
\node at (.5,0) {$\cdots$};
\end{tikzpicture} \ \ 
}
\quad \text{resp.} \quad
\KhEval{ \ \ 
\begin{tikzpicture}[anchorbase,scale=1]
\begin{scope}[shift={(0,.75)}]\fTbox{1.8}{\pi_i' T_i}\end{scope}
\begin{scope}[shift={(0,-.75)}]\fTbox{1.8}{r_y(\pi_i' S_i)}\end{scope}
%
\begin{scope}[shift={(1,0)}]
\dottedv{-.4} 
\dottedv{0} 
\rjf{0}{0}{\gamma}
\end{scope}
\node at (-.5,0) {$\cdots$};
\end{tikzpicture} \ \ 
}
\]
depending on whether $h=(\gamma,+)$ or $h=(\gamma,-)$, respectively.  
\end{remark}

\begin{example}\label{ex:annulus} 
We consider the case when $\Sigma$ is a standardly-oriented annulus, 
$\Pi$ consists of two arcs (one on each boundary component) 
whose orientations agree with that of $\Sigma$, and $\Gamma = \{\gamma\}$.
Let $T \in \Tan(\Sigma,\Pi;\Gamma)$ with $| T \cap \gamma |=n$ 
and its image under a standardization of the (lone) region $\cutS$ be depicted as follows:
\begin{equation}
\label{eq:tangle in annulus}
T = 
\begin{tikzpicture}[anchorbase,scale=1]
\draw[very thick, gray] (0,0) circle (.5);
\draw[very thick, gray] (0,0) circle (1.5);
\draw[thick,double] (0,0) circle (1);
\draw[thick,double] (0,-1.5) to (0,-.5);
\node at (0,-1.7) {\tiny$k$};
\node at (0,-.3) {\tiny$l$};
\node at (1.2,0) {\tiny$n$};
\begin{scope}[shift={(0,-1)}] \draw[fill=white] (0,0) circle (.25);\end{scope}
\node at (0,-1) {$T$};
\draw[<-, seam,very thick] (-1.5,0) to node[pos=.25,above=-2pt]{\scs$\gamma$} (-.5,0);
\draw [blue,ultra thick,domain=240:300,->] plot ({1.5*cos(\x)}, {1.5*sin(\x)});
\draw [blue,ultra thick,domain=225:315,<-] plot ({.5*cos(\x)}, {.5*sin(\x)});
\end{tikzpicture} \qquad,\qquad
\pi T =
\begin{tikzpicture}[anchorbase,scale=1]
\draw[thick, double] (.25,1) to [out=0,in=270] (.5,1.25) to [out=90,in=0] (-.5,2)
	to [out=180,in=90] (-1.5,0) node[below=-1pt]{\tiny$n$};
\draw[thick, double] (0,1.25) to [out=90,in=0] (-.5,1.5) to [out=180,in=90] (-1,0) node[below=-2pt]{\tiny$l$};
\draw[thick, double] (-.25,1) to [out=180,in=90] (-.5,0) node[below=-1pt]{\tiny$n$};
\draw[thick, double] (0,1) to (0,0) node[below=-2pt]{\tiny$k$};
\draw[thick,fill=white] (0,1) circle (.25);
\node at (0,1) {$T$};
\end{tikzpicture}
=
\begin{tikzpicture}[anchorbase,scale=1]
\begin{scope}[shift={(0,.8)}] \fTbox{1.1}{\pi T} \end{scope}
\draw[thick,double]
(.3,-.55) to (.3,.55) 
(.9,-.55) to (.9,.55) 
(-.3,-.55) to (-.3,.55) 
(-.9,-.55) to (-.9,.55);
\node at (-1.1,0) {\tiny$n$};
\node at (-.5,0) {\tiny$l$};
\node at (.1,0) {\tiny$n$};
\node at (.7,0) {\tiny$k$};
\end{tikzpicture} \, .
\end{equation}
Given another tangle $S \in \Tan(\Sigma,\del T;\{\gamma\})$, let
$m=|S \cap \gamma|$. The complex $\PMS_{\Sigma,\Pi;\Gamma}(T,S)$ is then
computed as:
\[
\PMS_{\Sigma,\Pi;\Gamma}(T,S) =	
q^{\frac{n+m+k+l}{2}} \KhEval{ \ 
\begin{tikzpicture}[anchorbase,scale=1]
\begin{scope}[shift={(0,1)}] \fTbox{1.1}{\pi T} \end{scope}
\begin{scope}[shift={(0,-1)}] \fTbox{1.1}{r_y(\pi S)} \end{scope}
\draw[thick,double]
(.3,-.75) to (.3,.75) 
(.9,-.75) to (.9,.75) 
(-.3,-.75) to (-.3,.75) 
(-.9,-.75) to (-.9,.75);
\node at (-1.1,.55) {\tiny$n$};
\node at (-.43,0) {\tiny$l$};
\node at (.1,.55) {\tiny$n$};
\node at (.775,0) {\tiny$k$};
\node at (-1.1,-.55) {\tiny$m$};
\node at (.1,-.55) {\tiny$m$};
\rjf{-.9}{0}{}
\ljf{.3}{0}{}
\end{tikzpicture}
\ }	
\simeq
q^{\frac{k+l}{2}} \KhEval{\ 
\begin{tikzpicture}[anchorbase,scale=1]
\begin{scope}[shift={(0,1)}] \fTbox{1.2}{\pi T} \end{scope}
\begin{scope}[shift={(0,-1)}] \fTbox{1.2}{r_y(\pi S)} \end{scope}
\draw[thick,double] (-.1,-.75) to[out=90,in=270] (.5,-.2) to (.5,.2) to[out=90,in=270] (-.1,.75);
\begin{scope} 
	\clip (-.5,.5) rectangle (.5,-.5);
	\draw[white, line width=4pt] 
		(-.1,-.3) to[out=0,in=90] (.5,-.75)
		(-.1,.3) to[out=0,in=270] (.5,.75);
	\end{scope}
\draw[thick,double]
(-.9,-.75) to[out=135,in=180] (-.7,-.3)
(-.9,.75) to[out=225,in=180] (-.7,.3)
(-.1,-.3) to[out=0,in=90] (.5,-.75)
(-.1,.3) to[out=0,in=270] (.5,.75);
\draw[thick,double] (.9,-.75) to (.9,.75);
\draw[fill=white] (-.7,-.5) rectangle (-.1,.5);
\node[rotate=90] at (-.4,0) {$\Bproj_{n+m}$};
\node at (.375,0) {\tiny$l$};
\node at (.775,0) {\tiny$k$};
\end{tikzpicture}
\ } \, .
\]

The second diagram here illustrates how one may express $\PMS_{\Sigma,\Pi;\Gamma}(T,S)$ entirely
using Rozansky projectors; in general, this will involve diagrams that involve tangles with crossings.
Such crossings determine chain complexes over $\bn$ via the formalism for Khovanov homology 
developed in \cite{BN2}.
Note, however, that such crossings are superfluous, up to homotopy equivalence:
the Rozansky projector is constructed from ``split tangles'', as indicated in the first diagram.
This description allows us to make contact with the 
description given in \S \ref{ss:intro summary},
i.e.~$q^{\frac{-(k+l)}{2}}\PMS_{\Sigma,\Pi;\Gamma}(T,S)$ is given by
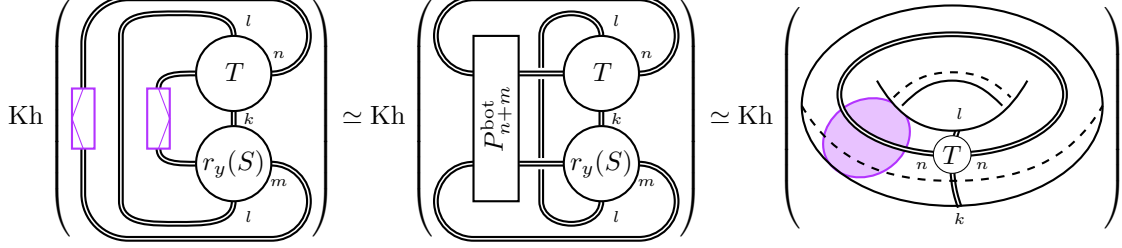
\begin{equation}
	\label{eq:annulusflex}
	\KhEval{
	\begin{tikzpicture}[anchorbase,scale=1]
		\draw[thick, double] 
		(0,-1.1) to (0,1.1)  
		(0,1.1) \ul (-0.75,1.4) \ld (-1.5,1.1) to (-1.5,-1.1) \dr (-.75,-1.4) \ru (0,-1.1) 
		(0.5,0.6) \ru (1,1.1) \ul (0.5,1.6) \pl (-1.4,1.6) \ld (-2,1) to (-2,0)
		(0.5,-0.6) \rd (1,-1.1) \dl (0.5,-1.6) \pl (-1.4,-1.6) \lu (-2,-1) to (-2,0)
		(-0.5,0.6) \ld (-1,0.4)
		(-0.5,-0.6) \lu (-1,-0.4)
		;
		\rjf{-2}{0}{}
		\ljf{-1}{0}{}
		\circlepair{0}{0}{}	
		\node at (0.6,-.85) {\tiny$m$};
		\node at (0.6,.85) {\tiny$n$};
		\node at (.2,0) {\tiny$k$};
		\node at (.2,1.3) {\tiny$l$};
		\node at (.2,-1.3) {\tiny$l$};
	\end{tikzpicture}
}	
\simeq
\KhEval{
	\begin{tikzpicture}[anchorbase,xscale=-1, yscale=1]
		\draw[thick, double] (1,-1.1) to (1,1.1)  (1,1.1) \ur (1.4,1.4) \rd (1.8,1.1) to (1.8,-1.1) \dl (1.4,-1.4) \lu (1,-1.1) ;
		\draw[thick, double] 
		(.5,.6) \lu (0,1.1) \ur (.5,1.6) to (2.7,1.6) \rd (3.2,1.1) \dl (2.7,.6)
		(.5,-.6) \ld (0,-1.1) \dr (.5,-1.6) to (2.7,-1.6) \ru (3.2,-1.1) \ul (2.7,-.6);
		\horwithedge{1}{0}{1.1}
		\circlepair{1}{0}{}
		\rtallbox{2.1}{0}{n+m}
		\node at (0.4,-.85) {\tiny$m$};
		\node at (0.4,.85) {\tiny$n$};
		\node at (.8,0) {\tiny$k$};
		\node at (.8,1.3) {\tiny$l$};
		\node at (.8,-1.3) {\tiny$l$};
		\end{tikzpicture}
}	 
\simeq
\KhEval{
	\begin{tikzpicture}[anchorbase,scale=1]
		\toruscdisk{-1.2}{.2}{.6}{.5}
		\draw[thick,double] (0.1,-.25) to [out=250,in=90] (0,-.65) to [out=270,in=110] (0.1,-1.35);
		\torus{0}{0}{2}{thick}
		\draw[thick,double] (0,.15) ellipse (1.5 and .8);
		\begin{scope}[shift={(0,-.65)}] \draw[fill=white] (0,0) circle (.25);\end{scope}
		\node  at (0,-.65) {$T$};
		\node at (0.4,-.8) {\tiny$n$};
		\node at (-0.4,-.8) {\tiny$n$};
		\node at (.1,-1.5) {\tiny$k$};
		\node at (.1,-.1) {\tiny$l$};
		\end{tikzpicture}
}	
\end{equation}
See \S \ref{s:affine BN} for additional structure in the annular case.
\end{example}

The graphical model described in this section makes clear that 
the cohomology of
the complexes $\PMS_{\Sigma,\Pi;\Gamma}(T,S)$ and 
(more generally) $\PMS_{\Sigma,\Pi;\Gamma}(T|V|S)$ agrees with
the Rozansky--Willis invariant \cite{rozansky2010categorification,MR4332675} of
certain links assembled from $T$ and $S$ (and $V$).

\subsection{Composition and associativity}
\label{ss:cflat ii}

Below, in Theorem \ref{thm:dgcatS} we establish that the complexes
$\PMS(T,S)$ from \eqref{eq:Cflatcx} can be regarded as the 
$\Hom$-complexes in a dg category. For this we will need maps 
$\cring\rightarrow \PMS(T,T)$ that determine the identity endomorphisms, 
and composition maps $\PMS(T,S)\otimes \PMS(S,R)\rightarrow \PMS(T,R)$. 
In fact, we will define generalizations of the latter for the 
complexes $\PMS(T|V|S)$.

For the following, 
recall the tensor product of bar complexes 
$B^T_S\in \Ch^-\left(\bni(\gpm)^{T}_{S}\right)$ from \eqref{eq:bar}.

\begin{definition}\label{def:alg struct on BT} 
Set $\one_T := (\one_{|\gamma_1 \cap T|},\one_{|\gamma_1 \cap T|},\ldots, 
		\one_{|\gamma_g \cap T|},\one_{|\gamma_g \cap T|}) \in \bn(\gpm)^T_T$
and let
\begin{itemize}
	\item $\eta_T \colon \one_T \rightarrow B^T_T$ be the map 
		in $\Ch^-\left(\bni(\gpm)^T_T\right)$
		inherited from the inclusions $(\one_n , \one_n) \rightarrow \Bar^n_n$
		in $\Ch^-\left((\bn_n^n)^{\op} \otimes \bn_n^n \right)$, and
	\item $\mu_S \colon 
		B^T_S \hComp B^S_R \rightarrow B^T_R$ be the map 
		in $\Ch^-\left(\bni(\gpm)^{T}_{R}\right)$ 
		given via the Eilenberg--Zilber shuffle product 
		$\mu_m\colon \Bar^n_m \hComp \Bar^m_k \rightarrow \Bar^n_k$
		from Proposition \ref{prop:Bar is algebra}.
		   \end{itemize}
	   \end{definition}
	   
\begin{lemma}\label{lemma:alg struct on BT} The multiplication maps $\mu_S$ and
unit maps $\eta_T$ from Definition \ref{def:alg struct on BT} are strictly
unital and associative.
\end{lemma}

\begin{proof}
Recall that we view $B_S^T$ as an object of $\Ch^-\left(\bni(\gpm)^{T}_{S}\right)$ 
using 
totalization \eqref{eq:totalization}, 
and the maps $\eta_T$ and $\mu_S$ are induced by corresponding maps 
in $\bigotimes_{i=1}^g \Ch^{-}(\bn^\op \otimes \bn)$.
It is a straightforward consequence of Proposition \ref{prop:Bar is algebra} that 
these latter maps are strictly unital and associative,
therefore the result follows from functoriality of totalization.
\end{proof}

\begin{definition}\label{def:Cflat-composition-and-units} 
Consider a triple of tangles $T,S,R\in \Tan(\Sigma,\Pi;\Gamma)$. 
Given $W\in \bn(\Pi)^{\partial T}_{\partial S}$ and $V\in \bn(\Pi)^{\partial S}_{\partial R}$ 
let
\begin{equation}
	\label{eq:Cflatcomposition}
\circ_S \colon \PMS(T|W|S) \otimes \PMS(S|V|R) \to \PMS(T|W\star V|R)
	\end{equation}
be the chain map given as the composite of the prism stacking maps
\[
\PMS_{\cutS,\cutP}(T|B^T_S, W|S) \otimes \PMS_{\cutS,\cutP}(S|B^S_R, V|R) 
		\rightarrow \PMS_{\cutS,\cutP}(T|B^T_S\star B^S_R, W\star V|R)
\]
induced by Definition~\ref{def:prismstackingmaps},
followed by the map
\[
\PMS_{\cutS,\cutP}(T|B^T_S\star B^S_R, W\star V|R)
	\to \PMS_{\cutS,\cutP}(T|B^T_R, W\star V|R)
\]
obtained by applying the functor 
$\PMS_{\cutS,\cutP}(T|-,W\star V|R) \colon \Ch^{-}(\bn(\gpm)^T_R) \rightarrow \Ch^{-}(\gModen{\Z}{\cring})$ 
to the multiplication map $\mu_S$ from Definition \ref{def:alg struct on BT}. 

Additionally, let
\begin{equation}
	\label{eq:Cflatunit}
	\epsilon_T\colon \cring \to \PMS(T|\one_\pp|T)
\end{equation}
be the chain map given as the composite
as the composite
\[	
	\cring \xrightarrow{e_T} \PMS_{\cutS,\cutP}(T|\one_{\Gamma\cap T}, \one_{\pp}|T)
			\rightarrow \PMS_{\cutS,\cutP}(T|B^T_T, \one_{\pp}|T)
\]
where $e_T$ denotes the tensor product of 
the units for the prism stacking maps from Definition~\ref{def:prismstackingmaps} 
and the second arrow is the functor 
$\PMS_{\cutS,\cutP}(T|-, \one_{\pp}|T) 
	\colon \Ch^-(\bn(\gpm))^T_T \rightarrow \Ch^{-}(\gModen{\Z}{\cring})$ 
applied to the unit map 
$\eta_T\colon \one_{\Gamma\cap T}\rightarrow B^T_T$ 
from Definition \ref{def:alg struct on BT}.
\end{definition}

\begin{example}
Specializing to the case $\partial T = \partial S = \partial R=\pp$ and $W=\one_{\pp}=V$, 
the chain map \eqref{eq:Cflatcomposition} takes the form:
\begin{equation}
	\label{eq:Cflatcompositioncat}
\circ_S \colon 
\PMS(T|S) \otimes \PMS(S|R) \to \PMS(T|R)
\end{equation}
Diagrammatically, the composition \eqref{eq:Cflatcompositioncat} can be pictured as 
(the tensor product over $1\leq i\leq R$ of) the maps
	\begin{gather}
		\label{eq:composition}
		\KhEval{
			\begin{tikzpicture}[anchorbase,scale=.6]
					\vertconn{0}{.5}{1}
					\vertconn{0}{2.75}{1}
					\tTbox{0}{3.75}{\pi_i T_i}
					\tTbox{0}{2.25}{r_y(\pi_i S_{i})}
					\tTbox{0}{1.5}{\pi_i S_i}
					\tTbox{0}{0}{r_y(\pi_i R_{i})}
					\node at (-1.75,3.25) {$\cdots$};
					\jf{-3}{3.25}{1}
					\jf{-.5}{3.25}{k}
					\node at (-1.75,1) {$\cdots$};
					\jf{-3}{1}{1}
					\jf{-.5}{1}{k}
			\end{tikzpicture}
		}\!
		\longrightarrow
		\KhEval{
			\begin{tikzpicture}[anchorbase,scale=.6]
				\vertconn{0}{.5}{3.25}
				\tTbox{0}{3.75}{\pi_i T_i}
				\tTbox{0}{0}{r_y(\pi_i R_{i})}
				\node at (-1.75,3.25) {$\cdots$};
				\jf{-3}{3.25}{1}
				\jf{-.5}{3.25}{k}
				\node at (-1.75,1) {$\cdots$};
				\jf{-3}{1}{1}
				\jf{-.5}{1}{k}
			\end{tikzpicture}
		}\! 
		\xrightarrow{\mu}
		\KhEval{
			\begin{tikzpicture}[anchorbase,scale=.6]
				\vertconn{0}{.5}{3.25}
				\tTbox{0}{3.75}{\pi_i T_i}
				\tTbox{0}{0}{r_y(\pi_i R_{i})}
				\node at (-1.75,3.25) {$\cdots$};
				\node at (-1.75,1) {$\cdots$};
				\jfbig{-3}{1}{1}
				\jfbig{-.5}{1}{k}
			\end{tikzpicture}
		}.
	\end{gather}
	where we have suppressed grading shifts present in \eqref{eq:homcomplex}.
\end{example}

We think of the composition maps \eqref{eq:Cflatcomposition} 
as generalizations of the prism stacking maps for disks from Definition~\ref{def:prismstackingmaps} 
to surfaces $\Sigma$ glued from disks. 
The following is an analogue of Lemma~\ref{lem:prismstacking-star} in our present setting.

	\begin{lem}\label{lem:Cflatcomposition-star} 
		For morphisms $f \colon V \to V'$ in $\bni(\Pi)^{\partial S}_{\partial R}$ 
		and $g \colon W \to W'$ in $\bn(\Pi)^{\partial T}_{\partial S}$
		we have a commutative square of chain maps:
		\[
			\begin{tikzcd}
				\PMS(T|W|S) \otimes \PMS(S|V|R)
				\arrow[r,"\circ_S"] 
				 \arrow[d,"\PMS(T|g|S)\otimes \PMS(S|f|R)"] 
				&  
				\PMS(T|W \star V|R)
				\arrow[d,"\PMS(T|g \star f|R)"]
				\\
				\PMS(T|W'|S) \otimes	\PMS(S|V'|R)
				\arrow[r,"\circ_S"] 
				&  
				\PMS(T|W'\star V'|R)
			\end{tikzcd}
		\]
In other words, the composition maps 
	$\circ_S \colon \PMS(T|W|S) \otimes \PMS(S|V|R) \to \PMS(T|W\star V|R)$ 
	from \eqref{eq:Cflatcomposition} 
	form the components of a natural transformation:
		\[
			\begin{tikzcd}[column sep=3cm]
				\bni(\Pi)^{\partial T}_{\partial S} \otimes \bni(\Pi)^{\partial S}_{\partial R}
				\arrow[r,"{\PMS(T|-|S)\otimes \PMS(S|-|R)}"] 
				 \arrow[d,"\star"] 
				& 
				\Ch^-\left(\gModen{\Z}{\cring}\right)\otimes \Ch^-\left(\gModen{\Z}{\cring}\right)
				\arrow[d,"\otimes"] 
				\arrow[dl,double,"\circ_S"] 
				\\
				\bni(\Pi)^{\partial T}_{\partial R}  
				\arrow[r,swap, "{\PMS(T|-|R)}"] 
				 &  
				 \Ch^-\left(\gModen{\Z}{\cring}\right)
			\end{tikzcd}
		\]
	\end{lem}
	\begin{proof} Analogous to the proof of
	Proposition~\ref{prop:Cflatassocunit} below, 
	but easier and thus omitted.
	\end{proof}
	
The composition maps $\circ_S$ enjoy the expected associativity
	and unitality properties as well.

\begin{prop}[Associativity and unitality]
	\label{prop:Cflatassocunit}
Given a quadruple of tangles $Q,R,S,T\in \Tan(\Sigma,\Pi;\Gamma)$, 
as well as 1-morphisms 
$V_S^T\in \bn(\Pi)^{\partial T}_{\partial S}$, 
$V_R^S\in \bn(\Pi)^{\partial S}_{\partial R}$, 
and $V_Q^R\in \bn(\Pi)^{\partial R}_{\partial Q}$. 
the composition maps \eqref{eq:Cflatcomposition} are associative and unital.
Precisely, the composites
		\begin{align*}
		\circ_R \circ (\circ_S \otimes \id) &\colon 
			\PMS(T|V_S^T|S)\otimes \PMS(S|V_R^S|R)\otimes \PMS(R|V_Q^R|Q) \to 
			\PMS(T|(V_S^T \star V_R^S)\star V_Q^R|Q) \\
		\circ_S \circ (\id\otimes \circ_R) &\colon
			\PMS(T|V_S^T|S)\otimes \PMS(S|V_R^S|R)\otimes \PMS(R|V_Q^R|Q) \to 
			\PMS(T|V_S^T\star (V_R^S\star V_Q^R)|Q)
		\end{align*}
are equal (up to the coherent isomorphism induced by 
$(V_S^T \star V_R^S)\star V_Q^R \cong V_S^T\star (V_R^S\star V_Q^R)$)
and the composites
		\begin{align*}
			\circ_S\circ (\id \otimes e_S) &\colon 
			\PMS(T|V_S^T|S) \otimes \cring \to 
			\PMS(T|V_S^T\star \one|S)\\
			\circ_T \circ (e_T \otimes \id) &\colon
			\cring \otimes \PMS(T|V_S^T|S) \to 
			\PMS(T|\one \star V_S^T|S)
		\end{align*}
are both identity maps (after adjusting by the coherent isomorphisms induced by 
$V_S^T \hComp \one \cong V_S^T \cong \one \hComp V_S^T$).
\end{prop}
\begin{proof}
For book-keeping purposes, we will denote
\[
V_Q^S = V_R^S \hComp V_Q^R \, , \quad V_R^T = V_S^T \hComp V_R^S 
\, , \quad V_Q^T = (V_S^T \star V_R^S)\star V_Q^R \cong V_S^T\star (V_R^S\star V_Q^R) \, .
\]
Given $I,J \in \{Q,R,S,T\}$, we let $X_I^J = (B_I^J , V_I^J)$, 
which we regard as an object of $\Ch^-(\bn(\cutP)_I^J)$.
The Eilenberg--Zilber product from Definition \ref{def:alg struct on BT} 
then gives a collection of chain maps
\begin{equation}\label{eq:Xijk}
X^K_J\star X^J_I \rightarrow  X^K_I \, .
\end{equation}

For reasons of space, abbreviate by writing $(-|-|-):= \PMS_{\cutS,\cutP}(-|-|-)$, 
so in particular \eqref{eq:CflatcxW} simply reads as 
$\PMS(T|V_S^T|S) = (T|X_S^T|S)$.
Further, the tensor product $\otimes_\cring$ 
will be denoted $\cdot$, and horizontal composition symbols $\hComp$ will be omitted. 
In these conventions, consider the diagram:
\[
\begin{tikzpicture}\tikzstyle{every node}=[font=\small]
\node (00) at (0,0) {$(T|X_S^T|S )\cdot (S|X_R^S|R) \cdot (R|X_Q^R|Q)$};
\node (01) at (6,0) {$(T|X_S^T X_R^S|R)\cdot (R|X_Q^R|Q)$};
\node (02) at (11.5,0) {$(T|X_R^T|R)\cdot (R|X_Q^R|Q)$};
\node (10) at (0,-2) {$(T|X_S^T|S) \cdot (S|X_R^TX_Q^R|Q)$};
\node (11) at (6,-2) {$(T|X_S^T X_R^S X_Q^R|Q)$};
\node (12) at (11.5,-2) {$(T|X_R^T X_Q^R|Q)$};
\node (20) at (0,-4) {$(T|X_S^T|S) \cdot (S|X_Q^S|Q)$};
\node (21) at (6,-4) {$(T|X_S^T X_Q^S|Q)$};
\node (22) at (11.5,-4) {$(T|X_Q^T|Q)$};
\node (A) at (3,-1) {I};
\node (D) at (9,-3) {IV};
\node (B) at (9,-1) {II};
\node (C) at (3,-3) {III};
\path[-stealth]
(00) edge node[above] {$(1)$} (01)
(01) edge node[above] {$(2)$} (02)
(10) edge node[above] {$(3)$} (11)
(11) edge node[above] {$(4)$} (12)
(20) edge node[above] {$(5)$} (21)
(21) edge node[above] {$(6)$} (22)
(00) edge node[right] {$(7)$} (10)
(10) edge node[right] {$(8)$} (20)
(01) edge node[right] {$(9)$} (11)
(11) edge node[right] {$(10)$} (21)
(02) edge node[right] {$(11)$} (12)
(12) edge node[right] {$(12)$} (22);
\end{tikzpicture}
\]
in which the arrows with odd labels are (the maps induced on complexes by)
the prism stacking maps from Definition~\ref{def:prismstackingmaps},
and the arrows with even labels are induced from \eqref{eq:Xijk}.

Square I commutes by Lemma~\ref{lem:assocunit-prismstacking}, 
which gives associativity of the prism stacking maps.
Squares II and III commute by Lemma~\ref{lem:prismstacking-star}.
(In both cases here, we extend from results involving 
one standard disk to the disjoint union of disks $(\cutS,\cutP)$ by tensor product.)
Square IV commutes by Lemma~\ref{lemma:alg struct on BT}. 
Thus, the outer square commutes, which is exactly the statement of associativity.
The proof of unitality is similar, thus omitted.
\end{proof}

\subsection{Dg categories for surfaces with specified points on the boundary}
\label{ss:dg cat of surface}

The results of the preceding section allow us to associate a dg category to 
a seamed marked surface $(\Sigma,\Pi;\Gamma)$ 
with a standard collection of points $\pp \subset A(\Pi)$.
Recall from \eqref{eq:Cflatcx} that in this setting we denote 
$\PMS(T,S) = \PMS(T| \one_\pp |S)$.

\begin{thm}
	\label{thm:dgcatS}
Let  $(\Sigma,\Pi;\Gamma)$ be a seamed marked surface and let $\pp \subset A(\Pi)$ be standard.
	The composition maps \eqref{eq:Cflatcompositioncat} and units \eqref{eq:Cflatunit} 
	define a dg category $\sCat(\Sigma,\pp;\Gamma)$ with set of objects $\Tan(\Sigma,\pp;\Gamma)$ and
	morphism complexes $\Hom_{\sCat(\Sigma,\pp;\Gamma)}(S,T):=\PMS(T,S)$.
\end{thm}
\begin{proof}
The associativity and unitality of the composition are obtained as special cases of Proposition~\ref{prop:Cflatassocunit}.
\end{proof}

When $\Sigma=D$ is a disk and $\Gamma=\emptyset$, we have
	$\sCat(D,\mathbf{p};\emptyset)=\bn(D,\mathbf{p})$, 
	the Bar-Natan category of the disk $D$ with specified boundary $\mathbf{p}$.
More generally, we can view $\sCat(\Sigma,\pp;\Gamma)$ as a generalization 
of the (na\"{i}ve) Bar-Natan category of $(\Sigma,\pp)$.
Precisely, we have:

\begin{thm}
\label{thm:surfaceBN}
There is a dg functor from $\sCat(\Sigma,\pp;\Gamma)$ to its cohomology
	category $H^0(\sCat(\Sigma,\mathbf{p};\Gamma))$ which is the identity on
	objects.  Moreover, this latter category is canonically isomorphic to the
	full subcategory of the Bar-Natan category $\bn(\Sigma, \mathbf{p})$ 
	of tangles in $(\Sigma,\pp)$ with standard, transverse intersection with $\Gamma$.
	Thus, $H^0(\sCat(\Sigma,\mathbf{p};\Gamma))$ is equivalent to $\bn(\Sigma, \mathbf{p})$.
\end{thm}
\begin{proof}
The $\Hom$-complexes in $\sCat(\Sigma,\pp;\Gamma)$ are supported in cohomological degrees $\leq 0$, 
therefore they project onto their zeroth cohomologies. 
This defines the dg functor from the statement.

The objects of $H^0(\sCat(\Sigma,\mathbf{p};\Gamma))$ 
	and the full subcategory in the statement of the theorem are both $\Tan(\Sigma,\mathbf{p};\Gamma)$, 
	the set of tangles in $(\Sigma,\pp)$ with standard, transverse intersection with $\Gamma$.
	Since every tangle in $(\Sigma,\pp)$ is isotopic to such a tangle, 
	the inclusion of the full subcategory of such tangles into 
	$\bn(\Sigma, \mathbf{p})$ is essentially surjective. 
	
Let $S,T \in \Tan(\Sigma,\mathbf{p};\Gamma))$. We first note
	that the zeroth cohomology of the morphism complex 
	$\PMS(T,S)=\PMS_{\cutS,\cutP}(T | B_S^T , \one_\pp | S)$ is just the group 
	of zero-chains $\PMS^0(T,S)$, modulo boundaries. 
	A zero-chain in $\PMS(T,S)$ is a linear combination of pure tensors whose factors 
	take values in the prism space 
	$\PMS_{\std}(\pi_i(T|_{D_i}) |
		V^{\pm}_{h_{i,1}},\ldots,V^{\pm}_{h_{i,k_{i}}} | \pi_i(S|_{D_i}))$
	where $\cutP|_{D_i} = \{h_{i,1} , \ldots , h_{i,k_{i}} \}$ and 
	$V_{h_{i,j}} = \one$ when $h_{i,j} \in \Pi$.
	As described in Remark \ref{rem:Pfunc}, we can view elements of the latter as 
	Bar-Natan cobordisms in the prism $(\Dst,\Ast) \times I$ with appropriate boundary conditions. 
	Using the standardization $\pi_i$, these determine cobordisms in $(D_i,A(\cutP|_{D_i})) \times I$
	with appropriate boundary conditions. 
	Since the terms of \eqref{eq:BarY} in degree zero are direct sums of terms of the form
	$(a_0,a_0) \in \bn^{\op} \otimes \bn$, 
	we see that if $(\gamma,-)$ and $(\gamma,+)$ lie in regions 
	$D_{i_{-}}$ and $D_{i_{+}}$, then the cobordisms in $(D_{i_{\pm}},A(\cutP|_{D_{i_\pm}})) \times I$ 
	are ``glue-able''; i.e.~we can glue them to obtain a cobordism in $D_{i_{-}} \cup_{\gamma} D_{i_{+}}$.
	
	Performing all such gluings, we obtain a linear combination of Bar-Natan
	cobordism in $\Sigma\times [0,1]$ from $S$ to $T$. 
	Since the Bar-Natan relations in the disk are also satisfied in $\bn(\Sigma,\pp)$, 
	we obtain a well-defined $\cring$-linear map
	\[
		\phi\colon \PMS^0(T,S) \to \Hom_{\bn(\Sigma, \mathbf{p})}(S,T)	 \, .
	\]
	Moreover, it is easy to see that the zero-boundaries lie in the kernel of $\phi$.
	Explicitly, the image of the zero-boundaries is spanned by differences of pairs of Bar-Natan 
	cobordisms in $\Sigma \times [0,1]$ that only differ via an isotopy in a tubular neighborhood 
	of $\gamma \times [0,1]$ for $\gamma \in \Gamma$. 
	Hence, there is an induced map:
	\[
		\overline{\phi}\colon H^0(\PMS(T,S)) \to \Hom_{\bn(\Sigma, \mathbf{p})}(S,T)	
	\]
	It is straightforward to construct the inverse to this map:
	we apply an isotopy to a Bar-Natan cobordism in $\Hom_{\bn(\Sigma, \mathbf{p})}(S,T)$
	to obtain such a cobordism that intersects $\Gamma \times [0,1]$ transversely, 
	then cutting to obtain elements in $\PMS_{\std}(\pi_i(T|_{D_i}) |
		V^{\pm}_{h_{i,1}},\ldots,V^{\pm}_{h_{i,k_{i}}} | \pi_i(S|_{D_i}))$
	for appropriate $V_h$'s.
\end{proof}

We finish this section with a direct observation that we will use in Section~\ref{sec:spin}.

\begin{proposition}
	\label{prop:equivop}
The dg category $\sCat(\Sigma,\mathbf{p};\Gamma)$ is isomorphic to its
opposite. The isomorphism is implemented by the dg functor $\sCat(\Sigma,\mathbf{p};\Gamma)\rightarrow
\sCat(\Sigma,\mathbf{p};\Gamma)^\op$ defined on objects by $T\mapsto T$ and
on morphism complexes by the vertical reflection $r_y$ applied to all diagrams
\eqref{eq:homcomplex}, combined with the isomorphism $r_y(B^T_S)\cong B^S_T$
from Proposition~\ref{prop:symmetries of B}. \qed
\end{proposition}

\subsection{Modules for seamed marked surfaces}
\label{ss:MM of surface}

Thus far, we have constructed dg categories $\sCat(\Sigma,\mathbf{p};\Gamma)$
depending on a specific set of boundary points $\mathbf{p}$, 
with $\Hom$-complexes given by $\PMS(T,S)$. 
Now, we use the more general chain
complexes $\PMS(T|V|S)$ to associate to a seamed marked surface 
dg bimodules that relate the categories $\sCat(\Sigma,\mathbf{p};\Gamma)$ 
for different sets of boundary points $\pp$.
We will assemble these invariant of the seamed marked surface 
$(\Sigma, \Pi;\Gamma)$ into a 2-functor, 
valued in an appropriate dg Morita bicategory, 
generalizing the construction in \S \ref{ss:arbitrary disks}.

\begin{definition}
\label{def:dg morita}
Given dg categories $\AS$, $\BS$, an \emph{$(\AS,\BS)$-bimodule} is a dg functor $\AS\otimes \BS^{\op}
\rightarrow \dgModen{\Z\times\Z}{\cring}$.
The $(\AS,\BS)$-bimodules form a dg category, denoted $\Bim_{\AS,\BS}$, 
with morphisms given by dg natural transformations.  
Given $M\in \Bim_{\AS,\BS}$ and $N\in \Bim_{\BS,\CS}$, 
their tensor product by the same formula as \eqref{eq:tensor product}, where now the 
cokernel is taking in the category of complexes of graded $\cring$-modules 
and degree zero chain maps.

We let $\MorDGCat$ denote the dg Morita bicategory, wherein
\begin{itemize}
\item The objects of $\MorDGCat$ are small dg categories, and
\item the $1$-morphism category $\AS\xleftarrow{M}\BS$ is the dg category of $(\AS,\BS)$-bimodules. 
\end{itemize}
The composition of $2$-morphisms is simply the composition of dg natural
transformations, and the composite of two $1$-morphisms $\AS\xleftarrow{M}\BS$ and
$\BS\xleftarrow{N}\CS$ is the tensor product\footnote{Here we use the na\"{i}ve 
(underived) tensor product of bimodules. It turns out that all bimodules we consider will be 
projective from the right and left, hence this agrees with the derived tensor product.
See Lemma \ref{lemma:tangle slide}.} 
$\AS\xleftarrow{M\otimes_\BS N}\CS$.
\end{definition}

Observe that the bicategory $\DGCat[\Z \times \Z]$ embeds in $\MorDGCat$ by 
sending a dg functor $F \colon \CS \to \DS$ to the $(\DS,\CS)$-bimodule 
$\Hom_\DS (F(-),-)$.

As in Remark \ref{rmk:left action}, we can describe dg bimodules via their left and right action maps.

\begin{definition}\label{def:our bimodules}
Fix a seamed marked surface $(\Sigma,\Pi;\Gamma)$.  
For each standard set of points $\pp \subset A(\Pi)$, 
let $\MM(\pp):=\sCat(\Sigma,\pp;\Gamma)$ be the associated dg category. 
For each 1-morphism $V\in \bn(\Pi)^{\qq}_{\pp}$, 
let $\MM(V)$ denote the $(\MM(\qq),\MM(\pp))$-bimodule $\PMS(-|V|-)$, i.e.
\[
\MM(V) \colon (T,S) \mapsto \PMS(T|V|S),
\]
with action maps:
\begin{equation}
	\label{eq:dgactionmaps}
	\begin{aligned}
		\circ_T \colon \PMS(T'|T) \otimes \PMS(T|V|S) \to \PMS(T'|V|S)\\
		\circ_S \colon \PMS(T|V|S) \otimes \PMS(S|S') \to \PMS(T|V|S')
	\end{aligned}
\end{equation} 
obtained as specializations of \eqref{eq:Cflatcomposition}.  
When we wish to include data $(\Sigma,\Pi;\Gamma)$ in the notation, we will write $\MM_{\Sigma,\Pi;\Gamma}$.
\end{definition}

We now establish some basic facts concerning the bimodules $\MM(V)$.
Given a seamed marked surface $(\Sigma,\Pi;\Gamma)$, 
abbreviate $A = A(\Pi)$ and
choose a homeomorphism
\begin{equation}\label{eq:boundary action}
\Sigma \cup_{A\times\{1\}} (A \times I) \xrightarrow{\cong} \Sigma
\end{equation}
which restricts to the identity on $A\times\{0\}$. 
Given $T\in \Tan(\Sigma,\qq;\Gamma)$ and $V\in \bn(\Pi)_{\pp}^\qq$, 
let $T\star V\in\Tan(\Sigma,\pp;\Gamma)$ denote the ``composition of tangles,''
i.e.~the image of $T\cup V$ under the homeomorphism \eqref{eq:boundary action}.
(Here we tacitly identify objects in $\bn(\Pi)_{\pp}^\qq$ with tangles in $A \times I$.)

The following establishes that the bimodules $\MM(V)$ are \emph{sweet} 
in the sense of \cite[Definition 1]{MR1928174}, 
that is, finitely-generated and projective from the left and right.

\begin{lemma}\label{lemma:tangle slide}
Suppose we have tangles $T,S\in \Tan(\Sigma,\Pi;\Gamma)$ with $\pp=\partial S$ and $\qq=\partial T$. 
If $V\in \bn(\Pi)^\qq_\pp$, then
\[
{}_T\MM(V) \cong {}_{T\star V}\MM(\one_\pp) \, , \quad  
\MM(V)_S \cong \MM(\one_\qq)_{S\star r_y(V)} \, .
\]
\end{lemma}
\begin{proof}
Clear after unpacking the definitions.
\end{proof}

Next, we observe that $\MM(-)$ respects the (horizontal) composition of $1$-morphisms.

\begin{proposition}\label{prop:tensoring tangles}
Given 1-morphisms $W\in \bn(\Pi)^{\rr}_{\qq}$ and $V\in \bn(\Pi)^{\qq}_{\pp}$ we have
\[
\MM(W)\otimes_{\MM(\qq)} \MM(V) \cong \MM(W\star V).
\]
\end{proposition}
\begin{proof}
The composition \eqref{eq:Cflatcomposition}
defines a map of bimodules 
$\MM(V)\otimes_{\MM(\pp)} \MM(W) \rightarrow \MM(V\star W)$. 
To show that this is an isomorphism of bimodules, we fix $T$ and $R$ and compute
\[
{}_T\MM(W)\otimes_{\MM(\qq)} \MM(V)_R \cong {}_{T\star W}\MM(\one_\qq)\otimes_{\MM(\qq)} \MM(V)_R  
	\cong {}_{T\star W} \MM(V)_R \cong {}_T\MM(W\star V)_R \, .
\]
by Lemma \ref{lemma:tangle slide}.
\end{proof}

We are now ready for the main construction.

\begin{thm}\label{thm:final}
The assignments $\pp\mapsto \MM(\pp)$, $V\mapsto \MM(V)$ 
extend to a 2-functor 
\[
\MM_{\Sigma,\Pi;\Gamma}\colon \bni(\Pi) \To \MorDGCat \, . 
\]
\end{thm}
\begin{proof}
First, on the level of $2$-morphisms, a 2-morphism $f\colon V \to V'$ in $\bni(\Pi)$ is sent to the
	bimodule homomorphism $\PMS(-|V|-)\to \PMS(-|V'|-)$ defined component-wise 
	by \eqref{eq:MTScx}. 
It follows from Lemma~\ref{lem:Cflatcomposition-star} that
	such a map commutes with the action maps \eqref{eq:dgactionmaps}, 
	hence constitute a bimodule homomorphism.
	
It is clear that the assignment on $2$-morphisms respects the composition of
$2$-morphisms. The composition of $1$-morphisms in $\bni(\Pi)$ is
compatible with the composition of bimodules by Proposition \ref{prop:tensoring tangles}. 
The compatibility of horizontal
and vertical composition on $2$-morphisms has been verified in
Lemma~\ref{lem:Cflatcomposition-star}. 
\end{proof}

\begin{remark}\label{rem:HH} 
Theorem \ref{thm:final} defines $\MM_{\Sigma,\Pi;\Gamma}$ as a
	module for the bicategory $\bni(\Pi)$ valued in the bicategory $\MorDGCat$.
	Similarly, the 2-functor for the cut surface $\MM_{\cutS,\cutP;\emptyset}$
	can be considered as a module for 
	$\bni(\cutP) \cong \bn(\gpm)\otimes \bni(\Pi)$. 
	Recall that $\bn(\gpm)$ is a tensor product of
	copies of $\bn^\op \otimes \bn$, one for each seam $\gamma \in \Gamma$. If we
	consider the tensor product $\bni(\Gamma):=\bigotimes_{\gamma\in \Gamma} \bn$, 
	then $\bni(\gpm) \cong \bn(\Gamma)^\op \otimes \bn(\Gamma)$ can be
	interpreted as the \emph{enveloping category} of $\bn(\Gamma)$ and
	$\MM_{\cutS,\cutP;\emptyset}$ as a $\bn(\Gamma)$-\emph{bi}module valued
	in $\bni(\Pi)$-modules.  We then arrive at the slogan: 
	\begin{center}\emph{The (glued) surface module $\MM_{\Sigma,\Pi;\Gamma}$ is computed as the Hochschild chains for the \\seam category $\bni(\Gamma)$ with coefficients in the bimodule $\MM_{\cutS,\cutP;\emptyset}$ associated with the cut surface}.\end{center}
	\end{remark}

\subsection{Coarsening, gluing, and seam reversal} 
\label{s:coarsening}

In this section, we construct quasi-equivalences that relate the dg categories associated to seamed marked surfaces
$(\Sigma,\Pi;\Gamma)$ and $(\Sigma,\Pi;\Gamma')$ with the same underlying marked surface but with different 
sets of seams.

\begin{construction}
	\label{constr:coarsening}
Let $(\Sigma,\Pi;\Gamma)$ be a seamed marked surface and let $\gamma\in \Gamma$ be such
	that $(\Sigma,\Pi;\Gamma\smallsetminus\{\gamma\})$ is 
	again a seamed marked surface. 
Define $\Gamma':=\Gamma\setminus\{\gamma\}$ and let
	$D_i$, $D_j$ denote the two regions of $(\Sigma,\Pi;\Gamma)$ that are separated
	by the arc $\gamma$. 
Given tangles $S\in \Tan(\Sigma,\Pi;\Gamma) \subset \Tan(\Sigma,\Pi;\Gamma')$
	and $T\in \Tan(\Sigma,\Pi;\Gamma)\subset \Tan(\Sigma,\Pi;\Gamma')$
	and $V\in \bni(\Pi)^{\partial T}_{\partial S}$, 
	abbreviate $\PMStwo(T|V|S):= \PMS_{\Sigma,\Pi;\Gamma'}(T|V|S)$ 
	while retaining our usual abbreviation $\PMS(T|V|S):= \PMS_{\Sigma,\Pi;\Gamma}(T|V|S)$.

We now define the \emph{coarsening} chain map:
\[\mathrm{coarsen}_\gamma(S|V|T)\colon \PMS(T|V|S)\to \PMStwo(T|V|S)\] by describing its
action on the relevant factors appearing in \eqref{eq:homcomplex}. 
On all tensor factors corresponding to regions $k\in I\setminus\{i,j\}$ it is the identity. To
describe the action on the tensor factors indexed by $i$ and $j$, 
we assume that $D_i$ is to the left of $\gamma$ (as we travel in the direction of its orientation)
while $D_j$ is to right.
Further, we may assume that both tensor
factors are presented as in Remark~\ref{rem:rotatefactor} applied to 
the cut seams $(\gamma,\pm)$. 
The coarsening map is then given by:
\begin{multline} 
	\label{eq:coarsen}
		\KhEval{
		\begin{tikzpicture}[anchorbase,xscale=1]
			\begin{scope}[shift={(-7,0)},xscale=-1]
				\vertconnbox{0}{-.5}{1}
				\node at (-1.75,0) {$\cdots$};
				\ljf{-3}{0}{\gamma}
				\jf{-.5}{0}{\cdot}
				\twopitangleboxes{0}{0}{j}
		\end{scope}
		\vertconnbox{0}{-.5}{1}
			\node at (-1.75,0) {$\cdots$};
			\ljf{-3}{0}{\gamma}
			\jf{-.5}{0}{\cdot}
			\twopitangleboxes{0}{0}{i}
			\end{tikzpicture}
			}
					\\ \\				
					=
					\KhEval{
						\begin{tikzpicture}[anchorbase,xscale=1]
							\begin{scope}[shift={(-7,0)},xscale=-1]
								\vertconnbox{0}{-.5}{1}
								\node at (-1.75,0) {$\cdots$};
								\ljf{-3}{0}{\gamma}
								\jf{-.5}{0}{\cdot}
								\twopitangleboxes{0}{0}{j}
						\end{scope}
						\vertconnbox{0}{-.5}{1}
							\node at (-1.75,0) {$\cdots$};
							\ljf{-3}{0}{\gamma}
							\jf{-.5}{0}{\cdot}
							\twopitangleboxes{0}{0}{i}
							\filldraw[white] (-4.2,-.48) rectangle (-2.8,.48); 
							\draw[double, thick] 
							(-4.01,0.5) \dr (-3.8,.3) to (-3.2,.3) \ru (-2.99,.5)
							(-4.01,-0.5) \ur (-3.8,-.3) to (-3.2,-.3) \rd (-2.99,-.5)
							;
							\filldraw[white] (-3.8,-.4) rectangle (-3.2,.4);
							\draw[thick] (-3.8,-.4) rectangle (-3.2,.4);
							\node[rotate=90] at (-3.5,0) {$\Bproj$};
							\end{tikzpicture}
							}
									\\ \\
									\xrightarrow{\e}
									\KhEval{
										\begin{tikzpicture}[anchorbase,xscale=1]
											\begin{scope}[shift={(-7,0)},xscale=-1]
												\vertconnbox{0}{-.5}{1}
												\node at (-1.75,0) {$\cdots$};
												\ljf{-3}{0}{\gamma}
												\jf{-.5}{0}{\cdot}
												\twopitangleboxes{0}{0}{j}
										\end{scope}
										\vertconnbox{0}{-.5}{1}
											\node at (-1.75,0) {$\cdots$};
											\ljf{-3}{0}{\gamma}
											\jf{-.5}{0}{\cdot}
											\twopitangleboxes{0}{0}{i}
											\filldraw[white] (-4.2,-.48) rectangle (-2.8,.48); 
											\draw[double, thick] 
											(-4.01,0.5) \dr (-3.8,.3) to (-3.2,.3) \ru (-2.99,.5)
											(-4.01,-0.5) \ur (-3.8,-.3) to (-3.2,-.3) \rd (-2.99,-.5);
											\end{tikzpicture}
											}
	\end{multline}
where we omit the relevant $q$-shifts for simplicity.
The last map in \eqref{eq:coarsen} is the counit from Definition \ref{def:bproj counit} and 
the codomain is canonically identified
with the tensor factor of the merged region $D_i\cup D_j$ in $\PMStwo(T|V|S)$ via
sphericality maps.
\end{construction}

\begin{prop}
	\label{prop:coarsenhomotopyequiv}
	In the setting of Construction~\ref{constr:coarsening}, the coarsening chain
	map $\mathrm{coarsen}_\gamma(S|V|T)$ is a homotopy equivalence.
\end{prop}
\begin{proof} 
First, consider a variation of the construction of the complexes
	$\PMS(T|V|S)$ and $\PMStwo(T|V|S)$, where we set all components of the
	differential to zero, except for the ones coming from the bar complex associated with
	the seam $\gamma$. 
	Denote these complexes respectively by $\PMS_0(T|V|S)$ and $\PMStwo_0(T|V|S)$  
	and note that the coarsening map still commutes with the (modified) differential, 
	thus givens a chain map between these complexes.

In fact, this chain map
	\begin{equation}
		\label{eqref:coarsenchain}
	\mathrm{coarsen}_\gamma(S|V|T) \colon \PMS_0(T|V|S)\to \PMStwo_0(T|V|S)
		\end{equation} 
is a homotopy equivalence by Proposition~\ref{prop:centrality}, 
since the projector is connected to objects of
through-degree zero, as is manifestly shown in \eqref{eq:coarsen}. 
This is equivalent to the cone of \eqref{eqref:coarsenchain} being contractible. 
By homological perturbation (see e.g.~\cite[Basic Perturbation Lemma]{Markl}), 
the cone of 
\[
	\mathrm{coarsen}_\gamma(S|V|T)\colon \PMS(T|V|S)\to
	\PMStwo(T|V|S)	
\]
is again contractible, thus this map is a homotopy equivalence.
\end{proof}

\begin{remark}
An analogue of the coarsening chain maps in Construction~\ref{constr:coarsening}
	is possible even when $\Gamma\setminus\{\gamma\}$ no longer yields a 
	seamed marked surface. 
This requires a different local model than \eqref{eq:coarsen} and
	will typically not yield homotopy equivalences. 
We will not use such chain maps in this paper.
\end{remark}

\begin{proposition}
	\label{prop:coarsencomp}
	The coarsening chain maps from Construction~\ref{constr:coarsening}
	intertwine with the composition maps from
	Definition~\ref{def:Cflat-composition-and-units}. Namely, in the setting of
	the latter definition, and retaining notation from
	Construction~\ref{constr:coarsening}, the coarsening maps fit as vertical
	arrows into strictly commuting diagrams of the form:
	\[
		\begin{tikzcd}
			\PMS(T|W|S) \otimes \PMS(S|V|R)
			\arrow[r,"\circ_S"] 
			 \arrow[d,""] 
			&  
			\PMS(T|W\star V|R)
			\arrow[d,""]
			\\
			\PMStwo(T|W|S) \otimes \PMStwo(S|V|R)
			\arrow[r,"\circ_S"] 
			&  
			\PMStwo(T|W\star V|R)
		\end{tikzcd}
	\]
\end{proposition}
\begin{proof}
Considering \eqref{eq:coarsen} and \eqref{eq:composition}, 
we see that this is a consequence of Lemma~\ref{lem:counitsquare}.  
	\end{proof}
	
The following is obvious from the definition of the coarsening maps.

\begin{proposition}
	The coarsening chain maps $\mathrm{coarsen}_\gamma(S|V|T)$ from
	Construction~\ref{constr:coarsening} are natural in the argument $V$, 
	i.e.~they form the components of a natural transformation
		\[
		\PMS_{\Sigma,\Pi;\Gamma}(T|-|S) \To \PMS_{\Sigma,\Pi;\Gamma'}(T|-|S) 
			\]
	of functors $\bni(\Pi)^{\partial T}_{\partial S} \rightarrow \Ch^-\big(\gModen{\Z}{\cring}\big)$.
	\end{proposition}

In summary, we obtain:
		
		\begin{cor}
			\label{cor:coarsening-transfo}
			The coarsening maps $\mathrm{coarsen}_\gamma(S|V|T)$ from
			Construction~\ref{constr:coarsening} assemble to a pseudonatural transformation			
			$\MM_{\Sigma,\Pi;\Gamma} \To \MM_{\Sigma,\Pi;\Gamma'}$
			between the 2-functors from Theorem~\ref{thm:final}. \qed
		\end{cor}

In particular, for every object $\mathbf{p}\in \bni(\Pi)$, 
	Corollary \ref{cor:coarsening-transfo} provides
	a bimodule from $\sCat(\Sigma,\mathbf{p};\Gamma)$ to
	$\sCat(\Sigma,\mathbf{p};\Gamma')$.
	These bimodules are actually realized by dg functors:

		\begin{cor}\label{cor:coarsening-dg-functors} 
		For each object $\pp \in \bni(\Pi)$, 
		the chain maps $\mathrm{coarsen}_\gamma(S|\one_\pp|T)$ assemble into a dg functor
		\[
		\mathrm{coarsen}_\gamma\colon \sCat(\Sigma,\mathbf{p};\Gamma) 
			\to \sCat(\Sigma,\mathbf{p};\Gamma')
			\]
		which is the identity on objects and a quasi-equivalence.
		\end{cor}
		
\begin{proof}
Proposition~\ref{prop:coarsencomp} gives that the
	coarsening maps are compatible with composition and thus constitute a dg
	functor $\mathrm{coarsen}_\gamma$ which is the identity on objects. 
By Proposition~\ref{prop:coarsenhomotopyequiv} they are also homotopy
	equivalences and hence quasi-isomorphisms. 
Finally, every object of $\sCat(\Sigma,\mathbf{p};\Gamma')$ 
	(a tangle with standard intersection with $\Gamma'$) 
	is isomorphic to an object in the image of $\mathrm{coarsen}_\gamma$ 
	(consider an isotopic tangle that has standard intersection with $\Gamma$). 
Thus $\mathrm{coarsen}_\gamma$ is (quasi-)essentially surjective, 
	and hence a quasi-equivalence.
\end{proof}

\begin{remark} A pseudonatural transformation between 2-functors that assigns equivalences to
objects is known as a \emph{pseudonatural equivalence}.
Corollary~\ref{cor:coarsening-dg-functors} expresses that
$\mathrm{coarsen}_\gamma$ assigns \emph{quasi-}equivalences (between dg
categories) to objects. 
Thus, the transformations in Corollary \ref{cor:coarsening-transfo} 
should be called \emph{pseudonatural quasi-equivalences}. 
Under a suitable localization of the target Morita bicategory, 
they become honest pseudonatural equivalences. 
\end{remark}

Before we consider the effect of reversing seam orientation on our dg categories, 
we pause to record a variant of Construction~\ref{constr:coarsening}. 
This produces gluing maps that relate
the dg categories associated to a seamed marked surface $(\Sigma, \Pi;\Gamma)$ 
and the seamed marked surface that results from gluing together two boundary arcs 
$\ma, \ma' \in \Pi$.
Since none of the subsequent results in this paper rely on this material, 
our discussion is somewhat terse.

\begin{construction}
	\label{const:gluing}
Given a seamed marked surface $(\Sigma,\Pi;\Gamma)$ and a pair 
	$\{\ma,\ma'\}\in \Pi$ of arcs with opposite orientation relative to the orientation of $\Sigma$, 
	let $\Sigma'$ denote the result of gluing $\Sigma$ along $\ma$ and $\ma'$.
Set $\Pi':=\Pi \smallsetminus \{\ma,\ma'\}$ and $\Gamma'=\Gamma \cup \{\ma\}$. 
Further, let $\mathbf{p},\mathbf{q} \subset A(\Pi)$ be standard subsets
	such that $|\mathbf{p} \cap \ma| = |\mathbf{p} \cap \ma'|$ and 
	$|\mathbf{q} \cap \ma| = |\mathbf{q} \cap \ma'|$.
Fix $S\in \Tan(\Sigma,\mathbf{p};\Gamma)$ and 
	$T\in \Tan(\Sigma,\mathbf{q};\Gamma)$ 
	and let the objects obtained from the 
gluing\footnote{Here, we may need to adjust the parametrization 
of $\ma$ or $\ma'$ in order for the subsets $\mathbf{p} \cap \ma$ 
and $\mathbf{p} \cap \ma'$ (and similarly for $\qq$) to agree after gluing.} 
be denoted $\mathbf{p}',\mathbf{q}', S', T'$. 
Lastly, suppose that $V \in \bn(\Pi)_{\del S}^{\del T}$ has the property that its components $V_\ma$
	and $V_{\ma'}$ indexed by $\ma$ and $\ma'$, satisfy $r_x(V_\ma)= V_{\ma'}$. 
Given such $V$, we let $V' \in \bn(\Pi')^{\partial T'}_{\partial S'}$ 
denote the result of omitting the components $V_\beta$ and $V_{\beta'}$. 

Proceeding analogously to \eqref{eq:coarsen}, we define the 
\emph{gluing} chain map:
\begin{equation}
	\label{eq:glue}
\mathrm{glue}_{V_\ma',V_{\ma}} \colon 
	\PMS_{\Sigma,\Pi;\Gamma}(T|V|S) \to \PMS_{\Sigma',\Pi';\Gamma'}(T'|V'|S')
	\end{equation}
by describing its action on the relevant tensor factors. 
Again, it is the identity on tensor factors corresponding to regions that 
do not intersect $\ma$ or $\ma'$.
On the remaining factors, we again first
realize the appropriate tensor factors as facing each other
(assuming, for simplicity that two regions are involved) along the $\ma$ and
$\ma'$ components.
Then, we map $r_x(V_\ma) \boxtimes V_{\ma}$ into the degree zero chain group 
of $\iota_m^n(\Bar^n_m) = q^{-\frac{1}{2}(m+n)} \Bproj_{m+n}$
(where here $V_\ma \in \bn_m^n$).

In the graphical language of Section~\ref{sec:graphicalmodel}, 
we map the pair of ``black boxes'' labelled
with $V_\ma$ and $V_{\ma'}$ into a pair of ``purple boxes'' in the same position. 
This latter description also works if $\ma$ and $\ma'$ border the same region.
\end{construction}

\begin{rem}
For each pair $\{V_\ma,V_{\ma'} \}$ with $V_{\ma'} = r_x(V_\ma)$, 
we can assemble the gluing maps into a pseudonatural transformation between 2-functors.
For this, we first restrict the domain of $\MM_{\Sigma,\Pi;\Gamma}$ to 
$\bn(\Pi \smallsetminus \{\ma, \ma'\}) = \bn(\Pi')$. 
Note that, in this setup, a $1$-morphism $V'$ determines a $1$-morphism $V \in \bn(\Pi)$ 
by placing $V_\ma$ and $V_{\ma'}$ in the appropriate (missing) entries.
The components of 
\[
\mathrm{glue}_{V_\ma',V_{\ma}} \colon
	(\MM_{\Sigma,\Pi;\Gamma})|_{\bn(\Pi \smallsetminus \{\ma, \ma'\})} 
	\To \MM_{\Sigma',\Pi';\Gamma'}
	\]
are then given by \eqref{eq:glue}. In fact, via these gluing maps we can see
$\MM_{\Sigma',\Pi';\Gamma'}$ as a kind of derived self-tensor product of the $\bn$-\emph{multimodule} $\MM_{\Sigma,\Pi;\Gamma}$ 
along the copies of $\bn$ acting at the arcs $\ma$ and $\ma'$.
\end{rem}

Lastly, we define equivalences associated with reversing the orientation of a seam in a 
seamed marked surface.

\begin{construction}
	\label{const:reorient}
	Given a seamed marked surface $(\Sigma,\Pi;\Gamma)$ and a seam $\gamma\in \Gamma$, 
	let $\Gamma'$ be the set of seams obtained from $\Gamma$ by replacing $\gamma$ with its
	orientation reversal. 
	For $S,T,V$ chosen as in Construction~\ref{constr:coarsening} and again denoting 
	$\PMStwo(T|V|S):= \PMS_{\Sigma,\Pi;\Gamma'}(T|V|S)$, 
	the \emph{reorientation} isomorphism
	\[
	\mathrm{reorient}_\gamma(S|V|T)\colon \PMS(T|V|S)\to \PMStwo(T|V|S)
		\]
	given via the isomorphism of Lemma~\ref{lem:revpurple}
	(i.e.~induced by precomposing with the symmetries \eqref{BS1} and \eqref{BS3} 
	from Proposition~\ref{prop:symmetries of B} in the tensor factors associated with $\gamma$).
\end{construction}

The locality of the construction immediately gives the following.

\begin{cor}
	\label{cor:reorientation-transfo}
	The reorientation isomorphisms $\mathrm{reorient}_\gamma(S|V|T)$ from
	Construction~\ref{const:reorient} assemble to a canonical 
	pseudo-natural equivalence between the 2-functors	
		\[\MM_{\Sigma,\Pi;\Gamma} \To \MM_{\Sigma,\Pi;\Gamma'}\] from
	Theorem~\ref{thm:final}. 
These equivalences intertwine the coarsening and gluing transformations 
	and the equivalences for reversing orientation on distinct seams
	$\gamma,\gamma'\in \Gamma$ commute with each other. \qed
\end{cor}

From now on, we suppress the dependence of $\MM_{\Sigma,\Pi;\Gamma}$ on 
	the orientations of the elements of $\Gamma$, 
	as all such choices yield canonically equivalent 2-functors, 
	compatible with coarsening and gluing transformations.

\subsection{Coherence for coarsening}
\label{s:coherence}

In Section \ref{s:coarsening}, we constructed equivalences between the dg
	categories (and more generally, the $2$-functors) associated to seamed
	marked surfaces with the same underlying marked surface $(\Sigma, \Pi)$ but
	differing sets of seams $\Gamma$ and $\Gamma'$. We now restrict our
	attention to certain seam sets and use classical results of Harer
	\cite{MR830043} to assign dg categories (and more generally, $2$-functors as
	in Theorem \ref{thm:final}) to $(\Sigma, \Pi)$ that are canonical up to
	coherent quasi-equivalence. This approach is inspired by \cite{DyKa,HKK}.

Let $(\Sigma, \Pi;\Gamma)$ be a seamed marked surface, 
	and consider the triple $(X,\BB,\CB)$ associated to $(\Sigma, \Pi)$ in Remark \ref{rem:truncation}.
Recall that $X$ is obtained by collapsing each components of 
	$\del \Sigma \smallsetminus \operatorname{int} A(\Pi)$ 
	to a point, $\BB$ consists of the resulting points whose preimages are arcs, and $\CB$ consists of the 
	resulting points whose preimages are circles.
The set of seams $\Gamma$ determines a collection of arcs $\Gamma_X$ in $X$ 
	which we now consider as \emph{unoriented}.
These arcs have boundary in $\BB \cup \CB$ 
	and they do not intersect in $X \smallsetminus (\BB \cup \CB)$.
As in Construction \ref{const:cutting along arc}, we may consider the result 
	$\overline{X \smallsetminus \Gamma_X}$ of cutting $X$ along the arcs $\Gamma_X$.
The resulting regions are polygons with edges coming from $\Pi$ and (two copies of) $\Gamma_X$, 
	and we call these polygons \emph{unpuctured} if they do not contain a point of $\CB$ in their interior.

\begin{definition}
	\label{def:tessellatedS}
A seamed marked surface $(\Sigma, \Pi;\Gamma)$ is called \emph{tessellated} provided 
	no component of $\overline{X \smallsetminus \Gamma_X}$ is an unpunctured monogon 
	or an unpunctured digon.
\end{definition}

Note that, by our definition, a seamed marked surface
with $\Gamma = \emptyset$ is tessellated 
if (and only if) $\Sigma = D$ is a disk and $|\Pi| \geq 3$ or $\Pi = \emptyset$. 
As such, when $\Sigma = D$ (equivalently, $X$ is an unpunctured disk) 
some results we state will differ slightly from certain parts of the literature.
For example, one can check that a marked surface $(\Sigma, \Pi)$ may be tessellated 
as in Definition \ref{def:tessellatedS} if and only if 
the corresponding surface $(X,\BB \cup \CB)$ admits an ideal triangulation
in the sense of \cite[Definition 2.6]{FST}, 
or when $\Sigma = D$ and $|\Pi| = 0$ or $3$.
Given a tessellated surface $(\Sigma,\Pi;\Gamma)$, 
we will refer to the set of isotopy classes
\[
\TG = \{[\gamma] \mid \gamma \in \Gamma_{X} \}
\]
of unoriented arcs in $X$ as a \emph{tessellation} of $(\Sigma, \Pi)$.

\begin{definition}
	The \emph{tessellation poset} of a marked surface $(\Sigma,\Pi)$ is the
	partially ordered set $T(\Sigma,\Pi)$ whose elements are tessellations, 
	ordered by containment.
\end{definition}

It is useful to think of this poset as a category with objects (equivalence classes of)
tessellated surfaces $(\Sigma, \Pi;\Gamma)$ and with a unique 
\emph{tessellation coarsening} morphism
$(\Sigma, \Pi; \Gamma_1) \to (\Sigma, \Pi; \Gamma_2)$ 
whenever $\Gamma_2$ can be obtained by removing some seams from $\Gamma_1$, 
while still defining a tessellation
(rigorously: if $\TG_2$ is a subset of $\TG_1$). 
Clearly, tessellation coarsenings are generated under composition by 
\emph{elementary coarsenings} that remove a single seam and, thus, 
join two\footnote{Removing an arbitrary single seam may 
also merge one disk into an annulus, but this
would not result in a tessellation, so this would not be a coarsening morphism.}
regions. 
If two seams can be removed simultaneously, while still
resulting in a tessellation, the corresponding elementary coarsenings commute.
In fact, $T(\Sigma,\Pi)$ has no additional relations, not even of the higher kind.
Precisely, the following result is established in \cite{MR830043}; 
see also 
\cite{HatcherT},
\cite[Theorem 3.7]{FST}, 
and \cite[Proposition 3.3.9]{DyKa}.

\begin{proposition}
	\label{prop:tessellationposetcontractible}
If $(\Sigma,\Pi)$ is a marked surface that admits a tessellation,
	then the geometric realization of the nerve of the tessellation poset $T(\Sigma,\Pi)$ is contractible. 
	\qed
\end{proposition}

\begin{rem}\label{rem:disks} The reader familiar with the literature might
expect us to assert in Proposition \ref{prop:tessellationposetcontractible} that
when $\Sigma = D$ (i.e.~$X$ is an unpunctured disk), the tessellation poset is
not contractible, and instead is a sphere. Indeed, this would be the case if we
didn't allow the tessellation $\Gamma = \emptyset$. 
Allowing this tessellation (as we do) then gives that 
the geometric realization of the nerve of
$T(D,\Pi)$ is a ball.
\end{rem}

We will now use Proposition \ref{prop:tessellationposetcontractible} to construct the 
dg categories (and more generally, the $2$-functors) associated to marked surfaces 
that admit tessellations. The following records the remaining marked surfaces, 
which we will treat separately; see Remark \ref{rem:noT2}.

\begin{rem}
	\label{rem:noT}
The connected marked surfaces with non-empty boundary that do not admit tessellations are 
given in Table \ref{table:exceptions}, where, as always, $D$ denotes a disk.
\begin{table}[h]
	\begin{center}
	\begin{tabular}{|c|c||c|c|c|c|}\hline
		$\Sigma$ & $|\Pi|$ & $X$  & $|\BB|$ & $|\CB|$ & description of $(X,\BB,\CB)$ \\ \hline\hline 
		$S^1 \times [0,1]$ & $0$ & $S^2$  & $0$ & $2$ & twice punctured sphere  \\ \hline
		$D$ & $1$ & $D$  & $1$ & $0$ & unpunctured monogon \\ \hline
		$D$ & $2$ & $D$  & $2$ & $0$ & unpunctured digon \\ \hline
	\end{tabular}
	\caption{\label{table:exceptions}}
	\end{center}
	\end{table}
	
\noindent Note that, for us, $(D,\Pi)$ admits the single tessellation $\Gamma = \emptyset$ when $|\Pi|=0,3$.
\end{rem}

The following results will allow us to define functors from the tessellation poset.

\begin{thm}\label{thm:coarsencommute}
Let $(\Sigma,\Pi;\Gamma)$ be a tessellated surface and let
	$\gamma_1,\gamma_2\in \Gamma$ with $\gamma_1\neq \gamma_2$ be 
	such that for $\Gamma_i:=\Gamma\smallsetminus\{\gamma_i\}$ and
	$\Gamma':=\Gamma\smallsetminus\{\gamma_1,\gamma_2\}$ the triples
	$(\Sigma,\Pi;\Gamma_1)$, $(\Sigma,\Pi;\Gamma_2)$, and $(\Sigma,\Pi;\Gamma')$ 
	are again tessellated surfaces. 
Then, the coarsening transformations for $\gamma_1$ and $\gamma_2$ 
	form a strictly commuting square:
	\[
	\begin{tikzcd}
		& \MM_{\Sigma,\Pi;\Gamma_1} \arrow[dr, double, "\mathrm{coarsen}_{\gamma_2}"]
		&
		\\[-15pt]
		\MM_{\Sigma,\Pi;\Gamma}
		\arrow[ur, double,"\mathrm{coarsen}_{\gamma_1}"]
		\arrow[dr, double,"\mathrm{coarsen}_{\gamma_2}",swap]&&
		\MM_{\Sigma,\Pi;\Gamma'}
	\\[-15pt]
	&
	\MM_{\Sigma,\Pi;\Gamma_2}
	\arrow[ur, double,"\mathrm{coarsen}_{\gamma_1}",swap]
	&
	\end{tikzcd} \, .
	\]
\end{thm}
\begin{proof}
We will establish the commutativity on component $1$-morphisms $V \in \bni(\Pi)$, 
	namely the commutativity of:
	\[
	\begin{tikzcd}
		& \MM_{\Sigma,\Pi;\Gamma_1}(V) \arrow[dr, "\mathrm{coarsen}_{\gamma_2}"]
		&
		\\[-15pt]
		\MM_{\Sigma,\Pi;\Gamma}(V)
		\arrow[ur, "\mathrm{coarsen}_{\gamma_1}"]
		\arrow[dr, "\mathrm{coarsen}_{\gamma_2}",swap]&&
		\MM_{\Sigma,\Pi;\Gamma'}(V)
	\\[-15pt]
	&
	\MM_{\Sigma,\Pi;\Gamma_2}(V)
	\arrow[ur, "\mathrm{coarsen}_{\gamma_1}",swap]
	&
	\end{tikzcd}	
	\]
For $V=\one$, this implies the commutativity of the dg functors from
	Corollary~\ref{cor:coarsening-dg-functors},  and then the commutativity 
	for general $V$ using Lemma~\ref{lemma:tangle slide}.
Consequently, this establishes commutativity for the entire coarsening transformations. 
		
The assumption that $\Gamma'$ defines a tessellation 
	(thus, in particular, a seamed marked surface)
	implies that we are in one of two situations 
	from the perspective of $(\Sigma,\Pi;\Gamma)$:
	\begin{enumerate}
		\item There are four pairwise distinct regions $D_i,D_j,D_k,D_l$ with
		respect to $\Gamma$, where $\gamma_1$ separates $D_i$ and $D_j$ and
		$\gamma_2$ separates $D_k$ and $D_l$.
		\item There are three pairwise distinct regions $D_i,D_j,D_k$ with
		respect to $\Gamma$, where $\gamma_1$ separates $D_i$ and $D_j$ and
		$\gamma_2$ separates $D_j$ and $D_k$.
	\end{enumerate}
In the first case, commutativity follows because
	$\mathrm{coarsen}_{\gamma_1}$ and $\mathrm{coarsen}_{\gamma_2}$ act
	on completely independent tensor factors. 
	
In the second case, commutativity follows because we can find a planar model
	encompassing all three involved tensor factors (similar to the model for two
	tensor factors in \eqref{eq:coarsen}), in which the two coarsening maps are
	realised by applying counits to distant copies of $\Bproj$, 
	and these clearly commute. 
	In fact, we can choose any model for the $j$-th tensor factor, as
	long as we choose adapted (see Remark~\ref{rem:rotatefactor}) models for the
	$i$-th and $k$-th factors with purple boxes pointing left or right depending on
	whether these disks lie to the right or left of the arcs, respectively.
	We can then place these planar models in a suitable subregion of the model for the
	$j$-th factor and proceed as in \eqref{eq:coarsen}.
\end{proof}

Since the $2$-functors $\MM_{\Sigma,\Pi;\Gamma}$ depend only on the combinatorial data 
associated with the cut surface $\cutS,\cutP$, 
and since Corollary \ref{cor:reorientation-transfo} gives that these $2$-functors are independent of 
seam orientation, the following is immediate.

\begin{lem}
If $\Gamma$ and $\Gamma'$ are sets of seams in the marked surface $(\Sigma, \Pi)$ giving rise to the 
same tessellation, then there is a canonical identification of the $2$-functors 
$\MM_{\Sigma,\Pi;\Gamma}$ and $\MM_{\Sigma,\Pi;\Gamma'}$. \qed
\end{lem}

We thus arrive at the following.

\begin{thm}\label{thm:final2} 
Let $(\Sigma,\Pi)$ be a marked surface that admits a tessellation.
The assignment $(\Sigma,\Pi;\Gamma)\mapsto
	\MM_{\Sigma,\Pi;\Gamma}$ from Theorem~\ref{thm:final}  together with the coarsening transformations
from Corollary~\ref{cor:coarsening-transfo} yields a functor from the tessellation poset of $(\Sigma,\Pi)$ to the
category of 2-functors $\bni(\Pi) \To \MorDGCat$ and pseudonatural transformations between them:
\begin{equation}
	\label{eq:TPF}
\MM \colon T(\Sigma, \Pi) \to \mathrm{2Fun}\big(\bni(\Pi),\MorDGCat \big) \, .
\end{equation}
Moreover, the pseudonatural transformations are object-wise quasi-equivalences. \qed
\end{thm}

For marked surfaces $(\Sigma, \Pi)$ that admit tessellations,
the contractibility of the tessellation poset provided by Proposition \ref{prop:tessellationposetcontractible}
implies that a suitably defined homotopy colimit of the functor $\MM$ 
from Theorem \ref{thm:final2} can be considered to be 
canonically associated with $(\Sigma, \Pi)$, 
without the auxiliary choice of a system of seams $\Gamma$. 
On the level of objects, we can be more precise.
As shown in \cite{Tabuada}, 
the category of small (pointed) dg categories $\DGCat[\Z \times \Z]$ admits a model structure 
wherein the weak equivalences are the quasi-equivalences. 

\begin{definition}
Let $(\Sigma,\Pi)$ be a marked surface that admits a tessellation.
Given $\pp \subset A(\Pi)$, let $\sCat(\Sigma,\pp)$ denote the 
homotopy colimit of the functor
\begin{equation}
	\label{eq:TPF2}
\begin{aligned}
	T(\Sigma, \Pi) &\to \DGCat[\Z \times \Z] \\
	(\Sigma,\Pi;\Gamma) &\mapsto \sCat(\Sigma,\mathbf{p};\Gamma)
\end{aligned}
\end{equation}
obtained by evaluating \eqref{eq:TPF} at the object $\pp \in \BN(\Pi)$.
\end{definition}

In apparent contrast to $\sCat(\Sigma,\pp;\Gamma)$, 
the dg category $\sCat(\Sigma,\pp)$ is canonically associated to 
$(\Sigma,\Pi)$, independent of any choice of seams.
Nevertheless, the following shows that upon passing to the category 
$\Hqe$, which is the localization of $\DGCat[\Z \times \Z]$ at the quasi-equivalences,
the same can actually be said of $\sCat(\Sigma,\pp;\Gamma)$ 
for \emph{any choice} of tessellation.

\begin{thm}\label{thm:final3} 
Let $(\Sigma, \Pi; \Gamma)$ be a tessellated surface.
For each $\pp \in \bni(\Pi)$, there is a canonical quasi-equivalence
$\sCat(\Sigma,\pp;\Gamma) \to \sCat(\Sigma,\pp)$.
Consequently, if $(\Sigma, \Pi; \Gamma)$ and $(\Sigma, \Pi; \Gamma')$ 
are two tessellations of the same marked surface $(\Sigma, \Pi)$, 
there is a canonical isomorphism 
$\sCat(\Sigma,\pp;\Gamma) \cong \sCat(\Sigma,\pp;\Gamma')$ in $\Hqe$.
\end{thm}
\begin{proof}
In model categorical language, 
Theorem \ref{thm:final2} gives that the functor \eqref{eq:TPF2} 
is homotopically constant.
By Proposition \ref{prop:tessellationposetcontractible}, 
the (nerve of the) domain of \eqref{eq:TPF2} is contractible, 
and the result follows e.g.~from \cite[Corollary 29.2]{HToD}.
\end{proof}

\begin{remark}
	\label{rem:noT2}
Theorem \ref{thm:final3} associates a dg category to marked surfaces $(\Sigma,\Pi)$ 
that admit tessellations, 
i.e.~to all marked surfaces with the exception of those containing a connected component 
as listed in Remark \ref{rem:noT}.
We will discuss the special case of $(S^1 \times [0,1], \emptyset)$ 
in Remark \ref{rem:Connes}.
The remaining cases $(D,\Pi)$ with $|\Pi|=1$ or $2$ can be treated as follows.
Since the tessellation poset is empty, we may instead consider the 
``seamed marked surfaces'' poset, consisting of equivalence classes of seam sets $\Gamma$, 
where equivalence is generated by isotopy and seam reversal, 
and the partial order is again given by containment.
We again find that this poset is contractible, 
so the argument in the proof of Theorem \ref{thm:final3} applies.
In fact, as is the case for the tessellation poset for marked disks with $|\Pi| \neq 1, 2$, 
this poset has a terminal object given by $\Gamma = \emptyset$, 
so we can view the (dg) category $\sCat(D,\pp;\emptyset)$ as canonically associated to $(D,\Pi)$;
in fact, the choice of $\Pi$ is essentially immaterial.
\end{remark}

Theorems~\ref{thm:final2} and \ref{thm:final3} 
also allow us to construct an action of orientation
preserving diffeomorphisms of $\Sigma$ that fix the boundary $\partial \Sigma$ pointwise. 
Let $\diff(\Sigma, \del \Sigma)$ denote the group of such diffeomorphisms.

\begin{construction}[Sketch: action by diffeomorphisms]\label{constr:diffeos} 
Let $(\Sigma,\Pi)$ be a marked surface that admits a tessellation 
and let $\phi\in \diff(\Sigma)$. 
Given any tessellation $(\Sigma, \Pi; \Gamma)$, after acting by $\phi$
we again find that $(\Sigma, \Pi; \phi(\Gamma))\in T(\Sigma, \Pi)$. 
By inspecting Definition~\ref{def:our bimodules}, it is clear
that the geometric action of $\phi$ on tangles 
in $\Tan(\Sigma,\Pi;\Gamma)$ tautologically extends to an invertible transformation:
\[
\MM_{\phi}\colon \MM_{\Sigma,\Pi;\Gamma} \To \MM_{\Sigma,\Pi;\phi(\Gamma)} \, .
\]
Upon appropriate localization, 
we may post-compose with the essentially unique identification 
\[\MM_{\Sigma,\Pi;\phi(\Gamma)}\To \MM_{\Sigma,\Pi;\Gamma} \] 
and view $\MM_{\phi}$ as an automorphism of $\MM_{\Sigma,\Pi;\Gamma}$.
Further, another application of Proposition \ref{prop:tessellationposetcontractible} 
shows that $\MM_{\phi}$ does not depend on the auxiliary choice of $\Gamma$ 
and that these maps are compatible with composing 
diffeomorphisms---all understood in a suitably weak sense. 
\end{construction}

This construction only sketches how individual diffeomorphisms act and---by
discreteness of the construction---this only depends on the underlying mapping class.
We do conjecture, however, that an appropriate 
$\infty$-categorical setup for the homotopy colimit of the functor from Theorem~\ref{thm:final2} 
will support a homotopy-coherent action of the entire diffeomorphism group.

\section{The dg extended affine Bar-Natan category}
\label{s:affine BN}

Recall that the lone marked surface not covered by our results in \S \ref{s:coherence}
was the marked surface $(S^1 \times [0,1], \emptyset)$. 
In this section, we briefly elaborate on the categories associated to (marked) annuli.

\subsection{The annulus invariant}
\label{ss:annulus and htr}
We begin by making the setup from Example \ref{ex:annulus} more precise.
Let us model $S^1$ as the unit circle in $\C$ 
and consider the annulus $S^1 \times
[0,1]$ and $\Gamma=\{\gamma\}$ with $\gamma = \{\mathsf{i}\} \times [0,1]$ 
(here $\mathsf{i} \in \C$ denotes the imaginary unit).
Let $\ma \colon [0,1] \to S^1$ be a positively-oriented arc whose image lies in 
some small neighborhood of $-\mathsf{i}$
and let $\Pi = \{ \ma_0, \ma_1\}$ for $\ma_j = \ma \times \{j\}$.
By slight abuse of notation, 
let $\pp^n$ denote the standard set of $n$ points in $\ma$.
(Technically, since $\pp^n \subset (0,1)$, 
this should be written $\ma(\pp^n)$.)

\begin{definition}\label{def:annulus setup}
Consider the marked surface $\big(S^1 \times [0,1], \Pi = \{ \ma_0, \ma_1\}\big)$.
Given $k,l \in \N$, 
set $\pp^k_l:=\pp^k \times\{1\} \sqcup \pp^l \times\{0\} \subset \Pi$.
The \emph{dg extended affine Bar-Natan category}
is the dg category
\[
\Affbn^k_l:=\sCat(S^1\times[0,1],\pp^k_l;\{\gamma\}).
\]
\end{definition}

\begin{remark}
A typical object $T\in \Affbn^k_l$ is pictured in \eqref{eq:tangle in annulus}.
\end{remark}

Given a $\K$-linear bicategory $\CB$, 
the work in \cite[Section 6.4]{2002.06110} defines its \emph{dg horizontal trace} 
$\hTr^{\dg}(\CB)$, 
a dg analogue of the usual horizontal trace of a bicategory 
(see e.g.~\cite{BHLZ,QR2} for the latter). 
When $\CB = \bn$, 
we now show that
the dg horizontal trace is a special case of Definition \ref{def:annulus setup}.

\begin{thm}
	\label{thm:dgtrace}
	There is an equivalence of dg categories $\Affbn^0_0 \cong \hTr^{\dg}(\bn)$.
	\end{thm}
	
In the following, we assume some familiarity with \cite{2002.06110}.
	
	\begin{proof}
	The objects of $\Affbn^0_0$ and $\hTr^{\dg}(\bn)$ are both parametrized by a
	non-negative integer $n\in \N$ and an object $T\in \bn_n^n$. 
We identify such tangles with a $2$-partitioned cap tangle as follows:
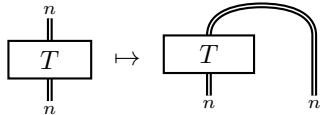
\[
		\begin{tikzpicture}[anchorbase]
			\draw[thick,double] (0,-.55) node[below=-2pt]{\scs$n$} 
				to (0,.55) node[above=-2pt]{\scs$n$};
		\draw[thick, fill=white] (-.55,-.25) rectangle (.55,.25);
		\node at (0,0) {$T$};
		\end{tikzpicture}
		\;\;
		\mapsto
		\;\;
		\begin{tikzpicture}[anchorbase]
			\draw[thick,double] 	
			(0,-.55) node[below=-2pt]{\scs$n$} to (0,.25) \ur (.7,.65) \rd (1.4,.25) 
				to (1.4,-.55) node[below=-2pt]{\scs$n$};
			\draw[thick, fill=white] (-.6,-.25) rectangle (.6,.25);
			\node at (0,0) {$T$};
			\end{tikzpicture}
\]
The complexes of morphisms from an object $S\in \bn_m^m$ to an object $T\in
\bn_n^n$ in $\hTr^{\dg}(\bn)$ is defined by first considering morphism spaces
\[
\Hom_{\bn}(S\star Z, Z' \star T) 
	\cong q^{\frac{m+n}{2}} 	
	\KhEval{
	\begin{tikzpicture}[anchorbase]
		\draw[thick,double] 	
		(0,-.75) to (0,.25) \ur (.7,.65) \rd (1.4,.25) to (1.4,-1) \dl (.7,-1.4)\lu (0,-1);
		\draw[thick, fill=white] (-.6,-.25) rectangle (.6,.25);
		\draw[thick, fill=white] (-.6,-.5) rectangle (.6,-1);
		\draw[thick, fill=white] (.8,-.25) rectangle (2,.25);
		\draw[thick, fill=white] (.8,-.5) rectangle (2,-1);
		\node at (0,0) {$Z'$};
		\node at (1.4,0) {$r_x(S)$};
		\node at (0,-.75) {$T$};
		\node at (1.4,-.75) {$r_x(Z)$};
		\end{tikzpicture}
	}
	\cong q^{\frac{m+n}{2}} 	
	\KhEval{
	\begin{tikzpicture}[anchorbase]
		\draw[thick,double] 	
		(0,-1.25) to (0,.25) \ur (.7,.65) \rd (1.4,.25) to (1.4,-1.75) \dl (.7,-2.15)\lu (0,-1.75);
		\draw[thick, fill=white] (-.6,-.25) rectangle (.6,.25);
		\draw[thick, fill=white] (-.6,-.5) rectangle (.6,-1);
		\draw[thick, fill=white] (-.6,-1.25) rectangle (.6,-1.75);
		\draw[thick, fill=white] (.8,-.5) rectangle (2,-1);
		\node at (0,0) {$T$};
		\node at (1.4,-.75) {$r_{xy}(Z')$};
		\node at (0,-.75) {$r_y(Z)$};
		\node at (0,-1.5) {$r_y(S)$};
		\end{tikzpicture}
	}
\]
for $Z,Z' \in \bn_n^m$
and then contracting with the bar complex $\Bar^n_m$. 
Using the (purple box) notation from Definition \ref{def:tbar} 
and Remarks \ref{rmk:einstein} and \ref{rmk:purple boxes merge to roz},
as well as the symmetries \eqref{BS1} and \eqref{BS2} from Lemma~\ref{prop:symmetries of B},
we obtain the isomorphism
\[
	\Hom_{\hTr^{\dg}(\bn)}(S,T) 
\cong q^{\frac{m+n}{2}} 	
	\KhEval{
	\begin{tikzpicture}[anchorbase]
		\draw[thick,double] 	
		(0,-1.25) to (0,.25) \ur (.7,.65) \rd (1.4,.25) to (1.4,-1.75) \dl (.7,-2.15)\lu (0,-1.75);
		\draw[thick, fill=white] (-.6,-.25) rectangle (.6,.25);
		\draw[thick, fill=white] (-.6,-1.25) rectangle (.6,-1.75);
		\rjf{0}{-.75}{}
		\ljf{1.4}{-.75}{}
		\node at (0,0) {$T$};
		\node at (0,-1.5) {$r_y(S)$};
		\end{tikzpicture}
	}
=
	\KhEval{
	\begin{tikzpicture}[anchorbase]
		\draw[thick,double] 	
		(0,-.25) to (0,.25) \ur (.7,.65) \rd (1.4,.25) to (1.4,-0.25) \dl (.7,-0.65)\lu (0,-0.25)
		(0,-1.25) \ur (.7,-.85) \rd (1.4,-1.25) to (1.4,-1.75) \dl (.7,-2.15)\lu (0,-1.75);
		\draw[thick, fill=white] (-.6,-.25) rectangle (.6,.25);
		\draw[thick, fill=white] (-.6,-1.25) rectangle (.6,-1.75);
		\draw[thick,fill=white] (0.4,-.3) rectangle (1,-1.2);
		\node[rotate=90] at (.7,-.75) {$\Bproj_{n+m}$};
		\node at (0,0) {$T$};
		\node at (0,-1.5) {$r_y(S)$};
		\end{tikzpicture}
	}
	=
	\Hom_{\Affbn^0_0}(S,T) 
	\]
	Moreover, this isomorphism is compatible with composition, which in both
	categories is defined via the Eilenberg--Zilber shuffle product on bar
	complexes, 
	c.f.~\cite[Section 6.4]{2002.06110} and Proposition~\ref{prop:Bar is algebra}.
	\end{proof}

\begin{remark}
	\label{rem:Connes}
Theorem \ref{thm:dgtrace}	 also holds if we replace the dg category
\[
\Affbn_0^0 = \sCat(S^1 \times [0,1], \emptyset \subset \{\ma_0,\ma_1\}; \{ \gamma \})
\]
associated to the marked surface $(S^1 \times [0,1] , \{\ma_0,\ma_1\})$
with the (similarly notated) dg category 
\[
\sCat(S^1 \times [0,1], \emptyset; \{ \gamma \})
\]
associated to the marked surface $(S^1 \times [0,1] , \emptyset)$.
Indeed, upon inspection, we see that these categories are equal on the nose.

Recall that the case of the annulus with $\Pi=\emptyset$ 
	was excluded from the discussion in \S \ref{s:coherence} because
	Proposition~\ref{prop:tessellationposetcontractible} is not available. 
The above identification shows that we can view the category $\Affbn_0^0$ as being associated 
	to $(S^1 \times [0,1],\emptyset)$.
However, since the definition of $\Affbn_0^0$ required choosing a point in $S^1$ 
	(for us, $\mathsf{i} \in S^1$),
	there is an entire $S^1$-family of such categories, rather then a single one. 
These categories are all equivalent, but not up to unique equivalence. 
This ambiguity is well-known in the setting of the dg horizontal trace 
	and, more classically, Hochschild chains, 
	where it gives rise to the Connes differential; see e.g.~\cite[\S 5.6, \S 8.4]{2002.06110}.
	\end{remark}

\begin{remark}
The results in \cite[\S 6.2]{rozansky2010categorification} show that the bottom projectors
	$\Bproj_{n+m}$ admit a description as the stable limit of the Khovanov invariants 
	of powers of the full twist braid on $m+n$ strands.
Theorem~\ref{thm:dgtrace} therefore shows that we may
	approximate the cohomology of the morphism complexes in the horizontal
	trace $\hTr^{\dg}(\bn)$ by the Khovanov homology of link diagrams
	assembled from the annular links $S$ and $T$, with a number of full
	twists inserted in place of the bottom projector.
	\end{remark}

\subsection{Composing affine tangles}
\label{ss:composing in affbn}

This equivalence of categories in Theorem \ref{thm:dgtrace}
can be promoted to an equivalence of \emph{monoidal} dg categories.  
More generally, 
the dg extended affine Bar-Natan categories are expected to be 1-morphism
categories in a dg extended affine Bar-Natan bicategory\footnote{This may have an
interpretation as factorization homology of the monoidal Bar-Natan bicategory
over the circle.}. 
The horizontal composition
\[
	\Affbn_{m}^n \times \Affbn_{k}^m \to \Affbn_{k}^n
	\]
where $k,m,n\in \N$ would correspond to the operation of stacking annular tangles inside each other:

\[
\begin{tikzpicture}[anchorbase,scale=1]
\draw[very thick, gray] (0,0) circle (.25);
\draw[very thick, gray] (0,0) circle (1);
\draw[very thick, gray] (0,0) circle (1.75);
\draw[thick,double] (0,0) circle (.625);
\draw[thick,double] (0,0) circle (1.375);
\draw[thick,double] (0,-1.75) to (0,-.25);
\begin{scope}[shift={(0,-.625)}] \draw[fill=white] (0,0) circle (.25);\end{scope}
\node at (0,-.625) {$T_1$};
\begin{scope}[shift={(0,-1.375)}] \draw[fill=white] (0,0) circle (.25);\end{scope}
\node at (0,-1.375) {$T_2$};
\draw[->, seam,very thick] (0,.25) to (0,1.75);
\draw [blue,ultra thick,domain=240:300] plot ({1.75*cos(\x)}, {1.75*sin(\x)});
\draw [blue,ultra thick,domain=235:305] plot ({cos(\x)}, {sin(\x)});
\draw [blue,ultra thick,domain=225:315] plot ({.25*cos(\x)}, {.25*sin(\x)});
\end{tikzpicture}
\]
We will write $a(T_1,T_2)$ for the result of stacking the annular tangles $T_1$ and $T_2$ as in the figure.

We now describe a candidate for the annular stacking composition. 
On the level of objects, it is simply given by gluing tangles.  
On morphisms, the relevant chain map
is most clearly expressed in the language of \eqref{eq:annulusflex}. 
In this graphical language, the map
$\PMS(T_1,S_1)\otimes \PMS(T_2,S_2) \to \PMS(a(T_1,T_2),a(S_1,S_2))$ 
is given as the following composition 
(wherein we suppress grading shifts to save space):
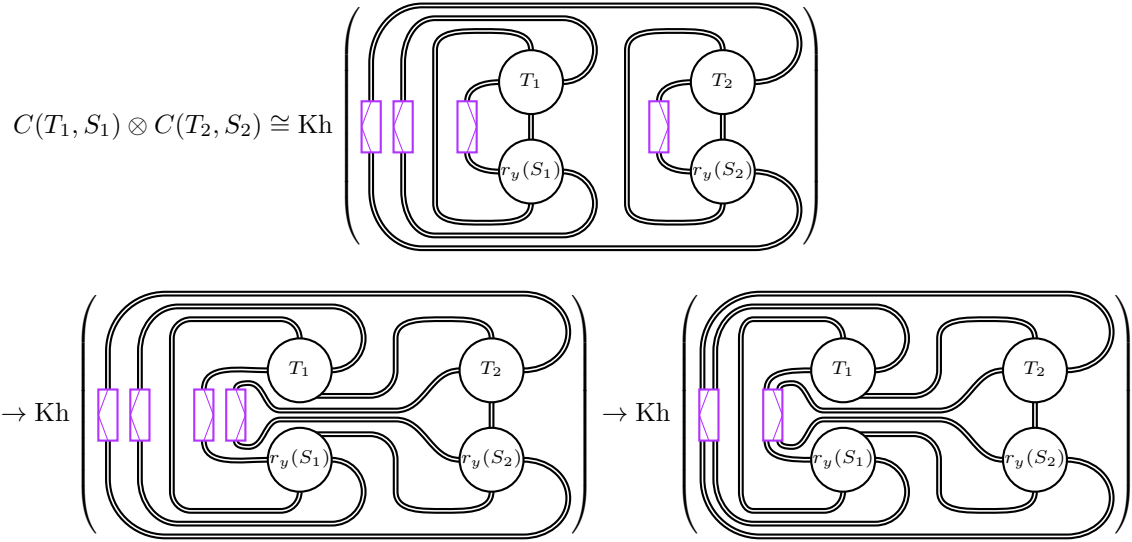
\begin{multline*}
\PMS(T_1,S_1)\otimes \PMS(T_2,S_2) \cong 
	\KhEval{
	\begin{tikzpicture}[anchorbase,scale=.85,smallnodes]
		\begin{scope}[shift={(-3,0)}]
		\draw[thick, double] 
		(0,-1.2) to (0,1.2)  
		(0,1.2) \ul (-0.75,1.5) \ld (-1.5,1.2) to (-1.5,-1.2) \dr (-.75,-1.5) \ru (0,-1.2) 
		(0.5,0.7) \ru (1,1.2) \ul (0.5,1.7) \pl (-1.5,1.7) \ld (-2,1) to (-2,0)
		(0.5,-0.7) \rd (1,-1.2) \dl (0.5,-1.7) \pl (-1.5,-1.7) \lu (-2,-1) to (-2,0)
		(-0.5,0.7) \ld (-1,0.4)
		(-0.5,-0.7) \lu (-1,-0.4)
		;
		\rjf{-2}{0}{}
		\ljf{-1}{0}{}
		\circlepairwide{0}{0}{1}	
		\end{scope}
		\draw[thick, double] 
		(0,-1.2) to (0,1.2)  
		(0,1.2) \ul (-0.75,1.5) \ld (-1.5,1.2) to (-1.5,-1.2) \dr (-.75,-1.5) \ru (0,-1.2) 
		(0.5,0.7) \ru (1.2,1.3) \ul (0.5,1.9) \pl (-4.7,1.9) \ld (-5.5,1) to (-5.5,0)
		(0.5,-0.7) \rd (1.2,-1.3) \dl (0.5,-1.9) \pl (-4.7,-1.9) \lu (-5.5,-1) to (-5.5,0)
		(-0.5,0.7) \ld (-1,0.4)
		(-0.5,-0.7) \lu (-1,-0.4)
		;
		\rjf{-5.5}{0}{}
		\ljf{-1}{0}{}
		\circlepairwide{0}{0}{2}	
	\end{tikzpicture}
} \\ \\
\to
\KhEval{
	\begin{tikzpicture}[anchorbase,scale=.85,smallnodes]
		\draw[thick, double] 
		(0,1.2) \ul (-0.75,1.5) \ld (-1.5,1.2) to (-1.5,.7) \dl (-1.9,.3) to (-2.8,.3)
		(0,-1.2) \dl (-0.75,-1.5) \lu (-1.5,-1.2) to (-1.5,-.7) \ul (-1.9,-.3) to (-2.8,-.3) 
		;
		\begin{scope}[shift={(-3,0)}]
		\draw[thick, double] 
		(0,1.2) \ul (-.3,1.5) to (-1.7,1.5) \ld (-2,1.2) to (-2,-1.2) \dr (-1.7,-1.5) to (-.3,-1.5) \ru (0,-1.2) 
		(0.5,0.7) \ru (1,1.2) \ul (0.5,1.7) \pl (-1.9,1.7) \ld (-2.5,1) to (-2.5,0)
		(0.5,-0.7) \rd (1,-1.2) \dl (0.5,-1.7) \pl (-1.9,-1.7) \lu (-2.5,-1) to (-2.5,0)
		(-0.5,0.7) \ld (-1.5,0.4)
		(-0.5,-0.7) \lu (-1.5,-0.4)
		;
		\rjf{-2.5}{0}{}
		\ljf{-1.5}{0}{}
		\circlepairwide{0}{0}{1}	
		\end{scope}
		\draw[thick, double] 
		(0,-1.2) to (0,1.2)  
		(0.5,0.7) \ru (1.2,1.3) \ul (0.5,1.9) \pl (-5.1,1.9) \ld (-6,1) to (-6,0)
		(0.5,-0.7) \rd (1.2,-1.3) \dl (0.5,-1.9) \pl (-5.1,-1.9) \lu (-6,-1) to (-6,0)
		(-0.5,0.7) \pl (-1.5,.07) to (-3.3,.07) \pl (-3.8,.5)  \ld (-4,.4)
		(-0.5,-0.7) \pl (-1.5,-.07) to (-3.3,-.07) \pl (-3.8,-.5)  \lu (-4,-.4)
		;
		\rjf{-6}{0}{}
		\ljf{-4}{0}{}
		\circlepairwide{0}{0}{2}	
	\end{tikzpicture}
}
\to
\KhEval{
	\begin{tikzpicture}[anchorbase,scale=.85,smallnodes]
		\draw[thick, double] 
		(0,1.2) \ul (-0.75,1.5) \ld (-1.5,1.2) to (-1.5,.7) \dl (-1.9,.3) to (-2.8,.3)
		(0,-1.2) \dl (-0.75,-1.5) \lu (-1.5,-1.2) to (-1.5,-.7) \ul (-1.9,-.3) to (-2.8,-.3) 
		(0.5,0.7) \ru (1.2,1.3) \ul (0.5,1.9) \pl (-4.3,1.9) \ld (-5.2,1) to (-5.2,0)
		(0.5,-0.7) \rd (1.2,-1.3) \dl (0.5,-1.9) \pl (-4.3,-1.9) \lu (-5.2,-1) to (-5.2,0)
		;
		\begin{scope}[shift={(-3,0)}]
		\draw[thick, double] 
		(0,1.2) \ul (-.3,1.5) to (-1.4,1.5) \ld (-1.6,1.2) to (-1.6,-1.2) \dr (-1.4,-1.5) to (-.3,-1.5) \ru (0,-1.2) 
		(0.5,0.7) \ru (1,1.2) \ul (0.5,1.7) \pl (-1.4,1.7) \ld (-2,1) to (-2,0)
		(0.5,-0.7) \rd (1,-1.2) \dl (0.5,-1.7) \pl (-1.4,-1.7) \lu (-2,-1) to (-2,0)
		(-0.5,0.7) \ld (-1.2,0.4)
		(-0.5,-0.7) \lu (-1.2,-0.4)
		;
		\rjf{-2.1}{0}{}
		\circlepairwide{0}{0}{1}	
		\end{scope}
		\draw[thick, double] 
		(0,-1.2) to (0,1.2)  
		(-0.5,0.7) \pl (-1.5,.07) to (-3.3,.07) \pl (-3.8,.5)  \ld (-4,.4)
		(-0.5,-0.7) \pl (-1.5,-.07) to (-3.3,-.07) \pl (-3.8,-.5)  \lu (-4,-.4)
		;
		\ljf{-4.1}{0}{}
		\circlepairwide{0}{0}{2}	
	\end{tikzpicture}
}
\end{multline*}

Here, the isomorphism uses sphericality and symmetric monoidality of $\bn_0^0$
and the first arrow is a (composition of) saddle morphism(s). 
The last arrow uses yet another composition of bottom projectors/bar complexes,
this time induced by the monoidal structure $\boxtimes$ on $\bn$;
compare with Proposition \ref{prop:Bar is algebra} and Remark~\ref{rmk:roz is coalg}, 
where the composition was based on $\star$. 
It is clear that for a triple gluing, the chain maps implementing the gluing on the
level of $2$-morphims are as associative as this composition of Rozansky projectors.

\section{Spin networks}
\label{sec:spin}

As discussed in \S \ref{s:intro}, 
	we expect the dg categories $\sCat(\Sigma,\pp)$ to 
	be the invariants assigned to surfaces by a higher-categorical version of 
	the Turaev--Viro TFT.
The vector spaces associated to (appropriately decorated) surfaces $\Sigma$ by the latter 
	admit a description in terms of the Temperley--Lieb skein module of $\Sigma$;
	see e.g.~\cite{SN-TV}.
Given an ideal triangulation of $\Sigma$, 
	this skein module has a basis consisting of
	so-called ``spin networks,'' which are built from Jones--Wenzl projectors.
In this section, we construct a categorified analogue of this spin network basis.
The most striking result is Theorem \ref{thm:pairing}, 
	which shows that the (twisted) $\Hom$-pairing on $\sCat(\Sigma,\pp)$
	recovers a bilinear form appearing in various formulations 
	of TFTs derived from the Temperley--Lieb category.

While we have worked over a general commutative ring $\cring$ thus far, 
	we assume for the duration of this section 
	that $\cring$ is an integral domain,
and we let $\mathbb{K}$ denote the fraction field of $\cring$.
In order to avoid reproducing large swaths of background, 
	we will also assume that the reader is familiar with the relevant parts 
	of the categorification literature concerning the categorification of Jones--Wenzl projectors, 
	e.g.~the papers \cite{CooperKrushkal} and \cite{Roz}.
We will also regularly employ homological perturbation theory;
	see e.g.~\cite{Markl}.

\subsection{The setup}
\label{ss:spin network setup}

We first introduce an appropriate category of one-sided twisted complexes, with some restrictions on the gradings.  

\begin{definition}
Denote the power set of $\Z\times \Z$ by $\mathcal{P}(\Z\times \Z)$.
Let $O_{\leftangle} \subset \mathcal{P}(\Z\times \Z)$ 
be the set of all subsets $S\subset \Z \times \Z$ satisfying the following conditions:
\begin{enumerate}
	\item[(a)] $S$ is bounded from the right, i.e. $S\subset \Z_{\leq M} \times \Z$
	for some $M\in \Z$.
	\item[(b)] $S$ has finite intersection with every vertical line $\{x\}\times
	\Z$ for $x\in \Z$. 
	\item[(c)] We have $\lim_{x\to -\infty} \min(y\in \Z|(x,y)\in
	S)=\infty$. 
\end{enumerate}
Here the minimum of the empty set is declared to be $\infty$.
\end{definition}

\begin{lemma} 
	\label{lem:Oproperties}
The set $O_{\leftangle}$ satisfies the following properties:
\begin{itemize}
	\item Every finite subset of $\Z\times \Z$ is in $O_{\leftangle}$.
	\item For every $y\in \Z$, each $S\in O_{\leftangle}$ 
		has finite intersection with the horizontal line $\Z\times \{y\}$.
	\item Every $S\in O_{\leftangle}$ is bounded from below, 
		i.e.~$S\subset  \Z \times \Z_{\geq m}$ for some $m\in \Z$.
	\item If $S\in O_{\leftangle}$ and $S'\subset S$ then $S'\in O_{\leftangle}$.
	\item If $S_1,S_2\in O_{\leftangle}$ then $S_1\cup S_2\in O_{\leftangle}$.
	\item If $S_1,S_2\in O_{\leftangle}$ then $S_1+S_2\in O_{\leftangle}$.
	\end{itemize}
\end{lemma}
	\begin{proof} Straightforward. For property (b) for $S_1+S_2$, 
	note that by (a) there are only finitely many pairs of vertical lines in $S_1$ and $S_2$
	that may contribute to a vertical line in $S_1+S_2$. 
	\end{proof}

\begin{definition}\label{def:angle shaped over k}
The \emph{support} of a complex $C \in\dgMod{\Z \times \Z}{\cring}$, 
denoted $\supp(C)$, is the subset of indices 
$(i,j)\in \Z\times \Z$ in which 
$C$ is nonzero.
We say that $C$ is \emph{angle-shaped} if the chain modules $C^{i,j}$ of $C$ 
are 
free of finite rank over
$\cring$ and $\supp(C)\in O_{\leftangle}$.
We denote by
$\dgMod{\leftangle}{\cring}$ (respectively $\dgModen{\leftangle}{\cring}$)
the full subcategories of 
$\dgMod{\Z \times \Z}{\cring}$ (resp. $\dgModen{\Z \times \Z}{\cring}$)
consisting of all angle-shaped complexes. For $(i,j)\in \Z\times \Z$ we let
\begin{equation}\label{eq:poincare and euler}
P(C):=\sum_{i,j} t^iq^j \mathrm{dim}(\mathbb{K}\otimes H^{i,j}(C)) \in \Z\llbracket t\inv, q\rrbracket[t,q\inv] 
\, , \quad
\chi(C):= P(C)|_{t=-1}\in \Z\llbracket q\rrbracket[q\inv] 
\end{equation}
be the 
\emph{graded Poincar\'e series} (respectively \emph{graded Euler characteristic}) of $C$, where $\mathbb{K}$ is the fraction field of $\cring$.
\end{definition}

Note that the angle shaped condition guarantees that the 
Laurent series in \eqref{eq:poincare and euler} are well-defined.

\begin{lem}
	\label{lem:CDangle}
If $C,D\in \dgMod{\leftangle}{\cring}$, 
then $C \otimes D \in \dgMod{\leftangle}{\cring}$ 
and we have the identities
$P(C\otimes D)=P(C)P(D)$ and $\chi(C\otimes D)=\chi(C)\chi(D)$.
\end{lem}
\begin{proof}
We have that
$\supp(C\otimes D) = \supp(C) + \supp(D)$, 
hence Lemma \ref{lem:Oproperties} implies that 
$C \otimes D \in \dgMod{\leftangle}{\cring}$.
The rest is a straightforward application of the K\"unneth theorem. 
\end{proof}

We will need a generalization of Definition \ref{def:angle shaped over k} 
to certain one-sided twisted complexes over our dg categories $\CS(\Sigma,\pp;\Gamma)$.

\begin{definition}\label{def:angle shaped twisted complexes}
Fix a collection of objects $\Chi$ in a differential $(\Z \times \Z)$-graded category $\CS$.
We say that a one-sided twisted complex\footnote{Recall from \S \ref{ss:setup} 
that when forming $\Ch^{-}(\CS)$ in this setting, 
we formally adjoin shifts of objects, so $X$ is a well-defined object in $\Ch^{-}(\CS)$.} 
$X=\tw_\a(\bigoplus_{i\in I} t^{k_i}q^{l_i}X_i) \in \Ch^{-}(\CS)$ is 
\emph{angle-shaped with respect to $\Chi$}
provided
\begin{itemize}
\item $X_i \in \Chi$ for all $i \in I$, 
\item the set $\{i \in I \mid (k_i,l_i) = (k,l)\}$ is finite for each $(k,l) \in \Z \times \Z$, and
\item the set of shift exponents $(k_i,l_i)$ is in $O_{\leftangle}$.
\end{itemize}
\end{definition}

If the classes of the objects in $\Chi$ are $\Z[q,q^{-1}]$-linearly independent 
in the Grothendieck group $K_0(\CS)$
of the additive completion of $\CS$ (with shifts adjoined), 
then such $X$ have a well-defined Euler characteristic
\[
[X] :=\sum_{i\in I} (-1)^{k_i} q^{l_i} [X_i] 
	\in \Z\llbracket q \rrbracket [q\inv] \otimes_{\Z[q,q^{-1}]} K_0(\CS) \, .
\]
Note that every complex in the category $\dgModen{\leftangle}{\cring}$ 
is angle-shaped 
in the sense of Definition \ref{def:angle shaped twisted complexes}
with respect to the singleton collection $\{ \cring \}$, and 
the Euler characteristic $\chi(C)$ from \eqref{eq:poincare and euler} will 
coincide with $[\mathbb{K}\otimes_\cring C]$, where $\mathbb{K}$ is the 
fraction field of $\cring$.

\begin{definition}\label{def:angle shaped cats}
Let $(\Sigma, \Pi; \Gamma)$ be a seamed marked surface and let $\pp \subset A(\Pi)$ be standard.
A tangle $T \in \Tan(\Sigma,\Pi;\Gamma)$ is said to be \emph{minimal} 
if it does not contain any contractible circle components.
Let $\CS^{\leftangle}(\Sigma,\pp;\Gamma)$ denote the dg category consisting of 
complexes that are angle-shaped with respect to the collection of all minimal tangles.
\end{definition}

Since $\bn_m^n$ is equal to $\CS([0,1]^2,\pp;\emptyset)$ for appropriate $\pp$,
we can also consider $\CS^{\leftangle}(\bn_m^n)$.
Note in particular that there is a canonical quasi-equivalence 
$\CS^{\leftangle}(\bn_0^0) \cong \dgModen{\leftangle}{\cring}$.

Let $P_{n,n} \in \Ch^{-}(\bn_n^n)$ denote the complexes constructed in 
	\cite{CooperKrushkal} and \cite{Roz} that categorify the $n$-strand Jones--Wenzl projector.
These complexes admit an abstract characterization akin to Proposition~\ref{prop:Punique}. 
Namely, they are uniquely characterized (up to homotopy) by the conditions that 
	they are supported in non-positive homological degree, 
	contain exactly one copy of the identity tangle in bidegree $(0,0)$, 
	and that they annihilate tangles in $\bn_n^n$ of ``non-maximal through-degree''
	(i.e.~any tangle that does not have a shift-of-identity summand).
In particular, this implies that there is a homotopy equivalence 
	$P_{n,n} \hComp P_{n,n} \xrightarrow{\simeq} P_{n,n}$.

\begin{proposition}\label{prop:P-n-n angle shaped}
The projectors $P_{n,n}$ admit an angle-shaped model 
	(i.e.~$P_{n,n}$ is homotopy equivalent to a complex in $\Ch^{\leftangle}(\bn^n_n)$).
\end{proposition}
\begin{proof}
This is clear from the constructions in \cite{CooperKrushkal} and \cite{Roz}.
\end{proof}

\begin{proposition}\label{prop:P-2n-0 angle shaped}
The projectors $P_{2n,0}$ admit an angle-shaped model.
\end{proposition}
\begin{proof}
See \cite[Section 8]{rozansky2010categorification}.
\end{proof}

Our next result asserts that complexes built from angle shaped complexes 
using horizontal and vertical composition in $\Ch(\bn)$ are again angle-shaped.
These operations, and the pivotality of $\bn$, allow us to define dg functors
\begin{equation}
	\label{eq:SM}
F \colon \bn(D_1,\pp_1)\otimes \cdots \otimes \bn(D_r,\pp_r)\rightarrow \bn(D_0,\pp_0)
\end{equation}
in the following setup.
Here $D_0$ is a disk containing disjoint sub-disks $D_1,\ldots,D_r$ in its interior, 
and we obtain a family of functors as in \eqref{eq:SM}, 
one for each (crossingless) tangle in $D_0 \smallsetminus \cup_{i=1}^r D_i$ 
with boundary at the specified points $\cup_{i=0}^r \pp_i$.
The functor is given by ``plugging in'' tangles and cobordisms to the relevant input disks;
see e.g.~\cite[Section 5]{BN2}.

\begin{proposition}\label{prop:canopolisAngle}
The functors \eqref{eq:SM} preserve angle-shaped complexes, 
i.e.~they extend to functors
\[
F \colon \CS^{\leftangle}(D_1,\pp_1)\otimes \cdots \otimes \CS^{\leftangle}(D_r,\pp_r) \rightarrow \CS^{\leftangle}(D_0,\pp_0).
\]
\end{proposition}
\begin{proof}
This follows similarly to Lemma \ref{lem:CDangle} after noticing that
the set of possible $q$-shifts that occur when expressing 
$F(T_1,\ldots,T_r)$ as a sum of shifts of minimal tangles
(as the $T_i$ themselves range over all minimal tangles in $\Tan(D_i,\pp_i)$) 
is bounded from above and below.
\end{proof}

Combining Propositions \ref{prop:P-2n-0 angle shaped} 
	and \ref{prop:canopolisAngle} yields the following.
		
		\begin{proposition}
			\label{prop:cxangle}
			The complexes $\PMS_{\Sigma,\Pi;\Gamma}(T|V|S)$ from \eqref{eq:CflatcxW} are angle-shaped. 
		\end{proposition}
	
In particular, this proves Corollary~\ref{cor:fg} from the introduction.
	
\begin{proof}
From the description of $\PMS_{\Sigma,\Pi;\Gamma}(T|V|S)$ as in \eqref{eq:homcomplex}, 
	it is clear that 
\[	
\PMS_{\Sigma,\Pi;\Gamma}(T|V|S) = F(\tBar^{n_1}_{m_1},\ldots, \tBar^{n_G}_{m_G})
	\]
	for an appropriate functor $F$ as in \eqref{eq:SM}.
Each complex $\tBar^{n_i}_{m_i}\in \Ch^-(\bn^n_m\otimes \bn^n_m)$ 
	has an angle shaped model by Proposition \ref{prop:P-2n-0 angle shaped}, 
	and so Proposition \ref{prop:canopolisAngle} completes the proof.
\end{proof}

\subsection{A projector formula}

In the remaining sections, we will find a generating set of objects for
$\CS^{\leftangle}(\Sigma,\mathbf{p},\Gamma)$ 
which is orthogonal for the (symmetrized) $\Hom$-pairing. 
The main tool for showing orthogonality is Lemma \ref{lem:CKcup} below. 
Before stating it, we need a bit of setup.

\begin{definition}
	For $n \in \N$, consider the following objects in $\Ch^-(\bn^{2n}_{2n})$:
	\[
	I_{n} := 
	\begin{tikzpicture}[anchorbase]
	\draw[thick, double] (-.75,-.7) to  (-.75,.7);
	\draw[thick, double] (.75,-.7) to (.75,.7);
	\draw[thick, fill=white] (-1.25,-.25) rectangle (-0.25, .25);
	\draw[thick, fill=white] (1.25,-.25) rectangle (0.25, .25);
	\node at (-.75,0) {$P_{n,n}$};
	\node at (.75,0) {$P_{n,n}$};
	\node at (1,.6) {\scriptsize $n$};
	\node at (1,-.6) {\scriptsize $n$};
	\node at (-1,.6) {\scriptsize $n$};
	\node at (-1,-.6) {\scriptsize $n$};
	\end{tikzpicture}\ 
	\, , \quad 
	C_n := q^n \begin{tikzpicture}[anchorbase,xscale=.75]
	\draw[thick, double] (-.75,-.7)  \ur  (0,-.15)   \rd (.75,-.7);
	\draw[thick, double] (-.75,.7) \dr  (0,.15) \ru (.75,.7);
	\node at (1,.6) {\scriptsize $n$};
	\node at (1,-.6) {\scriptsize $n$};
	\node at (-1,.6) {\scriptsize $n$};
	\node at (-1,-.6) {\scriptsize $n$};
	\end{tikzpicture}
	\]
	\end{definition}

\begin{conv}\label{conv:P is algebra}
For the remainder of this section we will choose a model for $P_{n,n}$ for which
the aforementioned homotopy equivalence 
$\mu\colon P_{n,n} \hComp P_{n,n}\xrightarrow{\simeq} P_{n,n}$ is strictly associative 
with strict unit $\eta\colon \one_n\rightarrow P_{n,n}$ 
(this can always be done).
\end{conv}

\begin{lemma}\label{lemma:image of I}
For a complex $X\in \Ch^-(\bn^{2n}_{2n})$, the following are equivalent: 
\begin{enumerate}
\item $X$ is homotopy equivalent both to $X\simeq I_n\star X\star I_n$ and 
	to a complex of through-degree zero.
\item $X$ is homotopy equivalent to a one-sided twisted complex constructed from $I_n\star C_n\star I_n$.
\end{enumerate}
\end{lemma}
\begin{proof}
Clearly (2) implies (1).  For the converse, note that if $Y\in
\bn^{2n}_{2n}$ has through-degree zero then $I_{n} \star Y\star I_{n}$ is
homotopy equivalent to a direct sum of shifts of $I_n\star C_n\star I_n$.
Applying this to the chain objects of $X$ yields
(by homological perturbation) the implication from (1) to (2).
\end{proof}

\begin{lemma}\label{lemma:PC(I)} 
For $n\in \N$ there is a complex 
	$\PB_{C_n}(I_n)\in \Ch^-(\bn^{2n}_{2n})$ which is uniquely characterized 
	(up to homotopy equivalence) by the following two properties:
\begin{enumerate}
\item $\PB_{C_n}(I_n)$ is (homotopy equivalent to) a one-sided twisted complex constructed from $I_n\star C_n\star I_n$.
\item There exists a degree zero chain map $\e\colon \PB_{C_n}(I_n)\rightarrow I_n$ with $C_n \star \Cone(\e)\simeq 0 \simeq \Cone(\e)\star C_n$.
\end{enumerate}
\end{lemma}
\begin{proof}
Throughout this proof, we abbreviate by dropping the subscripts on $C_n$ and
$I_n$, and by omitting the symbols $\star$.  The proof of uniqueness follows
along standard lines; see e.g.~\cite[\S 3.3]{hogancamp2020constructing}. For
existence we may define $\PB_{C_n}(I_n)$ by an appropriate analogue of the bar
construction.  

To begin, 
Convention \ref{conv:P is algebra} guarantees that the homotopy equivalence 
$I I \xrightarrow{\simeq} I$ is strictly associative (and unital).
Next we form the following one-sided twisted complex
\begin{equation}
	\label{eq:PB}
\PB_C(I) := \left(\cdots \rightarrow t^{-2}ICICICI \rightarrow t\inv ICICI \rightarrow ICI\right) 
	= \tw_{\d}\left(\bigoplus_{m\geq 0} t^{-m} I(CI)^{m+1}\right).
\end{equation}
in which the twist $\d$ is a sum of $(-1)^i$ times the maps 
\[
\begin{tikzcd}
(IC)^{i} ICI (CI)^{m-i} \ar[r] & (IC)^{i} II (CI)^{m-i} \ar[r] & (IC)^{i} I (CI)^{m-i}
\end{tikzcd}
\]
induced from the iterated saddle map $C\rightarrow \one_{2n}$ 
and the equivalence $II\xrightarrow{\simeq} I$.
It is straightforward to verify that the
above is a well-defined complex satisfying (1) and (2).
\end{proof}

\begin{lemma}
	\label{lem:CKcup}
For $m,n\in\N_0$ we have homotopy equivalences
	\begin{equation}
		\label{eq:magic}
				\begin{tikzpicture}[anchorbase]
					\draw[thick, double] (-.75,-1.35) \pu  (-.75,1.35);
					\draw[thick, double] (.75,-1.35) \pu  (.75,1.35);
					\draw[thick, fill=white] (-1.25,.25) rectangle (1.25, -.25);
					\draw[thick, fill=white] (-1.25,.5) rectangle (-.25, 1);
					\draw[thick, fill=white] (.25,.5) rectangle (1.25, 1);
					\draw[thick, fill=white] (-1.25,-.5) rectangle (-.25, -1);
					\draw[thick, fill=white] (.25,-.5) rectangle (1.25, -1);
					\node at (0,0) {$P_{n+m,0}$};
					\node at (-.75,.75) {$P_{n,n}$};
					\node at (.75,.75) {$P_{m,m}$};
					\node at (-.75,-.75) {$P_{n,n}$};
					\node at (.75,-.75) {$P_{m,m}$};
					\node at (1,1.25) {\scriptsize $m$};
					\node at (1,-1.25) {\scriptsize $m$};
					\node at (-1,1.25) {\scriptsize $n$};
					\node at (-1,-1.25) {\scriptsize $n$};
					\end{tikzpicture}
						=
				q^{\frac{n+m}{2}}
				\begin{tikzpicture}[anchorbase]
					\draw[thick, double] (-.75,-1.35) to (-.75,-.75) \ur  (-.35,-.25) to (.35,-.25)  \rd (.75,-.75) to (.75,-1.35);
					\draw[thick, double] (-.75,1.35) to (-.75,.75) \dr  (-.35,.25) to (.35,.25) \ru (.75,.75) to (.75,1.35);
					\draw[thick, fill=white] (-1.25,.5) rectangle (-.25, 1);
					\draw[thick, fill=white] (.25,.5) rectangle (1.25, 1);
					\draw[thick, fill=white] (-1.25,-.5) rectangle (-.25, -1);
					\draw[thick, fill=white] (.25,-.5) rectangle (1.25, -1);
					\node at (-.75,.75) {$P_{n,n}$};
					\node at (.75,.75) {$P_{m,m}$};
					\node at (-.75,-.75) {$P_{n,n}$};
					\node at (.75,-.75) {$P_{m,m}$};
					\ujf{.25}{0}{}
					\djf{-.25}{0}{}
					\node at (1,1.25) {\scriptsize $m$};
					\node at (1,-1.25) {\scriptsize $m$};
					\node at (-1,1.25) {\scriptsize $n$};
					\node at (-1,-1.25) {\scriptsize $n$};
					\end{tikzpicture}
					\xrightarrow{\simeq}
					\begin{cases}
					\PB_{C_n}(I_n) & \text{if} \quad n=m\\
					0 & \text{if} \quad n\neq m
					\end{cases}.
				\end{equation}
In particular, $\PB_{C_n}(I_n)$ has a model in $\CS^{\leftangle}(I\times I, \pp^n\times\{0,1\})$.
\end{lemma}
\begin{proof}
The first equality in \eqref{eq:magic} is simply Definition \ref{def:bottom proj}. 
If $n\neq m$ then each complex of the form
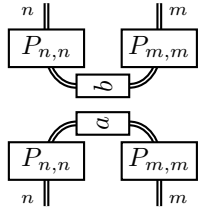
\[
\begin{tikzpicture}[anchorbase]
\draw[thick, double] (-.75,-1.35) to (-.75,-.75) \ur  (-.35,-.25) to (.35,-.25)  \rd (.75,-.75) to (.75,-1.35);
\draw[thick, double] (-.75,1.35) to (-.75,.75) \dr  (-.35,.25) to (.35,.25) \ru (.75,.75) to (.75,1.35);
\draw[thick, fill=white] (-1.25,.5) rectangle (-.25, 1);
\draw[thick, fill=white] (.25,.5) rectangle (1.25, 1);
\draw[thick, fill=white] (-1.25,-.5) rectangle (-.25, -1);
\draw[thick, fill=white] (.25,-.5) rectangle (1.25, -1);
\node at (-.75,.75) {$P_{n,n}$};
\node at (.75,.75) {$P_{m,m}$};
\node at (-.75,-.75) {$P_{n,n}$};
\node at (.75,-.75) {$P_{m,m}$};
\draw[thick, fill=white] (-.35,-.4) rectangle (.35,-.1);
\draw[thick, fill=white] (-.35,.4) rectangle (.35,.1);
\node[rotate = 90] at (0,-.25) {$a$};
\node[rotate = 90] at (0,.25) {$b$};
\node at (1,1.25) {\scriptsize $m$};
\node at (1,-1.25) {\scriptsize $m$};
\node at (-1,1.25) {\scriptsize $n$};
\node at (-1,-1.25) {\scriptsize $n$};
\end{tikzpicture}
\]
is contractible, 
since either $P_{m,m}$ or $P_{n,n}$ must annihilate both $a$ and $b$.
An application of homological perturbation thus establishes
the statement for $n\neq m$.

Now assume $n=m$. 
Lemma \ref{lemma:image of I} implies that $I_n\star P_{2n,0}\star I_n$ 
satisfies the first condition in Lemma \ref{lemma:PC(I)}.
For the second condition, 
Proposition \ref{prop:Punique} gives that the cone of the counit map
$\epsilon_{2n} \colon P_{2n,0}\rightarrow \one_{2n}$
annihilates any through-degree zero object. 
It follows that the cone of $I_n\star P_{2n,0}\star I_n \xrightarrow{I_n \hComp \epsilon_{2n} \hComp I_n} I_n$ 
annihilates $I_n \hComp C_n \hComp I_n$, 
so $I_n\star P_{2n,0}\star I_n \simeq P_{C_n}(I_n)$ by Lemma \ref{lemma:PC(I)}.

Finally, the last assertion follows from
Propositions \ref{prop:P-n-n angle shaped}, \ref{prop:P-2n-0 angle shaped}, and \ref{prop:canopolisAngle}.
\end{proof}

\begin{definition}\label{def:En}
Let $E_n := \Hom(\one_n,P_{n,n})$, 
with strictly unital associative algebra structure 
induced from that on $P_{n,n}$.
(See Convention \ref{conv:P is algebra}.)
\end{definition}

The following is a straightforward consequence of \eqref{eq:magic}.
\begin{corollary}
\label{cor:inverseunknot}
We have
\begin{equation}\label{eq:PCI and bar}
\begin{tikzpicture}[anchorbase]
\draw[thick, double] (-.75,-1.35) \pu  (-.75,1.35);
\draw[thick, double] (.75,-1.35) \pu  (.75,1.35);
\draw[thick, fill=white] (-1.25,.25) rectangle (1.25, -.25);
\draw[thick, fill=white] (-1.25,.5) rectangle (-.25, 1);
\draw[thick, fill=white] (.25,.5) rectangle (1.25, 1);
\draw[thick, fill=white] (-1.25,-.5) rectangle (-.25, -1);
\draw[thick, fill=white] (.25,-.5) rectangle (1.25, -1);
\node at (0,0) {$P_{2n,0}$};
\node at (-.75,.75) {$P_{n,n}$};
\node at (.75,.75) {$P_{n,n}$};
\node at (-.75,-.75) {$P_{n,n}$};
\node at (.75,-.75) {$P_{n,n}$};
\node at (1,1.25) {\scriptsize $n$};
\node at (1,-1.25) {\scriptsize $n$};
\node at (-1,1.25) {\scriptsize $n$};
\node at (-1,-1.25) {\scriptsize $n$};
\end{tikzpicture}
\ \simeq \ 
\Bar(E_n) \otimes_{E_n\otimes E_n} 
\left(
q^n
\begin{tikzpicture}[anchorbase,scale=.75]
\draw[thick, double] (-.75,-1.35) to (-.75,-.75) \ur  (0,-.25)  \rd (.75,-.75) to (.75,-1.35);
\draw[thick, double] (-.75,1.35) to (-.75,.75) \dr  (0,.25) \ru (.75,.75) to (.75,1.35);
\draw[thick, fill=white] (-1.25,.5) rectangle (-.25, 1);
\draw[thick, fill=white] (-1.25,-.5) rectangle (-.25, -1);
\node at (-.75,.75) {\small$P_{n,n}$};
\node at (-.75,-.75) {\small$P_{n,n}$};
\node at (-1,1.25) {\scriptsize $n$};
\node at (-1,-1.25) {\scriptsize $n$};
\end{tikzpicture} \
\right)
\end{equation}
where $\mathrm{Bar}(E_n)$ denotes the $2$-sided bar complex of $E_n$ 
and we view the final term as an $(E_n \otimes E_n)$-module by letting each 
factor of $E_n$ act on the two copies of $P_{n,n}$ via the multiplication map from 
Convention \ref{conv:P is algebra}.
\end{corollary}

\begin{proof}
The proof relies on the following two observations.  
The first is a simplification of $\PB_{C_n}(I_n)$ using idempotence of $P_{n,n}$.  
Precisely, abbreviate by writing $I_n':= P_{n,n}\boxtimes \one_n$, then the unit map
$\one_n \to P_{n,n}$ gives us a chain map $\theta \colon I_n'\rightarrow I_n$.
The mapping cone of $\theta$ is homotopy equivalent to a one-sided
twisted complex constructed from complexes of the form 
$P_{n,n}\boxtimes X$ where $X\in \bn^n_n$ has non-maximal through-degree. 
Such terms become contractible after composing with $C_n$ on
the left or right, so it follows that the following maps are homotopy
equivalences
\[
\theta\star \id_{C_n}\colon I_n'\star C_n\xrightarrow{\simeq} I_n\star C_n 
\, , \quad \id_{C_n}\star \theta\colon C_n\star I_n'\xrightarrow{\simeq} C_n\star I_n \, .
\]
Next, we observe that
\[
E_n := \Hom(\one_n,P_{n,n})\cong q^n \KhEval{\begin{tikzpicture}[anchorbase]
			\draw[thick, double] (.5,0) circle (.5);
			\draw[thick, fill=white] (-.5,-.25) rectangle (0.5, .25);
			\node at (0,0) {$P_{n,n}$};
		\end{tikzpicture}} \, .
\]
Hence, if we define $\PB_{C_n}(I_n')$ by replacing $I_n$ by $I_n'$ in \eqref{eq:PB}
then (essentially, by definition) we have an isomorphism
\[
\Bar(E_n)\otimes_{E_n\otimes E_n} I_n' \star C_n\star I_n' \cong \PB_{C_n}(I_n') \, .
\]
The result now follows using the 
homotopy equivalence $\PB_{C_n}(I_n')\xrightarrow{\simeq} \PB_{C_n}(I_n)$ induced from $\theta$.
\end{proof}

\begin{remark}\label{rem:inverseunknot}
Computing Euler characteristic of both sides of \eqref{eq:inverseunknot}, 
we find
\begin{equation}\label{eq:inverseunknot}
\left[
	\begin{tikzpicture}[anchorbase]
		\draw[thick, double] (-.75,-1.35) \pu  (-.75,1.35);
		\draw[thick, double] (.75,-1.35) \pu  (.75,1.35);
		\draw[thick, fill=white] (-1.25,.25) rectangle (1.25, -.25);
		\draw[thick, fill=white] (-1.25,.5) rectangle (-.25, 1);
		\draw[thick, fill=white] (.25,.5) rectangle (1.25, 1);
		\draw[thick, fill=white] (-1.25,-.5) rectangle (-.25, -1);
		\draw[thick, fill=white] (.25,-.5) rectangle (1.25, -1);
		\node at (0,0) {$P_{2n,0}$};
		\node at (-.75,.75) {$P_{n,n}$};
		\node at (.75,.75) {$P_{n,n}$};
		\node at (-.75,-.75) {$P_{n,n}$};
		\node at (.75,-.75) {$P_{n,n}$};
		\node at (1,1.25) {\scriptsize $n$};
		\node at (1,-1.25) {\scriptsize $n$};
		\node at (-1,1.25) {\scriptsize $n$};
		\node at (-1,-1.25) {\scriptsize $n$};
	\end{tikzpicture}
\right]
= \frac{1}{[n+1]} \left[\  
	\begin{tikzpicture}[anchorbase]
		\draw[thick, double] (-.75,-1.35) to (-.75,-.75) \ur  (0,-.25)  \rd (.75,-.75) to (.75,-1.35);
		\draw[thick, double] (-.75,1.35) to (-.75,.75) \dr  (0,.25) \ru (.75,.75) to (.75,1.35);
		\draw[thick, fill=white] (-1.25,.5) rectangle (-.25, 1);
		\draw[thick, fill=white] (-1.25,-.5) rectangle (-.25, -1);
		\node at (-.75,.75) {$P_{n,n}$};
		\node at (-.75,-.75) {$P_{n,n}$};
		\node at (-1,1.25) {\scriptsize $n$};
		\node at (-1,-1.25) {\scriptsize $n$};
	\end{tikzpicture} \ 
\right],
	\end{equation}
where $[k] = \frac{q^k-q^{-k}}{q-q\inv}$ denotes the usual quantum integer. This is reminiscent of well-known formulas in the graphical
	calculus for modular categories, namely the expression for wrapping a
	$0$-framed Kirby-colored unknot around a pair of strands labeled by two
	simple objects; see e.g.~\cite[Figure 3.10]{MR3617439},
	\cite[(3.1.21)]{MR1797619}, and \cite[Figure 1.4]{De_Renzi_2022}.
		\end{remark}

\subsection{Spin networks}
\label{ss:spin-networks}

We will consider graphs in certain seamed marked surfaces, 
and will adopt the following model for graphs. 
A graph will be regarded as a tuple $(V,H,v,\tau)$ where $V$ is the set of \emph{vertices}, 
$H$ is the set of \emph{half edges}, 
$v\colon H\rightarrow V$ is a map (we think of the half edge $h$ as attached to the vertex $v(h)$), 
and $\tau$ is an involution $\tau \colon H\rightarrow H$.
In this setup, an \emph{edge} is a $\tau$-orbit in $H$, denoted $[h]=\{h,\tau(h)\}$. 
If $h=\tau(h)$, then $[h]$ is called a \emph{boundary edge}; 
otherwise $[h]$ is called an \emph{internal edge}.

\begin{definition}
A \emph{triangulated surface} is a tessellated surface $(\Sigma,\Pi;\Gamma)$ 
for which each region in $\overline{X \smallsetminus \Gamma_X}$ 
(in the notation of Definition \ref{def:tessellatedS})
is a triangle.
\end{definition}

In this setting, we will refer to the corresponding tessellation as a \emph{triangulation} of $\Sigma$.
Note that triangulations necessarily correspond to minimal 
elements of the tessellation poset $T(\Sigma,\Pi)$
i.e.~those with a maximal number of seams.
Further, given a triangulation $(\Sigma,\Pi;\Gamma)$ and a (standard) subset $\pp \subset A(\Pi)$, 
Theorem \ref{thm:final3} provides a canonical quasi-equivalence 
$\sCat(\Sigma,\pp;\Gamma) \xrightarrow{\simeq} \sCat(\Sigma,\pp)$.
Nevertheless, we will continue to work with the category $\sCat(\Sigma,\pp;\Gamma)$ 
since our categorical spin networks will depend on the choice of $\Gamma$ via the following.

\begin{definition}
The \emph{dual graph} to the triangulation $(\Sigma,\Pi;\Gamma)$ is the graph $N_\Gamma$ wherein:
\begin{enumerate}
\item the vertex set is $V(N_\Gamma):=\Reg(\Sigma,\Pi;\Gamma)$ 
(note: each region corresponds to a triangle in $\overline{X \smallsetminus \Gamma_X}$).
\item the set of half edges $H(N_\Gamma)$ is the set $\cutP=\Pi \sqcup \gpm$.
\item the map from half-edges to vertices sends $\ma\in \cutP$ to the region $D$ it abuts, 
and the involution on half-edges fixes elements of $\Pi$ and sends $(\gamma,+)\leftrightarrow (\gamma,-)$.
\end{enumerate}
\end{definition}

In particular, the set of edges in $N_\Gamma$ is $\Pi \sqcup \Gamma$, 
with $\Pi$ being the set of boundary edges and $\Gamma$ being the set of internal edges.
We may view $N_\Gamma$ as embedded in either $\Sigma$ or $X$; 
see the left diagram in Figure \ref{fig:spin} for an example.
(This depicts $N_\Gamma$ in $X$; to visualize it in $\Sigma$, 
remove a neighborhood of each corner on the boundary.)

\begin{figure}[ht]
\[
\begin{tikzpicture}[anchorbase]
	\def\radius{5em}
	\draw[very thick]
	(0:\radius) to (60:\radius) to (120:\radius) to (180:\radius) to (240:\radius) to (300:\radius) to (360:\radius);
	\draw[very thick]
	(180:\radius) to (0:\radius);
	\draw[very thick]
	(180:\radius) to (60:\radius);
	\draw[very thick]
	(180:\radius) to (-60:\radius);
	\def\smallradius{4.33em}
	%
	\fill[red] (60:2em) circle[radius=2pt];
	\fill[red] (120:3.3em) circle[radius=2pt];
	\fill[red] (240:3.3em) circle[radius=2pt];
	\fill[red] (300:2em) circle[radius=2pt];
	\draw[very thick, red]
	(30:\smallradius) to (60:2em);
	\draw[very thick, red]
	(90:\smallradius) to (120:3.3em);
	\draw[very thick, red]
	(150:\smallradius) to (120:3.3em);
	\draw[very thick, red]
	(210:\smallradius) to (240:3.3em);
	\draw[very thick, red]
	(270:\smallradius) to (240:3.3em);
	\draw[very thick, red]
	(330:\smallradius) to (300:2em);
	\draw[very thick, red]
	(120:3.3em) to (60:2em) to (300:2em) to (240:3.3em) ;
\end{tikzpicture}
\qquad,\qquad
\begin{tikzpicture}[scale=.52,smallnodes,anchorbase]
	\def\radius{4}
	\def\smallradius{2}
	\draw[very thick]
	(270:\radius) to (30:\radius) to (150:\radius) to (270:\radius);
	\draw[thick,red, double] (0,2) to (0,1);
	\draw[thick,red,double] (100:1) to[out=270,in=30] (200:1);
	\node at (-.3,1.65) {$a$};
	\node at (-.6,.35) {$i$};
	\rotatebox{120}{
	\draw[thick,red, double] (0,2) to (0,1);
	\draw[thick,red,double] (100:1) to[out=270,in=30] (200:1);
	\draw[very thick,red,fill=white] (-0.6,1-.25) rectangle (0.6,1+.25);
	\node[rotate=-120] at (.3,1.65) {$b$};
	\node[rotate=-120] at (-.6,.35) {$j$};
	}
	\rotatebox{240}{
	\draw[thick,red, double] (0,2) to (0,1);
	\draw[thick,red,double] (100:1) to[out=270,in=30] (200:1);
	\draw[very thick,red,fill=white] (-0.6,1-.25) rectangle (0.6,1+.25);
	\node[rotate=-240] at (-.3,1.65) {$c$};
	\node[rotate=-240] at (-.6,.35) {$k$};
	}
	\draw[very thick,red,fill=white] (-0.6,1-.25) rectangle (0.6,1+.25);
\end{tikzpicture}
\]
\caption{\label{fig:spin} Left: the dual graph associated to a triangulation. Right: a local spin network.}
\end{figure}
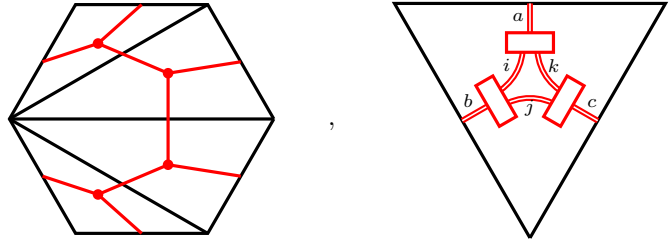

Now we turn our attention to spin networks.

\begin{definition}\label{def:spin-network}
Let $(\Sigma,\Pi;\Gamma)$ be a triangulated surface, with dual graph $N_\Gamma$.  
A function $\aa \colon \Pi\sqcup \gpm \rightarrow \N$ is called an \emph{admissible coloring} if
\begin{itemize}
\item $\aa$ assigns the same value to $(\gamma,+)$ and $(\gamma,-)$ for all $\gamma\in \Gamma$, 
	so $\aa$ may be regarded as a labelling of the edges of $N_\Gamma$.
\item for each $D \in \Reg(\Sigma,\Pi;\Gamma)$ with boundary arcs $\cutP|_D = \{\ma_1,\ma_2,\ma_3\}$ 
	and colors $a_i=\aa(\ma_i)$, 
	we have $a_1+a_2+a_3 \in 2\N$ and $a_1+a_2\geq a_3$, as well as all permutations thereof.
\end{itemize}
A \emph{spin network} in $(\Sigma,\Pi)$ is a pair $(N_\Gamma,\aa)$, 
	where $N_\Gamma$ is the dual graph of a triangulation of $\Sigma$ and $\aa$ is an admissible coloring.
\end{definition}

We will visualize a spin network as the diagram 
obtained by gluing together the local pictures as shown\footnote{Therein, 
$i,j,k \in \N$ are the unique solutions to the equations $i+k=a$, $i+j=b$, and $j+k=c$, 
which exist if and only if $(a,b,c)$ satisfy the second condition in Definition \ref{def:spin-network}.} 
on the right side of Figure~\ref{fig:spin}.
Interpreting each box with $a$ incoming and outgoing strands as the categorified projector 
$P_{a,a}$ will yield an object $\nabla_{\Gamma,\aa} \in \CS^{\leftangle}(\Sigma,\pp;\Gamma)$.  
The following makes this construction precise.

\begin{definition}\label{def:spin-network-complexes}
Let $(N_\Gamma,\aa)$ be a spin network associated to the triangulation $(\Sigma,\Pi;\Gamma)$.
Let $\pp(\aa)\subset A(\Pi)$ denote the standard (as in Definition \ref{def:tangle in SMS}) 
	set of boundary points for which $|\pp(\aa)\cap\ma|=\aa(\ma)$ for all $\ma\in \Pi$.  
Let $F_{\Gamma,\aa}\colon \bigotimes_{\ma\in \cutP} \bn^{\aa(\ma)}_{\aa(\ma)}\rightarrow \CS(\Sigma,\pp(\aa);\Gamma)$ 
	be the dg functor defined locally by
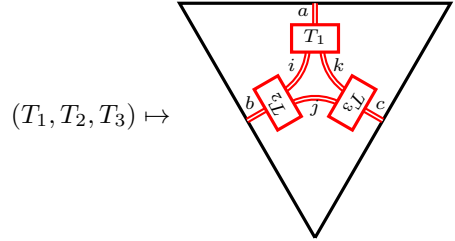
\[
(T_1,T_2,T_3)\mapsto 
\begin{tikzpicture}[scale=.52,smallnodes,anchorbase]
	\def\radius{4}
	\def\smallradius{2}
	\draw[very thick]
	(270:\radius) to (30:\radius) to (150:\radius) to (270:\radius);
	\draw[thick,red, double] (0,2) to (0,1);
	\draw[thick,red,double] (100:1) to[out=270,in=30] (200:1);
	\node at (-.3,1.75) {$a$};
	\node at (-.6,.35) {$i$};
	\rotatebox{120}{
	\draw[thick,red, double] (0,2) to (0,1);
	\draw[thick,red,double] (100:1) to[out=270,in=30] (200:1);
	\draw[very thick,red,fill=white] (-0.6,1-.25) rectangle (0.6,1+.45);
	\node at (0,1.1) {\scriptsize $T_2$};
	\node[rotate=-120] at (.3,1.75) {$b$};
	\node[rotate=-120] at (-.6,.35) {$j$};
	}
	\rotatebox{240}{
	\draw[thick,red, double] (0,2) to (0,1);
	\draw[thick,red,double] (100:1) to[out=270,in=30] (200:1);
		\draw[very thick,red,fill=white] (-0.6,1-.25) rectangle (0.6,1+.45);
	\node at (0,1.1) {\scriptsize $T_3$};
	\node[rotate=-240] at (-.3,1.75) {$c$};
	\node[rotate=-240] at (-.6,.35) {$k$};
	}
	\draw[very thick,red,fill=white] (-0.6,1-.25) rectangle (0.6,1+.45);
	\node at (0,1.1) {\scriptsize $T_1$};
\end{tikzpicture}
\]
Extending this functor to angle-shaped complexes,
the \emph{simple object} associated to $(N_\Gamma,\aa)$ is defined by
\[
L_{\Gamma,\aa} := F_{\Gamma,\aa}\left(\begin{cases}
\one_{\aa(\ma)} & \text{ if $\ma\in \gpm$ is an internal half-edge}\\
P_{\aa(\ma),\aa(\ma)} & \text{ if $\ma\in \Pi$ is a boundary half-edge}
\end{cases}\right) 
\in \Ch^{-}\!\big(\sCat(\Sigma,\pp(\aa); \Gamma) \big) \, .
\]
The \emph{costandard object} associated to $(N_\Gamma,\aa)$ is defined by
\[
\nabla_{\Gamma,\aa}:= F_{\Gamma,\aa}\left(P_{\aa(\ma),\aa(\ma)} \text{ for all } \ma\in \cutP \right)
\in \Ch^{-}\!\big(\sCat(\Sigma,\pp(\aa); \Gamma) \big) \, .
\]
\end{definition}

\begin{remark}
In the situation where $\aa(\ma) = 0,1$ for $\ma\in \Pi$, 
the ``projector'' $P_{\aa(\ma),\aa(\ma)}$ appearing in the definition of $L_{\Gamma,\aa}$ 
is either $\one_1$ or $\one_0$.
\end{remark}

\begin{remark}
	\label{rem:standardOb}
The terminology of simple and costandard objects 
is inspired by the notion of highest weight categories.  
Aspects of that theory are visible in our considerations.
To further the analogy, one would also want to consider the family of \emph{standard objects}.
In our situation, these would be the complexes $\Delta_{\Gamma,\aa}$ obtained by evaluating 
$F_{\Gamma,\aa}$ on the family of \emph{dual} projectors $P_{a,a}^\vee$, 
which lie in the category $\Ch^+(\bn_a^a)$ of bounded-below complexes over $\bn_a^a$.  
Since the $\Hom$-complexes in $\CS(\Sigma,\pp(\aa);\Gamma)$ are bounded above, 
the formation of infinite one-sided twisted complexes with cohomological shifts shifts tending to $+\infty$ 
should be treated with caution.
In this paper we will not consider these $\Delta_{\Gamma,\aa}$ any further, 
and any connection with highest weight categories will be treated as purely inspiration.
\end{remark}

For fixed $N_\Gamma$ and $\pp\subset A(\Pi)$, 
the set of spin networks $(N_\Gamma,\aa)$ with $\pp(\aa)=\pp$ form a partially ordered set with 
$\aa \leq \bb$ if $\aa(\ma)\leq \bb(\ma)$ for all $\ma\in \cutP$. 
We write $\aa<\bb$ if $\aa\leq \bb$ and $\aa\neq \bb$.  

\begin{definition}\label{def:partial-order-on-networks}
Let $(\Sigma, \Pi; \Gamma)$ be a triangulated surface and let $\pp \subset A(\Pi)$ be standard.
Let $\CS^{\leftangle}_{\leq \aa}(\Sigma,\pp;\Gamma)$ be the full subcategory of 
$\Ch^{-}\!\big(\sCat(\Sigma,\pp(\aa); \Gamma) \big)$
consisting of complexes that are 
angle-shaped with respect to the collection $\{L_{\Gamma,\bb}\}_{\bb\leq \aa}$.  
Define $\CS^{\leftangle}_{< \aa}(\Sigma,\pp;\Gamma)$ similarly.
\end{definition}

\begin{lemma}
	\label{lem:spin-basis-change}
There is a canonical chain map $L_{\Gamma,\aa}\rightarrow \nabla_{\Gamma,\aa}$ 
whose mapping cone is homotopy equivalent to an object of $\CS^{\leftangle}_{< \aa}(\Sigma,\pp(\aa);\Gamma)$. 
\end{lemma}
\begin{proof}
The chain map is induced by the unit map $\one_a \to P_{a,a}$.
The cone of this map is angle-shaped, with respect to the collection of 
minimal tangles in $\bn_a^a$ of non-maximal through-degree.
This, 
together with an argument analogous to 
the proof of Proposition \ref{prop:canopolisAngle},
implies the second statement.
\end{proof}

For ease of exposition, 
we now restrict to the following setup; 
see Remark \ref{rem:genspin} below for comments on the general case.

\begin{definition}\label{def:admissible-triangulations}
Let $\Sigma$ be a compact surface with a finite set of points $\pp\subset \partial\Sigma$.  A triangulation 
$(\Sigma,\Pi;\Gamma)$ will be called \emph{weakly (resp.~strictly) $\pp$-admissible} if each $\ma\in \Pi$ 
contains at most one (resp.~exactly one) point of $\pp$, and this point is standard.
\end{definition}

\begin{lemma}
	\label{lem:gen}
Let $\Sigma$ be a surface with a finite set of points $\pp\subset \partial \Sigma$.
If $(\Sigma,\Pi;\Gamma)$ is a fixed, weakly $\pp$-admissible triangulation, 
then either of the families of objects 
$\{L_{\Gamma,\aa}\}_{\pp(\aa)=\pp}$ 
or $\{\nabla_{\Gamma,\aa}\}_{\pp(\aa)=\pp}$ 
generate the category $\CS^{\leftangle}(\Sigma,\pp;\Gamma)$ 
with respect to forming angle-shaped one-sided twisted complexes.
\end{lemma}
\begin{proof}
We first claim that each $\nabla_{\Gamma,\aa}$ is angle-shaped, 
i.e.~is homotopy equivalent to an object of $\CS^{\leftangle}(\Sigma,\pp(\aa);\Gamma)$. 
This follows from the fact that each $P_{n,n}$ has an angle-shaped model, 
together with a generalization of Proposition \ref{prop:canopolisAngle} 
to the functors $F_{\Gamma,\aa}$ from Definition \ref{def:spin-network-complexes}.

Next, by a standard argument using the ``circle-removal'' isomorphism in $\BN$, 
every object of $\CS(\Sigma,\pp;\Gamma)$ 
is homotopy equivalent 
to a direct sum of grading shifts of objects $L_{\Gamma,\pp}$.
Hence the latter objects generate $\CS^{\leftangle}(\Sigma,\pp;\Gamma)$.
A straightforward induction using Lemma~\ref{lem:spin-basis-change}
allows us to write $L_{\Gamma,\aa}$ as an angle-shaped one-sided twisted complex
constructed from $\nabla_{\Gamma,\bb}$ with $\bb\leq \aa$, 
hence the costandard objects generate $\CS^{\leftangle}(\Sigma,\pp;\Gamma)$.
\end{proof}

\begin{example}
Consider the disk with 6 marked points on its boundary. 
To write down a set of generators of the associated dg category, 
we must first choose a triangulation of the 6-gon, e.g.~as on the left side of Figure~\ref{fig:spin}. 
Then, the generating objects are parametrized by spin
networks whose underlying graph is the dual graph (depicted in red), 
and whose boundary spins are all $1$.
\end{example}

\subsection{Orthogonality of spin networks}
\label{ss:orthogonality}
To conclude, we formulate and prove results on the 
orthogonality and ``norm'' of the complexes $\nabla_{\Gamma,\aa}$. 
We first introduce the \emph{pairing} involved in such a statement.

\begin{definition}\label{def:symmetrized hom pairing} 
For $(\Sigma,\mathbf{p};\Gamma)$ as in Section~\ref{s:surfaceinvt} we recall the isomorphism
			$\sCat(\Sigma,\mathbf{p};\Gamma)\xrightarrow{\sigma}
			\sCat(\Sigma,\mathbf{p};\Gamma)^{\op}$ from Proposition \ref{prop:equivop} 
	and define the \emph{symmetrized hom pairing} on $\sCat(\Sigma,\mathbf{p};\Gamma)$ to be the dg functor:
	\[
		\sCat(\Sigma,\mathbf{p};\Gamma) \otimes \sCat(\Sigma,\mathbf{p};\Gamma) 
			\xrightarrow{\id \otimes \sigma} 
		\sCat(\Sigma,\mathbf{p};\Gamma) \otimes \sCat(\Sigma,\mathbf{p};\Gamma)^{\op} 
			\xrightarrow{\Hom} \dgModen{\Z \times \Z}{\cring}
			\]
			which sends $X,Y\in \sCat(\Sigma,\mathbf{p};\Gamma)$ to
			$\ip{X,Y}:=\PMS(X|\one|\sigma(Y))$.
			\end{definition}

Proposition \ref{prop:equivop} now implies that we have natural isomorphism $\ip{X,Y}\cong \ip{Y,X}$ in $\dgModen{\leftangle}{\cring}$ and, 
thus, the identities $P\big( \ip{X,Y} \big)=P\big( \ip{Y,X} \big)$ and $\chi\big( \ip{X,Y} \big)=\chi\big( \ip{Y,X} \big)$
of Poincar\'{e} series and Euler characteristics.

\begin{lemma} The symmetrized $\Hom$-pairing from
	Definition~\ref{def:symmetrized hom pairing} extends to a dg functor
	\begin{equation}
		\label{eq:symhom}
	\ip{-,-} \colon \CS^{\leftangle}(\Sigma,\mathbf{p},\Gamma) \otimes \CS^{\leftangle}(\Sigma,\mathbf{p},\Gamma) 
		\to \dgModen{\leftangle}{\cring} \, .
		\end{equation}
	\end{lemma}
	\begin{proof}
	Lemma~\ref{lem:Oproperties} and Proposition~\ref{prop:canopolisAngle} 
	enable the extension of the symmetrized $\Hom$-pairing 
	to angle-shaped one-sided twisted complexes.
	\end{proof}

\begin{lemma}
	\label{lem:orth}
Retain the setup of Lemma \ref{lem:gen}.
Then for the symmetrized $\Hom$-pairing from \eqref{eq:symhom} we have:
\[\langle \nabla_{\Gamma,\aa} , \nabla_{\Gamma,\bb}\rangle \simeq 0\]
for all $\pp$-admissible colorings $\aa,\bb$ with $\pp(\aa)=\pp=\pp(\bb)$ and $\aa\neq \bb$.
\end{lemma}
\begin{proof}
Since the colorings agree on the boundary but are different, 
there must be an internal edge of $N_{\Gamma}$
(i.e.~a seam $\gamma \in \Gamma$) on which they differ. 
Using the graphical model from \S \ref{sec:graphicalmodel} for the symmetrized $\Hom$-pairing, 
and focusing on the seam $\gamma$ as in \eqref{constr:coarsening}, 
the vanishing now follows directly from \eqref{eq:magic}. 
\end{proof}

\begin{thm}
	\label{thm:pairing}
Retain the setup of Lemma \ref{lem:gen}. 
The objects $\nabla_{\Gamma,\aa}$ as $\aa$ ranges over all colorings with $\pp(\aa)=\pp$ 
generate $\CS^{\leftangle}(\Sigma,\mathbf{p},\Gamma)$ with respect to the
formation of angle-shaped one-sided twisted complexes 
and are pairwise orthogonal with respect to the symmetrized $\Hom$-pairing. 
Further, on the level of Euler characteristics, the self-pairing satisfies:
\begin{equation}\label{eq:self-pairing}
	\chi\left(\ip{\nabla_{\Gamma,\aa},\nabla_{\Gamma,\aa}}\right) =
\frac{\prod_{v \in V(N_\Gamma)} \chi\left(\Theta(a_v,b_v,c_v)\right)}{\prod_{e \in \Gamma} \chi\left(\Omega(a_e)\right)} \, .
\end{equation}
Here $(a_v,b_v,c_v)$ are the colors of the edges incident to the vertex $v \in V(N_\gamma)$, 
$a_e$ is the color of the internal edge $e \in \Gamma$, and we are using the shorthand:
\[
\Theta(a,b,c):=
\KhEval{ 
	\begin{tikzpicture}[anchorbase]
		\draw[thick,double] 
		(-1.75,.25) \ur (0,.85) \rd (1.75,.25)
		(-1.75,-.25) \dr (0,-.85) \ru (1.75,-.25)
		(-1.25,.25) \ur (-.75,.5) \rd (-.25,.25)
		(-1.25,-.25) \dr (-.75,-.5) \ru (-.25,-.25)
		(.25,.25) \ur (.75,.5) \rd (1.25,.25)
		(.25,-.25) \dr (.75,-.5) \ru (1.25,-.25)
		;
		\draw[thick, fill=white] (-.5,-.25) rectangle (0.5, .25);
		\draw[thick, fill=white] (-2,-.25) rectangle (-1, .25);
		\draw[thick, fill=white] (1,-.25) rectangle (2, .25);
		\node at (-1.5,0) {$P_{a,a}$};
		\node at (0,0) {$P_{b,b}$};
		\node at (1.5,0) {$P_{c,c}$};
	\end{tikzpicture}
} \, , \quad
\Omega(a):=
 \KhEval{ 
	\begin{tikzpicture}[anchorbase]
		\draw[thick, double] (.5,0) circle (.5);
		\draw[thick, fill=white] (-.5,-.25) rectangle (0.5, .25);
		\node at (0,0) {$P_{a,a}$};
	\end{tikzpicture}
}.
\]

\end{thm}
\begin{proof}
The first two statements summarize Lemma~\ref{lem:gen} and Lemma~\ref{lem:orth}.
The statement about the Euler characteristic of the self-pairing follows from
the same considerations as the proof of Lemma~\ref{lem:orth}, except that we use
the $m=n$ case of the homotopy equivalence \eqref{eq:magic} at each seam.
Consequently, Remark~\ref{rem:inverseunknot} shows that
we can remove the bottom projector corresponding to
each internal edge $e \in \Gamma$ of $N_{\Gamma}$ 
of the spin network and replace it with a cup-cap,
provided we compensate with a factor of $\frac{1}{[a_e+1]}$.
What remains is the Euler characteristic of a 
tensor product of evaluations of projector-colored theta graphs, 
one for each triangle in the triangulation or, equivalently, 
one for each vertex $v \in V(N_{\Gamma})$.
The result then follows since
$[\Omega(a)]$ is exactly the value of the closure of the $a$-strand Jones--Wenzl projector, 
which equals $[a+1]$.
\end{proof}

\begin{remark}
The pairing formula \eqref{eq:self-pairing} is reminiscent of various
hermitian pairings on Kauffman bracket skein algebras and
Turaev--Viro/Reshetikhin--Tuarev invariants of surfaces, 
which admit bases of spin networks, 
see e.g.~\cite[Theorem 4.11]{BHMV} and \cite[\S 2.2]{MR2387432}. 
We hence view 
Theorem \ref{thm:pairing} as justification for viewing 
the present work as a categorified analogue of these theories 
(for the surfaces we consider).

More precisely, such \emph{hermitian} pairings would coincide with 
the decategorification of the usual (\emph{un-symmetrized}) $\Hom$-pairing 
(which is sesquilinear with respect to grading shifts)
after tensoring with $\C$, and specializing $q$ to a root of unity. 
However, it is non-trivial to extend the ordinary $\Hom$-pairing on
$\sCat(\Sigma,\mathbf{p};\Gamma)$ to sufficiently infinite complexes to
allow for its evaluation (in \emph{both} arguments) on spin networks built from the projectors $P_{a,a}$.
In fact, we do not expect the analogous orthogonality statement in this setting. 
Instead, there is a second type of spin networks that are modeled on 
the dual projectors $P_{a,a}^{\vee}$ mentioned above in Remark \ref{rem:standardOb}.
These projectors generate a suitable category $\CS^{\rightangle}(\Sigma,\mathbf{p};\Gamma)$ 
of bounded below angle-shaped complexes, but which decategorify to the same
elements under the procedure outlined above. 
Theorem~\ref{thm:pairing} then expresses that the $\Hom$-pairing extends to a perfect pairing 
\[
\Hom\colon \CS^{\leftangle}(\Sigma,\mathbf{p};\Gamma) \times \CS^{\rightangle}(\Sigma,\mathbf{p};\Gamma) 
	\to \dgModen{\leftangle}{\cring}
\] 
with spin networks and dual spin networks forming respective 
generating sets that are dual with respect to the pairing.
\end{remark}

\begin{remark}
	\label{rem:genspin}
The results of Lemma \ref{lem:gen} and Theorem \ref{thm:pairing} hold 
without restricting to $\pp$-admissible triangulations, 
provided we work with an appropriate substitute for the category $\CS^{\leftangle}(\Sigma,\mathbf{p},\Gamma)$. 
In more detail, given an arbitrary triangulated surface $(\Sigma,\Pi;\Gamma)$ and an admissible coloring $\aa$, 
we can consider the dg category wherein the $\Hom$-space between (appropriate) tangles $T,S$ is given by 
the complex $\PMS_{\Sigma,\Pi;\Gamma}(T| (P_{\aa(\ma),\aa(\ma)})_{\ma \in \Pi} |S)$. 
The composition map now also involves the multiplication map $P_{a,a} \hComp P_{a,a} \to P_{a,a}$.
The aforementioned results now hold in the category of angle-shaped complexes 
(e.g.~with respect to the $L_{\Gamma,\aa}$) over this category.
\end{remark}

\bibliographystyle{plain}
\bibliography{pw}

\begin{thebibliography}{10}

\bibitem{MR3692883}
Rina Anno and Timothy Logvinenko.
\newblock Spherical {DG}-functors.
\newblock {\em J. Eur. Math. Soc. (JEMS)}, 19(9):2577--2656, 2017.
\newblock \arxiv{1309.5035}.

\bibitem{MR2370224}
Marta Asaeda and Charles Frohman.
\newblock A note on the {B}ar-{N}atan skein module.
\newblock {\em Internat. J. Math.}, 18(10):1225--1243, 2007.
\newblock \arxiv{math/0602262}.

\bibitem{MR2113902}
Marta~M. Asaeda, J\'{o}zef~H. Przytycki, and Adam~S. Sikora.
\newblock Categorification of the {K}auffman bracket skein module of
  {$I$}-bundles over surfaces.
\newblock {\em Algebr. Geom. Topol.}, 4:1177--1210, 2004.
\newblock \arxiv{math/0409414}.

\bibitem{MR2827825}
Denis Auroux.
\newblock Fukaya categories and bordered {H}eegaard-{F}loer homology.
\newblock In {\em Proceedings of the {I}nternational {C}ongress of
  {M}athematicians. {V}olume {II}}, pages 917--941. Hindustan Book Agency, New
  Delhi, 2010.

\bibitem{MR1355899}
John~C. Baez and James Dolan.
\newblock Higher-dimensional algebra and topological quantum field theory.
\newblock {\em J. Math. Phys.}, 36(11):6073--6105, 1995.
\newblock \arxiv{q-alg/9503002}.

\bibitem{MR1797619}
Bojko Bakalov and Alexander Kirillov, Jr.
\newblock {\em Lectures on tensor categories and modular functors}, volume~21
  of {\em University Lecture Series}.
\newblock American Mathematical Society, Providence, RI, 2001.

\bibitem{BN2}
Dror Bar-Natan.
\newblock Khovanov's homology for tangles and cobordisms.
\newblock {\em Geom. Topol.}, 9:1443--1499, 2005.
\newblock \arxiv{math.GT/0410495}.

\bibitem{barrett2018gray}
John~W. Barrett, Catherine Meusburger, and Gregor Schaumann.
\newblock Gray categories with duals and their diagrams, 2018.
\newblock \arxiv{1211.0529}.

\bibitem{BHLZ}
Anna Beliakova, Kazuo Habiro, Aaron~D. Lauda, and Marko {\v Z}ivkovi{\'c}.
\newblock Trace decategorification of categorified quantum {$\germ{sl}_2$}.
\newblock {\em Math. Ann.}, 367(1-2):397--440, 2017.
\newblock \arxiv{1404.1806}.

\bibitem{BHMV}
C.~Blanchet, N.~Habegger, G.~Masbaum, and P.~Vogel.
\newblock Topological quantum field theories derived from the {K}auffman
  bracket.
\newblock {\em Topology}, 34(4):883--927, 1995.

\bibitem{MR2475122}
Jeffrey Boerner.
\newblock A homology theory for framed links in {I}-bundles using embedded
  surfaces.
\newblock {\em Topology Appl.}, 156(2):375--391, 2008.
\newblock \arxiv{0810.5566}.

\bibitem{HToD}
W.~Chach\'{o}lski and J.~Scherer.
\newblock Homotopy theory of diagrams.
\newblock {\em Mem. Amer. Math. Soc.}, 155(736), 2002.

\bibitem{CooperKrushkal}
B.~Cooper and V.~Krushkal.
\newblock Categorification of the {J}ones--{W}enzl projectors.
\newblock {\em Quantum Topol.}, 3(2):139--180, 2012.

\bibitem{Cooper_2015}
Benjamin Cooper and Matt Hogancamp.
\newblock An exceptional collection for {K}hovanov homology.
\newblock {\em Algebraic \& Geometric Topology}, 15(5):2659--2707, November
  2015.

\bibitem{costello2004ainfinity}
Kevin Costello.
\newblock The {A}-infinity operad and the moduli space of curves, 2004.
\newblock \arxiv{math/0402015}.

\bibitem{MR2298823}
Kevin Costello.
\newblock Topological conformal field theories and {C}alabi-{Y}au categories.
\newblock {\em Adv. Math.}, 210(1):165--214, 2007.

\bibitem{MR1295461}
Louis Crane and Igor~B. Frenkel.
\newblock Four-dimensional topological quantum field theory, {H}opf categories,
  and the canonical bases.
\newblock volume~35, pages 5136--5154. 1994.
\newblock \arxiv{hep-th/9405183}.

\bibitem{CraneYetter}
Louis Crane and David Yetter.
\newblock A categorical construction of {$4$}d topological quantum field
  theories.
\newblock In {\em Quantum topology}, volume~3 of {\em Ser. Knots Everything},
  pages 120--130. World Sci. Publ., River Edge, NJ, 1993.

\bibitem{De_Renzi_2022}
Marco De~Renzi.
\newblock Non-semisimple extended topological quantum field theories.
\newblock {\em Memoirs of the American Mathematical Society}, 277(1364), May
  2022.

\bibitem{MR4045965}
Christopher~L. Douglas, Robert Lipshitz, and Ciprian Manolescu.
\newblock Cornered {H}eegaard {F}loer homology.
\newblock {\em Mem. Amer. Math. Soc.}, 262(1266):v+124, 2019.

\bibitem{MR3180613}
Christopher~L. Douglas and Ciprian Manolescu.
\newblock On the algebra of cornered {F}loer homology.
\newblock {\em J. Topol.}, 7(1):1--68, 2014.

\bibitem{douglas2018fusion}
Christopher~L. Douglas and David~J. Reutter.
\newblock Fusion 2-categories and a state-sum invariant for 4-manifolds, 2018.
\newblock \arxiv{1812.11933}.

\bibitem{DyKa}
T.~Dyckerhoff and M.~Kapranov.
\newblock Triangulated surfaces in triangulated categories.
\newblock {\em J. Eur. Math. Soc. (JEMS)}, 20(6):1473--1524, 2018.
\newblock \arxiv{1306.2545}.

\bibitem{EML}
S.~Eilenberg and S.~Mac Lane.
\newblock On the groups $h(\pi,n)$, {I}.
\newblock {\em Ann. of Math. (2)}, 58(1), 1953.

\bibitem{thesis-Fadali}
Lyla Fadali.
\newblock {\em Bar-Natan Skein Modules in Black and White}.
\newblock PhD thesis, UC San Diego, 2016.

\bibitem{FST}
S.~Fomin, M.~Shapiro, and D.~Thurston.
\newblock Cluster algebras and triangulated surfaces. {I}. {C}luster complexes.
\newblock {\em Acta Math.}, 201(1):83--146, 2008.

\bibitem{MR2994995}
Daniel~S. Freed.
\newblock The cobordism hypothesis.
\newblock {\em Bull. Amer. Math. Soc. (N.S.)}, 50(1):57--92, 2013.
\newblock \arxiv{1210.5100}.

\bibitem{MR4696498}
J\"{u}rgen Fuchs, Christoph Schweigert, and Yang Yang.
\newblock {\em String-net construction of {RCFT} correlators}, volume~45 of
  {\em SpringerBriefs in Mathematical Physics}.
\newblock Springer, Cham, [2022] \copyright 2022.

\bibitem{GNR}
E.~Gorsky, A.~Negut, and J.~Rasmussen.
\newblock Flag {H}ilbert schemes, colored projectors and {K}hovanov-{R}ozansky
  homology.
\newblock {\em Adv. Math.}, 2021.
\newblock \arxiv{1608.07308}.

\bibitem{2002.06110}
Eugene Gorsky, Matthew Hogancamp, and Paul Wedrich.
\newblock Derived traces of {S}oergel categories.
\newblock {\em Int. Math. Res. Not. IMRN}, 2021.
\newblock \arxiv{2002.06110}.

\bibitem{2019arXiv190404481G}
Eugene Gorsky and Paul Wedrich.
\newblock Evaluations of annular {K}hovanov-{R}ozansky homology.
\newblock {\em Math. Z.}, 303(1):Paper No. 25, 57, 2023.
\newblock \arxiv{1904.04481}.

\bibitem{HKK}
F.~Haiden, L.~Katzarkov, and M.~Kontsevich.
\newblock Flat surfaces and stability structures.
\newblock {\em Publ. Math. Inst. Hautes \'{E}tudes Sci.}, 126:247--318, 2017.
\newblock \arxiv{1409.8611}.

\bibitem{MR830043}
John~L. Harer.
\newblock The virtual cohomological dimension of the mapping class group of an
  orientable surface.
\newblock {\em Invent. Math.}, 84(1):157--176, 1986.

\bibitem{HatcherT}
A.~Hatcher.
\newblock On triangulations of surfaces.
\newblock {\em Topology Appl.}, 40(2):189--194, 1991.
\newblock
  \href{https://pi.math.cornell.edu/~hatcher/Papers/TriangSurf.pdf}{https://pi.math.cornell.edu/~hatcher/Papers/TriangSurf.pdf}.

\bibitem{hogancamp2020constructing}
Matthew Hogancamp.
\newblock Constructing categorical idempotents, 2020.
\newblock \arxiv{2002.08905}.

\bibitem{HRW1}
Matthew {Hogancamp}, D.~E.~V. {Rose}, and Paul {Wedrich}.
\newblock A skein relation for singular {S}oergel bimodules, 2021.
\newblock \arxiv{2107.08117}.

\bibitem{HRW3}
Matthew {Hogancamp}, D.~E.~V. {Rose}, and Paul {Wedrich}.
\newblock A {K}irby color for {K}hovanov homology, 2022.
\newblock \arxiv{2210.05640}, to appear in \textit{J. Eur. Math. Soc.}

\bibitem{MR2503518}
Uwe Kaiser.
\newblock Frobenius algebras and skein modules of surfaces in 3-manifolds.
\newblock In {\em Algebraic topology---old and new}, volume~85 of {\em Banach
  Center Publ.}, pages 59--81. Polish Acad. Sci. Inst. Math., Warsaw, 2009.

\bibitem{MR3220480}
Uwe Kaiser.
\newblock On constructions of generalized skein modules.
\newblock In {\em Knots in {P}oland. {III}. {P}art 1}, volume 100 of {\em
  Banach Center Publ.}, pages 153--172. Polish Acad. Sci. Inst. Math., Warsaw,
  2014.

\bibitem{kaiser2022barnatan}
Uwe Kaiser.
\newblock Bar-{N}atan theory and tunneling between incompressible surfaces in
  3-manifolds, 2022.
\newblock \arxiv{2211.01937}.

\bibitem{Kho}
M.~Khovanov.
\newblock A categorification of the {J}ones polynomial.
\newblock {\em Duke Math. J.}, 101(3):359--426, 2000.
\newblock \arxiv{math.QA/9908171}.

\bibitem{MR1928174}
Mikhail Khovanov.
\newblock A functor-valued invariant of tangles.
\newblock {\em Algebr. Geom. Topol.}, 2:665--741, 2002.
\newblock \arxiv{math.GT/0103190}.

\bibitem{Kho3}
M.~Khovanov$\phantom{a}\!\!\!$.
\newblock $\mathfrak{sl}(3)$ link homology.
\newblock {\em Algebr. Geom. Topol.}, 4:1045--1081, 2004.
\newblock \arxiv{math.QA/0304375}.

\bibitem{SN-TV}
Alexander Kirillov, Jr.
\newblock String-net model of {T}uraev-{V}iro invariants.
\newblock \arxiv{1106.6033}.

\bibitem{lipshitz2023floer}
Robert Lipshitz, Peter Ozsv{\'a}th, and Dylan Thurston.
\newblock Floer homology beyond borders, 2023.
\newblock \arxiv{2307.16330}.

\bibitem{MR3827056}
Robert Lipshitz, Peter~S. Ozsvath, and Dylan~P. Thurston.
\newblock Bordered {H}eegaard {F}loer homology.
\newblock {\em Mem. Amer. Math. Soc.}, 254(1216):viii+279, 2018.
\newblock \mathscinet{MR3827056} \doi{10.1090/memo/1216} \arxiv{0810.0687}.

\bibitem{liu2024braided}
Yu~Leon Liu, Aaron Mazel-Gee, David Reutter, Catharina Stroppel, and Paul
  Wedrich.
\newblock A braided monoidal $(\infty,2)$-category of {S}oergel bimodules,
  2024.
\newblock \arxiv{2401.02956}.

\bibitem{MR2555928}
Jacob Lurie.
\newblock On the classification of topological field theories.
\newblock In {\em Current developments in mathematics, 2008}, pages 129--280.
  Int. Press, Somerville, MA, 2009.
\newblock \arxiv{0905.0465}.

\bibitem{MSV}
M.~Mackaay, M.~Sto{\v{s}}i{\'c}, and P.~Vaz.
\newblock {$\mathfrak{sl}_{N}$}-link homology {$(N\geq 4)$} using foams and the
  {K}apustin--{L}i formula.
\newblock {\em Geom. Topol.}, 13(2):1075--1128, 2009.
\newblock \arxiv{0708.2228}.

\bibitem{manion2020higher}
Andrew Manion and Raphael Rouquier.
\newblock Higher representations and cornered {H}eegaard {F}loer homology,
  2020.
\newblock \arxiv{2009.09627}.

\bibitem{2020arXiv200908520M}
Ciprian Manolescu and Ikshu Neithalath.
\newblock Skein lasagna modules for 2-handlebodies.
\newblock {\em J. Reine Angew. Math.}, 788:37--76, 2022.
\newblock \mathscinet{MR4445546} \doi{10.1515/crelle-2022-0021}
  \arxiv{2009.08520}.

\bibitem{2206.04616}
Ciprian Manolescu, Kevin Walker, and Paul Wedrich.
\newblock Skein lasagna modules and handle decompositions, 2022.
\newblock \arxiv{2206.04616}.

\bibitem{MR2387432}
Julien March\'{e} and Majid Narimannejad.
\newblock Some asymptotics of topological quantum field theory via skein
  theory.
\newblock {\em Duke Math. J.}, 141(3):573--587, 2008.

\bibitem{Markl}
M.~Markl.
\newblock Ideal perturbation lemma.
\newblock {\em Comm. Algebra}, 29(11):5209--5232, 2001.
\newblock \arxiv{math/0002130}.

\bibitem{2019arXiv190712194M}
Scott Morrison, Kevin Walker, and Paul Wedrich.
\newblock Invariants of 4-manifolds from {K}hovanov-{R}ozansky link homology.
\newblock {\em Geom. Topol.}, 26(8):3367--3420, 2022.
\newblock \arxiv{1907.12194}.

\bibitem{QR}
H.~Queffelec and D.E.V. Rose.
\newblock The $\mathfrak{sl}_n$ foam $2$-category: a combinatorial formulation
  of {K}hovanov--{R}ozansky homology via categorical skew {H}owe duality.
\newblock {\em Adv. Math.}, 302:1251--1339, 2016.
\newblock \arxiv{1405.5920}.

\bibitem{2209.08794}
Hoel Queffelec.
\newblock Gl2 foam functoriality and skein positivity, 2022.
\newblock \arxiv{2209.08794}.

\bibitem{QR2}
Hoel Queffelec and David Rose.
\newblock Sutured annular {K}hovanov-{R}ozansky homology.
\newblock {\em Transactions of the American Mathematical Society},
  370(2):1285--1319, 2018.
\newblock \arxiv{1506.08188}.

\bibitem{MR3934851}
Hoel Queffelec and Paul Wedrich.
\newblock Extremal weight projectors.
\newblock {\em Math. Res. Lett.}, 25(6):1911--1936, 2018.
\newblock \arxiv{1701.02316}.

\bibitem{1806.03416}
Hoel {Queffelec} and Paul {Wedrich}.
\newblock {Khovanov homology and categorification of skein modules}.
\newblock {\em Quantum Topol.}, 12(1):129--209, 2021.
\newblock \arxiv{1806.03416}.

\bibitem{ren2024khovanov}
Qiuyu Ren and Michael Willis.
\newblock Khovanov homology and exotic $4$-manifolds, 2024.
\newblock \arxiv{2402.10452}.

\bibitem{MR1091619}
Nicolai Reshetikhin and Vladimir~G. Turaev.
\newblock Invariants of {$3$}-manifolds via link polynomials and quantum
  groups.
\newblock {\em Invent. Math.}, 103(3):547--597, 1991.

\bibitem{reutter2020semisimple}
David Reutter.
\newblock Semisimple 4-dimensional topological field theories cannot detect
  exotic smooth structure, 2020.
\newblock \arxiv{2001.02288}.

\bibitem{RoW}
Louis-Hadrien Robert and Emmanuel Wagner.
\newblock A closed formula for the evaluation of foams.
\newblock {\em Quantum Topol.}, 11(3):411--487, 2020.
\newblock \arxiv{1702.04140}.

\bibitem{Roz}
L.~Rozansky.
\newblock An infinite torus braid yields a categorified {J}ones--{W}enzl
  projector.
\newblock {\em Fund. Math.}, 225(1):305--326, 2014.

\bibitem{rozansky2010categorification}
Lev Rozansky.
\newblock A categorification of the stable {SU}(2)
  {W}itten-{R}eshetikhin-{T}uraev invariant of links in {S2 x S1}, 2010.
\newblock \arxiv{1011.1958}.

\bibitem{Russell}
Heather Russell.
\newblock The {B}ar-{N}atan skein module of the solid torus and the homology of
  $(n,n)$ {S}pringer varieties.
\newblock {\em Geom. Dedicata}, 2009.

\bibitem{stroppel2022categorification}
Catharina Stroppel.
\newblock Categorification: tangle invariants and tqfts, 2022.
\newblock \arxiv{2207.05139}.

\bibitem{sullivan2024kirby}
Ian~A. Sullivan and Melissa Zhang.
\newblock Kirby belts, categorified projectors, and the skein lasagna module of
  $s^{2}\times{S^{2}}$, 2024.
\newblock \arxiv{2402.01081}.

\bibitem{Tabuada}
G.~Tabuada.
\newblock A {Q}uillen model structure on the category of dg categories.
\newblock {\em C. R. Math. Acad. Sci. Paris}, 340:15--19, 2005.

\bibitem{TuraevViro}
V.~G. Turaev and O.~Ya. Viro.
\newblock State sum invariants of {$3$}-manifolds and quantum {$6j$}-symbols.
\newblock {\em Topology}, 31(4):865--902, 1992.

\bibitem{MR2253441}
Vladimir Turaev and Paul Turner.
\newblock Unoriented topological quantum field theory and link homology.
\newblock {\em Algebr. Geom. Topol.}, 6:1069--1093, 2006.

\bibitem{MR3617439}
Vladimir~G. Turaev.
\newblock {\em Quantum invariants of knots and 3-manifolds}, volume~18 of {\em
  De Gruyter Studies in Mathematics}.
\newblock De Gruyter, Berlin, 2016.
\newblock Third edition [of MR1292673].

\bibitem{MR4332675}
Michael Willis.
\newblock Khovanov homology for links in {$\#^r(S^2\times S^1)$}.
\newblock {\em Michigan Math. J.}, 70(4):675--748, 2021.

\bibitem{MR2941830}
Rumen Zarev.
\newblock {\em Bordered {S}utured {F}loer {H}omology}.
\newblock ProQuest LLC, Ann Arbor, MI, 2011.
\newblock Thesis (Ph.D.)--Columbia University.

\end{thebibliography}

\end{document}